\documentclass[a4paper,draft]{amsart}
\usepackage{amssymb,verbatim,dsfont,accents}
\usepackage{xypic}
\xyoption{all}
\usepackage{verbatim}

\DeclareMathAlphabet{\mathpzc}{OT1}{pzc}{m}{it}

\newtheorem{theorem}{Theorem}[section]
\newtheorem{lemma}[theorem]{Lemma}
\newtheorem{proposition}[theorem]{Proposition}
\newtheorem{corollary}[theorem]{Corollary}

\theoremstyle{definition}
\newtheorem{definition}[theorem]{Definition}
\newtheorem{notations}[theorem]{Notations}
\newtheorem{notation}[theorem]{Notation}
\newtheorem{example}[theorem]{Example}

\theoremstyle{remark}
\newtheorem{remark}[theorem]{Remark}

\DeclareMathOperator{\BC}{BC}

\DeclareMathOperator{\BN}{BN}

\DeclareMathOperator{\BP}{BP}

\DeclareMathOperator{\cc}{c}

\DeclareMathOperator{\gc}{gc}

\DeclareMathOperator{\Ho}{H}

\DeclareMathOperator{\HH}{HH}

\DeclareMathOperator{\HC}{HC}

\DeclareMathOperator{\HP}{HP}

\DeclareMathOperator{\HN}{HN}

\DeclareMathOperator{\Hom}{Hom}

\DeclareMathOperator{\ide}{id}

\DeclareMathOperator{\End}{End}

\DeclareMathOperator{\pc}{pc}

\DeclareMathOperator{\ps}{p\wh{s}}

\DeclareMathOperator{\s}{s}

\DeclareMathOperator{\Sh}{Sh}

\DeclareMathOperator{\Tot}{Tot}

\newcommand{\xcirc}{\hspace{-0.5pt} \circ \hspace{-0.5pt}}
\newcommand{\ot}{\otimes}

\newcommand{\sub}{\subseteq}
\newcommand{\wt}{\widetilde}
\newcommand{\wh}{\widehat}
\newcommand{\ov}{\overline}

\newcommand{\ba}{\mathbf a}
\newcommand{\bx}{\mathbf x}
\newcommand{\bc}{\mathbf c}

\newcommand{\bh}{\mathbf h}
\newcommand{\bv}{\mathbf v}
\newcommand{\byy}{\mathbf y}
\newcommand{\bz}{\mathbf z}

\newcommand{\de}{\delta}
\newcommand{\De}{\Delta}
\newcommand{\ep}{\epsilon}
\newcommand{\si}{\sigma}
\newcommand{\cX}{\mathcal X}
\newcommand{\cY}{\mathcal Y}

\newcommand{\nesp}{\hspace{-0.3pt}}

\newcommand{\nespp}{\hspace{-0.5pt}}

\begin{document}

\title[Cyclic homology of Brzezi\'nski's crossed products]{Cyclic homology of Brzezi\'nski's crossed products and of braided Hopf crossed products}

\author[G. Carboni]{Graciela Carboni}
\address{C\'\i clo B\'asico Com\'un\\ Universidad de Buenos Aires\\ Ciudad Universitaria-Pabell\'on~3\\ (C1428EGA) Buenos Aires, Argentina.}
\curraddr{}
\email{gcarbo@dm.uba.ar}

\author[J. A. Guccione]{Jorge A. Guccione}
\address{Departamento de Matem\'atica\\ Facultad de Ciencias Exactas y Naturales\\ Universidad de Buenos Aires\\ Ciudad Universitaria-Pabell\'on~1\\ (C1428EGA) Buenos Aires, Argentina.}
\email{jjgucci@dm.uba.ar}

\author[J. J. Guccione]{Juan J. Guccione}
\address{Departamento de Matem\'atica\\ Facultad de Ciencias Exactas y Naturales\\ Universidad de Buenos Aires\\ Ciudad Universitaria-Pabell\'on~1\\ (C1428EGA) Buenos Aires, Argentina.}
\email{jjgucci@dm.uba.ar}

\thanks{The research of G. Carboni, J. A. Guccione and J.J Guccione was supported by  UBACYT 095 and PIP 112-200801-00900 (CONICET)}

\author[C. Valqui]{Christian Valqui}
\address{Pontificia Universidad Cat\'olica del Per\'u - Instituto de Matem\'atica y Ciencias Afi\-nes, Secci\'on Matem\'aticas, PUCP, Av. Universitaria 1801, San Miguel, Lima 32, Per\'u.}
\email{cvalqui@pucp.edu.pe}
\thanks{The research of C. Valqui was supported by PUCP-DAI-2009-0042, Lucet 90-DAI-L005 and SFB 478 U. M\"unster, Konrad Adenauer Stiftung.}

\begin{abstract} Let $k$ be a field, $A$ a unitary associative $k$-algebra and $V$ a $k$-vector space endowed with a distinguished element $1_V$. We obtain a mixed complex, simpler that the canonical one, that gives the Hochschild, cyclic, negative and periodic homology of a crossed product $E:=A\#_f V$, in the sense of Brzezi\'nski. We actually work in the more general context of relative cyclic homology. Specifically, we consider a subalgebra $K$ of $A$ that satisfies suitable hypothesis and we find a mixed complex computing the Hochschild, cyclic, negative and periodic homology of $E$ relative to $K$. Then, when $E$ is a cleft braided Hopf crossed product, we obtain a simpler mixed complex, that also gives the Hochschild, cyclic, negative and periodic homology of $E$.
\end{abstract}

\subjclass[2010]{primary 16E40; secondary 16T05}

\keywords{Crossed products, Hochschild (co)homology, Cyclic homology}

\dedicatory{}

\maketitle

\section*{Introduction}
The problem of develop tools to compute the cyclic homology of smash products algebras $A\# k[G]$, where $G$ is a group, was considered in \cite{F-T}, \cite{N} and \cite{G-J}. For instance, in the first paper it was obtained a spectral sequence converging to the cyclic homology of $A\# k[G]$. In \cite{G-J}, this result was derived from the theory of paracyclic modules and cylindrical modules developed by the authors. The main tool for this computation was a version for cylindrical modules of Eilenberg-Zilber theorem. In \cite{A-K} this theory was used to obtain a Feigin-Tsygan type spectral sequence for smash products $A\# H$, of a Hopf algebra $H$ with an $H$-module algebra~$A$.

It is natural to try to extend this result to the general crossed products $A\#_f H$ introduced in \cite{B-C-M} and \cite{D-T}, and to more general algebras such as Hopf Galois extensions. In \cite{J-S} the relative to $A$ cyclic homology of a Galois $H$ extension $C/A$ was studied, and the results obtained was applied to the Hopf crossed products $A\#_f H$, giving the absolute cyclic homology when $A$ is a separable algebra. As far as we know, \cite{K-R} was the first work dealing with the absolute cyclic homology of a crossed product $A\#_f H$, with $A$ non separable and $f$ non trivial. In that paper the authors  get a Feigin-Tsygan type spectral sequence for a crossed products $A\#_f H$, under the hypothesis that $H$ is cocommutative and $f$ takes values in $k$. Finally, the main results established in~\cite{K-R} were extended in~\cite{C-G-G} to the general Hopf crossed products $A\#_f H$ introduced in~\cite{B-C-M} and~\cite{D-T}. In particular were constructed two spectral sequences converging to the cyclic homology of $A\#_f H$. The second one, which is valid  under the hypothesis that $f$ takes values in $k$, generalize those obtained in~\cite{A-K} and~\cite{K-R}.

Let $k$ be a field. An associative and unital $k$-algebra $E$ is an {\em smash product} of two associative and unital algebras $A$ and $B$ if the underlying vector space of $E$ is $A\ot_k B$, the maps
$$
\xymatrix @R=-2pt {A\rto & E\\
a \ar@{|->}[0,1] & a\ot_k 1}
\quad
\xymatrix @R=-2pt{\save[]+<0pc,0pc> \Drop{}\txt{and}\restore}
\quad
\xymatrix @R=-2pt {A\rto & E\\
b \ar@{|->}[0,1] & 1\ot_k b}
$$
are morphisms of algebras and $(a\ot_k 1)(1\ot_k b)=a\ot_k b$ for all $a\in A$ and $b\in B$. Let $R\colon B\ot_k A \to A\ot_k B$ be the map defined by $R(b\ot_k a):=(1\ot_k b)(a\ot_k 1)$. It is evident that each smash product $E$ is complete determined by the map $R$. This justify the notation $A\#_R B$ for $E$. An smash product $A\#_R B$ is called {\em strong} if $R$ is bijective. A different generalization of the results established in~\cite{A-K} was obtained in~\cite{Z-H}, where it was found a mixed complex, simpler than the canonical one, that computes the type cyclic homology groups of a strong smash product algebra. Using this the authors construct a spectral sequence that converges to the cyclic homology of $A\#_R B$. The Hochschild (co)homology of strong smash products was studied in~\cite{G-G1}.

Let $V$ be a $k$-vector space endowed with a distinguished element $1$ and $A$ an associative and unital $k$-algebra. We say that an algebra $E$ with underlying vector space $A\ot_k V$ is a Brzezi\'nski's crossed product of $A$ with $V$ if it is associative with unit $1\ot_k 1$, the map
$$
\xymatrix @R=-2pt {A\rto & E\\
a \ar@{|->}[0,1] & a\ot_k 1},
$$
is a morphism of algebras and the left $A$-module structure of $A\ot_k V$ induced by this map is the canonical one. Brzezi\'nski's crossed products are a wide generalization of Hopf crossed products and smash products of algebras (the relation between smash products and Brzezi\'nski's crossed products of algebras is analogous to the relation between group smash products and group crossed products). The goal of this work is to present a mixed complex $\bigl(\wh{X}_*,\wh{d}_*,\wh{D}_*\bigr)$, simpler than the canonical one, that gives the Hochschild, cyclic, negative and periodic homologies of a Brzezi\'nski's crossed product of $A$ with $V$. This result generalizes
the main results of~\cite[Section~2]{C-G-G} and~\cite{Z-H}. Moreover in this case our complex also works when the smash product is not strong. We actually work in the more general context of relative cyclic homology. Specifically, we consider a subalgebra $K$ of $A$ that satisfies suitable conditions, and we find a mixed complex computing the Hochschild, cyclic, negative and periodic homology groups of $E$ relative to $K$ (which we simply call the Hochschild, cyclic, negative and periodic homology groups of the $K$-algebra $E$). Of course, when $K$ is separable, this gives the absolute homologies. Our main result is Theorem~\ref{complejo mezclado que da la homologia ciclica}, in which is proved that $\bigl(\wh{X}_*,\wh{d}_*, \wh{D}_*\bigr)$ is homotopically equivalent to the canonical normalized mixed complex of $E$. As an application we obtain four spectral sequences converging to the cyclic homology of the $K$-algebra $E$. The first one generalizes those given in~\cite[Section~3.1]{C-G-G} and~\cite[Theorem~4.7]{Z-H}, and the third one those of~\cite{A-K}, \cite{K-R} and~\cite[Section 3.2]{C-G-G}. As far as we know, the results of the core of this paper (Sections~3, 4, 5 and 6) applies to all the up to date existent types of crossed products of algebras with braided Hopf algebras, in particular to the underlying algebras of the crossed product bialgebras  considered in~\cite{B-D} and to the $L$-$R$ smash products introduced in~\cite{B-G-G-S} and~\cite{B-S}. In sections~7, 8, 9, 10 and 11 we consider the cleft braided Hopf crossed products introduced in~\cite{G-G2}. The main result of these sections is that when $E$ is a cleft braided Hopf crossed product, $\bigl(\wh{X}_*,\wh{d}_*,\wh{D}_*\bigr)$ is isomorphic to a simpler mixed complex $\bigl(\ov{X}_*,\ov{d}_*,\ov{D}_*\bigr)$.

\smallskip

Our method of proof is different from that used in \cite{G-J}, \cite{A-K}, \cite{K-R} and \cite{Z-H}, since they are based in the results obtained in \cite{G-G1} and the Perturbation Lemma instead of a generalization of the Eilenberg-Zilber theorem.

\smallskip

Finally we want to point out that in this paper we also study the Hochschild homology and cohomology of $E$ with coefficients in an arbitrary $E$-bimodule $M$. More precisely, we obtain complexes, simpler that the canonical ones, that compute the Hochschild homology and cohomology of $E$ with coefficients in $M$. Using them we get spectral sequences that generalize the Hochschild-Serre spectral sequences (\cite{H-S}), and we get some result about the cup product of the Hochschild cohomology of $E$ and cap product of the Hochschild homology of $E$ with coefficients in $M$.

\section{preliminaries}
In this article we work in the category of vector spaces over a field $k$. Then we assume implicitly that all the maps are $k$-linear maps. The tensor product over $k$ is denoted by $\ot_k$. Given a $k$-vector space $V$ and $n\ge 1$, sometimes we let $V^{\ot_k^n}$ denote the $n$-fold tensor product $V\ot_k\cdots\ot_k V$. Given $k$-vector spaces $U,V,W$ and a map $f\colon V\to W$ we write $U\ot_k f$ for  $\ide_U\ot_k f$ and $f\ot_k U$ for $f\ot_k
\ide_U$. We assume that the reader is familiar with the notions of algebra, coalgebra, module and comodule. Unless otherwise explicitly established we assume that the algebras are associative unitary and the coalgebras are coassociative counitary. Given an algebra $A$ and a coalgebra $C$, we let
$$
\mu\colon A\ot_k A \to A,\quad \eta\colon k \to A,\quad \Delta\colon C\to C\ot_k C\quad \text{and}\quad \ep\colon C\to k
$$
denote the multiplication, the unit, the comultiplication and the counit, respectively, specified with a subscript if necessary. Moreover, given $k$-vector spaces $V$ and $W$, we let $\tau\colon V\ot_k W\to W\ot_k V$ denote the flip $\tau(v\ot_k w)= w\ot_k v$.

\smallskip

In this article we use the nowadays well known graphic calculus for monoidal and braided categories. As usual, morphisms will be composed from up to down and tensor products will be represented by horizontal concatenation in the corresponding order. The identity map of a $k$-vector space will be represented by a vertical line, and the flip by the diagram
\begin{equation}\label{s0}
\spreaddiagramcolumns{-1.6pc}\spreaddiagramrows{-1.6pc}
\objectmargin{0.0pc}\objectwidth{0.0pc}
\def\objectstyle{\sssize}
\def\labelstyle{\sssize}
\grow{\xymatrix{
               \ar@{-}[2,2]&&\ar@{-}[2,-2]\\
               &&\\
               &&
}}
\grow{\xymatrix{{}\save[]+<0pc,-0.4pc>*\txt{.} \restore}}
\end{equation}
Given an algebra $A$, the diagrams
\begin{equation}\label{s1}
\spreaddiagramcolumns{-1.6pc}\spreaddiagramrows{-1.6pc}
\objectmargin{0.0pc}\objectwidth{0.0pc}
\def\objectstyle{\sssize}
\def\labelstyle{\sssize}
\grow{\xymatrix{
               \ar@{-}`d/4pt [1,1] `[0,2] [0,2] &&\\
               &\ar@{-}[1,0]\\
               &
}}
\grow{\xymatrix{\\
\txt{,}
}}\quad
\grow{\xymatrix{
             \circ\save\go+<0pt,-1pt>\Drop{}\ar@{-}[2,0]+<0pt,3.5pt>\restore\\\\&
}}
\grow{\xymatrix{\\
\txt{,}
}}\quad
\grow{\xymatrix{
               \ar@{-}`d/4pt [1,2][1,2] && \ar@{-}[2,0]\\
               &&\\
               &&
}}
\quad\grow{\xymatrix{{}\save[]+<0pc,-0.4pc>*\txt{and}
\restore}}\quad
\grow{\xymatrix{
               \ar@{-}[2,0] && \ar@{-}`d/4pt [1,-2][1,-2]\\\\
               &
}}
\end{equation}
stand for the multiplication map, the unit, the action of $A$ on a left $A$-module and the action of $A$ on a right $A$-module, respectively. Given a coalgebra $C$, the comultiplication, the counit, the coaction of $C$ on a right $C$-comodule and the coaction of $C$ on a left $C$-comodule will be represented by the diagrams
\begin{equation}\label{s2}
\spreaddiagramcolumns{-1.6pc}\spreaddiagramrows{-1.6pc}
\objectmargin{0.0pc}\objectwidth{0.0pc}
\def\objectstyle{\sssize}
\def\labelstyle{\sssize}
\grow{\xymatrix{
              &\ar@{-}[1,0]\\
              & \ar@{-}`l/4pt [1,-1] [1,-1] \ar@{-}`r [1,1] [1,1]\\
              &&
}}
\grow{\xymatrix{\\
\txt{,}
}}\quad
\grow{\xymatrix{
              \ar@{-}[2,0]+<0pt,2pt>\\\\
              \save\go+<0pt,2pt>\Drop{\circ}\restore
              &
}}
\grow{\xymatrix{\\
\txt{,}
}}\quad
\grow{\xymatrix{
               \ar@{-}[2,0] \\
               \ar@{-}`r/4pt [1,2] [1,2] \\
               &&
}}
\quad\grow{\xymatrix{{}\save[]+<0pc,-0.4pc>*\txt{and}
\restore}}\quad
\grow{\xymatrix{
               &&\ar@{-}[2,0]\\
               &&\ar@{-}`l/4pt [1,-2] [1,-2]\\
               &&
}}
\grow{\xymatrix{\\
\txt{,}
}}
\end{equation}
respectively.

Consider a $k$-linear map $c\colon V\ot_k W \to W\ot_k V$. If $V$ is an algebra, then we say that {\em $c$ is compatible with the algebra structure of $V$} if
$$
c \xcirc (\eta\ot_k W) = W\ot_k \eta\quad\text{and} \quad c \xcirc (\mu\ot_k W)=  (W\ot_k \mu)\xcirc(c\ot_k V)\xcirc(V\ot_k c).
$$
If $V$ is a coalgebra, then we say that $c$ is {\em compatible with the coalgebra structure of $V$} if
$$
(W\ot_k \epsilon)\xcirc c = \epsilon\ot_k W\quad\text{and}\quad (W\ot_k \Delta) \xcirc c = (c\ot_k V)\xcirc (V\ot_k c)\xcirc (\Delta\ot_k W).
$$
Finally, if $W$ is an algebra or a coalgebra, then we introduce the notion that $c$ is compatible with the structure of $W$ in the obvious way.

\subsection{Brzezi\'nski's crossed products}
In this subsection we recall a very general definition of crossed product, introduced in \cite{Br}, and its basic properties. For the proofs we refer to \cite{Br} and \cite{B-D}. Throughout this paper $A$ is a unitary algebra and $V$ is a $k$-vector space equipped with a distinguished element $1\in V$.

\begin{definition} Given maps $\chi\colon V\ot_k A\to A\ot_k V$ and $\mathcal{F}\colon V\ot_k V\to A\ot_k V$, we let $A\# V$ denote the algebra (in general non associative and non unitary) whose underlying $k$-vector space is $A\ot_k V$ and whose multiplication map is given by
$$
\mu_{A\# V}:= (\mu_{\!A}\ot_k V)\xcirc(\mu_{\!A}\ot_k \mathcal{F})\xcirc(A\ot_k \chi \ot_k V).
$$
The element $a\ot_k v$ of $A\# V$ will usually be written $a\# v$. The algebra $A\# V$ is called a {\em crossed product} if it is associative with $1\# 1$ as identity.
\end{definition}

\begin{definition}\label{cociclo} Let $\chi\colon V\ot_k A\to A\ot_k V$ and $\mathcal{F}\colon V\ot_k V\to A\ot_k V$ be maps.

\smallskip

\begin{enumerate}

\item $\chi$ is a {\em twisting map} if it is compatible with the algebra structure of $A$ and $\chi(1\ot_k a) = a\ot_k 1$.

\smallskip

\item $\mathcal{F}$ is {\em normal} if $\mathcal{F}(1\ot_k v) = \mathcal{F}(v\ot_k 1) = 1\ot_k v$.

\smallskip

\item $\mathcal{F}$ is a {\em cocycle that satisfies the twisted module condition} if
$$
\spreaddiagramcolumns{-1.6pc}\spreaddiagramrows{-1.6pc}
\objectmargin{0.0pc}\objectwidth{0.0pc}
\def\objectstyle{\sssize}
\def\labelstyle{\sssize}
\qquad\qquad\grow{\xymatrix@!0{
     \ar@{-}[2,0] &&& \ar@{-}@<0.40pc>[1,0] \ar@{-}@<-0.40pc>[1,0]&\\
     &&&*+<0.1pc>[F]{\,\,\mathcal{F}\,\,} \ar@{-}@<0.40pc>[3,0] \ar@{-}@<-0.40pc>[1,0] &\\
     \ar@/^0.1pc/ @{-}[2,2] \ar@/_0.1pc/ @{-}[2,2]&& \ar@/^0.1pc/ @{-}[2,-2]
                                                  \ar@/_0.1pc/ @{-}[2,-2]&& \ar@{-}[2,0]\\
     \\
     \ar@{-}[2,0]&&& \ar@{-}@<0.40pc>[1,0] \ar@{-}@<-0.40pc>[1,0]&\\
     &&&*+<0.1pc>[F]{\,\,\,\mathcal{F}\,\,\,} \ar@{-}@<0.40pc>[3,0] \ar@{-}@<-0.40pc>[1,0] &\\
     \ar@{-}`d/4pt [1,1] `[0,2] [0,2] &&&& \ar@{-}[1,0]\\
     &\ar@{-}[1,0]&&&\ar@{-}[1,0]\\
     &&&&
     }}
\grow{\xymatrix{
               \\\\\\
               {}\save[]+<0pc,-0.4pc>*\txt{$=$} \restore
}}
\grow{\xymatrix@!0{
     \\
     &\ar@{-}@<0.40pc>[1,0] \ar@{-}@<-0.40pc>[1,0]&&&\ar@{-}[2,0]\\
     &*+<0.1pc>[F]{\,\,\,\mathcal{F}\,\,\,} \ar@{-}@<0.40pc>[1,0] \ar@{-}@<-0.40pc>[1,0] &&&\\
     \ar@{-}[2,0]&&& \ar@{-}@<0.40pc>[1,0] \ar@{-}@<-0.40pc>[1,0]&\\
     &&&*+<0.1pc>[F]{\,\,\,\mathcal{F}\,\,\,} \ar@{-}@<0.40pc>[3,0] \ar@{-}@<-0.40pc>[1,0] &\\
     \ar@{-}`d/4pt [1,1] `[0,2] [0,2] &&&& \ar@{-}[1,0]\\
     &\ar@{-}[1,0]&&&\ar@{-}[1,0]\\
     &&&&
     }}
\quad\grow{\xymatrix{
               \\\\\\
               {}\save[]+<0pc,-0.4pc>*\txt{and} \restore
}}\quad
\grow{\xymatrix@!0{
     \ar@{-}[2,0] &&  \ar@/^0.1pc/ @{-}[2,2] \ar@/_0.1pc/ @{-}[2,2]&& \ar@/^0.1pc/ @{-}[2,-2]
                                                  \ar@/_0.1pc/ @{-}[2,-2]\\
     \\
     \ar@/^0.1pc/ @{-}[2,2] \ar@/_0.1pc/ @{-}[2,2]&& \ar@/^0.1pc/ @{-}[2,-2]
                                                  \ar@/_0.1pc/ @{-}[2,-2]&& \ar@{-}[2,0]\\
     \\
     \ar@{-}[2,0]&&& \ar@{-}@<0.40pc>[1,0] \ar@{-}@<-0.40pc>[1,0]&\\
     &&&*+<0.1pc>[F]{\,\,\,\mathcal{F}\,\,\,} \ar@{-}@<0.40pc>[3,0] \ar@{-}@<-0.40pc>[1,0] &\\
     \ar@{-}`d/4pt [1,1] `[0,2] [0,2] &&&& \ar@{-}[1,0]\\
     &\ar@{-}[1,0]&&&\ar@{-}[1,0]\\
     &&&&
     }}
\grow{\xymatrix{
               \\\\\\
               {}\save[]+<0pc,-0.4pc>*\txt{$=$} \restore
}}
\grow{\xymatrix@!0{
     \\
     &\ar@{-}@<0.40pc>[1,0] \ar@{-}@<-0.40pc>[1,0]&&&\ar@{-}[2,0]\\
     &*+<0.1pc>[F]{\,\,\,\mathcal{F}\,\,\,} \ar@{-}@<0.40pc>[1,0] \ar@{-}@<-0.40pc>[1,0] &&&\\
     \ar@{-}[2,0] &&  \ar@/^0.1pc/ @{-}[2,2] \ar@/_0.1pc/ @{-}[2,2]&& \ar@/^0.1pc/ @{-}[2,-2]
                                                  \ar@/_0.1pc/ @{-}[2,-2]\\
     \\
     \ar@{-}`d/4pt [1,1] `[0,2] [0,2] &&&& \ar@{-}[2,0]\\
     &\ar@{-}[1,0]&&&\\
     &&&&
     }}
\grow{\xymatrix{
               \\\\\\
               {}\save[]+<0pc,-0.4pc>*\txt{,} \restore
}}
\quad\grow{\xymatrix{
               \\\\\\
               {}\save[]+<0pc,-0.4pc>*\txt{where } \restore
}}
\grow{\xymatrix{
               \\
               \\
               \\
     \ar@/^0.1pc/ @{-}[2,2] \ar@/_0.1pc/ @{-}[2,2]&& \ar@/^0.1pc/ @{-}[2,-2]
                                                  \ar@/_0.1pc/ @{-}[2,-2]\\
     \\
&& }}
\grow{\xymatrix{
               \\\\\\
               {}\save[]+<0pc,-0.4pc>*\txt{$=\chi$ and} \restore
}}
\grow{\xymatrix@!0{
     \\
     \\
     \\
     &\ar@{-}@<0.40pc>[1,0] \ar@{-}@<-0.40pc>[1,0]&\\
     &*+<0.1pc>[F]{\,\,\,\mathcal{F}\,\,\,} \ar@{-}@<0.40pc>[1,0] \ar@{-}@<-0.40pc>[1,0] &\\
     &&
}}
\grow{\xymatrix{
               \\\\\\
               {}\save[]+<0pc,-0.4pc>*\txt{$= \mathcal{F}$.} \restore
}}
$$
More precisely, the first equality says that $\mathcal{F}$ is a cocycle and the second one says that $\mathcal{F}$ satisfies the twisted module condition.

\end{enumerate}

\end{definition}

\begin{theorem}[Brzezi\'nski]  The algebra $A\# V$ is a crossed product is and only if $\chi$ is a twisting map and $\mathcal{F}$ is a normal cocycle that satisfies the twisted module condition.
\end{theorem}

Note that the multiplication of a crossed product have the following property:
\begin{equation}
(a\# 1)(b\# v) = ab\# v.\label{eq1}
\end{equation}
In particular $a\mapsto a\# 1$ is an injective morphism of $k$-algebras. We consider $A$ as a subalgebra of $A\# V$ via this map. Conversely, each $k$-algebra with underlying vector space $A\ot_k V$, whose multiplication map satisfies~\eqref{eq1}, is a crossed product. The twisting map $\chi$ and the cocycle $\mathcal{F}$ are given by
$$
\chi(v\ot_k a) = (1\# v)(a\# 1)\quad\text{and}\quad \mathcal{F}(v\ot_k w) = (1\# v)(1\# w).
$$

\begin{definition} Let $A\# V$ be a crossed product with associated twisting map $\chi$ and cocycle $\mathcal{F}$, and let $R$ be a subalgebra of $A$. We say that:

\begin{itemize}

\smallskip

\item[-] $R$ is {\em stable under $\chi$} if $\chi(V\ot_k R)\subseteq R\ot_k V$.

\smallskip

\item[-] {\em $\mathcal{F}$ takes its values in $R\ot_k V$} if $\mathcal{F}(V\ot_k V) \subseteq R\ot_k V$.

\end{itemize}

\end{definition}

\subsection{Braided Hopf crossed products}
Braided bialgebras and braided Hopf algebras were introduced by Majid (see his survey~\cite{M}). In this subsection, we make a quick review of this subject following the intrinsic presentation given by Takeuchi in~\cite{T}. Then, we review the concept of braided Hopf crossed products introduced in~\cite{G-G2}. Let $V$ be a $k$-vector space. Recall that a map $c\in \End_k(V^{\ot_k^2})$ is called a braiding operator of $V$ if it satisfies the equality
$$
(c\ot_k V)\xcirc (V\ot_k c)\xcirc (c\ot_k V) = (V\ot_k c)\xcirc (c\ot_k V)\xcirc (V\ot_k c).
$$

\begin{definition} A {\em braided bialgebra} is a $k$-vector space $H$, endowed with an algebra structure, a coalgebra structure and a bijective braiding operator $c$ of $H$, called the braid of $H$, such that: $c$ is compatible with the algebra and coalgebra structures of $H$, $\eta$ is a coalgebra morphism, $\ep$ is an algebra morphism and
$$
\De\xcirc\mu = (\mu\ot_k \mu)\xcirc(H\ot_k c \ot_k H)\xcirc(\De \ot_k \De).
$$
Moreover, if there exists a map $S\colon H\to H$, which is the convolution inverse of the identity map, then we say that $H$ is a {\em braided Hopf algebra} and we call $S$ the {\em antipode} of $H$.
\end{definition}

Usually $H$ denotes a braided bialgebra, understanding the structure maps, and $c$ denotes its braid.

\begin{definition}\label{transposition} Let $H$ be a braided bialgebra and $A$ an algebra. A  {\em transposition of $H$ on $A$} is a bijective twisting map $s\colon H\ot_k A \to A\ot_k H$ which is compatible with bialgebra structure of $H$. That is, $s$ is a twisting map that satisfies the equation
$$
(s\ot_k H)\xcirc (H\ot_k s)\xcirc (c\ot_k A) = (A\ot_k c)\xcirc (s\ot_k H)\xcirc (H\ot_k s)\quad\text{ (compatibility of $s$ with $c$)}
$$
and it is compatible with the algebra and coalgebra structures of $H$.
\end{definition}

\begin{remark}\label{transposicion inversa} It is easy to see that if $s$ is a transposition then $s^{-1}$ is compatible with the algebra and coalgebra structures of $H$, with the algebra structure of $A$ and that
$$
(H\ot_k s^{-1})\xcirc (s^{-1}\ot_k H )\xcirc (A\ot_k c^{-1}) = (c^{-1}\ot_k H )\xcirc (H\ot_k s^{-1}) \xcirc (s^{-1}\ot_k H).
$$
\end{remark}

\begin{definition}\label{accion debil} Let $s\colon H\ot_k A\to A\ot_k H$ be a transposition. A {\em weak $s$-action of $H$ on $A$} is a map $\rho\colon H\ot_k A\to A$, that satisfies:

\begin{enumerate}

\smallskip

\item $\rho\xcirc (H\ot_k\mu) =\mu\xcirc(\rho\ot_k\rho)\xcirc (H\ot_k s\ot_k A)\xcirc (\De\ot_k A\ot_k A)$,

\smallskip

\item $\rho(h\ot_k 1)= \ep(h)1$, for all $h\in H$,

\smallskip

\item $\rho(1\ot_k a) = a$, for all $a\in A$,

\smallskip

\item $s\xcirc (H\ot_k\rho) = (\rho\ot_k H)\xcirc (H\ot_k s)\xcirc (c\ot_k A)$.

\medskip

\end{enumerate}
An {\em $s$-action} is a weak $s$-action which satisfies $\rho\xcirc (H\ot_k\rho) =\rho\xcirc (\mu\ot_k A)$.
\end{definition}

\begin{remark}\label{propiedad de accion debil} It is easy to see that if $\rho$ is a weak $s$-action of $H$ on $A$, then
$$
(H\ot_k \rho)\xcirc (c^{-1}\ot_k A)\xcirc (H\ot_k s^{-1}) = s^{-1}\xcirc (\rho\ot_k H).
$$
\end{remark}

We will use the diagrams
\begin{equation}\label{s3}
\spreaddiagramcolumns{-1.6pc}\spreaddiagramrows{-1.6pc}
\objectmargin{0.0pc}\objectwidth{0.0pc}
\def\objectstyle{\sssize}
\def\labelstyle{\sssize}
\grow{\xymatrix{
\ar@{-}[1,1]+<-0.1pc,0.1pc> && \ar@{-}[2,-2]\\
&&\\
&&\ar@{-}[-1,-1]+<0.1pc,-0.1pc>
}}
\grow{\xymatrix{
\\
\txt{,}
}}\quad
\grow{\xymatrix{ \ar@{-}[2,2]&&\ar@{-}[1,-1]+<0.1pc,0.1pc>\\
 &&\\
 \ar@{-}[-1,1]+<-0.1pc,-0.1pc>&&}}
\grow{\xymatrix{
\\
\txt{,}
}}\quad
\grow{\xymatrix{
\ar@{-}[1,1]+<-0.125pc,0.0pc> \ar@{-}[1,1]+<0.0pc,0.125pc>&&\ar@{-}[2,-2]\\
&&\\
&&\ar@{-}[-1,-1]+<0.125pc,0.0pc>\ar@{-}[-1,-1]+<0.0pc,-0.125pc>
}}
\grow{\xymatrix{
\\
\txt{,}
}}\quad
\grow{\xymatrix{
\ar@{-}[2,2]&&\ar@{-}[1,-1]+<0.125pc,0.0pc> \ar@{-}[1,-1]+<0.0pc,0.125pc>\\
&&\\
\ar@{-}[-1,1]+<-0.125pc,0.0pc> \ar@{-}[-1,1]+<0.0pc,-0.125pc>&&
}}
\quad\grow{\xymatrix{{}\save[]+<0pc,-0.4pc>*\txt{and}
\restore}}\quad
\grow{\xymatrix{
               \ar@{-}`d/4pt [1,2][1,2] && \ar@{-}[2,0]\\
               &&\\
               &&
}}
\end{equation}
to denote the braid $c$ of $H$, its inverse $c^{-1}$, the transposition $s$, its inverse $s^{-1}$, and the weak $s$-action $\rho$, respectively.

\begin{definition} Let $s\colon H\ot_k A \to A\ot_k H$ be a transposition, $\rho\colon H\ot_k A\to A$ a weak $s$-action and $f\colon H\ot_k H \to A$ a $k$-linear map. We say that $f$ is {\em normal} if $f(1\ot_k x)=f(x\ot_k 1)= \ep(x)$ for all $x\in H$, and that $f$ is a {\em cocycle that satisfies the twisted module condition} if
$$
\spreaddiagramcolumns{-1.6pc}\spreaddiagramrows{-1.6pc}
\objectmargin{0.0pc}\objectwidth{0.0pc}
\def\objectstyle{\sssize}
\def\labelstyle{\sssize}
\grow{\xymatrix@!0{
   &\ar@{-}[1,0]&&&&\ar@{-}[1,0]&&&&\ar@{-}[1,0]\\
                     & \ar@{-}`l/4pt [1,-1] [1,-1] \ar@{-}`r [1,1] [1,1]&&&&
                     \ar@{-}`l/4pt [1,-1] [1,-1] \ar@{-}`r [1,1] [1,1]&&&&
                     \ar@{-}`l/4pt [1,-1] [1,-1] \ar@{-}`r [1,1] [1,1]&\\
  \ar@{-}[4,0]&&\ar@{-}[1,1]+<-0.1pc,0.1pc> && \ar@{-}[2,-2]&&\ar@{-}[1,1]+<-0.1pc,0.1pc> &&
                                                                \ar@{-}[2,-2]&&\ar@{-}[4,0]\\
  &&&&&&&&&&\\
  &&\ar@{-}[2,0]&&\ar@{-}[-1,-1]+<0.1pc,-0.1pc>\ar@{-}[1,1]+<-0.1pc,0.1pc> && \ar@{-}[2,-2]
                           && \ar@{-}[-1,-1]+<0.1pc,-0.1pc>\ar@{-}[2,0]&&\\
  &&&&&&&&&&\\
  \ar@{-}[2,1]&&\ar@{-}`d/4pt [1,1] `[0,2] [0,2]&&&&\ar@{-}[-1,-1]+<0.1pc,-0.1pc> \ar@{-}[2,1]
                                                       &&\ar@{-}`d/4pt [1,1] `[0,2] [0,2]&&\\
  &&&{\bullet}\ar@{-}[1,0]&&&&&&\ar@{-}[1,0]&\\
  &\ar@{-}`d/4pt [1,2][1,2] && \ar@{-}[2,0]&&&&\ar@{-}`d/4pt [1,1] `[0,2] [0,2]&&\\
  &&&&&&&&{\bullet}\ar@{-}[3,0] &&\\
  &&&\ar@{-}[2,1]&&&&&&&\\
  &&&&&&&&&&\\
  &&&&\ar@{-}`d/4pt [1,2] `[0,4] [0,4] &&&&&&\\
  &&&&&&\ar@{-}[1,0]&&&&\\
  &&&&&&
}}
\grow{\xymatrix{
\\\\\\\\\\
\txt{$=$} }}
\grow{\xymatrix@!0{
\\\\
  &\ar@{-}[1,0]&&&&\ar@{-}[1,0]&&\ar@{-}[6,0]\\
  & \ar@{-}`l/4pt [1,-1] [1,-1] \ar@{-}`r [1,1] [1,1]&&&&\ar@{-}`l/4pt [1,-1] [1,-1]
                                                               \ar@{-}`r [1,1] [1,1]&&\\
  \ar@{-}[2,0]&&\ar@{-}[1,1]+<-0.1pc,0.1pc> && \ar@{-}[2,-2]&&\ar@{-}[2,0]&\\
  &&&&&&&\\
  \ar@{-}`d/4pt [1,1] `[0,2] [0,2]&&&&\ar@{-}[-1,-1]+<0.1pc,-0.1pc>
                                                      \ar@{-}`d/4pt [1,1] `[0,2] [0,2]&&&\\
  &{\bullet}\ar@{-}[1,0]&&&&\ar@{-}[1,0]&&\\
  &\ar@{-}[2,1]&&&&\ar@{-}`d/4pt [1,1] `[0,2] [0,2]&&\\
  &&&&&&{\bullet}\ar@{-}[1,0]&\\
  &&\ar@{-}`d/4pt [1,2] `[0,4] [0,4] &&&&&\\
  &&&&\ar@{-}[1,0]&&&\\
  &&&&&&
}}
\quad
\grow{\xymatrix{
\\\\\\\\\\
\txt{and} }} \quad
\grow{\xymatrix@!0{
  &\ar@{-}[1,0]&&&&\ar@{-}[1,0]&&&\ar@{-}[4,0]\\
  & \ar@{-}`l/4pt [1,-1] [1,-1] \ar@{-}`r [1,1] [1,1]&&&&\ar@{-}`l/4pt [1,-1] [1,-1]
                                                               \ar@{-}`r [1,1] [1,1]&&\\
  \ar@{-}[6,0]&&\ar@{-}[1,1]+<-0.1pc,0.1pc> && \ar@{-}[2,-2]&&\ar@{-}[2,0]&\\
  &&&&&&&\\
  \ar@{-}[4,0]&&\ar@{-}[4,0]&&\ar@{-}[-1,-1]+<0.1pc,-0.1pc>\ar@{-}[2,0]&&
          \ar@{-}[1,1]+<-0.125pc,0.0pc> \ar@{-}[1,1]+<0.0pc,0.125pc>&&\ar@{-}[2,-2]\\
  &&&&&&&&\\
  &&&&\ar@{-}[1,1]+<-0.125pc,0.0pc> \ar@{-}[1,1]+<0.0pc,0.125pc>&&\ar@{-}[2,-2]
        &&\ar@{-}[2,0]\ar@{-}[-1,-1]+<0.125pc,0.0pc>\ar@{-}[-1,-1]+<0.0pc,-0.125pc>\\
  &&&&&&&&\\
  \ar@{-}[2,2]&&\ar@{-}`d/4pt [1,2][1,2] && \ar@{-}[2,0]
                   &&\ar@{-}[-1,-1]+<0.125pc,0.0pc>\ar@{-}[-1,-1]+<0.0pc,-0.125pc>
  \ar@{-}`d/4pt [1,1] `[0,2] [0,2] &&\\
  &&&&&&&{\bullet}\ar@{-}[1,0]&\\
  &&\ar@{-}`d/4pt [1,2][1,2] && \ar@{-}[2,0]&&&\ar@{-}[2,-1]&\\
  &&&&&&&&\\
  &&&&\ar@{-}`d/4pt [1,1] `[0,2] [0,2]&&&&\\
  &&&&&\ar@{-}[1,0]&&&\\
  &&&&&&&&
}}
\grow{\xymatrix{
\\\\\\\\\\
\txt{$=$} }}
\grow{\xymatrix@!0{
\\
  &\ar@{-}[1,0]&&&&\ar@{-}[1,0]&&\ar@{-}[10,0]\\
  & \ar@{-}`l/4pt [1,-1] [1,-1] \ar@{-}`r [1,1] [1,1]&&&&\ar@{-}`l/4pt [1,-1] [1,-1]
                                                               \ar@{-}`r [1,1] [1,1]&&\\
  \ar@{-}[2,0]&&\ar@{-}[1,1]+<-0.1pc,0.1pc> && \ar@{-}[2,-2]&&\ar@{-}[2,0]&\\
  &&&&&&&\\
  \ar@{-}[2,1]&&\ar@{-}[2,1]&&\ar@{-}[-1,-1]+<0.1pc,-0.1pc>\ar@{-}`d/4pt [1,1] `[0,2] [0,2]
                                                                              &&&\ar@{-}[2,0]\\
  &&&&&\ar@{-}[1,0]&&\\
  &\ar@{-}`d/4pt [1,1] `[0,2] [0,2]&&&&\ar@{-}`d/4pt [1,2][1,2] && \ar@{-}[4,0]\\
  && {\bullet}\ar@{-}'[1,0][3,1]&&&&&\\
  &&&&&&&\\
  &&&&&&&\\
  &&&\ar@{-}`d/4pt [1,2] `[0,4] [0,4] &&&&\\
  &&&&&\ar@{-}[1,0]&&\\
  &&&&&&
}}
\grow{\xymatrix{
\\\\\\\\\\
\txt{,\qquad where } }} \grow{\xymatrix@!0{
\\\\\\\\\\\\
\ar@{-}`d/4pt [1,1] `[0,2] [0,2] &&\\
&{\bullet}\ar@{-}[1,0]\\
& }} \grow{\xymatrix{
\\\\\\\\\\
\txt{$= f$.} }}
$$
More precisely, the first equality is the {\em cocycle condition} and the second one is the {\em twisted module condition}. Finally, we say that $f$ is {\em compatible with $s$} if
$$
(f\ot_k H)\xcirc(H\ot_k c)\xcirc(c\ot_k H) = s\xcirc(H\ot_k f).
$$
\end{definition}

\begin{definition} Let $s\colon H\ot_k A \to A\ot_k H$ be a transposition, $\rho\colon H\ot_k A\to A$ a weak $s$-action, $f\colon H\ot_k H \to A$ a compatible with $s$ normal cocycle satisfying the twisted module condition, and $R$ a subalgebra of $A$. We say that a $R$ is {\em stable under $s$ and $\rho$} if $s(H\ot_k R)\subseteq R\ot_k H$ and $\rho(H\ot_k R)\subseteq R$, and we say that {\em $f$ takes its values in $R$} if $f(H\ot_k H)\subseteq R$.
\end{definition}

Let $H$ be a bialgebra, $A$ and algebra, $s\colon H\ot_k A \to A\ot_k H$ a transposition, $\rho\colon H\ot_k A \to A$ a weak $s$-action and $f\colon H\ot_k H \to A$ a compatible with $s$ normal cocycle that satisfies the twisted module condition. Let $\chi\colon H\ot_k A \to A\ot_k H$ and $\mathcal{F}\colon H\ot_k H \to A\ot_k H$ be the maps defined by
$$
\chi:=\! (\rho \ot_k H)\xcirc (H\ot_k s)\xcirc (\De \ot_k A)\quad\!\text{and} \quad\!\mathcal{F}:=\! (f\ot_k\mu)\xcirc (H\ot_k c\ot_k H)\xcirc(\De \ot_k \De).
$$
In~\cite[Section 9]{G-G2} it was proven that $\chi$ is a twisting map and $\mathcal{F}$ is a normal cocycle that satisfies the twisted module condition.

\smallskip

Let $R$ be a subalgebra of $A$. It is evident that if $R$ is stable under $s$ and $\rho$, then it is also stable under $\chi$, and that if $f$ takes its values in $R$, then $\mathcal{F}$ takes its values in $R\ot_k H$.

\begin{definition} The {\em braided Hopf crossed product $A\#_f H$ associated with $(s,\rho,f)$} is the Brzezi\'nski crossed product associated with $\chi$ and $\mathcal{F}$.
\end{definition}

Let $H\ot_c H$ be the coalgebra with underlying space $H\ot_k H$, comultiplication map $\De_{H\ot_c H}:=(H\ot_k c\ot_k H)\xcirc(\De_H \ot_k \De_H)$ and counit $\epsilon_{H\ot_c H} :=\epsilon_H\ot_k \epsilon_H$. An important class of braided Hopf crossed products are those with $H$ a braided Hopf algebra and whose cocycle $f\colon H\ot_c H\to A$ is convolution invertible. They are named cleft. In~\cite[Section 10]{G-G2} it was proven that $E$ is cleft if and only if the map $\gamma\colon H\to E$, defined by $\gamma(h) = 1\# h$, is convolution invertible. Moreover, in this case,
$$
\gamma^{-1} = (f^{-1}\ot_k H)\xcirc (S\ot_k H\ot_k S)\xcirc (H\ot_k c)\xcirc (c\ot_k H)\xcirc (\De_H\ot_k H)\xcirc \De_H.
$$

\subsection{comodule algebras}
\begin{definition}\label{comodulo algebra} Let $s\colon H\ot_k A\to A\ot_k H$ a transposition. Assume that $A$ is a right $H$-comodule with coaction $\nu$. We say that $(A,s)$ is a {\em right $H$-comodule algebra} if and only if

\begin{enumerate}

\smallskip

\item $(\nu\ot_k H)\xcirc s = (A\ot_k c) \xcirc (s\ot_k H) \xcirc (H\ot_k \nu)$,

\smallskip

\item $(\mu_{\!A}\ot_k \mu_{\!H})\xcirc (A\ot_k s\ot_k H)\xcirc (\nu \ot_k \nu) = \nu\xcirc \mu_{\!A}$,

\smallskip

\item $\nu(1) = 1\ot_k 1$.

\smallskip

\end{enumerate}
\end{definition}

Let $(A,s)$ and $(A',s')$ be $H$-comodule algebras. We say that a map $f\colon A\to A'$ is a {\em morphism of $H$-comodule algebras} from $(A,s)$ to $(A',s')$, if it is a morphism of algebras, a morphism of $H$-comodules and $s'\xcirc (H\ot_k f) = (f\ot_k H)\xcirc s$.

\begin{example}\label{ex 1.14} If $E = A\#_f H$ is a braided Hopf crossed product, then the map $\wh{s}\colon H\ot_k E\to E\ot_k H$ defined by $\wh{s}:= (A\ot_k c)\xcirc (s\ot_k H)$ is a transposition, and $(E,\wh{s})$, endowed with the comultiplication $\nu\colon E\to E\ot_k H$, defined by $\nu := A\ot_k \Delta_H$, is an $H$-braided comodule algebra. In particular $(H,c)$ is an $H$-braided comodule algebra with comultiplication $\Delta_H$. Moreover the map $\gamma\colon H\to E$ is a morphisms of $H$-comodule algebras from $(H,c)$ to $(E,\wh{s})$.
\end{example}

\begin{remark}\label{notacion para wh{s}} The maps $\wh{s}$ and $\wh{s}^{-1}$ will be represented by the same diagrams as the ones introduced in~\eqref{s3} for $s$ and $s^{-1}$, respectively.
\end{remark}

\subsection{Mixed complexes}
In this subsection we recall briefly the notion of mixed complex. For more details about this
concept we refer to \cite{K} and \cite{B}.

\smallskip

A {\em mixed complex} $(X,b,B)$ is a graded $k$-vector space $(X_n)_{n\ge 0}$, endowed with morphisms $b\colon X_n\to X_{n-1}$ and $B\colon X_n\to X_{n+1}$, such that
$$
b\xcirc b = 0,\quad B\xcirc B = 0\quad\text{and}\quad B \xcirc b + b\xcirc B = 0.
$$
A {\em morphism of mixed complexes} $f\colon (X,b,B)\to (Y,d,D)$ is a family of maps $f\colon X_n\to Y_n$, such that $d\xcirc f = f\xcirc b$ and $D\xcirc f= f\xcirc B$. Let $u$ be a degree~$2$ variable. A mixed complex $\cX = (X,b,B)$ determines a double complex
\[
\xymatrix{\\\\\\ \BP(\cX)=}\qquad
\xymatrix{
& \vdots \dto^-{b} &\vdots \dto^-{b}& \vdots \dto^-{b}& \vdots \dto^-{b}\\
\dots & X_3 u^{-1} \lto_-{B}\dto^-{b} & X_2 u^0\lto_-{B}\dto^-{b} & X_1 u\lto_-{B}\dto^-{b} &
X_0 u^2\lto_-{B} \\
\dots & X_2 u^{-1}\lto_-{B}\dto^-{b} & X_1 u^0\lto_-{B}\dto^-{b} & X_0 u\lto_-{B}\\
\dots & X_1 u^{-1}\lto_-{B}\dto^-{b} & X_0 u^0 \lto_-{B}\\
\dots & X_0 u^{-1} \lto_-{B},}
\]
where $b(\bx u^i):= b(\bx)u^i$ and $B(\bx u^i):= B(\bx)u^{i-1}$. By deleting the positively numbered columns we obtain a subcomplex $\BN(\cX)$ of $\BP(\cX)$. Let $\BN'(\cX)$ be the kernel of the canonical surjection from $\BN(\cX)$ to $(X,b)$. The quotient double complex $\BP(\cX)/\BN'(\cX)$ is denoted by $\BC(\cX)$. The homology groups $\HC_*(\cX)$, $\HN_*(\cX)$ and $\HP_*(\cX)$, of the total complexes of $\BC(\cX)$, $\BN(\cX)$ and $\BP(\cX)$ respectively, are called the {\em cyclic}, {\em negative} and {\em periodic homology groups} of $\cX$. The homology $\HH_*(\cX)$, of $(X,b)$, is called the {\em Hochschild homology} of $\cX$. Finally, it is clear that a morphism $f\colon \cX\to \cY$ of mixed complexes induces a morphism from the double complex $\BP(\cX)$ to the double complex $\BP(\cY)$.

\smallskip

Let $C$ be a $k$-algebra. If $K$ is a subalgebra of $C$ we will say that $C$ is a $K$-algebra. Throughout the paper we will use the following notations:

\begin{enumerate}

\smallskip

\item We set $\ov{C}:= C/K$. Moreover, given $c\in C$, we also denote by $c$ the class of $c$ in $\ov{C}$.

\smallskip

\item We use the unadorned tensor symbol $\ot$ to denote the tensor product $\ot_{\!K}$.

\smallskip

\item We write $\ov{C}^{\ot^l}:= \ov{C}\ot\cdots\ot \ov{C}$ ($l$-times).

\smallskip

\item Given $c_0,\dots, c_r \in C$ and $i<j$, we write $\bc_{ij}:= c_i\ot\cdots\ot c_j$.

\smallskip

\item Given a $K$-bimodule $M$, we let $M\ot$ denote the quotient $M/[M,K]$, where $[M,K]$ is the $k$-vector subspace of $M$ generated by all the commutators $m\lambda - \lambda m$, with $m\in M$ and $\lambda\in K$. Moreover, for $m\in M$, we let $[m]$ denote the class of $m$ in $M\ot$.

\smallskip

\end{enumerate}
By definition, the {\em normalized mixed complex of the $K$-algebra $C$} is the mixed complex $(C\ot \ov{C}^{\ot^*}\ot,b,B)$, where $b$ is the canonical Hochschild boundary map and the Connes operator $B$ is given by
$$
B\bigl([\bc_{0r}]\bigr):= \sum_{i=0}^r (-1)^{ir} [1\ot \bc_{ir}\ot \bc_{0,i-1}].
$$
The {\em cyclic}, {\em negative}, {\em periodic} and {\em Hochschild homology groups} $\HC^K_*(C)$, $\HN^K_*(C)$, $\HP^K_*(C)$ and $\HH^K_*(C)$ of $C$ are the respective homology groups of $(C\ot\ov{C}^{\ot^*}\ot,b,B)$.

\subsection{The perturbation lemma}
Next, we recall the perturbation lemma. We give the version introduced in \cite{C}.

\smallskip

A {\em homotopy equivalence data}
\begin{equation}
\xymatrix{(Y,\partial)\ar@<-1ex>[r]_-{i} & (X,d) \ar@<-1ex>[l]_-{p}}, \quad h\colon X_*\to
X_{*+1},\label{eq2}
\end{equation}
consists of the following:

\begin{enumerate}

\smallskip

\item Chain complexes $(Y,\partial)$, $(X,d)$ and quasi-isomorphisms $i$, $p$ between them,

\smallskip

\item A homotopy $h$ from $i\xcirc p$ to $\ide$.
\end{enumerate}

\smallskip

A {\em perturbation} $\de$ of~\eqref{eq2} is a map $\de\colon X_*\to X_{*-1}$ such that $(d+\de)^2 = 0$. We call it {\em small} if $\ide - \de\xcirc h$ is invertible. In this case we write $A = (\ide - \de\xcirc h)^{-1}\xcirc \de$ and we consider
\begin{equation}
\xymatrix{(Y,\partial^1)\ar@<-1ex>[r]_-{i^1} & (X,d+\de)\ar@<-1ex>[l]_-{p^1}}, \quad h^1\colon
X_*\to X_{*+1},\label{eq3}
\end{equation}
with
$$
\partial^1:= \partial + p\xcirc A\xcirc i,\quad i^1:= i + h\xcirc A\xcirc i,\quad p^1:= p + p\xcirc A\xcirc h,\quad h^1:= h + h\xcirc A\xcirc h.
$$
A {\em deformation retract} is a homotopy equivalence data such that $p\xcirc i = \ide$. A deformation retract is called {\em special} if $h\xcirc i = 0$, $p\xcirc h = 0$ and $h\xcirc h = 0$.

\smallskip

In all the cases considered in this paper the map $\de\xcirc h$ is locally nilpotent, and so $(\ide - \de\xcirc h)^{-1} = \sum_{n=0}^{\infty} (\de\xcirc h)^n$.

\begin{theorem}[\cite{C}]\label{lema de perturbacion} If $\de$ is a small perturbation of the homotopy equivalence data~\eqref{eq2}, then the perturbed data~\eqref{eq3} is a homotopy equivalence data. Moreover, if~\eqref{eq2} is a special deformation retract, then~\eqref{eq3} is also.
\end{theorem}

\section{A resolution for a Brzezi\'nski's crossed product}\label{SecRes}
Let $E:=A\# V$ be a Brzezi\'nski's crossed product with associated twisting map $\chi$ and cocycle $\mathcal{F}$, and let $K$ be a stable under $\chi$ subalgebra of $A$. Let $\Upsilon$ be the family of all the epimorphisms of $E$-bimodules which split as $(E,K)$-bimodule maps. In this section we construct a $\Upsilon$-projective resolution $(X_*,d_*)$, of $E$ as an $E$-bimodule, simpler than the normalized bar resolution of $E$. Moreover we will compute comparison maps between both resolutions. Recall that for all $K$-algebra $C$ we let $\ov{C}$ and $\ot$ denote $C/K$ and $\ot_K$, respectively. We also will use the following notations:

\begin{enumerate}

\smallskip

\item Given $x_0,\dots, x_r \in E$ and $i<j$, we write $\ov{\bx}_{ij}$ to mean $x_i\ot_{\!A}\cdots \ot_{\!A} x_j$, both in $E^{\ot_{\!A}^{j-i+1}}$ and in $(E/A)^{\ot_{\!A}^{j-i+1}}$.

\smallskip

\item We let $i_{\!A}\colon A\to E$ and $i_{\ov{A}}\colon \ov{A}\to \ov{E}$ denote the maps defined by $i_{\!A}(a):=a\# 1$ and $i_{\ov{A}}(a):=a\# 1$, respectively.

\smallskip

\item We set $\ov{V}:= V/k$. Moreover, given $v\in V$, we also denote by $v$ the class of $v$ in $\ov{V}$.

\smallskip

\item We write $\ov{V}^{\ot_k^l}:= \ov{V}\ot_k\cdots\ot_k\ov{V}$ ($l$-times).

\smallskip

\item Given $v_0,\dots, v_s \in V$ and $i<j$, we write $\bv_{ij}:= v_i\ot_k\cdots\ot_k v_j$.

\smallskip

\item We will denote by $\gamma$ any of the maps
$$
\xymatrix @R=-2pt {V\rto & E\\
v \ar@{|->}[0,1] & 1\# v},
\quad
\xymatrix @R=-2pt {V\rto & \ov{E}\\
v \ar@{|->}[0,1] & 1\# v}
\quad\xymatrix{{}\save[]+<0pc,-0.4pc>*\txt{or} \restore}\quad
\xymatrix @R=-2pt {V\rto & E/A\\
v \ar@{|->}[0,1] & 1\# v}.
$$
So, $\gamma(v)$ stands for $1\# v\in E$ or for its class in $\ov{E}$ or $E/A$. More generality, given $a\in A$ and $v\in V$ we will let $a\gamma(v)$ denote $a\# v\in E$ or its class in $\ov{E}$ or $E/A$.

\smallskip

\item We will denote by $\mathcal{V}$, $\mathcal{V}_{\!K}$ and $\mathcal{V}_{\!A}$ the image of $\gamma$ in $E$, $\ov{E}$ and $E/A$, respectively.

\smallskip

\item Given $\bv_{1j}\in V^{\ot_k^j}$, we write $\gamma(\bv_{1j})$ to mean $\gamma(v_i)\ot \cdots\ot\gamma(v_j)$ both in $E^{\ot^j}$ and in $\ov{E}^{\ot^j}$.

\smallskip

\item Given $\bv_{1j}\in V^{\ot_k^j}$, we write $\gamma_{\!A}(\bv_{ij})$ to mean $\gamma(v_i)\ot_{\!A}\cdots\ot_{\!A}\gamma(v_j)$ both in $E^{\ot_{\!A}^j}$ and in $(E/A)^{\ot_{\!A}^j}$.

\smallskip

\end{enumerate}
Note that $E/A \simeq A\ot_k \ov{V}$. We will use the following evident identifications
$$
A^{\ot^r}\ot_{\!A} E\!\simeq\!A^{\ot^r}\ot_k V,\quad\!\! E\ot_{\!A} (E/A)^{\ot_{\!A}^s} \!\simeq\! E\ot_k\ov{V}^{\ot_k^s}\quad\!\!\text{and}\quad\!\! E^{\ot_{\!A}^s}\!\simeq\! E^{\ot_{\!A}^i}\ot_k V^{\ot_k^{s-i}}.
$$
We consider $A^{\ot^r}\ot_k V$, $E\ot_k\ov{V}^{\ot_k^s}$ and $E^{\ot_{\!A}^i}\ot_k V^{\ot_k^{s-i}}$ as $E$-bimodules via the actions obtained by translation of structure. For all $r,s\ge 0$, we let $Y_s$ and $X_{rs}$ denote
$$
E\ot_{\!A} (E/A)^{\ot_{\!A}^s}\ot_{\!A} E\quad\text{and}\quad E\ot_{\!A} (E/A)^{\ot_{\!A}^s} \ot \ov{A}^{\ot^r} \ot E,
$$
respectively. By the above discussion
$$
Y_s \simeq \bigl(E\ot_k \ov{V}^{\ot_k^s}\bigr) \ot_{\!A} E\quad\text{and}\quad X_{rs} \simeq \bigl(E\ot_k\ov{V}^{\ot_k^s}\bigr)\ot \ov{A}^{\ot^r} \ot E.
$$
Consider the diagram of $E$-bimodules and $E$-bimodule maps
$$
\xymatrix{
\vdots \dto^-{-\partial_2} \\
Y_2\dto^-{-\partial_2} &X_{02}\lto_-{\nu_2} &X_{12}\lto_-{d^0_{12}} &\dots \lto_-{d^0_{22}}\\
Y_1 \dto^-{-\partial_1} &X_{01}\lto_-{\nu_1} &X_{11}\lto_-{d^0_{11}} &\dots \lto_-{d^0_{21}}\\
Y_0 & X_{00}\lto_-{\nu_0} & X_{10} \lto_-{d^0_{10}} & \dots \lto_-{d^0_{20}},
 }
$$
where $(Y_*,\partial_*)$ is the normalized bar resolution of the $A$-algebra $E$, introduced in \cite{G-S}; for each $s\ge 0$, the complex $(X_{*s},d^0_{*s})$ is $(-1)^s$-times the normalized bar resolution of the $K$-algebra $A$, tensored on the left over $A$ with $E\ot_{\!A} (E/A)^{\ot_{\!A}^s}$, and on the right over $A$ with $E$; and for each $s\ge 0$, the map $\nu_s$ is the canonical surjection. Each one of the rows of this diagram is contractible as a $(E,K)$-bimodule complex. A contracting homotopy
$$
\sigma^0_{0s}\colon Y_s \to X_{0s} \qquad\text{and}\qquad \sigma^0_{r+1,s}\colon X_{rs} \to
X_{r+1,s},
$$
of the $s$-th row, is given by
\begin{align*}
&\sigma^0_{0s}\bigl(\ov{\bx}_{0s}\ot_{\!A} \gamma(v)\bigr):= \ov{\bx}_{0s}\ot \gamma(v)
\intertext{and}
&\sigma^0_{r+1,s}\bigl(\ov{\bx}_{0s}\ot \ba_{1r}\ot a_{r+1}\gamma(v)\bigr):= (-1)^{r+s+1} \ov{\bx}_{0s}\ot \ba_{1,r+1}\ot \gamma(v).
\end{align*}
Let $\wt{u}\colon Y_0\to E$ be the multiplication map. The complex of $E$-bimodules
$$
\xymatrix{E & Y_0 \lto_-{-\wt{u}} & Y_1 \lto_-{-\partial_1} & Y_2 \lto_-{-\partial_2} & Y_3 \lto_-{-\partial_3} & Y_4 \lto_-{-\partial_4} & Y_5 \lto_-{-\partial_5} & \dots \lto_-{-\partial_6}}
$$
is also contractible as a complex of $(E,K)$-bimodules. A chain contracting homotopy
$$
\sigma_0^{-1}\colon E \to Y_0\qquad\text{and}\qquad \sigma^{-1}_{s+1}\colon Y_s\to Y_{s+1} \quad\text{($s\ge 0$)},
$$
is given by $\sigma^{-1}_{s+1}(\ov{\bx}_{0,s+1}):= (-1)^s \ov{\bx}_{0,s+1}\ot_{\!A} 1_E$.

\medskip

For $r\ge 0$ and $1\le l\le s$, we define $E$-bimodule maps $d^l_{rs}\colon X_{rs}\to X_{r+l-1,s-l}$ recursively on $l$ and $r$, by:
$$
d^l(\bz):= \begin{cases}
\sigma^0\xcirc\partial\xcirc\nu(\bz) &\text{if $l=1$ and $r=0$,}\\
- \sigma^0\xcirc d^1\xcirc d^0(\bz) &\text{if $l=1$ and $r>0$,}\\
- \sum_{j=1}^{l-1} \sigma^0\xcirc d^{l-j}\xcirc d^j(\bz) & \text{if $1<l$ and $r=0$,}\\
- \sum_{j=0}^{l-1} \sigma^0\xcirc d^{l-j}\xcirc d^j(\bz) &\text{if $1<l$ and $r>0$,}
\end{cases}
$$
for $\bz \in E\ot_{\!A} (E/A)^{\ot_{\!A}^s} \ot \ov{A}^{\ot^r} \ot K$.

\begin{theorem}\label{res nuestra} There is a $\Upsilon$-projective resolution of $E$
\begin{equation}
\xymatrix{E & X_0\lto_{-\mu} & X_1\lto_{d_1} &X_2 \lto_{d_2} &X_3\lto_{d_3} &X_4 \lto_{d_4} & \lto_{d_5} \dots,}\label{eq4}
\end{equation}
where $\mu\colon X_{00}\to E$ is the multiplication map,
$$
X_n:= \bigoplus_{r+s=n} X_{rs}\quad\text{and}\quad d_n:= \sum^n_{l=1} d^l_{0n} + \sum_{r=1}^n \sum^{n-r}_{l=0} d^l_{r,n-r}.
$$
\end{theorem}

\begin{proof} This follows immediately from~\cite[Corollary~A2]{G-G3}.
\end{proof}

In order to carry out our computations we also need to give an explicit contracting homotopy of the resolution~\eqref{eq4}. For this we define maps
$$
\sigma^l_{l,s-l}\colon Y_s \to X_{l,s-l}\quad\text{and}\quad \sigma^l_{r+l+1,s-l}\colon X_{rs} \to X_{r+l+1,s-l}
$$
recursively on $l$, by:
$$
\sigma^l_{r+l+1,s-l}:= - \sum_{i=0}^{l-1} \sigma^0 \xcirc d^{l-i} \xcirc\sigma^i\qquad \text{($0<l\le s$ and $r\ge -1$)}.
$$

\begin{proposition}\label{cont nuestra} The family
$$
\ov{\sigma}_0\colon E\to X_0,\qquad\ov{\sigma}_{n+1}\colon X_n\to X_{n+1}\quad\text{($n\ge 0$)},
$$
defined by $\ov{\sigma}_0:= \sigma_{00}^0\xcirc\sigma_0^{-1}$ and
$$
\ov{\sigma}_{n+1}:= - \sum_{l=0}^{n+1}\sigma_{l,n-l+1}^l\xcirc\sigma_{n+1}^{-1}\xcirc\nu_n + \sum_{r=0}^n \sum_{l=0}^{n-r}\sigma_{r+l+1,n-r-l}^l\qquad\text{($n\ge 0$)},
$$
is a contracting homotopy of~\eqref{eq4}.
\end{proposition}

\begin{proof} This is a direct consequence of~\cite[Corollary~A2]{G-G3}.
\end{proof}

\begin{notations}\label{not2.1} We will use the following notations:

\begin{enumerate}

\smallskip

\item For $j,l\ge 1$, we let $\chi_{_{jl}}\colon V^{\ot_k^j}\ot_k A^{\ot^l}\to A^{\ot^l} \ot_k V^{\ot_k^j}$ denote the map recursively defined by:
\begin{align*}
&\chi_{_{11}}:=\chi,\\
&\chi_{_{1,l+1}}:=\bigl(A^{\ot^l}\ot_k \chi\bigr)\xcirc \bigl(\chi_{_{1l}}\ot_k A \bigr),\\
&\chi_{_{j+1,l}}:=\bigl(\chi_{_{1l}}\ot_k V^{\ot_k^j}\bigr)\xcirc \bigl(V\ot_k \chi_{_{jl}}\bigr).
\end{align*}

\smallskip

\item Write $X'_{rs}:= E^{\ot_{\!A}^{s+1}} \ot A^{\ot^r} \ot E$. We let $u'_i\colon X'_{rs} \to X'_{r,s-1}$ denote the map defined by
\begin{align*}
\qquad\qquad &u'_i(\ov{\bx}_{0s}\ot\ba_{1r}\ot x):= \ov{\bx}_{0,i-1}\ot_{\!A} x_ix_{i+1}\ot_{\!A} \ov{\bx}_{i+1,s} \ot \ba_{1r}\ot x
\intertext{for $0\le i<s$, and}
& u'_s(\ov{\bx}_{0,s-1}\ot_{\!A} \gamma(v)\ot\ba_{1r}\ot x):= \sum_l \ov{\bx}_{0,s-1}\ot \ba_{1r}^{(l)}\ot \gamma(v^{(l)})x,
\end{align*}
where $\sum_l \ba_{1r}^{(l)}\ot_k v^{(l)} := \chi(v\ot_k\ba_{1r})$.

\smallskip

\item\label{notacion X_rs^u} Given a $K$-subalgebra $R$ of $A$ and $0\le u\le r$, we let $X^{Ru}_{rs}$ denote the $E$-subbimodule of $X_{rs}$ generated by all the simple tensors $1\ot_{\!A} \ov{\bx}_{1s}\ot \ba_{1r}\ot 1$, with at least $u$ of the $a_j$'s in $\ov{R}$.

\end{enumerate}

\end{notations}

\begin{theorem}\label{formula para d_1} The following assertions hold:

\begin{enumerate}

\smallskip

\item The map $d^1\colon X_{rs}\to X_{r,s-1}$ is induced by the map $\sum_{i=0}^s (-1)^i u'_i$.

\smallskip

\item Let $R$ be a stable under $\chi$ $K$-subalgebra of $A$. If $\mathcal{F}$ takes its values in $R\ot_k V$, then
$$
d^l\bigl(X_{rs}\bigr)\subseteq X^{R,l-1}_{r+l-1,s-l}
$$
for each $l\ge 1$.
\end{enumerate}

\end{theorem}

\begin{proof} Let $\bz \in E\ot_{\!A} (E/A)^{\ot_{\!A}^s} \ot \ov{A}^{\ot^r}\ot K$. The computation of $d_{rs}^1$ can be obtained easily by induction on $r$, using that
$$
d_{0s}^1(\bz) = \sigma_{0,s-1}^0\xcirc \partial_s \xcirc \nu_s^0(\bz)\quad\text{and}\quad d_{rs}^1(\bz) = -\sigma_{r,s-1}^0\xcirc d_{r-1,s}^1 \xcirc d_{rs}^0(\bz)\,\text{ for $r\ge 1$.}
$$
Item~(2) follows by induction on $r$ and $l$, using the recursive definition of $d_{rs}^l(\bz)$.
\end{proof}

\begin{remark} By item~(2) of the above theorem, if $\mathcal{F}$ takes its values in $K\ot_k V$, then $(X_*,d_*)$ is the total complex of the double complex $\bigl(X_{**},d_{**}^0,d_{**}^1 \bigr)$.
\end{remark}

\subsection{Comparison with the normalized bar resolution} Let $\bigl(E\ot \ov{E}^{\ot^*}\ot E,b'_*\bigr)$ be the normalized bar resolution of the $K$-algebra $E$. As it is well known, the complex
$$
\xymatrix{E & E\ot E\lto_-{\mu}& E\ot\ov{E}\ot E \lto_-{b'_1} &E\ot \ov{E}^{\ot^2}\ot E \lto_-{b'_2} &\dots \lto_-{b'_3}}
$$
is contractible as a complex of $(E,K)$-bimodules, with contracting homotopy
$$
\xi_0\colon E\to E\ot E,\qquad \xi_{n+1}\colon E\ot\ov{E}^{\ot^n}\ot E\to E\ot \ov{E}^{\ot^{n+1}} \ot E\quad\text{($n\ge 0$),}
$$
given by $\xi_n(\bx):= (-1)^n \bx\ot 1$. Let
$$
\phi_*\colon (X_*,d_*)\to \bigl(E\ot \ov{E}^{\ot^*}\ot E,b'_*\bigr)\quad\text{and}\quad \psi_*\colon \bigl(E\ot \ov{E}^{\ot^*}\ot E,b'_*\bigr)\to (X_*,d_*)
$$
be the morphisms of $E$-bimodule complexes, recursively defined by
\begin{align*}
&\phi_0:= \ide,\quad \psi_0:= \ide,\\
& \phi_{n+1}(\bz\ot 1):= \xi_{n+1}\xcirc\phi_n\xcirc d_{n+1}(\bz\ot 1)\\
\intertext{and}
&\psi_{n+1}(\bx\ot 1):= \ov{\sigma}_{n+1}\xcirc\psi_n\xcirc b'_{n+1}(\bx\ot 1).
\end{align*}

\begin{proposition}\label{homotopia} $\psi\xcirc\phi = \ide$ and $\phi\xcirc\psi$ is homotopically equivalent to the identity map. A homotopy $\omega_{*+1}\colon \phi_*\xcirc \psi_*\to \ide_*$ is recursively defined by
$$
\omega_1:= 0\quad\text{and}\quad\omega_{n+1}(\bx):= \xi_{n+1}\xcirc(\phi_n\xcirc \psi_n - \ide - \omega_n\xcirc b'_n)(\bx),
$$
for $\bx\in E\ot\ov{E}^{\ot^n}\ot K$.
\end{proposition}

\begin{proof} The proof of~\cite[Proposition~1.2.1]{G-G3} works in this context.
\end{proof}

\begin{remark}\label{re2.6} Since $\omega\bigl(E\ot\ov{E}^{\ot^{n-1}}\ot K\bigr)\subseteq E\ot \ov{E}^{\ot^n}\ot K$ and $\xi$ vanishes on $E\ot\ov{E}^{\ot^n}\ot K$,
$$
\omega(\bx\ot 1) = \xi\bigl(\phi\xcirc \psi(\bx\ot 1) - (-1)^n \omega(\bx)\bigr).
$$
\end{remark}

\subsection{The filtrations of $\mathbf{\boldsymbol{\bigl(E\ot\ov{E}^{\ot^*}\ot E,b'_*\bigr)}}$ and $\mathbf{\boldsymbol{(X_*, d_*)}}$} Let
$$
F^i(X_n):= \bigoplus_{0\le s\le i} X_{n-s,s}
$$
and let $F^i\bigl(E\ot \ov{E}^{\ot^n}\ot E\bigr)$ be the $E$-subbimodule of $E\ot \ov{E}^{\ot^n}\ot E$ generated by the tensors $1\ot \bx_{1n}\ot 1$ such that at least $n-i$ of the $x_j$'s belong to $\ov{A}$. The normalized bar resolution $\bigl(E\ot \ov{E}^{\ot^*}\ot E,b'_*\bigr)$ and the resolution $(X_*, d_*)$ are filtered by
\begin{align*}
& F^0\bigl(E\ot\ov{E}^{\ot^*}\ot E\bigr)\subseteq F^1\bigl(E\ot\ov{E}^{\ot^*}\ot E\bigr)\subseteq F^2\bigl(E\ot\ov{E}^{\ot^*}\ot E\bigr) \subseteq\dots
\intertext{and}
& F^0(X_*)\subseteq F^1(X_*)\subseteq F^2(X_*)\subseteq F^3(X_*)\subseteq F^4(X_*)\subseteq F^5(X_*) \subseteq \dots,
\end{align*}
respectively.

\begin{proposition}\label{phi, psi y omega preservan filtraciones} The maps $\phi$, $\psi$ and $\omega$ preserve filtrations.
\end{proposition}

\begin{proof} For $\phi$ this follows from Proposition~\ref{propA.5}. Let $Q^i_j := E\ot_{\!A} (E/A)^{\ot_{\!A}^i}\ot\ov{A}^{\ot^j}\ot K$. We claim that

\begin{enumerate}

\smallskip

\item[a)] $\ov{\si}\bigl(F^i(X_n)\bigr) \sub F^i(X_{n+1})$ for all $0\le i<n$,

\smallskip

\item[b)] $\ov{\si}\bigl(E\ot_{\!A} (E/A)^{\ot_{\!A}^i}\ot \ov{A}^{\ot^{n-i}}\ot A\bigr) \sub Q^i_{n+1-i} + F^{i-1}(X_{n+1})$ for all $0\le i\le n$,

\smallskip

\item[c)] $\ov{\si}(X_{0n}) \sub E\ot_{\!A} (E/A)^{\ot_{\!A}^{n+1}}\ot K+ F^n(X_{n+1})$,

\smallskip

\item[d)] $\psi\bigl(F^i(E\ot\ov{E}^{\ot^n}\ot E)\cap E\ot\ov{E}^{\ot^n}\ot K\bigr)\sub Q^i_{n-i} + F^{i-1}(X_n)$ for all $0\le i\le n$.

\medskip

\end{enumerate}
In fact~a), b) and c) follow immediately from the definition of $\ov{\si}_{n+1}$. Suppose d) is valid for $n$. Let
$$
\bx:=\bx_{0,n+1}\ot 1 \in F^i(E\ot\ov{E}^{\ot^{n+1}}\ot E)\cap E\ot \ov{E}^{\ot^{n+1}}\ot K\quad\text{where $0\le i\le n+1$.}
$$
Using~a), b) and the inductive hypothesis, we get that for $1\le j \le n$,
\begin{align*}
\ov{\si}\bigl(\psi(\bx_{0,j-1}\ot x_jx_{j+1}\ot \bx_{j+2,n+1}\ot 1)\bigr) & \sub \ov{\si}\bigl(Q^i_{n-i} + F^{i-1}(X_n)\bigr)\\
&\sub Q^i_{n+1-i} + F^{i-1}(X_{n+1}).
\end{align*}
Since $\psi(\bx) = \ov{\si}\xcirc\psi \xcirc b'(\bx)$, in order to prove~d) for $n+1$ we only must check that
$$
\ov{\si}(\psi(\bx_{0,n+1})) \sub Q^i_{n+1-i} + F^{i-1}(X_{n+1}).
$$
If $x_{n+1}\in A$, then using~a), b) and the inductive hypothesis, we get
\begin{align*}
\ov{\si}\bigl(\psi(\bx_{0,n+1})\bigr) & = \ov{\si}\bigl(\psi(\bx_{0n}\ot 1)x_{n+1}\bigr)\\
&\sub \ov{\si}\bigl(E\ot_{\!A} (E/A)^{\ot_{\!A}^i}\ot \ov{A}^{\ot^{n-i}}\ot A + F^{i-1}(X_n)\bigr)\\
&\sub Q^i_{n+1-i} + F^{i-1}(X_{n+1}),
\end{align*}
and if $x_{n+1}\notin A$, then $\bx_{0,n+1} \in F^{i-1}(E\ot\ov{E}^{\ot^n}\ot E)$, which together with~a), c) and the inductive hypothesis, implies that
$$
\ov{\si}\bigl(\psi(\bx_{0,n+1})\bigr) \sub \ov{\si}\bigl(F^{i-1}(X_n)\bigr) \sub Q^i_{n+1-i} + F^{i-1}(X_{n+1}).
$$
From~d) it follows immediately that $\psi$ preserves filtrations. Next, we prove that $\omega$ also does it. This is trivial for $\omega_1$, since $\omega_1=0$. Assume that $\omega_n$ does. Let
$$
\bx := \bx_{0n}\ot 1\in F^i(E\ot\ov{E}^{\ot^n}\ot E) \cap E\ot \ov{E}^{\ot^n}\ot K.
$$
By Remark~\ref{re2.6}, we know that
$$
\omega(\bx) = \xi\xcirc \phi\xcirc \psi(\bx) + (-1)^n \xi\xcirc\omega(\bx_{0n}).
$$
From~d) and the fact that $\phi$ preserve filtrations, we get
$$
\xi\xcirc\phi\xcirc \psi(\bx)  \in \xi\xcirc \phi\bigl(Q^i_{n-i} + F^{i-1}(X_n)\bigr) \sub\xi(F^{i-1}(E\ot\ov{E}^{\ot^n}\!\ot E))\sub F^i(E\ot\ov{E}^{\ot^n}\!\ot E),
$$
since $\xi(\phi(Q^i_{n-i})) \sub \xi(E\ot \ov{E}^{\ot^n}\ot K) = 0$. To finish the proof it remains to check that
$$
\xi\xcirc\omega\xcirc b'(\bx) \sub F^i(E\ot\ov{E}^{\ot^n}\ot E).
$$
Since, $\omega(E\ot \ov{E}^{\ot^{n-1}}\ot K) \sub E\ot\ov{E}^{\ot^n}\ot K$ by definition, we have
$$
\xi\xcirc\omega \xcirc b'(\bx) = (-1)^{n-1}\xi\xcirc\omega(\bx_{0n}).
$$
Hence, if $x_n\in A$, then
\begin{align*}
\xi\xcirc\omega\xcirc b'(\bx) &= (-1)^{n-1}\xi_{n+1}\bigl(\omega_n(\bx_{0,n-1}\ot 1)x_n\bigr)\\
&\sub \xi\bigl(F^i(E\ot\ov{E}^{\ot^n}\ot E)\cap E \ot \ov{E}^{\ot^n}\ot A\bigr)\\
&\sub F^i(E\ot\ov{E}^{\ot^{n+1}}\ot E),
\end{align*}
and if $x_n\notin A$, then $\bx_{0n}\in F^{i-1}(E\ot\ov{E}^{\ot^{n-1}}\ot E)$, and so
$$
\xi\xcirc\omega(\bx_{0n}) \sub \xi\bigl(F^{i-1}(E\ot\ov{E}^{\ot^n}\ot E)\bigr) \sub F^i(E\ot\ov{E}^{\ot^{n+1}}\ot E),
$$
as we want.
\end{proof}

\section{Hochschild homology of a Brzezi\'nski's crossed product}\label{Hochschild-Brzezinski} \label{HHBrezinski}
Let $E:=A\# V$ be a Brzezi\'nski's crossed product with associated twisting map $\chi$ and cocycle $\mathcal{F}$, and let $K$ be a stable under $\chi$ subalgebra of $A$. Recall that $\Upsilon$ is the family of all epimorphisms of $E$-bimodules which split as $(E,K)$-bimodule maps. Since $(X_*,d_*)$ is a $\Upsilon$-projective resolution of $E$, the Hochschild homology of the $K$-algebra $E$ with coefficients in an $E$-bimodule $M$ is the homology of $M\ot_{E^e}(X_*,d_*)$. For $r,s\ge 0$, write
$$
\wh{X}_{rs}(M):= M\ot_{\!A} (E/A)^{\ot_{\!A}^s}\ot \ov{A}^{\ot^r}\ot.
$$
It is easy to check that $\wh{X}_{rs}(M)\simeq M\ot_{E^e} X_{rs}$ via
$$
\xymatrix @R=-2pt {\wh{X}_{rs}(M)\rto & M\ot_{E^e} X_{rs}\\
[m\ot_{\!A} \ov{\bx}_{1s}\ot \ba_{1r}] \ar@{|->}[0,1] & m\ot_{E^e} (1\ot_{\!A}\ov{\bx}_{1s}\ot \ba_{1r}\ot 1)}.
$$
Let $\wh{d}^l_{rs}\colon \wh{X}_{rs}(M)\to \wh{X}_{r+l-1,s-l}(M)$ be the map induced by $\ide_M\ot_{E^e} d^l_{rs}$. Via the above identifications the complex $M\ot_{E^e}(X_*,d_*)$ becomes $(\wh{X}_*(M),\wh{d}_*)$, where
$$
\wh{X}_n(M):= \bigoplus_{r+s = n} \wh{X}_{rs}(M)\qquad\text{and}\qquad \wh{d}_n := \sum^n_{l=1} \wh{d}^l_{0n} + \sum_{r=1}^n \sum^{n-r}_{l=0} \wh{d}^l_{r,n-r}.
$$
Consequently, we have the following result:

\begin{theorem}\label{Hochschild homology} The Hochschild homology $\Ho^K_*(E,M)$, of the $K$-algebra $E$ with coefficients in $M$, is the homology of $(\wh{X}_*(M),\wh{d}_*)$.
\end{theorem}

\begin{remark} If $K$ is a separable $k$-algebra, then $\Ho^K_*(E,M)$ coincide with the absolute Hochschild homology $\Ho_*(E,M)$, of $E$ with coefficients in $M$.
\end{remark}

\begin{remark} If $K=A$, then $(\wh{X}_*(M),\wh{d}_*) = (\wh{X}_{0*}(M),\wh{d}^1_{0*})$.
\end{remark}

\begin{remark} In order to abbreviate notations we will write $\wh{X}_{rs}$ and $\wh{X}_n$ instead of $\wh{X}_{rs}(E)$ and $\wh{X}_n(E)$, respectively.
\end{remark}

For $r,s\ge 0$, let
$$
\wt{X}_{rs}(M):= M\ot_{\!A} E^{\ot_{\!A}^s}\ot A^{\ot^r}\ot.
$$
Similarly as for $\wh{X}_{rs}(M)$ we have canonical identifications
$$
\wt{X}_{rs}(M)\simeq M\ot_{E^e} X'_{rs}.
$$
For $0\le i\le s$, let
$$
\wt{u}_i\colon \wt{X}_{rs}(M)\to \wt{X}_{r,s-1}(M)
$$
be the map induced by $u'_i$. It is easy to see that
\begin{align*}
&\wt{u}_0\bigl([m\ot_{\!A}\ov{\bx}_{1s}\ot\ba_{1r}]\bigr) = [mx_1\ot_{\!A}\ov{\bx}_{2s}\ot \ba_{1r}],\\
&\wt{u}_i\bigl([m\ot_{\!A}\ov{\bx}_{1s}\ot\ba_{1r}]\bigr) = [m\ot_{\!A}\ov{\bx}_{1,i-1} \ot_{\!A} x_ix_{i+1} \ot_{\!A}\ov{\bx}_{i+1,s}\ot \ba_{1r}] \quad\text{for $0<i<s$}
\intertext{and}
&\wt{u}_s\bigl([m\ot_{\!A}\ov{\bx}_{1,s-1}\ot_{\!A} \gamma(v)\ot\ba_{1r}]\bigr) = \sum_l \bigl[\gamma(v^{(l)}) m\ot_{\!A}\ov{\bx}_{1,s-1}\ot\ba_{1r}^{(l)}\bigr],
\end{align*}
where $\sum_l \ba_{1r}^{(l)}\ot_k v^{(l)} := \chi(v\ot_k\ba_{1r})$.

\begin{notations}\label{not3.1} We will use the following notations:

\begin{enumerate}

\smallskip

\item We let $\ov{W}_n\subseteq \ov{W}'_n$ denote the $k$-vector subspace of $M\ot\ov{E}^{\ot^n}\ot$ generated by the classes in $M\ot\ov{E}^{\ot^n}\ot$ of the simple tensors $m\ot \bx_{1n}$ such that
$$
\#(\{j:x_j\notin \ov{A}\cup \mathcal{V}_{\!K}\})=0\quad\text{and}\quad \#(\{j:x_j\notin \ov{A}\cup \mathcal{V}_{\!K}\})\le 1,
$$
respectively.

\smallskip

\item Given a $K$-subalgebra $R$ of $A$, we let $\ov{C}^R_n$ denote the $k$-vector subspace of $M\ot \ov{E}^{\ot^n}\ot$ generated by the classes in $M\ot\ov{E}^{\ot^n}\ot$ of all the simple tensors $m\ot\bx_{1n}$ with some $x_i$ in~$\ov{R}$.

\smallskip

\item Given a $K$-subalgebra $R$ of $A$ and $0\le u\le r$, we let $\wh{X}^{Ru}_{rs}(M)$ denote the $k$-vector subspace of $\wh{X}_{rs}(M)$ generated by the classes in $\wh{X}_{rs}(M)$ of all the simple tensors \hbox{$m\ot_{\!A}\ov{\bx}_{1s}\ot \ba_{1r}$}, with at least $u$ of the $a_j$'s in $\ov{R}$. Moreover, we set $\wh{X}^{Ru}_n(M):=\bigoplus_{r+s=n}\wh{X}^{Ru}_{rs}(M)$.

\smallskip

\item For $j,l\ge 1$, we let
$$
\ov{\chi}_{_{jl}}\colon V^{\ot_k^j}\ot_k\ov{A}^{\ot^l}\to \ov{A}^{\ot^l}\ot_k V^{\ot_k^j}
$$
denote the map induced by the map $\chi_{_{jl}}$ introduced in Notations~\ref{not2.1}.

\smallskip

\item We let
$$
\Sh_{sr}\colon V^{\ot_k^s}\ot_k \ov{A}^{\ot^r}\to \ov{E}^{\ot^{r+s}}
$$
denote the map recursively define by:

\begin{itemize}

\smallskip

\item[-] $\Sh_{s0}:= \gamma^{\ot^s}$,

\smallskip

\item[-] $\Sh_{0r}:= i_{\ov{A}}^{\ot^r}$,

\smallskip

\item[-] If $r,s\ge 1$, then
$$
\quad\qquad\qquad\Sh_{sr}:=\sum_{i=0}^r (-1)^i\left(\Sh_{s-1,i}\ot\gamma\ot i_{\ov{A}}^{\ot^{r-i}} \right)\xcirc \Bigl(\ov{V}^{\ot^{s-1}}\ot \ov{\chi}_{_{1i}}\ot \ov{A}^{\ot^{r-i}}\Bigr),
$$
where $\ov{\chi}_{_{10}} := \ide_V$.

\end{itemize}

\smallskip

\end{enumerate}

\end{notations}

\begin{theorem}\label{formula para wh{d}_1} The following assertions hold:

\begin{enumerate}

\smallskip

\item The morphism $\wh{d}^0\colon \wh{X}_{rs}(M)\to \wh{X}_{r-1,s}(M)$ is $(-1)^s$-times the boundary map of the normalized chain Hochschild complex of the $K$-algebra $A$ with coefficients in $M\ot_{\!A} (E/A)^{\ot_{\!A}^s}$, considered as an $A$-bimodule via the left and right canonical actions.

\smallskip

\item The morphism $\wh{d}^1\colon \wh{X}_{rs}(M)\to \wh{X}_{r,s-1}(M)$ is induced by $\sum_{i=0}^s (-1)^i \wt{u}_i$.

\smallskip

\item Let $R$ be a stable under $\chi$ $K$-subalgebra of $A$. If $\mathcal{F}$ takes its values in $R\hspace{-0.7pt}\ot_{\hspace{-0.3pt} k}\! V$, then
$$
\wh{d}^l\bigl(\wh{X}_{rs}(M)\bigr)\subseteq \wh{X}^{R,l-1}_{r+l-1,s-l}(M)
$$
for each $l\ge 1$.
\end{enumerate}

\end{theorem}

\begin{proof} Item~(1) follows easily from the definition of $d^0$, and Items~(2) and~(3), from Theorem~\ref{formula para d_1}.
\end{proof}

Now it is convenient to note that $A\ot \ov{A}^{\ot^r}\ot M$ is an $E$-bimodule via
$$
a\gamma(v) \cdot (\ba_{0r}\ot m) \cdot a'\gamma(v') : = \sum_l a\ba_{0r}^{(l)} \ot \gamma(v^{(l)}) ma'\gamma(v'),
$$
where $\sum_l \ba_{0r}^{(l)}\ot_k v^{(l)} := \chi(v\ot_k\ba_{0r})$.

\begin{remark}\label{re3.5} Note that
\begin{align*}
&\Ho_r\bigl(\wh{X}_{*s}(M),\wh{d}^0_{*s}\bigr) = \Ho^K_r\bigr(A,M\ot_{\!A} (E/A)^{\ot_{\!A}^s} \bigl)
\intertext{and}
&\Ho_s\bigl(\wh{X}_{r*}(M),\wh{d}^1_{r*}\bigr) = \Ho^A_s\bigr(E,A\ot\ov{A}^{\ot^r}\ot M\bigl).
\end{align*}
\end{remark}

\begin{remark}\label{re3.4} By item~(3) of the above theorem, if $\mathcal{F}$ takes its values in $K\ot_k V$, then $(\wh{X}_*(M),\wh{d}_*)$ is the total complex of the double complex $\bigl(\wh{X}_{**}(M),\wh{d}_{**}^0, \wh{d}_{**}^1\bigr)$.
\end{remark}

\subsection{Comparison maps}\label{morfismos de comparacion en homologia} Let $\bigl(M\ot\ov{E}^{\ot^*}\ot,b_*\bigr)$ be the normalized Hochschild chain complex of the $K$-al\-gebra $E$ with coefficients in $M$. Recall that there is a canonical identification
$$
\bigl(M\ot\ov{E}^{\ot^*}\ot,b_*\bigr) \simeq M\ot_{E^e} \bigl(E\ot\ov{E}^{\ot^*}\ot E,b'_*\bigr).
$$
Let
$$
\wh{\phi}_*\colon (\wh{X}_*(M),\wh{d}_*)\rightarrow \bigl(M\ot\ov{E}^{\ot^*}\ot,b_*\bigr) \quad \text{and}\quad \wh{\psi}_*\colon \bigl(M\ot\ov{E}^{\ot^*}\ot,b_*\bigr)\rightarrow (\wh{X}_*(M),\wh{d}_*)
$$
be the morphisms of complexes induced by $\phi$ and $\psi$ respectively. By Proposition~\ref{homotopia} it is evident that $\wh{\psi}\xcirc\wh{\phi} = \ide$ and $\wh{\phi}\xcirc\wh{\psi}$ is homotopically equivalent to the identity map. An homotopy $\wh{\omega}_{*+1}\colon \wh{\phi}_*\xcirc\wh{\psi}_*\to \ide_*$ is the family of maps
$$
\Bigl(\wh{\omega}_{n+1}\colon M\ot\ov{E}^{\ot^n}\ot\longrightarrow M\ot\ov{E}^{\ot^{n+1}}\ot\Bigr)_{n\ge 0},
$$
induced by $\bigl(\omega_{n+1}\colon E\ot\ov{E}^{\ot^n}\ot E\longrightarrow E\ot \ov{E}^{\ot^{n+1}}\ot E\bigr)_{n\ge 0}$.

\subsection{The filtrations of $\mathbf{\boldsymbol{\bigl(M\ot\ov{E}^{\ot^*}\ot,b_*\bigr)}}$ and $\mathbf{\boldsymbol{(\wh{X}_*(M),\wh{d}_*)}}$}\label{filtraciones en homologia} Let
$$
F^i(\wh{X}_n(M)) := \bigoplus_{0\le s\le i} \wh{X}_{n-s,s}(M).
$$
The complex $(\wh{X}_*(M),\wh{d}_*)$ is filtered by
$$
F^0(\wh{X}_*(M))\subseteq F^1(\wh{X}_*(M))\subseteq F^2(\wh{X}_*(M))\subseteq F^3(\wh{X}_*(M))\subseteq F^4(\wh{X}_*(M))\subseteq\dots.
$$
Using this fact we obtain that there is a convergent spectral sequence
\begin{equation}
E^1_{rs} = \Ho^K_r(A,M\ot_A (E/A)^{\ot_{\!A}^s}) \Longrightarrow \Ho^K_{r+s}(E,M).\label{eq6}
\end{equation}
Let $F^i\bigl(M\ot\ov{E}^{\ot^n}\ot\bigr)$ be the $k$-vector subspace of $M\ot\ov{E}^{\ot^n} \ot$ generated by the classes in $M\ot\ov{E}^{\ot^n}\ot$ of the simple tensors $m\ot\bx_{1n}$ such that at least $n-i$ of the $x_i$'s belong to $\ov{A}$. The normalized Hochschild complex $\bigl(M\ot\ov{E}^{\ot^*}\ot,b_*\bigr)$ is filtered by
$$
F^0\bigl(M\ot\ov{E}^{\ot^*}\ot\bigr)\subseteq F^1\bigl(M\ot\ov{E}^{\ot^*}\ot\bigr)\subseteq F^2\bigl(M\ot\ov{E}^{\ot^*}\ot\bigr) \subseteq\dots.
$$
The spectral sequence associated to this filtration is called the homological Hochs\-child-Serre spectral sequence.

\begin{proposition}\label{phi, psi y omega preservan filtraciones en homologia} The maps $\wh{\phi}$, $\wh{\psi}$ and $\wh{\omega}$ preserve filtrations.
\end{proposition}

\begin{proof} This follows immediately from Proposition~\ref{phi, psi y omega preservan filtraciones}.
\end{proof}

\begin{corollary} The homological Hochschild-Serre spectral sequence is isomorphic to the spectral sequence~\eqref{eq6}.
\end{corollary}

\begin{proof} This follows immediately from Proposition~\ref{phi, psi y omega preservan filtraciones en homologia} and the comments following Remark~\ref{re3.4}.
\end{proof}

\begin{proposition}\label{formula para hat{phi} modulo la filtracion} Let $R$ be a stable under $\chi$ $K$-subalgebra of $A$. If $\mathcal{F}$ takes its values in $R\ot_k V$, then
$$
\wh{\phi}\bigl([m\ot_{\!A} \gamma_{\!A}(\bv_{1i})\ot \ba_{1,n-i}]\bigr) = \bigl[m\ot\Sh(\bv_{1i}\ot_k \ba_{1,n-i})\bigr] + [m\ot \bx],
$$
with $[m\ot \bx]\in F^{i-1}\bigl(M\ot \ov{E}^{\ot^n}\ot\bigr)\cap \ov{W}_n\cap \ov{C}^R_n$. In particular $\wh{\phi}$ preserve filtrations.
\end{proposition}

\begin{proof} This follows immediately from Proposition~\ref{propA.5}.
\end{proof}

In the next proposition we use the following notations:
\begin{align*}
&\ov{R}_i:= F^i\bigl(M\ot\ov{E}^{\ot^n}\ot\bigr)\setminus F^{i-1}\bigr(M \ot\ov{E}^{\ot^n} \ot\bigr)
\intertext{and}
&F_R^j(\wh{X}_n(M)):= F^j(\wh{X}_n(M)) \cap \wh{X}^{R1}_n(M).
\end{align*}

\begin{proposition}\label{propiedades de hat{psi}} Let $R$ be a stable under $\chi$ $K$-subalgebra of $A$ such that $\mathcal{F}$ takes its values in $R\ot_k V$. The following equalities hold:

\begin{enumerate}

\smallskip

\item $\wh{\psi}\bigl([m\ot\gamma(\bv_{1i})\ot\ba_{i+1,n}]\bigr) = [m\ot_{\!A} \gamma_{\!A}(\bv_{1i}) \ot\ba_{i+1,n}]$.

\smallskip

\item If $\bx = [m\ot \bx_{1n}]\in \ov{R}_i\cap \ov{W}_n$ and there is $1\le j\le i$ such that~$x_j\in A$, then~\hbox{$\wh{\psi}(\bx) = 0$}.

\smallskip

\item If $\bx = \bigl[m\ot \gamma(\bv_{1,i-1})\ot a_i\gamma(v_i)\ot \ba_{i+1,n}\bigr]$, then
\begin{align*}
\qquad\qquad \wh{\psi}(\bx) & \equiv \bigl[m\ot_{\!A} \gamma_{\!A}(\bv_{1,i-1})\ot_{\!A} a_i \gamma_{\!A}(v_i)\ot\ba_{i+1,n}\bigr]\\
& + \sum_l \bigl[\gamma(v_i^{(l)}) m\ot_{\!A}\gamma_{\!A}(\bv_{1,i-1})\ot a_i\ot \ba_{i+1,n}^{(l)}\bigr],
\end{align*}
module $F_R^{i-2}(\wh{X}_n(M))$, where $\sum_l \ba_{i+1,n}^{(l)}\ot_k v_i^{(l)} := \ov{\chi}(v_i \ot_k \ba_{i+1,n})$.

\smallskip

\item If $\bx = \bigl[m\ot \gamma(\bv_{1,j-1})\ot a_j\gamma(v_j)\ot \gamma(\bv_{j+1,i}) \ot \ba_{i+1,n}\bigr]$ with $j<i$, then
$$
\qquad\qquad\wh{\psi}(\bx)\equiv \bigl[m\ot_{\!A} \gamma_{\!A}(\bv_{1,j-1})\ot_{\!A} a_j
\gamma_{\!A}(v_j) \ot_{\!A}\gamma_{\!A}(\bv_{j+1,i})\ot\ba_{i+1,n}\bigr],
$$
module $F_R^{i-2}(\wh{X}_n(M)))$.

\smallskip

\item If $\bx = \bigl[m\ot \gamma(\bv_{1,i-1}) \ot \ba_{i,j-1}\ot a_j\gamma(v_j) \ot \ba_{j+1,n}\bigr]$ with $j>i$, then
$$
\wh{\psi}(\bx)\equiv  \bigl[\gamma(v_j^{(l)})m\ot_{\!A}\gamma_{\!A}(\bv_{1,i-1})\ot\ba_{ij}\ot \ba_{j+1,n}^{(l)}\bigr],
$$
module $F_R^{i-2}(\wh{X}_n(M))$, where $\sum_l \ba_{j+1,n}^{(l)}\ot_k v_j^{(l)} := \ov{\chi}(v_j \ot_k \ba_{j+1,n})$.

\smallskip

\item If $\bx = [m\ot \bx_{1n}]\in \ov{R}_i\cap \ov{W}'_n$ and there exists $1\le j_1 < j_2\le n$ such that $x_{j_1}\in A$ and $x_{j_2}\in \mathcal{V}_{\!K}$, then $\wh{\psi}(\bx) \in F_R^{i-2}(\wh{X}_n(M))$.

\end{enumerate}
\end{proposition}

\begin{proof} This follows immediately from Proposition~\ref{propA.7}.
\end{proof}

\begin{proposition}\label{sobre la imagen de hat{w}} If $\bx = [m\ot\bx_{1n}]\in F^i\bigl(M\ot \ov{E}^{\ot^n}\ot\bigr) \cap \ov{W}'_n$, then
$$
\wh{\omega}(\bx) = [m\ot\byy]\quad\text{with}\quad[m\ot\byy]\in F^i\bigl(M\ot \ov{E}^{\ot^{n+1}} \ot\bigr) \cap \ov{W}_{n+1}.
$$
\end{proposition}

\begin{proof} This follows immediately from Proposition~\ref{propA.9}.
\end{proof}

\section{Hochschild cohomology of a Brzezi\'nski's crossed product} \label{coHochschild-Brzezinski}
Let $M$ be an $E$-bimodule. Since $(X_*,d_*)$ is a $\Upsilon$-projective resolution of $E$, the Hochschild cohomology of the $K$-algebra $E$ with coefficients in $M$ is the cohomology of the cochain complex $\Hom_{E^e}\bigl((X_*,d_*),M\bigr)$.

For each $s\ge 0$, we let $\Hom_{\!A}((E/A)^{\ot_{\!A}^s},M)$ denote the abelian group of left $A$-linear maps from $(E/A)^{\ot_{\!A}^s}$ to $M$. Note that $\Hom_{\!A}((E/A)^{\ot_{\!A}^s}, M)$ is an $A$-bimodule via
$$
a\alpha(\ov{\bx}_{1s}):= \alpha(\ov{\bx}_{1s}a)\quad\text{and}\quad \alpha a(\ov{\bx}_{1s}):= \alpha(\ov{\bx}_{1s})a.
$$
For each $r,s\ge 0$, write
$$
\wh{X}^{rs}(M):= \Hom_{(A,K)}\bigl((E/A)^{\ot_{\!A}^s} \ot \ov{A}^{\ot^r},M \bigr)\simeq \Hom_{K^e}\bigl(\ov{A}^{\ot^r},\Hom_{\!A}((E/A)^{\ot_{\!A}^s},M)\bigr),
$$
It is easy to check that the $k$-linear map
$$
\zeta^{rs}\colon \Hom_{E^e}\bigl(X_{rs},M\bigr)\to \wh{X}^{rs}(M),
$$
given by
$$
\zeta(\alpha)(\ov{\bx}_{1s}\ot \ba_{1r}):= \alpha(1\ot_{\!A}\ov{\bx}_{1s}\ot \ba_{1r}\ot 1),
$$
is an isomorphism. For each $l\le s$, let
$$
\wh{d}_l^{rs}\colon \wh{X}^{r+l-1,s-l}(M) \to \wh{X}^{rs}(M)
$$
be the map induced by $\Hom_{E^e}(d^l_{rs},M)$. Via the above identifications the complex
$$
\Hom_{E^e}((X_*,d_*),M)
$$
becomes $(\wh{X}^*(M),\wh{d}^*)$, where
$$
\wh{X}^n(M):= \bigoplus_{r+s = n} \wh{X}^{rs}(M)\qquad\text{and}\qquad \wh{d}^n:= \sum^n_{l=1} \wh{d}_l^{0n} + \sum_{r=1}^n \sum^{n-r}_{l=0} \wh{d}_l^{r,n-r}.
$$
Consequently, we have the following result:

\begin{theorem}\label{Hochschild cohomology} The Hochschild cohomology $\Ho_K^*(E,M)$, of the $K$-algebra $E$ with coefficients in $M$, is the cohomology of $(\wh{X}^*(M),\wh{d}^*)$.
\end{theorem}

\begin{remark} If $K$ is a separable $k$-algebra, then $\Ho_K^*(E,M)$ coincide with the absolute Hochschild cohomology $\Ho^*(E,M)$, of $E$ with coefficients in $M$.
\end{remark}

\begin{remark} If $K=A$, then $(\wh{X}^*(M),\wh{d}^*) = (\wh{X}^{0*}(M),\wh{d}_1^{0*})$.
\end{remark}

\begin{remark} In order to abbreviate notations we will write $\wh{X}^{rs}$ and $\wh{X}^n$ instead of $\wh{X}^{rs}(E)$ and $\wh{X}^n(E)$, respectively.
\end{remark}

For each $s\ge 0$, we let $\Hom_{\!A}(E^{\ot_{\!A}^s},M)$ denote the abelian group of left $A$-linear maps from $E^{\ot_{\!A}^s}$ to $M$. Note that $\Hom_{\!A}(E^{\ot_{\!A}^s},M)$ is an $A$-bimodule via
$$
a\alpha(\ov{\bx}_{1s}):= \alpha(\ov{\bx}_{1s}a)\quad\text{and}\quad \alpha a(\ov{\bx}_{1s}):= \alpha(\ov{\bx}_{1s})a.
$$
For $r,s\ge 0$, let
$$
\wt{X}^{rs}(M) := \Hom_{(A,K)}\bigl(E^{\ot_{\!A}^s} \ot A^{\ot^r},M \bigr)\simeq \Hom_{K^e}\bigl(A^{\ot^r},\Hom_{\!A}(E^{\ot_{\!A}^s},M)\bigr).
$$
Similarly as for $\wh{X}^{rs}(M)$, we have canonical identifications
$$
\wt{X}^{rs}(M)\simeq \Hom_{E^e}\bigl(X'_{rs},M\bigr).
$$
For $0\le i\le s$, let
$$
\wt{u}^i\colon \wt{X}^{r,s-1}(M)\to \wt{X}^{rs}(M)
$$
be the map induced by $u'_i$. It is easy to see that
\begin{align*}
&\wt{u}^0(\alpha)(\ov{\bx}_{1s}\ot\ba_{1r}) = x_1\alpha(\ov{\bx}_{2s}\ot\ba_{1r}) ,\\
&\wt{u}^i(\alpha)(\ov{\bx}_{1,s+1}\ot\ba_{1r}) = \alpha(\ov{\bx}_{1,i-1}\ot_{\!A} x_ix_{i+1} \ot_{\!A}\ov{\bx}_{i+1,s}\ot \ba_{1r}) \quad\text{for $0<i<s$}
\intertext{and}
&\wt{u}^s(\alpha)(\ov{\bx}_{1,s-1}\ot_{\!A}\gamma(v)\ot\ba_{1r}) = \sum_l \alpha(\ov{\bx}_{1,s-1} \ot \ba_{1r}^{(l)})\gamma(v^{(l)}),
\end{align*}
where $\sum_l \ba_{1r}^{(l)}\ot_k v^{(l)} := \chi(v\ot_k\ba_{1r})$.

\begin{notations}\label{not4.3} We will use the following notations:

\begin{enumerate}

\smallskip

\item We let $F^i(\ov{E}^{\ot^n})$ denote the $K$-bimodule of $\ov{E}^{\ot^n}$ generated by the simple tensors $\bx_{1n}$ such that at least $n-i$ of the $x_i$'s belong to $A$.

\smallskip

\item We let $W^{\mathfrak{r}}_n$ denote the $K$-subbimodule of $\ov{E}^{\ot^n}$ generated by the simple tensors $\bx_{1n}$ such that $\#(\{j:x_j\notin \ov{A}\cup \mathcal{V}_{\!K}\})=0$.

\smallskip

\item Given a $K$-subalgebra $R$ of $A$, we let $C^{R\mathfrak{r}}_n$ denote the $K$-subbimodule of $\ov{E}^{\ot^n}$ generated by all the simple tensors $\bx_{1n}$ with some $x_i$ in~$\ov{R}$.

\smallskip

\item Given a $K$-subalgebra $R$ of $A$ and $0\le u\le r$, we let $\wh{X}_{Ru}^{rs}(M)$ denote the $k$-vector subspace of $\wh{X}^{rs}(M)$ consisting of all the $(A,K)$-linear maps
    $$
    \alpha\colon (E/A)^{\ot_{\!A}^s} \ot \ov{A}^{\ot^r}\to M,
    $$
    that factorize throughout the $(A,K)$-subbimodule
    $$
    \wh{X}^{R\mathfrak{r}u}_{r+u,s-u-1}
    $$
    of $(E/A)^{\ot_{\!A}^{s-u-1}} \ot \ov{A}^{\ot^{r+u}}$ generated by the simple tensors $\ov{\bx}_{1,s-u-1}\ot \ba_{1,r+u}$, with at least $u$ of the $a_j$'s in $\ov{R}$.

\end{enumerate}

\end{notations}

\begin{theorem}\label{formula para wh{d} 1 en cohomologia} The following assertions hold:

\begin{enumerate}

\smallskip

\item The morphism $\wh{d}_0\colon \wh{X}^{r-1,s}(M)\to \wh{X}^{rs}(M)$ is $(-1)^s$-times the coboundary map of the normalized cochain Hochschild complex of $A$ with coefficients in $\Hom_{\!A}((E/A)^{\ot_{\!A}^s},M)$, considered as an $A$-bimodule as at the beginning of this section.

\smallskip

\item The morphism $\wh{d}_1\colon \wh{X}^{r,s-1}(M)\to \wh{X}^{rs}(M)$ is induced by $\sum_{i=0}^s (-1)^i \wt{u}^i$.

\smallskip

\item Let $R$ be a stable under $\chi$ $K$-subalgebra of $A$. If $\mathcal{F}$ takes its values in $R\hspace{-0.7pt}\ot_{\hspace{-0.3pt} k}\! V$, then
$$
\wh{d}_l\bigl(\wh{X}^{r+l-1,s-l}(M)\bigr)\subseteq \wh{X}_{R,l-1}^{rs}(M),
$$
for all $l\ge 1$.
\end{enumerate}

\end{theorem}

\begin{proof} Item~(1) follows easily from the definition of $d^0$, and items~(2) and~(3), from Theorem~\ref{formula para d_1}.
\end{proof}

For each $r\ge 0$, we let $\Hom_{\!K}(A\ot \ov{A}^{\ot^r},M)$ denote the abelian group of right $K$-linear maps from $\ov{A}^{\ot^r}$ to $M$. Note that $\Hom_{\!K}(A\ot \ov{A}^{\ot^r},M)$ is an $E$-bimodule via
$$
\bigl(a\gamma(v) \cdot \alpha  \cdot a'\gamma(v')\bigr)(\ba_{0r}): = \sum_l \gamma(v^{(l)}) \alpha \bigl(a\ba_{0r}^{(l)}\bigr) a'\gamma(v'),
$$
where $\sum_l \ba_{0r}^{(l)}\ot_k v^{(l)} := \chi(v\ot_k\ba_{0r})$.

\begin{remark} Note that
\begin{align*}
&\Ho^r\bigl(\wh{X}^{*s}(M),\wh{d}_0^{*s}\bigr) = \Ho_K^r\bigr(A, \Hom_{\!A} ((E/A)^{\ot_{\!A}^s},M \bigl)
\intertext{and}
&\Ho^s\bigl(\wh{X}^{r*}(M),\wh{d}_1^{r*}\bigr) = \Ho_A^s\bigr(E,\Hom_{\!K}(A\ot \ov{A}^{\ot^r},M)\bigl).
\end{align*}
\end{remark}

\begin{remark}\label{re4.4} By item~(3) of the above theorem, if $\mathcal{F}$ takes its values in $K\ot_k V$, then $(\wh{X}^*(M),\wh{d}^*)$ is the total complex of the double complex $\bigl(\wh{X}^{**}(M),\wh{d}^{**}_0, \wh{d}^{**}_1\bigr)$.
\end{remark}

\subsection{Comparison maps}\label{morfismos de comparacion en cohomologia} Let $\bigl(\Hom_{K^e}(\ov{E}^{\ot^*},M),b^*\bigr)$ be the normalized Hochschild cochain complex of the $K$-algebra $E$ with coefficients in $M$. Recall that there is a canonical identification
$$
\bigl(\Hom_{K^e}(\ov{E}^{\ot^*},M),b^*\bigr) \simeq \Hom_{E^e}\bigl((E\ot\ov{E}^{\ot^*}\ot E,b'_*),M\bigr).
$$
Let
\begin{align*}
& \wh{\phi}^*\colon \bigl(\Hom_{K^e}(\ov{E}^{\ot^*},M),b^*\bigr)\longrightarrow (\wh{X}^*(M),\wh{d}^*)
\intertext{and}
& \wh{\psi}^*\colon (\wh{X}^*(M),\wh{d}^*)\longrightarrow \bigl(\Hom_{K^e}(\ov{E}^{\ot^*},M),b^* \bigr)
\end{align*}
be the morphisms of complexes induced by $\phi$ and $\psi$ respectively. By Proposition~\ref{homotopia} it is evident that $\wh{\phi}\xcirc\wh{\psi} = \ide$ and $\wh{\psi}\xcirc\wh{\phi}$ is homotopically equivalent to the identity map. An homotopy $\wh{\omega}^{*+1}\colon \wh{\psi}^*\xcirc\wh{\phi}^*\to \ide^*$ is the family of maps
$$
\Bigl(\wh{\omega}^{n+1}\colon \Hom_{K^e}(\ov{E}^{\ot^{n+1}},M)\longrightarrow \Hom_{K^e}(\ov{E}^{\ot^n},M)\Bigr)_{n\ge 0},
$$
induced by $\bigl(\omega_{n+1}\colon E\ot\ov{E}^{\ot^n}\ot E\longrightarrow E\ot \ov{E}^{\ot^{n+1}}\ot E\bigr)_{n\ge 0}$.

\subsection{The filtrations of $\mathbf{\boldsymbol{\bigl(\Hom_{K^e}(\ov{E}^{\ot^*},M),b^* \bigr)}}$ and $\mathbf{\boldsymbol{(\wh{X}^*(M),\wh{d}^*)}}$} Let
$$
F_i(\wh{X}^n(M)) := \bigoplus_{s\ge i} \wh{X}^{n-s,s}(M).
$$
The complex $(\wh{X}^*(M),\wh{d}^*)$ is filtered by
$$
F_0(\wh{X}^*(M))\supseteq F_1(\wh{X}^*(M))\supseteq F_2(\wh{X}^*(M))\supseteq F_3(\wh{X}^*(M))\supseteq F_4(\wh{X}^*(M))\supseteq\dots.
$$
Using this fact we obtain that there is a convergent spectral sequence
\begin{equation}
E_1^{rs} = \Ho_K^r(A,\Hom_A((E/A)^{\ot_{\!A}^s},M)) \Longrightarrow \Ho_K^{r+s}(E,M). \label{eq7}
\end{equation}
Let $F_i\bigl(\Hom_{K^e}(\ov{E}^{\ot^*},M)\bigr)$ be the $k$-submodule of $\Hom_{K^e}(\ov{E}^{\ot^*},M)$ consisting of all the maps $\alpha\in \Hom_{K^e}(\ov{E}^{\ot^*},M)$, such that $\alpha\bigl(F^i(\ov{E}^{\ot^*})\bigr) = 0$. The normalized Hochschild complex $\bigl(\Hom_{K^e}(\ov{E}^{\ot^*},M),b^*\bigr)$ is filtered by
\begin{equation}
F_0\bigl(\Hom_{K^e}(\ov{E}^{\ot^*},M)\bigr)\supseteq F_1\bigl(\Hom_{K^e}(\ov{E}^{\ot^*},M) \bigr) \supseteq\dots.\label{eq17}
\end{equation}
The spectral sequence associated to this filtration is called the cohomological Hochs\-child-Serre spectral sequence.

\begin{proposition}\label{phi, psi y omega preservan filtraciones en cohomologia} The maps $\wh{\phi}$, $\wh{\psi}$ and $\wh{\omega}$ preserve filtrations.
\end{proposition}

\begin{proof} This follows immediately from Proposition~\ref{phi, psi y omega preservan filtraciones}.
\end{proof}

\begin{corollary} The cohomological Hochschild-Serre spectral sequence is isomorphic to the spectral sequence~\eqref{eq7}.
\end{corollary}

\begin{proof} This follows immediately from Proposition~\ref{phi, psi y omega preservan filtraciones en cohomologia} and the comments following Remark~\ref{re4.4}.
\end{proof}

\begin{corollary} When $M=E$ the spectral sequence~\eqref{eq7} is multiplicative.
\end{corollary}

\begin{proof} This follows from the previous corollary and the fact that the filtration~\ref{eq17} satisfies $F_m\smile F_n\subseteq F_{m+n}$, where
\[
(\beta\smile \beta')(\bx_{1,m+n}) := \beta(\bx_{1m}) \beta'(\bx_{m+1,m+n}),
\]
for $\beta\in \Hom_{K^e}(\ov{E}^{\ot^m},E)$ and $\beta'\in\Hom_{K^e}(\ov{E}^{\ot^n},E)$,
\end{proof}

\begin{proposition}\label{formula para hat{phi} modulo la filtracion en cohomologia} Let $R$ be a stable under $\chi$ $K$-subalgebra of $A$. Assume that $\mathcal{F}$ takes its values in $R\ot_k V$. Then, for each $\beta\in \Hom_{K^e}(\ov{E}^{\ot^n},M)$, we have
$$
\wh{\phi}(\beta)(\gamma_{\!A}(\bv_{1i})\ot \ba_{1,n-i}) = \beta(\Sh(\bv_{1i}\ot_k \ba_{1,n-i})) + \beta(\bx),
$$
with $\bx\in F^{i-1}(\ov{E}^{\ot^n})\cap W^{\mathfrak{r}}_n\cap C^{R\mathfrak{r}}_n$.
\end{proposition}

\begin{proof} This follows immediately from Proposition~\ref{propA.5}.
\end{proof}

In the next proposition $R^{\mathfrak{r}}_i$ denotes $F^i\bigl(\ov{E}^{\ot^n}\bigr) \setminus F^{i-1}\bigr(\ov{E}^{\ot^n}\bigr)$.

\begin{proposition}\label{propiedades de hat{psi} en cohomologia} For all $\alpha\in\wh{X}^{n-i,i}(M)$, the following equalities hold:

\begin{enumerate}

\smallskip

\item $\wh{\psi}(\alpha)\bigl(\gamma(\bv_{1i})\ot\ba_{i+1,n}\bigr) = \alpha\bigl( \gamma_{\!A}(\bv_{1i})\ot\ba_{i+1,n}\bigr)$.

\smallskip

\item If $\bx_{1n}\in R^{\mathfrak{r}}_i\cap W^{\mathfrak{r}}_n$ and there is $j\le i$ such that $x_j\in A$, then $\wh{\psi}(\alpha)(\bx_{1n}) = 0$.

\smallskip

\end{enumerate}
\end{proposition}

\begin{proof} This follows immediately from items~(1) and~(2) of Proposition~\ref{propA.7}.
\end{proof}

\section{The cup and cap products for Brzezi\'nski's crossed products}
The aim of this section is to compute the cup product of $\HH_K^*(E)$ in terms of $(\wh{X}^*,\wh{d}^*)$ and the cap product of $\Ho^K_*(E,M)$ in terms of $(\wh{X}^*,\wh{d}^*)$ and $(\wh{X}_*(M),\wh{d}_*)$. First of all recall that by definition

\begin{itemize}

\smallskip

\item[-] the cup product of $\HH_K^*(E)$ is given in terms of $\bigl(\Hom_{K^e} \bigl(\ov{E}^*,E\bigr),b^*\bigr)$, by
\[
(\beta\smile \beta')(\bx_{1,m+n}) := \beta(\bx_{1m}) \beta'(\bx_{m+1,m+n}),
\]
for $\beta\in \Hom_{K^e}(\ov{E}^{\ot^m},E)$ and $\beta'\in\Hom_{K^e}(\ov{E}^{\ot^n},E)$,

\smallskip

\item[-] the cap product
$$
\Ho^K_n(E,M)\times \HH_K^m(E)\to \Ho^K_{n-m}(E,M)\qquad (m\le n),
$$
is defined in terms of $\bigl(M\ot\ov{E}^{\ot^*}\ot,b_*\bigr)$ and $\bigl(\Hom_{K^e} \bigl(\ov{E}^{\ot^*},E\bigr),b^*\bigr)$, by
$$
\ov{m\ot\bx_{1n}}\smallfrown \beta := \ov{m \beta(\bx_{1m})\ot\bx_{m+1,n}},
$$
where $\beta\in \Hom_{K^e}(\ov{E}^m,E)$. When $m>n$ we set $\ov{m\ot\bx_{1n}}\smallfrown \beta :=0$.
\end{itemize}

\begin{definition}\label{producto bullet} For $\alpha\in \wh{X}^{rs}$ and $\alpha'\in \wh{X}^{r's'}$ we define $\alpha\bullet \alpha'\in \wh{X}^{r+r',s+s'}$ by
$$
(\alpha\bullet\alpha')\bigl(\gamma_A(\bv_{1s''})\ot \ba_{1r''}\bigr) := \sum_i (-1)^{s'r} \alpha\bigl(\gamma_A(\bv_{1s})\ot \ba_{1r}^{(i)}\bigr) \alpha'\bigl(\gamma_A(\bv_{s+1,s''}^{(i)}) \ot \ba_{r+1,r''}\bigr),
$$
where $r'' = r+r'$, $s'' = s+s'$ and $\sum_i \ba_{1r}^{(i)}\ot_k \bv_{s+1,s''}^{(i)} := \ov{\chi}\bigl(\bv_{s+1,s''}\ot \ba_{1r}\bigr)$.
\end{definition}

\begin{theorem}\label{th cup product} Let $\alpha\in \wh{X}^{rs}$, $\alpha'\in \wh{X}^{r's'}$ and $n := r+r'+s+s'$. Let $R$ be a stable under $\chi$ $K$-subalgebra of $A$. If $\mathcal{F}$ takes its values in $R\ot_k V$, then
$$
\wh{\phi}\bigl(\wh{\psi}(\alpha)\smile \wh{\psi}(\alpha')\bigr) = \alpha\bullet\alpha' \qquad \text{module }\, \bigoplus_{i>s+s'}\wh{X}^{n-i,i}_{R(1)},
$$
where $\wh{X}_{R(1)}^{n-i,i}$ denotes the $k$-vector subspace of $\wh{X}^{n-i,i}$ consisting of all the $(A,K)$-li\-near maps
$$
\alpha\colon (E/A)^{\ot_{\!A}^i} \ot \ov{A}^{\ot^{n-i}}\to E,
$$
that factorize throughout $A\ot \bigl(W^{\mathfrak{r}}_n\cap C^{R\mathfrak{r}}_n\bigr)$, where $W^{\mathfrak{r}}_n$ and $C^{R\mathfrak{r}}_n$ are as in Notation~\ref{not4.3}.
\end{theorem}

\begin{proof} Let $r'',s''\in \mathds{N}$ such that $r''+s'' = n$, and let $\gamma_A(\bv_{1s''}) \ot \ba_{1r''}\in X_{r''s''}$. Set $T :=  \Sh(\bv_{1s''} \ot_k \ba_{1r''})$. By Proposition~\ref{formula para hat{phi} modulo la filtracion en cohomologia},
$$
\wh{\phi}\bigl(\wh{\psi}(\alpha)\smile \wh{\psi}(\alpha')\bigr)\bigl(\gamma_A(\bv_{1s''})\ot \ba_{1r''}\bigr) = \bigl(\wh{\psi}(\alpha)\smile \wh{\psi}(\alpha')\bigr)\bigl(T + \bx\bigr),
$$
with $\bx\in F^{s''-1}(\ov{E}^{\ot^n})\cap W^{\mathfrak{r}}_n\cap C^{R\mathfrak{r}}_n$. Since, by Theorem~\ref{propiedades de hat{psi} en cohomologia},

\begin{itemize}

\smallskip

\item[-] if $s''\le s+s'$, then $\bigl(\wh{\psi}(\alpha)\smile \wh{\psi}(\alpha') \bigr) (\bx) = 0$,

\smallskip

\item[-] if $s''\ne s+s'$, then$\bigl(\wh{\psi}(\alpha)\smile \wh{\psi}(\alpha')\bigr)(T) = 0$,

\smallskip

\item[-] if $s'' = s+s'$, then
$$
\qquad\qquad\!\! \bigl(\wh{\psi}(\alpha)\!\smile\! \wh{\psi}(\alpha')\bigr)(T)\! =\! \sum_i (-1)^{s'r} \alpha\bigl(\gamma_A(\bv_{1s})\ot \ba_{1r}^{(i)}\bigr) \alpha'\bigl(\gamma_A(\bv_{s+1,s''}^{(i)})\ot \ba_{r+1,r''}\bigr),
$$
where $\sum_i \ba_{1r}^{(i)}\ot_k \bv_{s+1,s''}^{(i)} := \ov{\chi}\bigl(\bv_{s+1,s''}\ot \ba_{1r}\bigr)$,

\smallskip

\end{itemize}
the result follows.
\end{proof}

\begin{corollary} If $\mathcal{F}$ takes its values in $K\ot_k V$, then the cup product of $\HH_K^*(E)$ is induced by the operation $\bullet$ in $(\wh{X}^*,\wh{d}^*)$.
\end{corollary}

\begin{proof} It follows from Theorem~\ref{th cup product}, since $\wh{X}^{n-i,i}_{K(1)} = 0$ for all $i$.
\end{proof}

\begin{definition}\label{accion bullet} Let $[m\ot_{\!A}\gamma_A(\bv_{1s})\ot \ba_{1r}] \in \wh{X}_{rs}(M)$ and $\alpha\in \wh{X}^{r's'}$. If $r'\le r$ and $s'\le s$, then we define $[m\ot_{\!A}\gamma_A(\bv_{1s})\ot \ba_{1r}]\bullet \alpha\in \wh{X}_{r-r',s-s'}(M)$ by
$$
[m\ot_{\!A}\gamma_A(\bv_{1s})\ot \ba_{1r}]\bullet \alpha\! :=\! \sum_i (-1)^{s''r'} \bigl[m \alpha\bigl(\gamma_A(\bv_{1s'})\ot \ba_{1r'}^{(i)}\bigr)\ot_{\!A}\gamma_A(\bv_{s'+1,s}^{(i)})\ot \ba_{r'+1,r}\bigr],
$$
where
$$
s'':=s-s'\quad\text{and}\quad \sum_i \ba_{1r'}^{(i)}\ot_k \bv_{s'+1,s}^{(i)} := \ov{\chi}\bigl(\bv_{s'+1,s}\ot \ba_{1r'}\bigr).
$$
Otherwise $[m\ot_{\!A}\gamma_A(\bv_{1s})\ot \ba_{1r}]\bullet \alpha :=0$.
\end{definition}

\begin{theorem}\label{th cap product} Let $[m\ot_{\!A}\gamma_A(\bv_{1s})\ot\ba_{1r}]\in \wh{X}_{rs}(M)$, $\alpha\in \wh{X}^{r's'}$ and n := r+s-r'-s'. Let $R$ be a stable under $\chi$ $K$-subalgebra of $A$. If $\mathcal{F}$ takes its values in $R\ot_k V$, then
$$
\wh{\psi}\bigl(\wh{\phi}([m\ot_{\!A}\gamma_A(\bv_{1s})\ot\ba_{1r}])\smallfrown \wh{\psi}(\alpha)\bigr)  = [m\ot_{\!A}\gamma_A(\bv_{1s})\ot\ba_{1r}] \bullet\alpha
$$
module
$$
\bigoplus_{i<s-s'}\Bigl(\wh{X}_{n-i,i}^{R1}(M) + M\alpha\bigl(X^{R\mathfrak{r}}_{r's'}\bigr) \ot_{\!A} (E/A)^{\ot_{\!A}^{s-s'}} \ot \ov{A}^{\ot^{r-r'}}\Bigr),
$$
where $X^{R\mathfrak{r}}_{r's'}$ denotes the $k$-vector subspace of $(E/A)^{\ot_{\!A}^{s'}} \ot \ov{A}^{\ot^{r'}}$ generated by all the simple tensors $m\ot_{\!A}\ov{\bx}_{1s'}\ot \ba_{1r'}$, with at least $1$ of the $a_j$'s in $\ov{R}$.
\end{theorem}

\begin{proof} By Proposition~\ref{formula para hat{phi} modulo la filtracion},
$$
\wh{\psi}\bigl(\wh{\phi}([m\ot_{\!A}\gamma_A(\bv_{1s})\ot\ba_{1r}])\smallfrown \wh{\psi}(\alpha)\bigr) = \wh{\psi}\bigl(\bigl([m\ot T] + [m\ot \bx_{1,r+s}]\bigr) \smallfrown \wh{\psi}(\alpha)\bigr),
$$
where
$$
T\!:=\! \Sh(\bv_{1s}\ot_k\! \ba_{1r})\quad\text{and}\quad[m\ot \bx_{1,r+s}]\in F^{s-1}\bigl(M\ot \ov{E}^{\ot^{r+s}}\ot\bigr)\cap \ov{W}_{r+s}\cap \ov{C}^R_{r+s}.
$$
Moreover, by Proposition~\ref{propiedades de hat{psi} en cohomologia}, we know that

\begin{itemize}

\smallskip

\item[-] If $s'>s$ or $r'>r$, then $[m\ot T]\smallfrown \wh{\psi}(\alpha) = 0$.

\smallskip

\item[-] If $s'\le s$ and $r'\le r$, then
$$
\qquad\qquad [m\ot T]\smallfrown \wh{\psi}(\alpha) = \sum_i (-1)^{r's'} m\ot \alpha \bigl(\bv_{1s'}\!\ot\! \ba_{1r'}^{(i)}\bigr)\ot \Sh\bigl(\bv_{s'+1,s}^{(i)}\!\ot_k\! \ba_{r'+1,r}\bigr),
$$
where $\sum_i \ba_{1r'}^{(i)}\ot_k \bv_{s'+1,s}^{(i)} := \ov{\chi}(\bv_{s'+1,s} \ot_k\ba_{1r'})$.

\smallskip

\item[-] If $s'\ge s$, then $[m\ot \bx_{1,r+s}]\smallfrown \wh{\psi}(\alpha) = 0$.

\smallskip

\item[-] If $s'<s$, then
$$
\qquad [m\ot \bx_{1,r+s}]\smallfrown \wh{\psi}(\alpha)\in F^{s-s'-1}\bigl(M\ot \ov{E}^{\ot^n}\ot\bigr)\cap \ov{W}_n\cap \bigl(\ov{C}^R_n + G_n\bigr),
$$
where $G_n := M\wh{\psi}(\alpha) \bigl(\ov{C}^{R\mathfrak{r}}_{r'+s'}\bigr)\ot \ov{E}^{\ot^n}$.

\smallskip

\end{itemize}
Now, in order to finish the proof it suffices to apply items~(1) and~(2) of Proposition~\ref{propiedades de hat{psi}}.
\end{proof}

\begin{corollary} If $\mathcal{F}$ takes its values in $K\ot_k V$, then in terms of the complexes $(\wh{X}_*(M),\wh{d}_*)$ and $(\wh{X}^*,\wh{d}^*)$, the cap product
$$
\Ho^K_n(E,M)\times \HH_K^m(E)\to \Ho^K_{n-m}(E,M),
$$
is induced by $\bullet$.
\end{corollary}

\begin{proof} It follows immediately from the previous theorem.
\end{proof}

\section{Cyclic homology of a Brzezi\'nski's crossed product}\label{cyclic-Brzezinski}
The aim of this section is to construct a mixed complex giving the cyclic homology of $E$, whose underlying Hochschild complex is $(\wh{X}_*,\wh{d}_*)$.

\begin{lemma}\label{B circ wh{omega} circ B circ wh{phi}=0} Let $B_*\colon E\ot\ov{E}^{\ot^*}\ot\to E\ot\ov{E}^{\ot^{*+1}}\ot$ be the Connes operator. The composition $B \xcirc \wh{\omega}\xcirc B\xcirc \wh{\phi}$ is the zero map.
\end{lemma}

\begin{proof} Let $\bx:= [x_0\ot_{\!A}\gamma_{\!A}(\bv_{1i})\ot\ba_{1,n-i}]\in \wh{X}_{n-i,i}$. By Proposition~\ref{formula para hat{phi} modulo la filtracion}, we know that
$$
\wh{\phi}(\bx)\in F^i\bigl(E\ot \ov{E}^{\ot^n}\ot \bigr)\cap \ov{W}_n.
$$
Hence
$$
B\xcirc \wh{\phi}(\bx)\in \bigl(K\ot \ov{E}^{\ot^{n+1}}\ot\bigr) \cap F^{i+1}\bigl(E\ot
\ov{E}^{\ot^{n+1}}\ot \bigr) \cap \ov{W}'_{n+1},
$$
and so, by Proposition~\ref{sobre la imagen de hat{w}}
$$
\wh{\omega}\xcirc B\xcirc \wh{\phi}(\bx)\in \bigl(K\ot\ov{E}^{\ot^{n+1}}\ot\bigr)\cap F^{i+1} \bigl(E\ot\ov{E}^{\ot^{n+1}}\ot\bigr) \cap \ov{W}_{n+2}\subseteq \ker{B},
$$
as desired.
\end{proof}

For each $n\ge 0$, let $\wh{D}_n\colon \wh{X}_n\to \wh{X}_{n+1}$ be the map $\wh{D} := \wh{\psi}\xcirc B \xcirc \wh{\phi}$.

\begin{theorem}\label{complejo mezclado que da la homologia ciclica} $\bigl(\wh{X}_*,\wh{d}_*,\wh{D}_*\bigr)$ is a mixed complex that gives the Hochschild, cyclic, negative and periodic homologies of the $K$-algebra $E$. Moreover we have chain complexes maps
\begin{equation*}
\xymatrix{{}\save[]+<-44pt,0pt>\Drop{\Tot\bigl(\BP(\wh{X}_*,\wh{d}_*,\wh{D}_*)\bigr)}\restore \ar@<-1ex>[rr]_-{\wh{\Phi}} && {{}\save[]+<63pt,0pt> \Drop{\Tot\bigl(\BP(E\ot\ov{E}^{\ot^*}\ot, b_*,B_*)\bigr)}} \restore \ar@<-1ex>[ll]_-{\wh{\Psi}}},
\end{equation*}
given by
$$
\wh{\Phi}_n(\bx u^i):= \wh{\phi}(\bx)u^i + \wh{\omega}\xcirc B\xcirc\wh{\phi}(\bx)u^{i-1}\quad \text{and}\quad \wh{\Psi}_n (\bx u^i):= \sum_{j\ge 0} \wh{\psi}\xcirc (B\xcirc\wh{\omega})^j (\bx) u^{i-j}.
$$
These maps satisfy $\wh{\Psi}\xcirc \wh{\Phi} = \ide$ and $\wh{\Phi}\xcirc\wh{\Psi}$ is homotopically equivalent to the identity map. A homotopy $\wh{\Omega}_{*+1}\colon\wh{\Phi}_* \xcirc\wh{\Psi}_*\to \ide_*$ is given by
$$
\wh{\Omega}_{n+1}(\bx u^i):= \sum_{j\ge 0}\wh{\omega}\xcirc (B\xcirc\wh{\omega})^j(\bx)u^{i-j}.
$$
\end{theorem}

\begin{proof} This result generalizes~\cite[Theorem~2.4]{C-G-G}, and the proof given in that paper works in our setting.
\end{proof}

\begin{remark} If $K$ is a separable $k$-algebra, then $\bigl(\wh{X}_*,\wh{d}_*, \wh{D}_*\bigr)$ is a mixed complex that gives the Hochschild, cyclic, negative and periodic absolute homologies of $E$.
\end{remark}

In the next proposition we use the notation $F_R^j(\wh{X}_n) := F^j(\wh{X}_n)\cap \wh{X}^{R1}_n(E)$ introduced above Proposition~\ref{propiedades de hat{psi}}.

\begin{proposition}\label{Connes operator} Let $R$ be a stable under $\chi$ $K$-subalgebra of $A$ such that $\mathcal{F}$ takes its values in $R\ot_k V$. The Connes operator $\wh{D}$ satisfies:

\begin{enumerate}

\smallskip

\item If $\bx = [a_0\gamma_{\!A}(\bv_{0i})\ot \ba_{1,n-i}]$, then
$$
\qquad \wh{D}(\bx) = \sum_{j=0}^i\sum_l (-1)^{i+ji} \bigl[1\ot_{\!A} \gamma_{\!A} (\bv_{j+1,i}^{(l)})\ot_{\!A} a_0\gamma_{\!A}(\bv_{0j})\ot \ba_{1,n-i}^{(l)}\bigr],
$$
module $F_R^i(\wh{X}_{n+1})$, where $\sum_l \ba_{1,n-i}^{(l)}\ot_k \bv_{j+1,i}^{(l)} \!:=\! \ov{\chi}(\bv_{j+1,i}\ot_k\ba_{1,n-i})$.

\smallskip

\item If $\bx = [a_0\ot_{\!A} \gamma_{\!A}(\bv_{1i})\ot \ba_{1,n-i}]$, then
$$
\qquad\wh{D}(\bx) = \sum_{j=0}^{n-i}\sum_l (-1)^{jn+ji+n} \bigl[1\ot_{\!A}\gamma_{\!A}(\bv_{1i}^{(l)})\ot \ba_{j+1,n-i}\ot a_0\ot \ba_{1j}^{(l)}\bigr],
$$
module $F_R^{i-1}(\wh{X}_{n+1})$, where $\sum_l \ba_{1j}^{(l)}\ot_k \bv_{1i}^{(l)} := \ov{\chi}(\bv_{1i}\ot_k \ba_{1j})$.

\end{enumerate}

\end{proposition}

\begin{proof} (1)\enspace We must compute $\wh{D}(\bx) = \wh{\psi}\xcirc B \xcirc \wh{\phi}(\bx)$. By Proposition~\ref{formula para hat{phi} modulo la filtracion},
$$
\qquad \wh{D}(\bx) = \wh{\psi}\xcirc B\bigl([a_0\gamma(v_0)\ot\Sh(\bv_{1i}\ot_k \ba_{1,n-i})]\bigr) + \wh{\psi}\xcirc B\bigl([a_0\gamma(v_0)\ot \bx]\bigr),
$$
where $[a_0\gamma(v_0)\ot\bx]\in F^{i-1}\bigl(E\ot \ov{E}^{\ot^n}\ot\bigr)\cap \ov{W}_n \cap\ov{C}^R_n$. Now

\begin{itemize}

\smallskip

\item[-] $B\bigl([a_0\gamma(v_0)\ot\Sh(\bv_{1i}\ot_k\ba_{1,n-i})]\bigr)$ is a sum of classes in $E\ot \ov{E}^{\ot^{n+1}}\ot$ of simple tensors $1\ot \byy_{1,n+1}$, with $n-i$ of the $y_j$'s in $i_{\ov{A}}(\ov{A})$, $i$ of the $y_j$'s in $\mathcal{V}_{\!K}$ and one $y_j\notin i_{\ov{A}}(\ov{A})\cup \mathcal{V}_{\!K}$.

\smallskip

\item[-] $B\bigl([a_0\gamma(v_0)\ot\bx]\bigr)$ is a sum of classes in $E\ot \ov{E}^{\ot^{n+1}}\ot$ of simple tensors $1\ot\bz_{1,n+1}$, with at least $n-i+1$ of the $z_j$'s in $i_{\ov{A}}(\ov{A})$ and exactly one $z_j$ in $\ov{E}\setminus (i_{\ov{A}}(\ov{A})\cup \mathcal{V}_{\!K})$.

\smallskip
\end{itemize}
The result follows now easily from the definition of $\Sh$ and items~(3)--(6) of Proposition~\ref{propiedades de hat{psi}}.

\smallskip

\noindent (2)\enspace As in the proof of item~(1) we have
$$
\wh{D}(\bx) = \wh{\psi}\xcirc B\bigl([a_0\ot \Sh(\bv_{1i}\ot_k \ba_{1,n-i})]\bigr) + \wh{\psi}\xcirc B\bigl([a_0\ot \bx]\bigr),
$$
where $[a_0\ot \bx]\in F^{i-1}\bigl(E\ot \ov{E}^{\ot^n}\ot\bigr)\cap \ov{W}_n\cap \ov{C}^R_n$. Now

\begin{itemize}

\smallskip

\item[-] $B\bigl([a_0\ot\Sh(\bv_{1i}\ot_k\ba_{1,n-i})]\bigr)$ is a sum of classes in $E\ot \ov{E}^{\ot^{n+1}}\ot$ of simple tensors $1\ot \byy_{1,n+1}$, with $n-i+1$ of the $y_j$'s in $i_{\ov{A}}(\ov{A})$ and $i$ of the $y_j$'s in~$\mathcal{V}_{\!K}$.

\smallskip

\item[-] $B\bigl([a_0\ot\bx]\bigr)$ is a sum of classes in $E\ot \ov{E}^{\ot^{n+1}}\ot$ of simple tensors $1\ot\bz_{1,n+1}$, with each $z_j$ in $i_{\ov{A}}(\ov{A})\cup \mathcal{V}_{\!K}$ and at least $n-i+2$ of the $z_j$'s in $i_{\ov{A}}(\ov{A})$.

\smallskip

\end{itemize}
The result follows now easily from the definition of $\Sh$ and items~(1) and~(2) of Proposition~\ref{propiedades de hat{psi}}.
\end{proof}

\begin{corollary} If $K=A$, then $\bigl(\wh{X}_*,\wh{d}_*,\wh{D}_*\bigr) = \bigl(\wh{X}_{0*},\wh{d}^1_{0*},\wh{D}_{0*}\bigr)$, where
$$
\wh{D}_{0n}\bigl([a\gamma_{\!A}(\bv_{0n})]\bigr) = \sum_{j=0}^n (-1)^{n+jn} \bigl[1\ot_{\!A} \gamma_{\!A} (\bv_{j+1,n})\ot_{\!A} a\gamma_{\!A}(\bv_{0j})\bigr].
$$
\end{corollary}

\subsection{The spectral sequences} The first of the following spectral sequences generalizes those obtained in \cite[Section 3.1]{C-G-G} and \cite[Theorem 4.7]{Z-H}, while the third one generalizes those obtained in \cite{A-K}, \cite{K-R} and \cite[Section 3.2]{C-G-G}. Let
$$
\wh{d}^0_{rs}\colon \wh{X}_{rs}\to\wh{X}_{r-1,s}\quad\text{and}\quad\wh{d}^1_{rs}\colon \wh{X}_{rs}\to \wh{X}_{r,s-1}
$$
be as at the beginning of Section~\ref{Hochschild-Brzezinski} and let
$$
\wh{D}^0_{rs}\colon \wh{X}_{rs}\to \wh{X}_{r,s+1}
$$
be the map defined by
$$
\wh{D}^0\bigl([a_0\gamma_{\!A}(\bv_{0s})\ot \ba_{1r}]\bigr) = \sum_{j=0}^s\sum_l (-1)^{s+js} \bigl[1\ot_{\!A} \gamma_{\!A} (\bv_{j+1,s}^{(l)})\ot_{\!A} a_0\gamma_{\!A}(\bv_{0j})\ot \ba_{1r}^{(l)}\bigr],
$$
where $\sum_l \ba_{1r}^{(l)}\ot_k \bv_{j+1,s}^{(l)} = \ov{\chi}(\bv_{j+1,s}\ot_k \ba_{1r})$.

\subsubsection{The first spectral sequence} Recall from Remark~\ref{re3.5} that
$$
\Ho_r\bigl(\wh{X}_{*s},\wh{d}^0_{*s}\bigr) = \Ho^K_r\bigr(A,E\ot_{\!A} (E/A)^{\ot_{\!A}^s} \bigl).
$$
Let
\begin{align*}
\breve{d}_{rs}\colon \Ho^K_r\bigr(A,E\ot_{\!A} (E/A)^{\ot_{\!A}^s} \bigl) \longrightarrow \Ho^K_r\bigr(A,E\ot_{\!A} (E/A)^{\ot_{\!A}^{s-1}} \bigl)\\
\intertext{and}
\breve{D}_{rs}\colon \Ho^K_r\bigr(A,E\ot_{\!A} (E/A)^{\ot_{\!A}^s} \bigl) \longrightarrow \Ho^K_r\bigr(A,E\ot_{\!A} (E/A)^{\ot_{\!A}^{s+1}} \bigl)
\end{align*}
be the maps induced by $\wh{d}^1$ and $\wh{D}^0$, respectively.

\begin{proposition}\label{pepito} For each $r\ge 0$,
\[
\breve{\Ho}^K_r\bigr(A,E\ot_{\!A} (E/A)^{\ot_{\!A}^*} \bigl) := \Bigl(\Ho^K_r\bigr(A,E\ot_{\!A} (E/A)^{\ot_{\!A}^*}\bigr),\breve{d}_{r*},\breve{D}_{r*}\Bigr)
\]
is a mixed complex and there is a convergent spectral sequence
$$
(\mathpzc{E}^v_{sr},\mathpzc{d}^v_{sr})_{v\ge 0} \Longrightarrow \HC^K_{r+s}(E),
$$
such that $\mathpzc{E}^2_{sr} = \HC_s\Bigl(\breve{\Ho}^K_r\bigr(A,E\ot_{\!A} (E/A)^{\ot_{\!A}^*}\bigl)\Bigr)$ for all $r,s\ge 0$.
\end{proposition}

\begin{proof} For each $s,n\ge 0$, let
$$
\mathpzc{F}^s\bigl(\Tot(\BC(\wh{X},\wh{d},\wh{D})_n)\bigr) := \bigoplus_{j\ge 0} F^{s-2j}(\wh{X}_{n-2j}) u^j,
$$
where $F^{s-2j}(\wh{X}_{n-2j})$ is the filtration introduced in Section~\ref{filtraciones en homologia}. Consider the spectral sequence $(\mathpzc{E}^v_{sr},\mathpzc{d}^v_{sr})_{v\ge 0}$, associated with the filtration
$$
\mathpzc{F}^0\bigl(\Tot(\BC(\wh{X}_*,\wh{d}_*,\wh{D}_*))\bigr) \subseteq \mathpzc{F}^1\bigl(\Tot(\BC(\wh{X}_*,\wh{d}_*,\wh{D}_*))\bigr) \subseteq \cdots
$$
of $\Tot\bigl(\BC(\wh{X}_*,\wh{d}_*,\wh{D}_*)\bigr)$. A straightforward computation shows that

\begin{itemize}

\smallskip

\item[-] $\mathpzc{E}^0_{sr} = \bigoplus_{j\ge 0} \wh{X}_{r,s-2j} u^j$,

\smallskip

\item[-] $\mathpzc{d}^0_{sr}\colon \mathpzc{E}^0_{sr} \to \mathpzc{E}^0_{s,r-1}$ is $\bigoplus_{j\ge 0} \wh{d}^0_{r,s-2j} u^j$,

\smallskip

\item[-] $\mathpzc{E}^1_{sr} = \bigoplus_{j\ge 0} \Ho_r\bigl(\wh{X}_{*,s-2j}, \wh{d}^0_{*,s-2j}\bigr) u^j$,

\smallskip

\item[-] $\mathpzc{d}^1_{sr}\colon \mathpzc{E}^1_{sr} \to \mathpzc{E}^1_{s-1,r}$ is $\bigoplus_{j\ge 0} \breve{d}_{r,s-2j}u^j + \bigoplus_{j\ge 0} \breve{D}_{r,s-2j}u^{j-1}$.

\smallskip

\end{itemize}
From this it follows easily that $\breve{\Ho}^K_r\bigr(A,E\ot_{\!A} (E/A)^{\ot_{\!A}^*}\bigl)$ is a mixed complex and
$$
\mathpzc{E}^2_{sr} = \HC_s\Bigl(\breve{\Ho}^K_r\bigr(A,E\ot_{\!A} (E/A)^{\ot_{\!A}^*} \bigl)\Bigr).
$$
In order to finish the proof note that the filtration of $\Tot\bigl(\BC(\wh{X}_*,\wh{d}_*, \wh{D}_*)\bigr)$ introduced above is canonically bounded, and so, by Theorem~\ref{complejo mezclado que da la homologia ciclica}, the spectral sequence $(\mathpzc{E}^v_{sr}, \mathpzc{d}^v_{sr})_{v\ge 0}$ converges to the cyclic homology of the $K$-algebra $E$.
\end{proof}

\subsubsection{The second spectral sequence} For each $s\ge 0$, we consider the double complex
\[
\xymatrix{\\\\ \wh{\Xi}_s =}\qquad
\xymatrix{
\vdots \dto^-{\wh{d}^0} &\vdots \dto^-{\wh{d}^0}& \vdots \dto^-{\wh{d}^0}& \vdots \dto^-{\wh{d}^0}\\
\wh{X}_{3s}u^0\dto^-{\wh{d}^0} & \wh{X}_{3,s-1}u^1\lto_-{\wh{D}^0}\dto^-{\wh{d}^0} & \wh{X}_{3,s-2}u^2\lto_-{\wh{D}^0}\dto^-{\wh{d}^0} & \wh{X}_{3,s-3}u^3\lto_-{\wh{D}^0} \\
\wh{X}_{2s}u^0\dto^-{\wh{d}^0} & \wh{X}_{2,s-1}u^1\lto_-{\wh{D}^0}\dto^-{\wh{d}^0} & \wh{X}_{2,s-2}u^2\lto_-{\wh{D}^0}\\
\wh{X}_{1s}u^0\dto^-{\wh{d}^0} & \wh{X}_{1,s-1}u^1\lto_-{\wh{D}^0}\\
\wh{X}_{0s}u^0}
\]
where the module $\wh{X}_{0s}u^0$ is placed in the intersection of the $0$-th column and the $0$-th row.

\begin{proposition}\label{pepe} There is a convergent spectral sequence
$$
(E^v_{sr},\partial^v_{sr})_{v\ge 0} \Longrightarrow \HC^K_{r+s}(E),
$$
such that $E^1_{sr} = \Ho_r(\Tot(\wh{\Xi}_s))$ for all $r,s\ge 0$.
\end{proposition}

\begin{proof} For each $s,n\ge 0$, let
$$
F^s\bigl(\Tot(\BC(\wh{X},\wh{d},\wh{D})_n)\bigr) := \bigoplus_{j\ge 0} F^{s-j}(\wh{X}_{n-2j})
u^j,
$$
where $F^{s-j}(\wh{X}_{n-2j})$ is the filtration introduced in Section~\ref{filtraciones en homologia}. Consider the spectral sequence $(E^v_{sr},\partial^v_{sr})_{v\ge 0}$, associated with the filtration
$$
F^0\bigl(\Tot(\BC(\wh{X}_*,\wh{d}_*,\wh{D}_*))\bigr) \subseteq F^1\bigl(\Tot(\BC(\wh{X}_*,\wh{d}_*,\wh{D}_*))\bigr) \subseteq \cdots
$$
of $\Tot(\BC(\wh{X}_*,\wh{d}_*,\wh{D}_*))$. By definition
$$
E^0_{sr} = \wh{X}_{rs}u^0\oplus \wh{X}_{r-1,s-1}u \oplus \wh{X}_{r-2,s-2}u^2\oplus \wh{X}_{r-3,s-3}u^3\oplus\cdots
$$
and the boundary map $\partial^0_{sr}\colon E^0_{sr}\to E^0_{s,r-1}$ is induced by $\wh{d} + \wh{D}$. Con\-se\-quently, by Theorem~\ref{formula para wh{d}_1} and item~(1) of Proposition~\ref{Connes operator},
$$
(E^0_{s*},\partial^0_{s*}) = \Tot(\wh{\Xi}_s)\qquad\text{for all $s\ge 0$,}
$$
and so $E^1_{sr} = \Ho_r(\Tot(\wh{\Xi}_s))$ as desired. Finally, it is clear that $(E^v_{sr}, \partial^v_{sr})_{v\ge 0}$ converges to $\HC_{r+s}^K(E)$.
\end{proof}

\subsubsection{The third spectral sequence} Assume that $\mathcal{F}$ takes its values in $K\ot_k V$. Recall from Remark~\ref{re3.5} that
$$
\Ho_s\bigl(\wh{X}_{r*},\wh{d}^1_{r*}\bigr) = \Ho^A_s\bigr(E,A\ot\ov{A}^{\ot^r}\ot E\bigl).
$$
Let
\begin{align*}
\breve{d}_{rs}\colon \Ho^A_s\bigr(E,A\ot\ov{A}^{\ot^r}\ot E\bigl) \longrightarrow \Ho^A_s\bigr(E,A\ot\ov{A}^{\ot^{r-1}}\ot E\bigl) \bigl)\\
\intertext{and}
\breve{D}_{rs}\colon \Ho^A_s\bigr(E,A\ot\ov{A}^{\ot^r}\ot E\bigl)\longrightarrow \Ho^A_s\bigr(E,A\ot\ov{A}^{\ot^{r+1}}\ot E\bigl)\bigl)
\end{align*}
be the maps induced by $\wh{d}^0$ and $\wh{D}^0$, respectively.

\begin{proposition}\label{pepitos} For each $s\ge 0$,
\[
\breve{\Ho}^A_s\bigr(E,A\ot\ov{A}^{\ot^*}\ot E\bigl):= \Bigl(\Ho^A_s\bigr(E,A\ot\ov{A}^{\ot^*} \ot E\bigl),\breve{d}_{*s},\breve{D}_{*s}\Bigr)
\]
is a mixed complex and there is a convergent spectral sequence
$$
(\mathfrak{E}^v_{rs},\mathfrak{d}^v_{rs})_{v\ge 0} \Longrightarrow \HC^K_{r+s}(E),
$$
such that $\mathfrak{E}^2_{rs} = \HC_r\Bigl(\breve{\Ho}^A_s\bigr(E,A\ot\ov{A}^{\ot^*}\ot E\bigl)\Bigr)$ for all $r,s\ge 0$.
\end{proposition}

\begin{proof} For each $r,n\ge 0$, let
$$
\mathfrak{F}^r\bigl(\Tot(\BC(\wh{X},\wh{d},\wh{D})_n)\bigr) := \bigoplus_{j\ge 0} \mathfrak{F}^{r-j}(\wh{X}_{n-2j}) u^j,
$$
where
$$
\mathfrak{F}^{r-j}(\wh{X}_{n-2j}):= \bigoplus_{i\le r-j} \wh{X}_{i,n-i-2j}.
$$
Consider the spectral sequence $(\mathfrak{E}^v_{rs},\mathfrak{d}^v_{rs})_{v\ge 0}$, associated with the filtration
$$
\mathfrak{F}^0\bigl(\Tot(\BC(\wh{X}_*,\wh{d}_*,\wh{D}_*))\bigr) \subseteq \mathfrak{F}^1\bigl(\Tot(\BC(\wh{X}_*,\wh{d}_*,\wh{D}_*))\bigr) \subseteq \cdots
$$
of $\Tot\bigl(\BC(\wh{X}_*,\wh{d}_*,\wh{D}_*)\bigr)$. A straightforward computation shows that

\begin{itemize}

\smallskip

\item[-] $\mathfrak{E}^0_{rs} = \bigoplus_{j\ge 0} \wh{X}_{r-j,s-j} u^j$,

\smallskip

\item[-] $\mathfrak{d}^0_{rs}\colon \mathfrak{E}^0_{rs}\to \mathfrak{E}^0_{r,s-1}$ is $\bigoplus_{j\ge 0} \wh{d}^1_{r-j,s-j} u^j$,

\smallskip

\item[-] $\mathfrak{E}^1_{rs} = \bigoplus_{j\ge 0} \Ho_s\bigl(\wh{X}_{r-j,*-j}, \wh{d}^1_{r-j,*-j}\bigr) u^j$,

\smallskip

\item[-] $\mathfrak{d}^1_{rs}\colon \mathfrak{E}^1_{rs}\to \mathfrak{E}^1_{r-1,s}$ is $\bigoplus_{j\ge 0} \breve{d}_{r-j,s-j} u^j + \bigoplus_{j\ge 0} \breve{D}_{r-j,s-j} u^{s-j}$.

\smallskip

\end{itemize}
From this it follows easily that $\breve{\Ho}^A_s\bigr(E,A\ot\ov{A}^{\ot^*}\ot E\bigl)$ is a mixed complex and
$$
\mathfrak{E}^2_{rs} = \HC_r\Bigl(\breve{\Ho}^A_s\bigr(E,A\ot\ov{A}^{\ot^*}\ot E\bigl)\Bigr).
$$
In order to finish the proof note that the filtration of $\Tot\bigl(\BC(\wh{X}_*,\wh{d}_*, \wh{D}_*)\bigr)$ introduced above is canonically bounded, and so, by Theorem~\ref{complejo mezclado que da la homologia ciclica}, the spectral sequence $(\mathfrak{E}^v_{sr}, \mathfrak{d}^v_{sr})_{v\ge 0}$ converges to the cyclic homology of the $K$-algebra $E$.
\end{proof}

\subsubsection{The fourth spectral sequence}\label{segunda suc esp} Assume that $\mathcal{F}$ takes its values in $K\ot_k V$. Then the mixed complex $\bigl(\wh{X}_*,\wh{d}_*,\wh{D}_*\bigr)$ is filtrated by
\begin{equation}
\mathcal{F}^0\bigl(\wh{X}_*,\wh{d}_*,\wh{D}_*\bigr)\subseteq \mathcal{F}^1\bigl(\wh{X}_*,\wh{d}_*,\wh{D}_*\bigr)\subseteq \mathcal{F}^2\bigl(\wh{X}_*,\wh{d}_*,\wh{D}_*\bigr)\subseteq\cdots,\label{eq8}
\end{equation}
where
$$
\mathcal{F}^r(\wh{X}_n) := \bigoplus_{i\le r} \wh{X}_{i,n-i}.
$$
Hence, for each $r\ge 1$, we can consider the quotient mixed complex
$$
\wh{\mathfrak{X}}^r: =\frac{\mathcal{F}^r \bigl(\wh{X}_*,\wh{d}_*,\wh{D}_*\bigr)} {\mathcal{F}^{r-1}\bigl(\wh{X}_*,\wh{d}_*,\wh{D}_*\bigr)}.
$$
It is easy to check that the Hochschild boundary map of $\wh{\mathfrak{X}}^r$ is $\wh{d}^1_{r*}\colon\wh{X}_{r*}\to \wh{X}_{r,*-1}$ and that, by item~(1) of Proposition~\ref{Connes operator}, its Connes operator is $\wh{D}^0_{rs} \colon \wh{X}_{rs}\to \wh{X}_{r,s+1}$.

\begin{proposition}\label{pepe1} There is a convergent spectral sequence
$$
(\mathcal{E}_{rs}^v,\delta_{rs}^v)_{v\ge 0} \Longrightarrow \HC^K_{r+s}(E),
$$
such that $\mathcal{E}^1_{rs} = \HC_s(\wh{\mathfrak{X}}^r)$ for all $r,s\ge 0$.
\end{proposition}

\begin{proof} Let $(\mathcal{E}_{rs}^v,\delta_{rs}^v)_{v\ge 0}$ be the spectral sequence associated with the filtration
$$
\mathcal{F}^0\bigl(\Tot(\BC(\wh{X}_*,\wh{d}_*,\wh{D}_*))\bigr)\subseteq \mathcal{F}^1\bigl(\Tot(\BC(\wh{X}_*,\wh{d}_*,\wh{D}_*))\bigr)\subseteq\cdots,
$$
of $\Tot(\BC(\wh{X}_*,\wh{d}_*,\wh{D}_*)$, induced by~\eqref{eq8}. It is evident that
$$
\mathcal{F}^r\bigl(\Tot(\BC(\wh{X},\wh{d},\wh{D}))_n\bigr) = \bigoplus_{j\ge 0} \mathcal{F}^r(\wh{X}_{n-2j}) u^j.
$$
Hence,
$$
\mathcal{E}^0_{rs} = \wh{X}_{rs} u^0\oplus \wh{X}_{r,s-2}u \oplus \wh{X}_{r,s-4}u^2\oplus \wh{X}_{r,s-6}u^3\oplus\cdots
$$
and $\delta^0_{rs}\colon \mathcal{E}^0_{rs}\to \mathcal{E}^0_{r,s-1}$ is the map induced by $\wh{d} + \wh{D}$. Consequently,
$$
(\mathcal{E}_{rs}^0,\delta_{rs}^0) = \Tot\bigl(\BC(\wh{\mathfrak{X}}^r)\bigr),
$$
and so $\mathcal{E}^1_{rs} = \HC_s(\wh{\mathfrak{X}}^r)$ as desired. Finally, it is clear that $(\mathcal{E}_{rs}^v, \delta_{rs}^v)_{v\ge 0}$ converges to $\HC^K_{r+s}(E)$.
\end{proof}

\section{Hochschild homology of a cleft braided Hopf crossed product}\label{hom de inv}
Let $E:= A\#_f H$ be the braided Hopf crossed product associated with a triple $(s,\rho,f)$, consisting of a transposition $s\colon H\ot_k A \to A\ot_k H$, a weak $s$-action $\rho$ of $H$ on $A$, and a compatible with $s$ normal cocycle $f\colon H\ot_k H \to A$, that satisfies the twisted module condition. Let $K$ be a subalgebra of $A$ stable under $s$ and $\rho$, and let $M$ be an $E$-bimodule. In this section we show that if $H$ is a Hopf algebra and $E$ is cleft, then the complex $(\wh{X}_*(M),\wh{d}_*)$ of Section~\ref{Hochschild-Brzezinski} is isomorphic to a simpler complex $(\ov{X}_*(M),\ov{d}_*)$. In the sequel we will use the following notations:

\begin{enumerate}

\item For $s\ge 1$, we let $\pc_s\colon H^{\ot_k^{2s}}\to H^{\ot_k^{2s}}$ denote the map recursively defined by
\begin{align*}
&\pc_1:=\ide,\\
&\pc_s:=\bigl(H\ot_k \pc_{s-1}\ot_k H\bigr)\xcirc \bigl(H\ot_k c^{\ot_k^{s-1}}\ot_k H\bigr).
\end{align*}

\smallskip

\item For $s\ge 1$ we let $H^{\ot_c^s}$ denote the coalgebra with underlying space $H^{\ot_k^s}$, comultiplication $\Delta_{H^{\ot_{\!c}^{\!s}}} := \pc_s\xcirc \Delta^{\ot_k^s}$ and counit $\varepsilon_{H^{\ot_{\!c}^{\!s}}} := \varepsilon^{\ot_k^s}$. Note that $\Delta^{\ot_k^s}$ induces a $k$-linear map from $\ov{H}^{\ot_k^s}$ to $\ov{H}^{\ot_k^s}\ot_k \ov{H}^{\ot_k^s}$, that we will also denote with the symbol $\Delta_{H^{\ot_{\!c}^{\!s}}}$. A similar remark is valid for the maps $\s_{sr}$, $\gc_s$ and $\cc_{sr}$ introduced below.

\smallskip

\item Let $\wh{s}\colon H\ot_k E\to E\ot_k H$ be as in Example~\ref{ex 1.14}. For each $s\ge 1$, we let $\ps_s\colon (E\ot_k H)^{\ot_k^s}\to E^{\ot_k^s}\ot_k H^{\ot_k^s}$ denote the map recursively defined by
\begin{align*}
&\ps_1:=\ide,\\
&\ps_s := \bigl(E\ot_k \ps_{s-1}\ot_k H\bigr)\xcirc \bigl(E\ot_k \wh{s}^{\ot_k^{s-1}}\ot_k H \bigr).
\end{align*}

\smallskip

\item For $s,r\ge 1$, we let $\s_{sr}\colon H^{\ot_k^s}\ot_k A^{\ot^r}\to A^{\ot^r} \ot_k H^{\ot_k^s}$ denote the map recursively defined by:
\begin{align*}
&\s_{11}:= s,\\
&\s_{1,r+1}:=\bigl(A^{\ot^r}\ot_k s\bigr)\xcirc \bigl(\s_{1r}\ot_k A \bigr),\\
&\s_{s+1,r}:=\bigl(\s_{1r}\ot_k H^{\ot_k^s}\bigr)\xcirc \bigl(V\ot_k \s_{sr}\bigr).
\end{align*}

\smallskip

\item For $s\ge 2$, we let $\gc_s\colon H^{\ot_k^s}\to H^{\ot_k^s}$ denote the map recursively defined by:
\begin{align*}
&\gc_2:= c,\\
&\gc_{s+1}:=\bigl(H\ot_k \gc_s\bigr)\xcirc \cc_{s1},
\end{align*}
where $\cc_{sr}\colon H^{\ot_k^r}\ot_k H^{\ot_k^s}\to H^{\ot_k^s}\ot_k H^{\ot_k^r}$ is the map obtained mimicking the definition of $\s_{sr}$, but using $c$ instead of $s$.

\smallskip

\item Let $[M\ot \ov{A}^{\ot^r},K]_H\ot_k \ov{H}^{\ot_k^s}$ be the $k$-vector subspace of $M\ot \ov{A}^{\ot^r}\ot_k \ov{H}^{\ot_k^s}$ generated by the commutators
    $$
    \lambda m\ot \ba_{1r}\ot_k \bh_{1s} - \sum_i m\ot \ba_{1r}\lambda^{(i)} \ot_k \bh_{1s}^{(i)}\quad\text{with $\lambda\in K$,}
    $$
    where
    $$
    \sum_i \lambda^{(i)} \ot_k \bh_{1s}^{(i)} := \s(\bh_{1s}\ot_k\lambda).
    $$
    Given $m\in M$, $\ba_{1r}\in \ov{A}^{\ot^r}$ and $\bh_{1s}\in \ov{H}^{\ot_k^s}$, we let $[m\ot \ba_{1r}]_H\ot_k \bh_{1s}$ denote the class of $m\ot \ba_{1r}\ot_k \bh_{1s}$ in
    $$
    \ov{X}_{rs}(M):= \frac{M\ot\ov{A}^{\ot^r}\ot_k \ov{H}^{\ot_k^s}}{[M\ot \ov{A}^{\ot^r},K]_H\ot_k \ov{H}^{\ot_k^s}}.
    $$

\smallskip

\end{enumerate}

\begin{remark}\label{explicacion} Note that:

\begin {enumerate}

\item The map $\pc_s$ acts over each element $(h_1\ot_k l_1) \ot_k \cdots \ot_k (h_s\ot_k l_s)$ of $H^{\ot_k^{2s}}$, carrying the $l_i$'s to the right by means of reiterated applications of~$c$.

\smallskip

\item The map $\ps_s$ acts over each element $(a_1\# h_1\ot_k l_1) \ot_k \cdots \ot_k (a_s\# h_s\ot_k l_s)$ of $(E\ot_k H)^{\ot_k^s}$, carrying the $l_i$'s to the right by means of reiterated applications of~$\wh{s}$.

\smallskip

\item The map $\s_{sr}$ acts over each element $\bh_{1s}\ot_k \ba_{1r}$ of $H^{\ot_k^s}\ot_k A^{\ot^r}$, carrying the $h_i$'s to the right by means of reiterated applications of~$s$.

\smallskip

\item The map $\gc_s$ acts over each element $\bh_{1s}$ of $H^{\ot_k^s}$, carrying the $i$-th factor to the $s-i+1$-place by means of reiterated applications of~$c$.

\end{enumerate}
\end{remark}

\begin{remark}\label{acerca de nu} For each $s\in \mathds{N}$, we consider $E^{\ot_k^s}$ as a $H^{\ot_c^s}$-comodule via
$$
\boldsymbol{\nu}:=\ps_s\xcirc (A\ot_k \Delta)^{\ot_k^s}.
$$
Note that $\boldsymbol{\nu} \xcirc \gamma^{\ot_k^s} = (\gamma^{\ot_k^s}\ot_k H^{\ot_c^s})\xcirc \Delta_{H^{\ot_{\!c}^{\!s}}}$ and that $\boldsymbol{\nu}$ induce a coaction
\begin{equation}
\boldsymbol{\nu}_{\!A} \colon E^{\ot_{\!A}^s}\to E^{\ot_{\!A}^s}\ot_k H^{\ot_c^s},\label{eq13}
\end{equation}
such that
\begin{equation}
\boldsymbol{\nu}_{\!A} \xcirc \gamma^{\ot_{\!A}^s} = (\gamma^{\ot_{\!A}^s}\ot_k H^{\ot_c^s})\xcirc \Delta_{H^{\ot_{\!c}^{\!s}}},\label{eq14}
\end{equation}
where $\gamma^{\ot_{\!A}^s}\colon H^{\ot_c^s}\to E^{\ot_{\!A}^s}$ is the map given by $\gamma^{\ot_{\!A}^s}(\bh_{1s}) := \gamma_{\!A}(\bh_{1s})$. We will also use the symbol  $\boldsymbol{\nu}_{\!A}$ to denote the map from $(E/A)^{\ot_{\!A}^s}$ to $(E/A)^{\ot_{\!A}^s}\ot_k \ov{H}^{\ot_c^s}$ induced by~\eqref{eq13}. We will use the property~\eqref{eq14} freely in the sequel.
\end{remark}

\begin{remark}\label{nuevas notaciones} The maps $\Delta_{H^{\ot_{\!c}^{\!s}}}$, $\cc_{sr}$, $\cc_{sr}^{-1}$, $\s_{sr}$, $\s_{sr}^{-1}$ and $\boldsymbol{\nu}_{\!A}$ will be represented by the same diagrams as the ones introduced in~\eqref{s2} and~\eqref{s3} for $\Delta$, $c$, $c^{-1}$ $s$, $s^{-1}$ and $\nu$.
\end{remark}

\smallskip

For each $r,s\ge 0$, we define the map $\theta_{rs}\colon \wh{X}_{rs}(M)\to \ov{X}_{rs}(M)$, by
$$
\theta\bigl([m\ot_{\!A} \ov{\bx}_{1s}\ot \ba_{1r}]\bigr) := \sum_i (-1)^{rs} \bigl[mx_1^{(0)}\cdots x_s^{(0)} \ot \ba_{1r}^{(i)} \bigr]_H\ot_k \bx_{1s}^{(1)(i)},
$$
where
\begin{align*}
& x_1^{(0)}\ot_{\!A}\cdots\ot_{\!A} x_s^{(0)}\ot_k x_1^{(1)}\ot_k\cdots\ot_k x_s^{(1)} :=   \ov{\bx}_{1s}^{(0)}\ot_k \bx_{1s}^{(1)} = \boldsymbol{\nu}_{\!A}(\ov{\bx}_{1s})
\intertext{and}
& \sum_i \ov{\bx}_{1s}^{(0)}\ot_k\ba_{1r}^{(i)}\ot_k \bx_{1s}^{(1)(i)} := \ov{\bx}_{1s}^{(0)} \ot_k \s\bigl(\bx_{1s}^{(1)}\ot_k \ba_{1r}\bigr).
\end{align*}

\begin{proposition}\label{inversa de theta} The map $\theta_{rs}$ is invertible. Its inverse is the map $\vartheta_{rs}$, given by
$$
\vartheta(\bx) := \sum_{ij}(-1)^{rs} m \gamma^{-1}(h_s^{(i)(1)(j)}) \cdots \gamma^{-1}(h_1^{(i)(1)(j)})\ot_{\!A}\gamma_{\!A}(\bh_{1s}^{(i)(2)})\ot\ba_{1r}^{(i)},
$$
where
\begin{align*}
& \bx := [m\ot\ba_{1r}]_H\ot_k\bh_{1s},\\
& \sum_i \bh_{1s}^{(i)}\ot_k \ba_{1r}^{(i)} := \s^{-1}\bigl(\ba_{1r}\ot_k \bh_{1s}\bigr)
\intertext{and}
& \sum_i\sum_j\bh_{s1}^{(i)(1)(j)}\ot_k \bh_{1s}^{(i)(2)}\ot_k \ba_{1r}^{(i)} := \sum_i \bigl( \gc_s\ot_k H^{\ot_c^s}\bigr)\xcirc \Delta_{H^{\ot_{\!c}^{\!s}}}(\bh_{1s}^{(i)})\ot_k \ba_{1r}^{(i)}.
\end{align*}
\end{proposition}

\begin{proof} See Appendix~B.
\end{proof}

We will need the following generalization of the weak action $\rho$ of $H$ on $A$.

\begin{definition}\label{accion debil multiple} For all $r\in \mathds{N}$, we let $\rho_r\colon H\ot A^{\ot^r}\to A^{\ot^r}$ denote the map re\-cursively defined by
$$
\rho_1 := \rho\quad\text{and}\quad \rho_{r+1} = (\rho_r\ot \rho_1)\xcirc (H\ot_k \s_{1r}\ot A)\xcirc
(\Delta\ot A^{\ot^{r+1}}).
$$
For $h\in H$ and $a_1,\dots,a_r\in A$, we set $h\cdot \ba_{1r}:=\rho_r(h\ot_k \ba_{1r})$.
\end{definition}

Let $\ov{d}^l_{rs}\colon\ov{X}_{rs}(M)\to \ov{X}_{r+l-1,s-l}(M)$ be the map $\ov{d}^l_{rs} := \theta_{r+l-1,s-l}\xcirc \wh{d}^l_{rs}\xcirc \vartheta_{rs}$.

\begin{theorem}\label{th5.2} The Hochschild homology of the $K$-algebra $E$ with coefficients in $M$ is the homology of $(\ov{X}_*(M),\ov{d}_*)$, where
$$
\ov{X}_n(M) : = \bigoplus_{r+s = n} \ov{X}_{rs}(M)\qquad\text{and}\qquad \ov{d}_n := \sum^n_{l=1}\ov{d}^l_{0n} + \sum_{r=1}^n \sum^{n-r}_{l=0} \ov{d}^l_{r,n-r}.
$$
Moreover,
\begin{align*}
\ov{d}^0(\bx)& = [ma_1\ot \ba_{2r}]_H\ot_k \bh_{1s}\\
& + \sum_{i=1}^{r-1} (-1)^i [m\ot\ba_{1,i-1}\ot a_ia_{i+1} \ot \ba_{i+1,r}]_H \ot_k \bh_{1s}\\
& + \sum_i (-1)^r [a_r^{(i)}m\ot \ba_{1,r-1}]_H\ot_k \bh_{1s}^{(i)}
\intertext{and}
\ov{d}^1(\bx)& = (-1)^r \bigl[m\ep(h_1)\ot \ba_{1r}\bigr]_H \ot_k \bh_{2s}\\
& + \sum_{i=1}^{s-1}(-1)^{r+i} \bigl[m\ot \ba_{1r}\bigr]_H \ot_k \bh_{1,i-1}\ot_k h_ih_{i+1} \ot_k \bh_{i+2,s}\\
& + \sum_{jl} (-1)^{r+s}\bigl[\gamma(h_s^{(2)})m\gamma^{-1}(h_s^{(1)(j)(l)(1)}) \ot h_s^{(1)(j)(l)(2)}\cdot  \ba_{1r}^{(l)}\bigr]_H\ot_k \bh_{1,s-1}^{(j)},
\end{align*}
where
\begin{align*}
&\, \bx := [m\ot\ba_{1r}]_H\ot_k \bh_{1s},\\
& \sum_i \bh_{1s}^{(i)}\ot_k a_r^{(i)} := \s^{-1}(a_r \ot_k \bh_{1s}),\\
& \sum_j h_s^{(1)(j)}\ot_k \bh_{1s-1}^{(j)}\ot_k h_s^{(2)} := c\bigl(\bh_{1,s-1}\ot_k h_s^{(1)} \bigr) \ot_k h_s^{(2)}
\intertext{and}
&\sum_{jl} h_s^{(1)(j)(l)}\ot_k \ba_{1r}^{(l)}\ot_k \bh_{1,s-1}^{(j)}\ot_k h_s^{(2)} := \sum_j s^{-1} \bigl(\ba_{1r}\ot_k h_s^{(1)(j)}\bigr)\ot_k \bh_{1,s-1}^{(j)}\ot_k h_s^{(2)}.
\end{align*}

\end{theorem}

\begin{proof} See Appendix~B.
\end{proof}

\begin{remark} In order to abbreviate notations we will write $\ov{X}_{rs}$ and $\ov{X}_n$ instead of $\ov{X}_{rs}(E)$ and $\ov{X}_n(E)$, respectively.
\end{remark}

\begin{notation}\label{not hat X Ru} Given a $K$-subalgebra $R$ of $A$ and $0\le u\le r$, we let $\ov{X}^{Ru}_{rs}(M)$ denote the $k$-vector subspace of $\ov{X}_{rs}(M)$ generated by the classes in $\ov{X}_{rs}(M)$ of all the simple tensors $m\ot \ba_{1r}\ot_k \bh_{1s}$, with at least $u$ of the $a_j$'s in $\ov{R}$. Moreover, we set $\ov{X}^{Ru}_n(M) :=\bigoplus_{r+s=n} \ov{X}^{Ru}_{rs}(M)$.
\end{notation}

\begin{proposition} Let $R$ be a stable under $s$ and $\rho$ subalgebra of $A$. If $f$ takes its values in $R$, then
$$
\ov{d}^l\bigl(\ov{X}_{rs}(M)\bigr)\subseteq \ov{X}^{R,l-1}_{r+l-1,s-l}(M),
$$
for all $l\ge 1$.
\end{proposition}

\begin{proof} This is an immediate consequence of item~(3) of Theorem~\ref{formula para wh{d}_1}.
\end{proof}

\begin{remark} By the previous proposition, we know that if $f$ takes its values in $K$, then $(\ov{X}_*(M), \ov{d}_*)$ is the total complex of the double complex $\bigl(\ov{X}_{**}(M),\ov{d}_{**}^0, \ov{d}_{**}^1 \bigr)$.
\end{remark}

\subsection{The filtration of $\mathbf{\boldsymbol{(\ov{X}_*(M),\ov{d}_*)}}$} \label{filtraciones en homologia caso cleft} Let $F^i(\ov{X}_n(M)) := \bigoplus_{s\le i} \ov{X}_{n-s,s}(M)$. The chain complex $(\ov{X}_*(M),\ov{d}_*)$ is filtered by
\begin{equation}
F^0(\ov{X}_*(M))\subseteq F^1(\ov{X}_*(M))\subseteq F^2(\ov{X}_*(M))\subseteq F^3(\ov{X}_*(M))\subseteq\dots.\label{eq9}
\end{equation}

\begin{remark}\label{re def theta} By Proposition~\ref{inversa de theta} and the definition of $(\ov{X}_*(M),\ov{d}_*)$, the map
$$
\theta_*\colon (\wh{X}_*(M),\wh{d}_*)\to (\ov{X}_*(M),\ov{d}_*),
$$
given by $\theta_n = \sum_{r+s = n} \theta_{rs}$, is an isomorphism of chain complexes. It is evident that $\theta_*$ preserve filtrations. Consequently, the spectral sequence introduced in~\eqref{eq6} coincides with the spectral sequence associated with the filtration~\eqref{eq9}. Clearly the compositional inverse of $\theta_*$ is the map
$$
\vartheta_*\colon (\ov{X}_*(M),\ov{d}_*)\longrightarrow (\wh{X}_*(M),\wh{d}_*),
$$
defined by $\vartheta_n := \bigoplus_{r+s = n} \vartheta_{rs}$.
\end{remark}

\subsection{Comparison maps}\label{morfismos de comparacion en homologia: caso cleft} Let
$$
\ov{\phi}_*\colon (\ov{X}_*(M),\ov{d}_*)\rightarrow \bigl(M\ot\ov{E}^{\ot^*}\ot,b_*\bigr) \quad \text{and}\quad \ov{\psi}_*\colon \bigl(M\ot\ov{E}^{\ot^*}\ot,b_*\bigr)\rightarrow (\ov{X}_*(M),\ov{d}_*)
$$
be the morphisms of chain complexes defined by $\ov{\phi}_* := \wh{\phi}_*\xcirc \vartheta_*$ and $\ov{\psi}_* :=  \theta_*\xcirc \wh{\psi}_*$, respectively. By the comments in Subsection~\ref{morfismos de comparacion en homologia}, we know that $\ov{\psi}\xcirc \ov{\phi} = \ide$ and $\ov{\phi}\xcirc \ov{\psi} = \wh{\phi}\xcirc \wh{\psi}$ is homotopically equivalent to the identity. Moreover, by Proposition~\ref{phi, psi y omega preservan filtraciones en homologia} and Remark~\ref{re def theta}, the morphisms $\ov{\phi}$ and $\ov{\psi}$, and the homotopy $\wh{\omega}_{*+1}\colon \wh{\phi}_*\xcirc\wh{\psi}_*\to \ide_*$, preserve filtrations.

\section{Hochschild cohomology of a cleft braided Hopf crossed product}
Let $E:= A\#_f H$, $K$ and $M$ be as in Section~\ref{hom de inv}. In this section we show that if $E$ is cleft, then the complex $(\wh{X}^*(M),\wh{d}^*)$ of Section~\ref{coHochschild-Brzezinski} is isomorphic to a simpler complex $(\ov{X}^*(M),\ov{d}^*)$.

\smallskip

For each $r,s\ge 0$, we consider $\ov{A}^{\ot^r}\ot_k \ov{H}^{\ot_k^s}$ as a left $K^e$-module via
$$
(\lambda_1\ot_k \lambda_2)(\ba_{1r}\ot_k\bh_{1s}) = \sum_i \lambda_1\ba_{1r}\lambda_2^{(i)} \ot_k\bh_{1s}^{(i)},
ç$$
where $\sum_i \lambda_2^{(i)}\ot_k\bh_{1s}^{(i)} := \s(\bh_{1s}\ot_k\lambda_2)$. Let
$$
\ov{X}^{rs}(M) := \Hom_{K^e}\bigl(\ov{A}^{\ot^r}\ot_k \ov{H}^{\ot_k^s},M\bigr).
$$
For each $r,s\ge 0$, we define the map $\theta^{rs}\colon \ov{X}^{rs}(M)\to \wh{X}^{rs}(M)$, by
$$
\theta(\beta)(\ov{\bx}_{1s}\ot \ba_{1r}) := \sum_i (-1)^{rs} x_1^{(0)}\cdots x_s^{(0)}\beta\bigl(\ba_{1r}^{(i)}\ot_k \bx_{1s}^{(1)(i)}\bigr),
$$
where
\vspace{2.5\jot}
\begin{align*}
& x_1^{(0)}\ot_{\!A}\cdots\ot_{\!A} x_s^{(0)}\ot_k x_1^{(1)}\ot_k\cdots\ot_k x_s^{(1)} :=   \ov{\bx}_{1s}^{(0)}\ot_k \bx_{1s}^{(1)} := \boldsymbol{\nu}_{\!A}(\ov{\bx}_{1s})
\intertext{and}
& \sum_i \ov{\bx}_{1s}^{(0)}\ot_k\ba_{1r}^{(i)}\ot_k \bx_{1s}^{(1)(i)} := \ov{\bx}_{1s}^{(0)} \ot_k \s\bigl(\bx_{1s}^{(1)}\ot_k \ba_{1r}\bigr).
\end{align*}

\begin{proposition}\label{inversa de cotheta} The map $\theta^{rs}$ is invertible. Its inverse is the map $\vartheta^{rs}$ given by
$$
\vartheta(\alpha)(\ba_{1r}\ot_k\! \bh_{1s})\! =\! \sum_{ij}\!(-1)^{rs} \gamma^{-1}(h_s^{(i)(1)(j)})\cdots \gamma^{-1}(h_1^{(i)(1)(j)})\alpha \bigl(\gamma_{\!A}(\bh_{1s}^{(i)(2)})\ot\ba_{1r}^{(i)}\bigr),
$$
where
\begin{align*}
& \sum_i \bh_{1s}^{(i)}\ot_k \ba_{1r}^{(i)} := \s^{-1}\bigl(\ba_{1r}\ot_k \bh_{1s}\bigr)
\intertext{and}
&\sum_i\sum_j\bh_{s1}^{(i)(1)(j)}\ot_k \bh_{1s}^{(i)(2)}\ot_k \ba_{1r}^{(i)}:=\sum_i\bigl(\gc_s \ot_k H^{\ot_c^s}\bigr)\xcirc \Delta_{H^{\ot_{\!c}^{\!s}}}(\bh_{1s}^{(i)})\ot_k\ba_{1r}^{(i)}.
\end{align*}
\end{proposition}

\begin{proof} For $r,s\ge 0$, consider $X_{rs}$, $\wh{X}_{rs}(E\ot_k E)$ and $\ov{X}_{rs}(E\ot_k E)$ as in Sections~\ref{SecRes}, \ref{Hochschild-Brzezinski} and~\ref{hom de inv}, respectively. Notice that $(\ov{X}_*(E\ot_k E),\ov{d}_*)$ and $(\wh{X}_*(E\ot_k E),\wh{d}_*)$ are $E$-bimodule complexes via
$$
\lambda_1 \bigl([(e_1\ot_k e_2)\ot \ba_{1r}]_H\ot_k\bh_{1s}  \bigr)\lambda_2 :=  [(e_1\lambda_2 \ot_k \lambda_1 e_2)\ot \ba_{1r}]_H \ot_k \bh_{1s}
$$
and
$$
\lambda_1 \bigl([(e_1\ot_k e_2)\ot_{\!A} \ov{\bx}_{1s}\ot\ba_{1r}]\bigr)\lambda_2 := [(e_1\lambda_2 \ot_k \lambda_1 e_2)\ot_{\!A} \ov{\bx}_{1s} \ot_k \ba_{1r}].
$$
Let $\varrho_{rs}\colon X_{rs}\to \wh{X}_{rs}(E\ot_k E)$ be the $E$-bimodule isomorphisms defined by
$$
\varrho(e_2 \ot_{\!A} \bx_{1s}\ot \ba_{1r}\ot e_1) = [(e_1\ot_k e_2) \ot_{\!A} \bx_{1s}\ot \ba_{1r}],
$$
and let $\varpi^{rs}\colon \Hom_{E^e}\bigl(\ov{X}_{rs}(E\ot_k E),M\bigr) \to \ov{X}^{rs}(M)$ be the isomorphism given by
$$
\varpi(\alpha)(\ba_{1r}\ot_k \bh_{1s}):= \alpha\bigl([(1\ot_k 1)\ot\ba_{1r}]_H\ot_k \bh_{1s}\bigr).
$$
It is easy to see that the diagrams
\begin{equation}
\spreaddiagramcolumns{4pc}
\xymatrix{
\Hom_{E^e}\bigl(\ov{X}_{rs}(E\ot_k E),M\bigr) \rto^-{\Hom_{E^e}\bigl(\theta_{rs},M\bigr)} \ddto^-{\varpi^{rs}} & \Hom_{E^e}\bigl(\wh{X}_{rs}(E\ot_k E),M\bigr) \dto^{\Hom_{E^e}\bigl(\varrho_{rs},M\bigr)}\\
& \Hom_{E^e}\bigl(X_{rs},M\bigr)\dto^{\zeta^{rs}}\\
\ov{X}^{rs}(M)\rto^-{\theta^{rs}} & \wh{X}^{rs}(M)
}
\end{equation}\label{pp}
and
$$
\spreaddiagramcolumns{4pc}
\xymatrix{
\Hom_{E^e}\bigl(\wh{X}_{rs}(E\ot_k E),M\bigr)\dto^{\Hom_{E^e}(\varrho_{rs},M)} \rto^-{\Hom_{E^e}(\vartheta_{rs},M)} &\Hom_{E^e}\bigl(\ov{X}_{rs}(E\ot_k E),M\bigr) \ddto^-{\varpi^{rs}}\\
\Hom_{E^e}\bigl(X_{rs},M\bigr)\dto^{\zeta^{rs}} & \\
\wh{X}^{rs}(M)\rto^-{\vartheta^{rs}} & \ov{X}^{rs}(M),
}
$$
where

\begin{itemize}

\smallskip

\item[-] $\zeta^{rs}$ is the map introduced at the beginning of Section~\ref{coHochschild-Brzezinski},

\smallskip

\item[-] $\theta_{rs}$ and $\vartheta_{rs}$ are the morphisms introduced in Section~\ref{hom de inv},

\smallskip

\end{itemize}
commute. Hence $\theta^{rs}$ is invertible an $\vartheta^{rs}$ is its inverse.
\end{proof}

Let $\ov{d}_l^{rs}\colon \ov{X}^{r+l-1,s-l}(M)\to\ov{X}^{rs}(M)$ be the map $\ov{d}_l^{rs} := \vartheta^{rs}\xcirc \wh{d}_l^{rs}\xcirc \theta^{r+l-1,s-l}$.

\begin{theorem}\label{th7.1} The Hochschild cohomology of the $K$-algebra $E$ with coefficients in $M$ is the cohomology of $(\ov{X}^*(M),\ov{d}^*)$, where
$$
\ov{X}^n(M) := \bigoplus_{r+s = n} \ov{X}^{rs}(M)\qquad\text{and}\qquad \ov{d}^n := \sum^n_{l=1} \ov{d}_l^{0n} + \sum_{r=1}^n\sum^{n-r}_{l=0} \ov{d}_l^{r,n-r}.
$$
Moreover,
\begin{align*}
\ov{d}_0(\beta)(\bx) & = a_1\beta\bigl(\ba_{2r}\ot_k \bh_{1s}\bigr)\\
& + \sum_{i=1}^{r-1} (-1)^i \beta\bigl(\ba_{1,i-1}\ot a_ia_{i+1}\ot\ba_{i+1,r}\ot_k \bh_{1s} \bigr)\\
& + \sum_i (-1)^r \beta\bigl(\ba_{1,r-1}\ot_k \bh_{1s}^{(i)}\bigr)a_r^{(i)}
\intertext{and}
\ov{d}_1(\beta)(\bx)& = (-1)^r \ep(h_1)\beta\bigl(\ba_{1r}\ot_k \bh_{2s}\bigr)\\
& + \sum_{i=1}^{s-1}(-1)^{r+i} \beta\bigl(\ba_{1r}\ot_k \bh_{1,i-1}\ot_k h_ih_{i+1} \ot_k \bh_{i+2,s}\bigr)\\
& + \sum_{jl} (-1)^{r+s} \gamma^{-1}(h_s^{(1)(j)(l)(1)})\beta\bigl(h_s^{(1)(j)(l)(2)}\cdot \ba_{1r}^{(l)}\ot_k \bh_{1,s-1}^{(j)}\bigr)\gamma(h_s^{(2)}),
\end{align*}
where
\begin{align*}
&\, \bx := \ba_{1r}\ot_k \bh_{1s},\\
& \sum_i \bh_{1s}^{(i)}\ot_k a_r^{(i)} := \s^{-1}(a_r \ot_k \bh_{1s}),\\
& \sum_j h_s^{(1)(j)}\ot_k \bh_{1s-1}^{(j)}\ot_k h_s^{(2)} := c_{s-1,1}\bigl(\bh_{1,s-1}\ot_k h_s^{(1)}\bigr)\ot_k h_s^{(2)}
\intertext{and}
&\sum_{jl} h_s^{(1)(j)(l)}\ot_k \ba_{1r}^{(l)}\ot_k \bh_{1,s-1}^{(j)}\ot_k h_s^{(2)} := \sum_j s^{-1} \bigl(\ba_{1r}\ot_k h_s^{(1)(j)}\bigr)\ot_k \bh_{1,s-1}^{(j)}\ot_k h_s^{(2)}.
\end{align*}
\end{theorem}

\begin{proof} We will use the same notations as in the proof of Proposition~\ref{inversa de cotheta}. By that proposition and the definition of $(\ov{X}^*(M),\ov{d}^*)$, the map
$$
\theta^*\colon (\ov{X}^*(M),\ov{d}^*)\to (\wh{X}^*(M),\wh{d}^*),
$$
given by $\theta^n = \sum_{r+s = n} \theta^{rs}$, is an isomorphism of complexes. Hence, by the discussion at the beginning of Section~\ref{coHochschild-Brzezinski}, the cohomology of $(\ov{X}^*(M),\ov{d}^*)$ is the Hochschild cohomology of the $K$-algebra $E$ with coefficients in $M$. In order to complete the proof we must compute $\ov{d}_0$ and $\ov{d}_1$. Since, also
\begin{align*}
& \Hom_{E^e}(\theta_*,M) \colon \Hom_{E^e}\bigl((\ov{X}_*(E\ot_k\! E),\ov{d}_*),M\bigr) \longrightarrow \Hom_{E^e}((\wh{X}_*(E\ot_k\! E),\wh{d}_*),M),\\
& \Hom_{E^e}(\varrho_*,M)\colon \Hom_{E^e}((\wh{X}_*(E\ot_k\! E),\wh{d}_*),M) \longrightarrow \Hom_{E^e}((X_*,d_*),M)
\intertext{and}
& \zeta^*\colon \Hom_{E^e}((X_*,d_*),M) \longrightarrow \ov{X}^*(M),
\end{align*}
where
$$
\varrho_n := \sum_{r+s=n} \varrho_{rs}\quad\text{and}\quad \zeta^n := \sum_{r+s=n} \zeta^{rs},
$$
are isomorphisms of complexes, form the commutativity of the diagram~\eqref{pp}, it follows that
$$
\varpi^*\colon \Hom_{E^e}\bigl((\ov{X}_*(E\ot_k E),\ov{d}_*), M\bigr)\longrightarrow (\ov{X}^*(M),\ov{d}^*),
$$
where $\varpi^n := \sum_{r+s=n} \varpi^{rs}$, is also. Hence
\begin{align*}
&\ov{d}_0^{rs}(\beta)(\ba_{1r}\ot_k\bh_{1s}) = \varpi^{-1}(\beta)\bigl(\ov{d}^0_{rs}([(1\ot_k 1) \ot \ba_{1r}]_H\ot_k\bh_{1s})\bigr)
\intertext{and}
&\ov{d}_1^{rs}(\beta)(\ba_{1r}\ot_k\bh_{1s}) = \varpi^{-1}(\beta)\bigl(\ov{d}^1_{rs}([(1\ot_k 1)\ot \ba_{1r}]_H\ot_k\bh_{1s})\bigr).
\end{align*}
Now the desired result can be immediately obtained by using Theorem~\ref{th5.2}.
\end{proof}

\begin{notation} Given a $K$-subalgebra $R$ of $A$ and $0\le u\le r$, we let $\ov{X}_{Ru}^{rs}(M)$ denote the $k$-vector subspace of $\ov{X}^{rs}(M)$ consisting of all the $K^e$-linear maps
$$
\beta\colon \ov{A}^{\ot^r}\ot_k \ov{H}^{\ot_k^s} \to M
$$
that factorize throughout the $K^e$-subbimodule $\ov{X}^{R\mathfrak{r}u}_{r+u,s-u-1}$ of $\ov{A}^{\ot^{r+u}}\ot_k \ov{H}^{\ot_k^{s-u-r}}$ generated by the simple tensors $\ba_{1,r+u}\ot_k \bh_{1,s-u-1}$, with at least $u$ of the $a_j$'s in $\ov{R}$.
\end{notation}

\begin{proposition} Let $R$ be a stable under $s$ and $\rho$ subalgebra of $A$. If $f$ takes its values in $R$, then
$$
\ov{d}_l\bigl(\ov{X}^{r+l-1,s-l}(M)\bigr)\subseteq \ov{X}_{Ru}^{rs}(M),
$$
for all $l\ge 1$.
\end{proposition}

\begin{proof} This is an immediate consequence of item~(3) of Theorem~\ref{formula para wh{d} 1 en cohomologia}.
\end{proof}

\begin{remark} By the above proposition, if $f$ takes its values in $K$, then $(\ov{X}^*(M),\ov{d}^*)$ is the total complex of the double complex $\bigl(\ov{X}^{**}(M),\ov{d}^{**}_0,\ov{d}^{**}_1\bigr)$.
\end{remark}

\subsection{The filtration of $\mathbf{\boldsymbol{(\ov{X}^*(M),\ov{d}^*)}}$} Let $F_i(\ov{X}^n(M)) := \bigoplus_{s\ge i} \ov{X}^{n-s,s}(M)$. The cochain complex $(\ov{X}^*(M),\ov{d}^*)$ is filtered by
\begin{equation}
F_0(\ov{X}^*(M))\supseteq F_1(\ov{X}^*(M))\supseteq F_2(\ov{X}^*(M))\supseteq F_3(\ov{X}^*(M)) \supseteq\dots.\label{eq10}
\end{equation}

\begin{remark}\label{re def theta en cohomologia} By Proposition~\ref{inversa de cotheta} and the definition of $(\ov{X}^*(M),\ov{d}^*)$, the map
$$
\theta^*\colon (\ov{X}^*(M),\ov{d}^*)\longrightarrow (\wh{X}^*(M),\wh{d}^*),
$$
given by $\theta^n = \sum_{r+s = n} \theta^{rs}$, is an isomorphism of cochain complexes. It is evident that $\theta^*$ preserve filtrations. Consequently, the spectral sequence introduced in~\eqref{eq7} coincides with the spectral sequence associated with the filtration~\eqref{eq10}. Clearly the compositional inverse of $\theta^*$ is the map
$$
\vartheta^*\colon (\wh{X}^*(M),\wh{d}^*)\longrightarrow (\ov{X}^*(M),\ov{d}^*),
$$
defined by $\vartheta^n = \sum_{r+s = n} \vartheta^{rs}$.
\end{remark}

\subsection{Comparison maps}\label{morfismos de comparacion en cohomologia: caso cleft} Let
\begin{align*}
&\ov{\phi}^*\colon \bigl(\Hom_{K^e}(\ov{E}^{\ot^*},M),b^*\bigr)\longrightarrow (\ov{X}^*(M),\ov{d}^*)
\intertext{and}
&\ov{\psi}^*\colon (\ov{X}^*(M),\ov{d}^*)\longrightarrow \bigl(\Hom_{K^e}(\ov{E}^{\ot^*},M),b^* \bigr)
\end{align*}
be the morphisms of cochain complexes defined by $\ov{\phi}^*:= \vartheta^* \xcirc \wh{\phi}^*$ and $\ov{\psi}^*:= \wh{\phi}^* \xcirc \theta^*$, respectively. By the comments in Subsection~\ref{morfismos de comparacion en cohomologia}, we know that $\ov{\phi}\xcirc \ov{\psi} = \ide$ and $\ov{\psi}\xcirc \ov{\phi} = \wh{\psi}\xcirc \wh{\phi}$ is homotopically equivalent to the identity. Moreover, by Proposition~\ref{phi, psi y omega preservan filtraciones en cohomologia} and Remark~\ref{re def theta en cohomologia}, the morphisms $\ov{\phi}$ and $\ov{\psi}$, and the homotopy $\wh{\omega}^{*+1}\colon \ov{\psi}^*\xcirc \ov{\phi}^*\to \ide^*$, preserve filtrations.

\section{The cup and cap products for cleft crossed products}
Let $E:= A\#_f H$, $K$ and $M$ be as in Section~\ref{hom de inv}. Assume that $E$ is cleft. The aim of this section is to compute the cup product of $\HH_K^*(E)$ in terms of $(\ov{X}^*,\ov{d}^*)$ and the cap product of $\Ho^K_*(E,M)$ in terms of $(\ov{X}^*,\ov{d}^*)$ and $(\ov{X}_*(M),\ov{d}_*)$. We will use the diagrams introduced in~\eqref{s0}, \eqref{s1}, \eqref{s2}, \eqref{s3} and Remark~\ref{nuevas notaciones}. We will need the following generalization of the maps $\rho_r$ introduced in Definition~\ref{accion debil multiple}.

\begin{definition}\label{accion debil doble multiple} For all $r,s\in \mathds{N}$, we let $\rho_{sr}\colon H^{\ot_c^s}\ot A^{\ot^r}\to A^{\ot^r}$ denote the map recursively defined by
$$
\rho_{1r} := \rho_r\quad\text{and}\quad \rho_{s+1,r} = \rho_{1r}\xcirc (H\ot_k \rho_{sr}).
$$
For $h_1,\dots,h_s \in H$ and $a_1,\dots,a_r\in A$, we set $\bh_{1s}\cdot \ba_{1r}:= \rho_{sr}(\bh_{1s}\ot_k \ba_{1r})$.
\end{definition}

\begin{remark}\label{ultima diagramatica} The map $\rho_{sr}$ will be represented by the same diagram as $\rho$.
\end{remark}

\begin{notations} Let $B$ be a $k$-algebra. For all $n\in \mathds{N}$ we let $\mu_n\colon B^{\ot_k^n}\to B$ denote the map recursively defined by
$$
\mu_1 := \ide_B\quad\text{and}\quad \mu_{n+1}:= \mu_B\xcirc(\mu_n\ot_k B).
$$
\end{notations}

\begin{definition} For $\beta\in \ov{X}^{rs}$ and $\beta'\in \ov{X}^{r's'}$ we define $\beta\star \beta'\in \ov{X}^{r+r',s+s'}$ as $(-1)^{r's}$ times the map induced by
$$
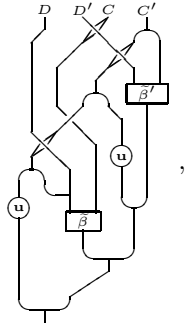
\spreaddiagramcolumns{-1.6pc}\spreaddiagramrows{-1.6pc}
\objectmargin{0.0pc}\objectwidth{0.0pc}
\def\objectstyle{\sssize}
\def\labelstyle{\sssize}
\grow{\xymatrix@!0{
&&\save\go+<0pt,3pt>\Drop{D}\restore \ar@{-}[1,0] &&&\save\go+<0pt,3pt>\Drop{D'}\restore \ar@{-}[2,2]&& \save\go+<0pt,3pt>\Drop{C}\restore \ar@{-}[1,-1]+<0.125pc,0.0pc> \ar@{-}[1,-1]+<0.0pc,0.125pc>&&&\save\go+<0pt,3pt>\Drop{C'}\restore  \ar@{-}[1,0]&\\
&&\ar@{-}[1,-1]&&&&&&&&\ar@{-}`l/4pt [1,-1] [1,-1] \ar@{-}`r [1,1] [1,1]&\\
&\ar@{-}[7,0]&&&&\ar@{-}[2,-2]\ar@{-}[-1,1]+<-0.125pc,0.0pc> \ar@{-}[-1,1]+<0.0pc,-0.125pc>&&\ar@{-}[2,2]&& \ar@{-}[1,-1]+<0.125pc,0.0pc> \ar@{-}[1,-1]+<0.0pc,0.125pc>&&\ar@{-}[3,0]+<0pc,-1pt>\\
&&&&&&&&&&&\\
&&&\ar@{-}[3,0]&&&&\ar@{-}[1,-1]\ar@{-}[-1,1]+<-0.125pc,0.0pc> \ar@{-}[-1,1]+<0.0pc,-0.125pc>&& \ar@{-}[1,0]+<0pc,-1pt>&&\\
&&&&&&\ar@{-}[1,0]&&&&&\\
&&&&&&\ar@{-}`l/4pt [1,-1] [1,-1] \ar@{-}`r [1,1] [1,1]&&&&*+<0.1pc>[F]{\,\,\,\wt{\beta}'\,\,\,}\ar@{-}[8,0]&\\
&&&\ar@{-}[1,1]+<-0.1pc,0.1pc> && \ar@{-}[2,-2]&&\ar@{-}[1,1]&&&&\\
&&&&&&&&\ar@{-}[2,0]+<0pc,-0.5pt>&&&\\
&\ar@{-}[2,2]&&\ar@{-}[1,-1]+<0.125pc,0.0pc> \ar@{-}[1,-1]+<0.0pc,0.125pc> &&\ar@{-}[-1,-1]+<0.1pc,-0.1pc>\ar@{-}[1,1]&&&&&&\\
&&&&&&\ar@{-}[5,0]+<0pc,-2pt>&&&&&\\
&\ar@{-}[1,0]\ar@{-}[-1,1]+<-0.125pc,0.0pc> \ar@{-}[-1,1]+<0.0pc,-0.125pc>&&\ar@{-}[1,1]&&&&&*+[o]+<0.40pc>[F]{\mathbf{u}}\ar@{-}[3,0]+<0pc,0pc>&&&\\
&\ar@{-}`l/4pt [1,-1] [1,-1] \ar@{-}`r [1,1] [1,1]&&&\ar@{-}[3,0]+<0pc,-2pt>&&&&&&&\\
\ar@{-}[1,0]+<0pc,-0.5pt>&&\ar@{-}`d/4pt [1,2][1,2]&&&&&&&&&\\
&&&&&&&&\ar@{-}`d/4pt [1,1] `[0,2] [0,2]&&&\\
*+[o]+<0.40pc>[F]{\mathbf{u}}\ar@{-}[7,0]+<0pc,0pc>&&&&&&&&&\ar@{-}[3,0]&&\\
&&&&&*+<0.1pc>[F]{\,\,\,\wt{\beta}\,\,\,}\ar@{-}[2,0]&&&&&&\\
&&&&&&&&&&&\\
&&&&&\ar@{-}`d/4pt [1,2] `[0,4] [0,4]&&&&&&\\
&&&&&&&\ar@{-}[1,0]&&&&\\
&&&&&&&\ar@{-}[2,-3]&&&&\\
&&&&&&&&&&&\\
\ar@{-}`d/4pt [1,2] `[0,4] [0,4]&&&&&&&&&&&\\
&&\ar@{-}[1,0]&&&&&&&&&\\
&&&&&&&&&&&
}}
\grow{\xymatrix@!0{
\\\\\\\\\\\\\\\\\\\\\\\\
\save\go+<0pt,0pt>\Drop{\txt{,}}\restore
}}
$$
where

\begin{itemize}

\smallskip

\item[-] $D:=A^{\ot_k^r}$, $D':=A^{\ot_k^{r'}}$, $C := H^{\ot_c^s}$ and $C' := H^{\ot_c^{s'}}$,

\smallskip

\item[-] $\wt{\beta}\colon D\ot_k C\to E$ and $\wt{\beta}'\colon D'\ot_k C'\to E$ are the maps induced by $\beta$ and $\beta'$, respectively,

\smallskip

\item[-] $\mathbf{u}:=\mu_{s'}\xcirc \gamma^{\ot_k^{s'}}$ and $\ov{\mathbf{u}}:=\mu_{s'} \xcirc \ov{\gamma}^{\ot_k^{s'}}\xcirc \gc_{s'}$, in which $\ov{\gamma}$ is the convolution inverse of $\gamma$.

\end{itemize}

\end{definition}

\begin{proposition}\label{cup producto en cleft} Let $\bullet$ be the operation introduced in Definition~\ref{producto bullet}. For each $\beta\in \ov{X}^{rs}$ and $\beta'\in \ov{X}^{r's'}$,
$$
\theta(\beta\star \beta') = \theta(\beta)\bullet \theta(\beta').
$$
\end{proposition}

\begin{proof} See Appendix~B.
\end{proof}

\begin{theorem}\label{th cup product en cleft} Let $\beta\in \ov{X}^{rs}$, $\beta'\in \ov{X}^{r's'}$ and $n := r+r'+s+s'$. Let $R$ be a stable under $s$ and $\rho$ $K$-subalgebra of $A$. If $f$ takes its values in $R$, then
$$
\ov{\phi}\bigl(\ov{\psi}(\beta)\smile \ov{\psi}(\beta')\bigr) = \beta\star\beta' \qquad \text{module }\, \bigoplus_{i>s+s'}\ov{X}^{n-i,i}_{R(1)},
$$
where $\ov{X}_{R(1)}^{n-i,i}$ denotes the $k$-vector subspace of $\ov{X}^{n-i,i}$ consisting of all the $K^e$-linear maps
$$
\beta\colon \ov{A}^{\ot^{n-i}}\ot_k \ov{H}^{\ot_k^{n-i}} \to E,
$$
that factorize throughout $W^{\mathfrak{r}}_n\cap C^{R\mathfrak{r}}_n$, where $W^{\mathfrak{r}}_n$ and $C^{R\mathfrak{r}}_n$ are as in Notation~\ref{not4.3}.
\end{theorem}

\begin{proof} This is an immediate consequence of Proposition~\ref{cup producto en cleft} and Theorem~\ref{th cup product}.
\end{proof}

\begin{corollary} If $f$ takes its values in $K$, then the cup product of $\HH_K^*(E)$ is induced by the operation $\star$ in $(\ov{X}^*,\ov{d}^*)$.
\end{corollary}

\begin{proof} It follows from Theorem~\ref{th cup product en cleft}, since $\ov{X}^{n-i,i}_{K(1)} = 0$ for all $i$.
\end{proof}

\begin{definition} Let $\beta\in \ov{X}^{r's'}$. For $r\ge r'$ and $s\ge s'$ we define
$$
\xymatrix @R=-2pt {\ov{X}_{rs}(M)\rto & \ov{X}_{r-r',s-s'}(M)\\
[m\ot\ba_{1r}]_H\ot_k \bh_{1s} \ar@{|->}[0,1] & ([m\ot\ba_{1r}]_H\ot_k \bh_{1s})\star \beta}
$$
as $(-1)^{r'(s-s')}$ times the morphism induced by
$$
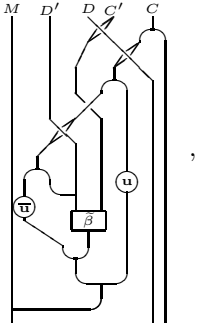
\spreaddiagramcolumns{-1.6pc}\spreaddiagramrows{-1.6pc}
\objectmargin{0.0pc}\objectwidth{0.0pc}
\def\objectstyle{\sssize}
\def\labelstyle{\sssize}
\grow{\xymatrix@!0{
\save\go+<0pt,3pt>\Drop{M}\restore \ar@{-}[24,0]&&&\save\go+<0pt,3pt>\Drop{D'}\restore \ar@{-}[8,0]&&& \save\go+<0pt,3pt>\Drop{D}\restore \ar@{-}[2,2] && \save\go+<0pt,3pt>\Drop{C'}\restore \ar@{-}[1,-1]+<0.125pc,0.0pc> \ar@{-}[1,-1]+<0.0pc,0.125pc> &&& \save\go+<0pt,3pt>\Drop{C}\restore \ar@{-}[1,0]&\\
&&&&&&&&&&&\ar@{-}`l/4pt [1,-1] [1,-1] \ar@{-}`r [1,1] [1,1]&\\
&&&&&&\ar@{-}[2,-1]\ar@{-}[-1,1]+<-0.125pc,0.0pc> \ar@{-}[-1,1]+<0.0pc,-0.125pc>&&\ar@{-}[2,2]&& \ar@{-}[1,-1]+<0.125pc,0.0pc> \ar@{-}[1,-1]+<0.0pc,0.125pc>&&\ar@{-}[22,0]\\
&&&&&&&&&&&&\\
&&&&&\ar@{-}[2,0]&&&\ar@{-}[1,0]\ar@{-}[-1,1]+<-0.125pc,0.0pc> \ar@{-}[-1,1]+<0.0pc,-0.125pc>&&\ar@{-}[1,1]&&\\
&&&&&&&&\ar@{-}`l/4pt [1,-1] [1,-1] \ar@{-}`r [1,1] [1,1]&&&\ar@{-}[19,0]&\\
&&&&&\ar@{-}[1,1]+<-0.1pc,0.1pc> && \ar@{-}[2,-2]&&\ar@{-}[6,0]+<0pc,-0.5pt>&&&\\
&&&&&&&&&&&&\\
&&&\ar@{-}[2,2]&& \ar@{-}[1,-1]+<0.125pc,0.0pc> \ar@{-}[1,-1]+<0.0pc,0.125pc>&&\ar@{-}[7,0]+<0pc,-1pt> \ar@{-}[-1,-1]+<0.1pc,-0.1pc>&&&&&\\
&&&&&&&&&&&&\\
&&&\ar@{-}[1,-1]\ar@{-}[-1,1]+<-0.125pc,0.0pc> \ar@{-}[-1,1]+<0.0pc,-0.125pc>&&\ar@{-}[5,0]+<0pc,-1pt> &&&&&&&\\
&&\ar@{-}[1,0]&&&&&&&&&&\\
&&\ar@{-}`l/4pt [1,-1] [1,-1] \ar@{-}`r [1,1] [1,1]&&&&&&&&&&\\
&\ar@{-}[1,0]+<0pc,-0.5pt>&&\ar@{-}`d/4pt [1,2][1,2]&&&&&&*+[o]+<0.4pc>[F]{\mathbf{u}}\ar@{-}[7,0]&&&\\
&&&&&&&&&&&&\\
&*+[o]+<0.37pc>[F]{\ov{\mathbf{u}}}\ar@{-}[1,0]&&&&&&&&&&&\\
&\ar@{-}[2,3]&&&&&*+<0.1pc>[F]{\,\,\,\widetilde{\beta}\,\,\,}\ar@{-}[2,0]&&&&&&\\
&&&&&&&&&&&&\\
&&&&\ar@{-}`d/4pt [1,1] `[0,2] [0,2]&&&&&&&&\\
&&&&&\ar@{-}[1,0]&&&&&&&\\
&&&&&\ar@{-}`d/4pt [1,2] `[0,4] [0,4]&&&&&&&\\
&&&&&&&\ar@{-}[1,0]&&&&&\\
&&&&&&&\ar@{-}`d/4pt [1,-7][1,-7]\ar@{-}`d/4pt [1,-2][1,-2]&&&&&\\
&&&&&&&&&&&&\\
&&&&&&&&&&&&&
}}
\grow{\xymatrix@!0{
\\\\\\\\\\\\\\\\\\\\\\
\save\go+<0pt,0pt>\Drop{\txt{,}}\restore
}}
$$
where

\begin{itemize}

\smallskip

\item[-] $D:=A^{\ot_k^r}$, $D':=A^{\ot_k^{r'}}$, $C := H^{\ot_c^s}$ and $C' := H^{\ot_c^{s'}}$,

\smallskip

\item[-] $\wt{\beta}\colon D'\ot_k C'\to E$ is the map induced by $\beta$,

\smallskip

\item[-] $\mathbf{u}:=\mu_{s'}\xcirc \gamma^{\ot_k^{s'}}$ and $\ov{\mathbf{u}}:=\mu_{s'} \xcirc \ov{\gamma}^{\ot_k^{s'}}\xcirc \gc_{s'}$, in which $\ov{\gamma}$ is the convolution inverse of $\gamma$.

\end{itemize}
If $r<r'$ or $s<s'$, then we set $([m\ot\ba_{1r}]_H\ot_k \bh_{1s})\star \beta:=0$.
\end{definition}

\begin{proposition}\label{cap producto en cleft} Let $\bullet$ be the action introduced in Definition~\ref{accion bullet}. The equality
$$
\vartheta\bigl(([m\ot\ba_{1r}]_H\ot_k \bh_{1s})\star \beta\bigr) = \vartheta \bigl([m\ot\ba_{1r}]_H \ot_k \bh_{1s} \bigr) \bullet \theta(\beta)
$$
holds for each $[m\ot\ba_{1r}]_H\ot_k \bh_{1s}\in \ov{X}_{rs}(M)$ and $\beta\in \ov{X}^{r's'}$.
\end{proposition}

\begin{proof} See Appendix~B.
\end{proof}

\begin{theorem}\label{th cap product en clef} Let $[m\ot\ba_{1r}]_H\ot_k \bh_{1s}\in \ov{X}_{rs}(M)$, $\beta\in \ov{X}^{r's'}$ and n:=r+s-r'-s'. Let $R$ be a stable under $\chi$ $K$-subalgebra of $A$. If $f$ takes its values in $R$, then
$$
\ov{\psi}\bigl(\ov{\phi}([m\ot\ba_{1r}]_H\ot_k \bh_{1s})\smallfrown \ov{\psi}(\beta)\bigr)  = ([m\ot\ba_{1r}]_H\ot_k \bh_{1s}) \star\beta
$$
module
$$
\bigoplus_{i<s-s'}\Bigl(\ov{X}_{n-i,i}^{R1}(M) + M\beta\bigl(\ov{X}^{R\mathfrak{r}}_{r's'} \bigr) \ot_{\!A} (E/A)^{\ot_{\!A}^{s-s'}} \ot \ov{A}^{\ot^{r-r'}}\Bigr),
$$
where $\ov{X}^{R\mathfrak{r}}_{r's'}$ denotes the $k$-vector subspace of $\ov{A}^{\ot^{r'}}\ot_k \ov{H}^{\ot_k^{s'}}$ generated by all the simple tensors $\ba_{1r}\ot_k \bh_{1,s}$, with at least $1$ of the $a_j$'s in $\ov{R}$.
\end{theorem}

\begin{proof} This is an immediate consequence of Proposition~\ref{cap producto en cleft} and Theorem~\ref{th cap product}.
\end{proof}

\begin{corollary} If $f$ takes its values in $K$, then in terms of $(\ov{X}_*(M),\ov{d}_*)$ and $(\ov{X}^*,\ov{d}^*)$, the cap product
$$
\Ho^K_n(E,M)\times \HH_K^m(E)\to \Ho^K_{n-m}(E,M),
$$
is induced by $\star$.
\end{corollary}

\begin{proof} It follows immediately from the previous theorem.
\end{proof}

\section{Cyclic homology of a cleft braided Hopf crossed product}
Let $E:= A\#_f H$, $K$ and $M$ be as in Section~\ref{hom de inv}. In this section we show that if $E$ is cleft, then the mixed complex $(\wh{X}_*,\wh{d}_*,\wh{D}_*)$ of Section~\ref{cyclic-Brzezinski} is isomorphic to a simpler mixed complex $(\ov{X}_*,\ov{d}_*,\ov{D}_*)$. We will use the diagrams introduced in~\eqref{s0}, \eqref{s1}, \eqref{s2}, \eqref{s3}, Remark~\ref{nuevas notaciones} and Remark~\ref{ultima diagramatica}.

\smallskip

Let $\theta_*\colon \wh{X}_*\to \ov{X}_*$ is the map introduced in Remark~\ref{re def theta} and $\vartheta_*$ is its inverse. Recall that $\theta_n = \bigoplus_{r+s=n}\theta_{rs}$. Hence $\vartheta_n = \bigoplus_{r+s=n}\vartheta_{rs}$, where $\vartheta_{rs}$ is the inverse of $\theta_{rs}$. For each $n\ge 0$, let $\ov{D}_n := \theta_{n+1}\xcirc \wh{D}_n\xcirc \vartheta_n$.

\begin{theorem}\label{complejo mezclado que da la homologia ciclica caso cleft} $\bigl(\ov{X}_*,\ov{d}_*,\ov{D}_*\bigr)$ is a mixed complex that gives the Hochschild, cyclic, negative and periodic homology of the $K$-algebra $E$. More precisely, the mixed complexes $\bigl(\ov{X}_*,\ov{d}_*,\ov{D}_*\bigr)$ and $\bigl(E\ot \ov{E}^{\ot^*},b_*,B_*\bigr)$ are homotopically
equivalent.
\end{theorem}

\begin{proof} Clearly $\bigl(\ov{X}_*,\ov{d}_*,\ov{D}_*\bigr)$ is a mixed complex and
$$
\theta_*\colon \bigl(\wh{X}_*,\wh{d}_*,\wh{D}_*\bigr)\to \bigl(\ov{X}_*,\ov{d}_*,\ov{D}_*\bigr)
$$
is an isomorphism of mixed complexes. So the result follows from Theorem~\ref{complejo mezclado que da la homologia ciclica}.
\end{proof}

We are now going to give a formula for $\ov{D}_n$. For $0\le j\le s$, let
$$
\wt{\tau}_j\colon M\ot\ov{A}^{\ot^r}\ot_k \ov{H}^{\ot_k^s}\longrightarrow M\ot\ov{A}^{\ot^r}\ot_k \ov{H}^{\ot_k^{s+1}}
$$
be the map induced by
$$
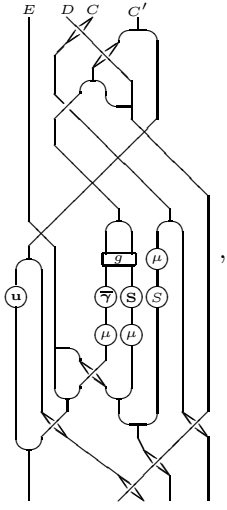
\spreaddiagramcolumns{-1.6pc}\spreaddiagramrows{-1.6pc}
\objectmargin{0.0pc}\objectwidth{0.0pc}
\def\objectstyle{\sssize}
\def\labelstyle{\sssize}
\grow{\xymatrix@!0{
&\save\go+<0pt,3pt>\Drop{E}\restore \ar@{-}[16,0] &&& \save\go+<0pt,3pt>\Drop{D}\restore \ar@{-}[2,2] &&  \save\go+<0pt,3pt>\Drop{C}\restore \ar@{-}[1,-1]+<0.125pc,0.0pc> \ar@{-}[1,-1]+<0.0pc,0.125pc> &&&\save\go+<0.2pc,3pt>\Drop{C'}\restore  \save[]+<0.2pc,0pc> \Drop{}\ar@{-}[1,0]+<0.2pc,0pc>\restore&&\\
&&&&&&&&& \save[]+<0.2pc,0pc> \Drop{}\ar@{-}`l/4pt [1,-1] [1,-1] \ar@{-}`r [1,2] [1,2]\restore&&\\
&&&&\ar@{-}[1,-1]\ar@{-}[-1,1]+<-0.125pc,0.0pc> \ar@{-}[-1,1]+<0.0pc,-0.125pc>&&\ar@{-}[2,2]&& \ar@{-}[1,-1]+<0.125pc,0.0pc> \ar@{-}[1,-1]+<0.0pc,0.125pc>&&& \ar@{-}[6,0] \\
&&&\ar@{-}[3,0]&&&&&&&&\\
&&&&&&\ar@{-}[-1,1]+<-0.125pc,0.0pc> \ar@{-}[-1,1]+<0.0pc,-0.125pc>\ar@{-}[1,0]&&\ar@{-}[1,1]&&&\\
&&&&&&\ar@{-}`l/4pt [1,-1] [1,-1] \ar@{-}`r [1,1] [1,1]&&&\ar@{-}[3,0]&&\\
&&&\ar@{-}[1,1]+<-0.1pc,0.1pc> && \ar@{-}[2,-2]&&\ar@{-}`d/4pt [1,2][1,2]&&&&\\
&&&&&&&&&&&\\
&&&\ar@{-}[2,0]&&\ar@{-}[-1,-1]+<0.1pc,-0.1pc>\ar@{-}[7,7]&&&&\ar@{-}[6,6]&&\ar@{-}[10,-10]\\
&&&&&&&&&&&&&&&\\
&&&\ar@{-}[5,5]&&&&&&&&&&&&\\
&&&&&&&&&&&&&&&\\
&&&&&&&&&&&&&&&\\
&&&&&&&&&&&&&&&\\
&&&&&&&&&&&&&&&\ar@{-}[17,0]\\
&&&&&&&&\ar@{-}[1,0]&&&&\ar@{-}[1,0]&&&\\
&\ar@{-}[2,2]&&&&&&&\ar@{-}`l/4pt [1,-1] [1,-1] \ar@{-}`r [1,1] [1,1]&&&&\ar@{-}`l/4pt [1,-1] [1,-1] \ar@{-}`r [1,1] [1,1]&&&\\
&&&&&&&\ar@{-}[1,0]+<0pc,-3pt>&&\ar@{-}[1,0]+<0pc,-3pt>&&\ar@{-}[1,0]+<0pc,-0.5pt>&&\ar@{-}[14,0]&&\\
&\ar@{-}[1,0]&&\ar@{-}[11,0]&&&&&&&&&&&&\\
&\ar@{-}`l/4pt [1,-1] [1,-1] \ar@{-}`r [1,1] [1,1]&&&&&&&*+<0.1pc>[F]{\,\,\,g\,\,\,} \ar@{-}@<0.40pc>[2,0]+<0pc,-0.5pt> \ar@{-}@<-0.40pc>[2,0]+<0pc,-0.5pt>&&& *+[o]+<0.40pc>[F]{\mu}\ar@{-}[2,0]+<0pc,-0.5pt> &&&&\\
\ar@{-}[1,0]+<0pc,0pc>&&\ar@{-}[11,0]&&&&&&&&&&&&&\\
&&&&&&&&&&&&&&&\\
*+[o]+<0.4pc>[F]{\mathbf{u}}\ar@{-}[11,0]+<0pc,0pc>&&&&&&&*+[o]+<0.37pc>[F]{\ov{\boldsymbol{\gamma}}} \ar@{-}[2,0]+<0pc,-0.5pt>&&*+[o]+<0.37pc>[F]{\mathbf{S}}\ar@{-}[2,0]+<0pc,-0.5pt>&& *+[o]+<0.37pc>[F]{S}\ar@{-}[9,0]&&&&\\
&&&&&&&&&&&&&&&\\
&&&&&&&&&&&&&&&\\
&&&&&&&*+[o]+<0.40pc>[F]{\mu}\ar@{-}[2,0]+<0pc,0pc>&&*+[o]+<0.40pc>[F]{\mu}\ar@{-}[4,0]+<0pc,0pc>&&&&&&\\
&&&\ar@{-}`r/4pt [1,2] [1,2]&&&&&&&&&&&&\\
&&&&&\ar@{-}[1,1]+<-0.125pc,0.0pc> \ar@{-}[1,1]+<0.0pc,0.125pc>&&\ar@{-}[2,-2]&&&&&&&&\\
&&&&&&&&&&&&&&&\\
&&&\ar@{-}`d/4pt [1,1] `[0,2] [0,2]&&&&\ar@{-}[-1,-1]+<0.125pc,0.0pc>\ar@{-}[-1,-1]+<0.0pc,-0.125pc> \ar@{-}`d/4pt [1,1] `[0,2] [0,2]&&&&&&&&\\
&&&&\ar@{-}[1,0]&&&&\ar@{-}[1,0]&&&&&&&\\
&&\ar@{-}[1,1]+<-0.125pc,0.0pc> \ar@{-}[1,1]+<0.0pc,0.125pc>&&\ar@{-}[2,-2]&&&&\ar@{-}`d/4pt [1,1]+<0.2pc,0pc> `[0,3] [0,3]&&&&&\ar@{-}[1,1]+<-0.125pc,0.0pc> \ar@{-}[1,1]+<0.0pc,0.125pc>&&\ar@{-}[2,-2]\\
&&&&&&&&&\save[]+<0.2pc,0pc> \Drop{}\ar@{-}[1,0]+<0.2pc,0pc>\restore &&&&&&\\
\ar@{-}`d/4pt [1,1] `[0,2] [0,2]&&&&\ar@{-}[3,4]\ar@{-}[-1,-1]+<0.125pc,0.0pc>\ar@{-}[-1,-1]+<0.0pc,-0.125pc> &&&&& \save[]+<0.2pc,0pc> \Drop{}\ar@{-}[1,1]\restore&&&&\ar@{-}[1,-1]&& \ar@{-}[-1,-1]+<0.125pc,0.0pc> \ar@{-}[-1,-1]+<0.0pc,-0.125pc>\ar@{-}[5,0]\\
&\ar@{-}[4,0]&&&&&&&&&\ar@{-}[1,1]+<-0.125pc,0.0pc> \ar@{-}[1,1]+<0.0pc,0.125pc>&&\ar@{-}[2,-2]&&&\\
&&&&&&&&&&&&&&&\\
&&&&&&&&\ar@{-}[1,1]+<-0.125pc,0.0pc> \ar@{-}[1,1]+<0.0pc,0.125pc>&&\ar@{-}[2,-2]&& \ar@{-}[-1,-1]+<0.125pc,0.0pc>\ar@{-}[-1,-1]+<0.0pc,-0.125pc>\ar@{-}[2,0]&&&\\
&&&&&&&&&&&&&&&\\
&&&&&&&&&&\ar@{-}[-1,-1]+<0.125pc,0.0pc>\ar@{-}[-1,-1]+<0.0pc,-0.125pc>&&&&&
}}
\grow{\xymatrix@!0{
\\\\\\\\\\\\\\\\\\\\\\\\\\\\\\\\\\\\\\
\save\go+<0pt,0pt>\Drop{\txt{,}}\restore
}}
$$
where

\begin{itemize}

\smallskip

\item[-] $D:=A^{\ot_k^r}$, $C := H^{\ot_c^j}$ and $C' := H^{\ot_c^{s-j}}$,

\smallskip

\item[-] $\ov{\boldsymbol{\gamma}}$ denotes the map $\ov{\gamma}^{\ot_k^{s-j}}$, where $\ov{\gamma}$ is the convolution inverse of~$\gamma$,

\smallskip

\item[-] $\mu$ denotes the maps $\mu_j$ and $\mu_{s-j}$,

\smallskip

\item[-] $\mathbf{u}$ denotes the map $\mu_{s-j} \xcirc \gamma^{\ot_k^{s-j}}$,

\smallskip

\item[-] $g$ denotes the map $\gc_{2s-2j}$ introduced in item~(5) of Section~\ref{hom de inv},

\smallskip

\item[-] $\mathbf{S}$ denote the map $S^{\ot_k^{s-j}}$.

\smallskip

\end{itemize}
and let $\ov{\tau}_j\colon \ov{X}_{rs}\to \ov{X}_{r,s+1}$ be the map induced by $\wt{\tau}_j$.

\smallskip

In the next theorem we use the notation $F_R^s(\ov{X}_{n+1}):=F^s(\ov{X}_{n+1})\cap \ov{X}^{R1}_{n+1}(E)$.

\begin{theorem}\label{morfismo D} Let $R$ be a stable under $s$ and $\rho$ subalgebra of $A$. If $f$ takes its values in $R$, then the map $\ov{D}_n\colon \ov{X}_{rs}\to \ov{X}_{n+1}$, where $r+s=n$, is given by
$$
\ov{D}_n = \sum_{j=0}^s (-1)^{r+s+js} \ov{\tau}_j
$$
module $F_R^s(\ov{X}_{n+1})$.
\end{theorem}

\begin{proof} See Appendix~B.
\end{proof}

Applying the previous theorem to the classical case (i.e. when $H$ is an standard Hopf algebra and $s\colon H\ot_k A\to A\ot_k H$ is the flip), we obtain an expression for $\ov{D}$ module $F_R^s(\ov{X}_{n+1})$, which is more convenient that the one given in~\cite[Theorem 3.3]{C-G-G}. Explicitly, we have:
\begin{align*}
\ov{D}(\bx) & = \sum_{j=0}^s (-1)^{r+s+js} \Bigl[\gamma(h_{j+1}^{(4)})\cdots \gamma(h_s^{(4)})a\gamma(h_0^{(1)})\gamma^{-1}(h_s^{(2)})\cdots \gamma^{-1}(h_{j+1}^{(2)})\\
&\,\ot h_{j+1}^{(3)}\cdot\bigl(\cdots h_{s-1}^{(3)}\cdot \bigl(h_s^{(3)}\cdot \ba_{1r}\bigr) \cdots \bigr)\Bigr]_H\!\ot_k\bh_{j+1,s}^{(5)}\ot_k h_0^{(2)} S(h_1^{(1)}\cdots h_s^{(1)})\ot_k \bh_{1j}^{(2)}
\end{align*}
module $F_R^s(\ov{X}_{n+1})$, where
\begin{align*}
&\bx := [a\gamma(h_0)\ot\ba_{1r}]_H\ot_k \bh_{1s},\\
& \bh_{1s}^{(1)}\ot_k \bh_{1s}^{(2)}\ot_k \bh_{j+1,s}^{(3)}\ot_k \bh_{j+1,s}^{(4)}\ot_k \bh_{j+1,s}^{(5)} := \bigr(\ide_{H^{\ot_{\!k}^j}}\ot_k \Delta^3_{H^{\ot_{\!c}^{\!s-j}}} \bigl)\xcirc \Delta_{H^{\ot_{\!c}^{\!s}}}(\bh_{1s})\\
\intertext{and}
&h\cdot \ba_{1r} := h^{(1)}\cdot a_1 \ot\cdots\ot h^{(r)}\cdot a_r,
\end{align*}
in which $h\cdot a$ denotes the weak action of $h\in H$ on $a\in A$.

\smallskip

\subsection{The spectral sequences} Let
$$
\ov{d}^0_{rs}\colon \ov{X}_{rs}\to\ov{X}_{r-1,s}\quad\text{and}\quad\ov{d}^1_{rs}\colon \ov{X}_{rs}\to \ov{X}_{r,s-1}
$$
be as above of Theorem~\ref{th5.2} and let
$$
\ov{D}^0_{rs}\colon \ov{X}_{rs}\to \ov{X}_{r,s+1}
$$
be the map defined by $\ov{D}^0(\bx) = \sum_{j=0}^{s-i} (-1)^{r+s+js} \ov{\tau}_j$.

\subsubsection{The first spectral sequence} Let
\begin{align*}
\check{d}_{rs}\colon \Ho_r\bigl(\ov{X}_{*s},\ov{d}^0_{*s}\bigr) \longrightarrow \Ho_r\bigl(\ov{X}_{*,s-1},\ov{d}^0_{*,s-1}\bigr)\\
\intertext{and}
\check{D}_{rs}\colon \Ho_r\bigl(\ov{X}_{*s},\ov{d}^0_{*s}\bigr) \longrightarrow \Ho_r \bigl(\ov{X}_{*,s+1},\ov{d}^0_{*,s+1}\bigr)
\end{align*}
be the maps induced by $\ov{d}^1$ and $\ov{D}^0$, respectively. Let
\begin{equation}
\mathpzc{F}^0\bigl(\Tot(\BC(\ov{X}_*,\ov{d}_*,\ov{D}_*))\bigr) \subseteq \mathpzc{F}^1\bigl(\Tot(\BC(\ov{X}_*,\ov{d}_*,\ov{D}_*))\bigr) \subseteq \cdots \label{eq15}
\end{equation}
be the filtration of $\Tot(\BC(\ov{X}_*,\ov{d}_*,\ov{D}_*))$, given by
$$
\mathpzc{F}^s\bigl(\Tot(\BC(\ov{X},\ov{d},\ov{D})_n)\bigr) := \bigoplus_{j\ge 0} F^{s-2j}(\ov{X}_{n-2j}) u^j,
$$
where $F^{s-2j}(\ov{X}_{n-2j})$ is the filtration introduce in Section~\ref{filtraciones en homologia caso cleft}. Since the isomorphism
$$
\theta_*\colon (\wh{X}_*,\wh{d}_*,\wh{D}_*)\longrightarrow (\ov{X}_*,\ov{d}_*,\ov{D}_*),
$$
satisfies
$$
\theta_n\Bigl(\mathpzc{F}^s\bigl(\Tot(\BC(\wh{X},\wh{d},\wh{D}))_n\bigr)\Bigr) = \mathpzc{F}^s\bigl(\Tot(\BC(\ov{X},\ov{d},\ov{D}))_n\bigr),
$$
where $\mathpzc{F}^s\bigl(\Tot(\BC(\wh{X},\wh{d},\wh{D}))_n\bigr)$ is as in the proof of Proposition~\ref{pepito}, the spectral sequence introduced in that proposition coincides with the one associated with the filtration~\eqref{eq15}. In particular
\[
\check{\Ho}_r\bigl(\ov{X}_{**},\ov{d}^0_{**}\bigr) := \Bigl(\Ho_r\bigl(\ov{X}_{**}, \ov{d}^0_{**}\bigr),\check{d}_{r*},\check{D}_{r*}\Bigr)
\]
is a mixed complex and
$$
\mathpzc{E}^2_{sr} = \HC_s\Bigl(\check{\Ho}_r\bigl(\ov{X}_{**},\ov{d}^0_{**}\bigr)\Bigr).
$$

\subsubsection{The second spectral sequence} For each $s\ge 0$, we consider the double complex
\[
\xymatrix{\\\\ \ov{\Xi}_s =}\qquad
\xymatrix{
\vdots \dto^-{\ov{d}^0} &\vdots \dto^-{\ov{d}^0}& \vdots \dto^-{\ov{d}^0}& \vdots \dto^-{\ov{d}^0}\\
\ov{X}_{3s}u^0\dto^-{\ov{d}^0} & \ov{X}_{3,s-1}u^1\lto_-{\ov{D}^0}\dto^-{\ov{d}^0} & \ov{X}_{3,s-2}u^2\lto_-{\ov{D}^0}\dto^-{\ov{d}^0} & \ov{X}_{3,s-3}u^3\lto_-{\ov{D}^0} \\
\ov{X}_{2s}u^0\dto^-{\ov{d}^0} & \ov{X}_{2,s-1}u^1\lto_-{\ov{D}^0}\dto^-{\ov{d}^0} & \ov{X}_{2,s-2}u^2\lto_-{\ov{D}^0}\\
\ov{X}_{1s}u^0\dto^-{\ov{d}^0} & \ov{X}_{1,s-1}u^1\lto_-{\ov{D}^0}\\
\ov{X}_{0s}u^0}
\]
where $\ov{X}_{0s} u^0$ is placed in the intersection of the $0$-th column and the $0$-th row. Let
\begin{equation}
F^0\bigl(\Tot(\BC(\ov{X}_*,\ov{d}_*,\ov{D}_*))\bigr) \subseteq F^1\bigl(\Tot(\BC(\ov{X}_*,\ov{d}_*,\ov{D}_*))\bigr) \subseteq \cdots \label{eq11}
\end{equation}
be the filtration of $\Tot(\BC(\ov{X}_*,\ov{d}_*,\ov{D}_*))$, given by
$$
F^s\bigl(\Tot(\BC(\ov{X},\ov{d},\ov{D}))_n\bigr) := \bigoplus_{j\ge 0} F^{s-j}(\ov{X}_{n-2j}) u^j,
$$
where $F^{s-j}(\ov{X}_{n-2j})$ is the filtration introduce in Section~\ref{filtraciones en homologia caso cleft}. Since the isomorphism
$$
\theta_*\colon (\wh{X}_*,\wh{d}_*,\wh{D}_*)\longrightarrow (\ov{X}_*,\ov{d}_*,\ov{D}_*),
$$
satisfies
$$
\theta_n\Bigl(F^s\bigl(\Tot(\BC(\wh{X},\wh{d},\wh{D}))_n\bigr)\Bigr) = F^s\bigl(\Tot(\BC(\ov{X},\ov{d},\ov{D}))_n\bigr),
$$
where $F^s\bigl(\Tot(\BC(\wh{X},\wh{d},\wh{D}))_n\bigr)$ is as in the proof of Proposition~\ref{pepe}, the spectral sequence introduced in that proposition coincides with the one associated with the filtration~\eqref{eq11}. In particular $\Tot(\wh{\Xi}_s)\simeq \Tot(\ov{\Xi}_s)$, and so $E^1_{sr} = \Ho(\Tot(\ov{\Xi}_s))$ for all $r,s\ge 0$.

\subsubsection{The third spectral sequence} Let
\begin{align*}
\check{d}_{rs}\colon \Ho_s\bigl(\ov{X}_{r*},\ov{d}^1_{r*}\bigr) \longrightarrow \Ho_s\bigl(\ov{X}_{r-1,*},\ov{d}^1_{r-1,*}\bigr)\\
\intertext{and}
\check{D}_{rs}\colon \Ho_s\bigl(\ov{X}_{r*},\ov{d}^1_{r*}\bigr) \longrightarrow \Ho_s \bigl(\ov{X}_{r,*+1},\ov{d}^1_{r,*+1}\bigr)
\end{align*}
be the maps induced by $\ov{d}^0$ and $\ov{D}^0$, respectively. Let
\begin{equation}
\mathfrak{F}^0\bigl(\Tot(\BC(\ov{X}_*,\ov{d}_*,\ov{D}_*))\bigr) \subseteq \mathfrak{F}^1\bigl(\Tot(\BC(\ov{X}_*,\ov{d}_*,\ov{D}_*))\bigr) \subseteq \cdots \label{eq16}
\end{equation}
be the filtration of $\Tot(\BC(\ov{X}_*,\ov{d}_*,\ov{D}_*))$, given by
$$
\mathfrak{F}^r\bigl(\Tot(\BC(\ov{X},\ov{d},\ov{D})_n)\bigr) := \bigoplus_{j\ge 0} \mathfrak{F}^{r-j}(\ov{X}_{n-2j}) u^j,
$$
where
$$
\mathfrak{F}^{r-j}(\ov{X}_{n-2j}):= \bigoplus_{i\le r-j} \ov{X}_{i,n-i-2j}.
$$
Since the isomorphism
$$
\theta_*\colon (\wh{X}_*,\wh{d}_*,\wh{D}_*)\longrightarrow (\ov{X}_*,\ov{d}_*,\ov{D}_*),
$$
satisfies
$$
\theta_n\Bigl(\mathfrak{F}^r\bigl(\Tot(\BC(\wh{X},\wh{d},\wh{D}))_n\bigr)\Bigr) = \mathfrak{F}^r\bigl(\Tot(\BC(\ov{X},\ov{d},\ov{D}))_n\bigr),
$$
where $\mathfrak{F}^r\bigl(\Tot(\BC(\wh{X},\wh{d},\wh{D}))_n\bigr)$ is as in the proof of Proposition~\ref{pepitos}, the spectral sequence introduced in that proposition coincides with the one associated with the filtration~\eqref{eq16}. In particular
\[
\check{\Ho}_s\bigl(\ov{X}_{**},\ov{d}^1_{**}\bigr) := \Bigl(\Ho_s\bigl(\ov{X}_{**}, \ov{d}^1_{**}\bigr),\check{d}_{*s},\check{D}_{*s}\Bigr)
\]
is a mixed complex and
$$
\mathfrak{E}^2_{rs} = \HC_r\Bigl(\check{\Ho}_s\bigl(\ov{X}_{**},\ov{d}^1_{**}\bigr)\Bigr).
$$

\subsubsection{The fourth spectral sequence} Assume that $f$ takes its values in $K$. Then the mixed complex $\bigl(\ov{X}_*,\ov{d}_*,\ov{D}_*\bigr)$ is filtrated by
\begin{equation}
\mathcal{F}^0\bigl(\ov{X}_*,\ov{d}_*,\ov{D}_*\bigr)\subseteq \mathcal{F}^1\bigl(\ov{X}_*,\ov{d}_*,\ov{D}_*\bigr)\subseteq \mathcal{F}^2\bigl(\ov{X}_*,\ov{d}_*,\ov{D}_*\bigr)\subseteq\cdots,\label{eq12}
\end{equation}
where
$$
\mathcal{F}^r(\ov{X}_n) := \bigoplus_{i\le r} \ov{X}_{i,n-i}.
$$
Hence, for each $r\ge 1$, we can consider the quotient mixed complex
$$
\ov{\mathfrak{X}}^r: =\frac{\mathcal{F}^r \bigl(\ov{X}_*,\ov{d}_*,\ov{D}_*\bigr)} {\mathcal{F}^{r-1}\bigl(\ov{X}_*,\ov{d}_*,\ov{D}_*\bigr)}.
$$
It is easy to check that the Hochschild boundary map of $\ov{\mathfrak{X}}^r$ is $\ov{d}^1_{r*}\colon\ov{X}_{r*}\to \ov{X}_{r,*-1}$, and that, by item~(1) of Theorem~\ref{morfismo D}, its Connes operator is $\ov{D}^0_{r*}\colon \ov{X}_{r*}\to \ov{X}_{r,*+1}$. Since the isomorphism
$$
\theta_*\colon (\wh{X}_*,\wh{d}_*,\wh{D}_*)\longrightarrow (\ov{X}_*,\ov{d}_*,\ov{D}_*),
$$
satisfies
$$
\theta_*\Bigl(\mathcal{F}^r\bigl(\wh{X}_*,\wh{d}_*,\wh{D}_*\bigr)\Bigr) = F^r\bigl(\ov{X}_*,\ov{d}_*,\ov{D}_*\bigr),
$$
where $\mathcal{F}^r\bigl(\wh{X}_*,\wh{d}_*,\wh{D}_*\bigr)$ is as in Section~\ref{segunda suc esp}, the spectral sequence introduced in Proposition~\ref{pepe1} coincides with the one associated with the filtration~\eqref{eq12}. In particular $\ov{\mathfrak{X}}^r\simeq \wh{\mathfrak{X}}^r$ and so $\mathcal{E}^1_{rs} = \HC_s(\ov{\mathfrak{X}}^r)$.

\section{Crossed products in Yetter-Drinfeld Categories}
The results established in this paper apply in particular to crossed products in the category ${}_L^L \mathcal{YD}$ of Yetter-Drinfeld modules over a Hopf algebra $L$. Next we consider the case where $L$ is a group algebra $k[G]$, with $k$ a field. Recall that a Yetter-Drinfeld module over $k[G]$ is a $k[G]$-module $M$, endowed with a $G$-gradation $M = \bigoplus_{\sigma\in G} M_{\sigma}$ such that ${}^{\sigma}\nesp  M_{\tau} = M_{\sigma\tau \sigma^{-1}}$ for all $\sigma,\tau\in G$. A morphism of Yetter-Drinfeld module over $k[G]$ is a $k[G]$-linear map $\varphi\colon M\to M'$ such that $\varphi (M_{\sigma})\subseteq M'_{\sigma}$ for all $\sigma\in G$. The category ${}^G_G \mathcal{YD}:={}^{k[G]}_{k[G]} \mathcal{YD}$, of Yetter-Drinfeld modules over $k[G]$, is a braided category. The unit object is the $k[G]$-trivial module $k$ concentrated in degree one; the tensor product $M\ot_k N$ of two Yetter-Drinfeld modules over $k[G]$ is the usual tensor product over $k$, endowed with the diagonal action and with the gradation defined by
$$
(M\ot_k N)_{\sigma} = \bigoplus_{\tau_1,\tau_2\in G\atop \tau_1\tau_2 = \sigma} M_{\tau_1}\ot_k M_{\tau_2};
$$
the associative and unit constraints are the usual ones; and the braided
$$
c_{MN}\colon M\ot_k N\to N\ot_k M,
$$
is given by $c_{MN}(m\ot n) = {}^\sigma\nesp  n\ot m$, for $m\in M_{\sigma}$. In this section we obtain formulas for the boundary maps $\overline{d}^0$ and $\overline{d}^1$ (see Theorem~\ref{th5.2}) and for the coboundary maps $\overline{d}_0$ and $\overline{d}_1$ (see Theorem~\ref{th7.1}), when $A\#_f H$ is a cleft crossed product in ${}^G_G \mathcal{YD}$, whose transposition $s\colon H\ot_k A \to A\ot_k H$ is $c_{HA}$. We do not give a formula for the Connes operator $\ov{D}$, because in general the computations are very involved.

\smallskip

We will use the following facts:

\begin{enumerate}

\medskip

\item An (associative and unitary) algebra in ${}^G_G \mathcal{YD}$ is a Yetter-Drinfeld module $A$ over $k[G]$, endowed with and associative and unitary multiplication such that
    $$
    \quad 1\in A_1,\quad A_{\sigma}A_{\tau} \subseteq A_{\sigma \tau},\quad {}^{\sigma}\nesp 1 = 1 \quad\text{and}\quad {}^{\sigma}\nespp (ab) = {}^{\sigma}\nespp a{}^{\sigma}\nesp b,
    $$
    for all $\sigma,\tau\in G$ and $a,b\in A$.

\smallskip

\item A (coassociative and counitary) coalgebra in ${}^G_G \mathcal{YD}$ is a Yetter-Drinfeld module $C$ over $k[G]$, endowed with and coassociative and counitary comultiplication such that
    \begin{xalignat*}{2}
    \qquad &\varepsilon(C_{\sigma}) = 0\text{\,\,if $\sigma\ne 1$}, && \Delta(C_{\sigma}) \subseteq \bigoplus_{\tau_1,\tau_2\in G\atop \tau_1\tau_2 = \sigma} C_{\tau_1}\ot_k C_{\tau_2},\\
    &\varepsilon ({}^{\sigma}\nespp c) = \varepsilon(c), && ({}^{\sigma}\nespp c)^{(1)}\ot_k ({}^{\sigma}\nespp c)^{(2)} = {}^{\sigma}\nespp (c^{(1)})\ot_k {}^{\sigma}\nespp (c^{(2)}),
    \end{xalignat*}
    for all $\sigma\in G$ and $c\in C$. Because of the compatibility between the gradation and the comultiplication of $C$, given $c\in C_{\sigma}$, we can write
    $$
    \Delta(c) = \sum_{\tau_1,\tau_2\in G\atop \tau_1\tau_2 = \sigma} c_{\tau_1}^{(1)}\ot_k c_{\tau_2}^{(2)},
    $$
    where $c_{\tau_1}^{(1)}\ot_k c_{\tau_2}^{(2)}\in C_{\tau_1}\ot_k C_{\tau_2}$ is a sum of simple tensors.

\smallskip

\item A a Yetter-Drinfeld module $H$ over $k[G]$ is a bialgebra in ${}^G_G \mathcal{YD}$ if it is an algebra and a coalgebra in ${}^G_G \mathcal{YD}$, whose counit $\varepsilon$ and comultiplication $\Delta$ satisfy
    \begin{align*}
    \quad &\varepsilon(1) = 1,\quad \varepsilon(hl) = \varepsilon(h)\varepsilon(l),\quad \Delta(1) = 1\ot 1
    \intertext{and}
    & (hl)_{\sigma\varsigma\tau^{-1}}^{(1)}\ot_k (hl)_{\tau}^{(1)} = \sum_{\upsilon,\nu\in G\atop \upsilon\nu = \tau} h^{(1)}_{\sigma \upsilon^{-1}} {}^{\upsilon}\nesp l^{(1)}_{\varsigma \nu^{-1}} \ot_k h^{(2)}_{\upsilon}l^{(2)}_{\nu},
    \end{align*}
    for all $h\in H_{\sigma}$, $l\in H_{\varsigma}$ and $\tau\in G$.

\smallskip

\item A bialgebra $H$ in ${}^G_G \mathcal{YD}$ is a Hopf algebra in ${}^G_G \mathcal{YD}$ if $\ide_H$ is invertible with respect to the convolution product in $\Hom_{{}^G_G \mathcal{YD}}(H,H)$. That is if there exists a map $S\colon H\to H$ of Yetter-Drinfeld modules, called the antipode of $H$, such that
    $$
    \quad S(h^{(1)})h^{(2)} = h^{(1)}S(h^{(2)})= \epsilon(h)1\quad\text{for all $h\in H$.}
    $$

\smallskip

\item Let $H$ be a Hopf algebra in ${}^G_G \mathcal{YD}$ and let $A$ be an algebra in ${}^G_G \mathcal{YD}$. In the sequel we let $s$ denote the braid $c_{HA}$. It is evident that $s$ is a transposition. A weak $s$-action of $H$ on $A$ is a map of Yetter-Drinfeld modules
    $$
    \quad \xymatrix @R=-2pt {H\ot_k A\rto^-{\rho} & A\\
    h\ot_k a \ar@{|->}[0,1] & h\cdot a}
    $$
    such that
    \begin{align*}
    \quad & 1\cdot a = a,\\
    & h\cdot 1 = \varepsilon(h)1,\\
    & l\cdot (ab) = \sum_{\varsigma,\tau\in G\atop \varsigma\tau = \sigma} (l^{(1)}_{\varsigma}\cdot {}^{\tau}\nespp  a)(l^{(2)}_{\tau}\cdot b),
    \end{align*}
    for $h\in H$, $l\in H_{\sigma}$ and $a,b\in A$.

\smallskip

\item Let $H$ be a Hopf algebra in ${}^G_G \mathcal{YD}$ and let $A$ be an algebra in ${}^G_G \mathcal{YD}$. Assume we have a weak $s$-action $\rho$ of $H$ on $A$. A map of Yetter-Drinfeld modules
    $$
    \quad f\colon H\ot_k H\to A
    $$
    is a normal cocycle that satisfies the twisted module condition if
    \begin{align*}
    \qquad &\,\,\, f(1\ot_k h) = f(h\ot_k 1) = \epsilon(h)1,\\[2.5pt]
    & \sum_{\sigma_1,\sigma_2,\tau_1,\tau_2,\upsilon_1,\upsilon_2\in G\atop \sigma_1\sigma_2 = \sigma, \, \tau_1\tau_2 = \tau,\,\upsilon_1\upsilon_2 = \upsilon} h^{(1)}_{\sigma_1} \cdot f\bigl({}^{\sigma_2}\nesp l^{(1)}_{\tau_1} \ot_k {}^{\sigma_2\tau_2}\nesp  m^{(1)}_{\upsilon_1}\bigr) f\bigl(h^{(2)}_{\sigma_2}\ot_k l^{(2)}_{\tau_2}m^{(2)}_{\upsilon_2}\bigr)\\
    &\phantom{\sum_{{\sigma_1\sigma_2 = \sigma \atop \tau_1\tau_2 = \tau}\atop \upsilon_1\upsilon_2 = \upsilon} h^{(1)}_{\sigma_1}} = \sum_{\sigma_1,\sigma_2, \tau_1,\tau_2\in G\atop \sigma_1\sigma_2 = \sigma, \,\tau_1\tau_2 = \tau} f\bigl(h^{(1)}_{\sigma_1}\ot_k {}^{\sigma_2}\nesp l^{(1)}_{\tau_1}\bigr) f\bigl(h^{(2)}_{\sigma_2} l^{(2)}_{\tau_2}\ot_k m \bigr)
    \intertext{and}
    &\sum_{\sigma_1,\sigma_2,\tau_1,\tau_2\in G\atop \sigma_1\sigma_2 = \sigma,\,  \tau_1\tau_2 = \tau} h^{(1)}_{\sigma_1} \cdot \bigl({}^{\sigma_2}\nesp l^{(1)}_{\tau_1} \cdot {}^{\sigma_2\tau_2}\nesp a \bigr)f\bigl(h^{(2)}_{\sigma_2} \ot_k l^{(2)}_{\tau_2}\bigr)\\
    &\phantom{\sum_{{\sigma_1\sigma_2 = \sigma \atop \tau_1\tau_2 = \tau}\atop \upsilon_1 \upsilon_2 = \upsilon} h^{(1)}_{\sigma_1}} = \sum_{\sigma_1,\sigma_2,\tau_1,\tau_2\in G \atop \sigma_1\sigma_2=\sigma,\,\tau_1\tau_2=\tau} f\bigl(h^{(1)}_{\sigma_1} \ot_k {}^{\sigma_2}\nesp l^{(1)}_{\tau_1}\bigr)\,\bigl(h^{(2)}_{\sigma_2} l^{(2)}_{\tau_2} \bigr)\cdot a,
    \end{align*}
    for $h\in H_{\sigma}$, $l\in H_{\tau}$, $m\in M_{\upsilon}$ and $a\in A$.

\smallskip

\item Let $H$ be a Hopf algebra in ${}^G_G \mathcal{YD}$ and let $A$ be an algebra in ${}^G_G \mathcal{YD}$. Assume that we have a weak $s$-action $\rho$ of $H$ on $A$ and a convolution invertible normal cocycle $f$ that satisfies the twisted module condition. By definition $E:=A\#_f H$ is the associative unitary $k$-algebra with underlying $k$-vector space $A\ot_k H$ and multiplication map
    $$
    \qquad\quad (a\# h)(b\# l) = \sum_{\sigma_1,\sigma_2,\sigma_3,\tau_1, \tau_2\in G \atop \sigma_1\sigma_2\sigma_3 = \sigma,\, \tau_1 \tau_2 = \tau} a (h^{(1)}_{\sigma_1}\cdot {}^{\sigma_2\sigma_3}\nesp b) f \bigl(h^{(2)}_{\sigma_2} \ot_k {}^{\sigma_3}\nesp l^{(1)}_{\tau_1}\bigr)\# h^{(3)}_{\sigma_3} l^{(2)}_{\tau_2},
    $$
    for $h\in H_{\sigma}$, $l\in H_{\tau}$ and $a\in A$.

\end{enumerate}

\begin{remark} Let $C$ be a coalgebra in ${}^G_G \mathcal{YD}$, $\sigma,\tau \in G$ and $c\in C_{\sigma}$. In relation with item~(2) above, note that
$$
\quad \sum_{v_1,v_2\in G\atop v_1v_2 = \sigma} {}^{\tau}\nespp (c_{v_1}^{(1)})\ot_k {}^{\tau}\nespp (c_{v_2}^{(2)}) = \sum_{v_1,v_2\in G\atop v_1v_2 = \sigma} ({}^{\tau}\nespp c)_{\tau v_1\tau^{-1}}^{(1)}\ot_k ({}^{\tau}\nespp c)_{\tau v_2\tau^{-1}}^{(2)}.
$$
To the sake of simplicity we set
$$
\quad \sum_{v_1,v_2\in G\atop v_1v_2 = \sigma} {}^{\tau}\nespp c_{v_1}^{(1)}\ot_k {}^{\tau}\nespp c_{v_2}^{(2)}:= \sum_{v_1,v_2\in G\atop v_1v_2 = \sigma} {}^{\tau}\nespp (c_{v_1}^{(1)})\ot_k {}^{\tau}\nespp (c_{v_2}^{(2)}).
$$
\end{remark}

\medskip

\noindent In order to give a convenient expression for the maps $\ov{d}_1$ and $\ov{d}^1$ in Theorems~\ref{d1 en homologia en YD} and~\ref{d1 en cohomologia en YD} respectively, we will need introduce the following notations:

\begin{itemize}

\smallskip

\item[-] Let $A$ be an algebra in ${}^G_G \mathcal{YD}$, $\tau\in G$ and $a_1,\dots,a_r\in A$. We set
$$
\quad {}^{\tau}\nespp \ba_{1r} := {}^{\tau}\nespp a_1\ot \cdots\ot {}^{\tau}\nespp a_r
$$

\smallskip

\item[-] Let $H$ be a Hopf algebra in ${}^G_G \mathcal{YD}$, $A$ be an algebra in ${}^G_G \mathcal{YD}$, $\sigma\in G$, $h\in H_{\sigma}$ and $a_1,\dots,a_r\in A$. We set
$$
\quad\qquad h \triangleright \ba_{1r} := \sum_{\tau_1,\dots, \tau_r\in G \atop \tau_1\cdots \tau_r = \sigma} h^{(1)}_{\tau_1}\cdot {}^{\tau_1^{-1}}\nespp a_1 \ot h^{(2)}_{\tau_2}\cdot {}^{\tau_2^{-1}\tau_1^{-1}}\nespp a_2\ot\cdots \ot h^{(r)}_{\tau_r}\cdot {}^{\tau_r^{-1}\cdots \tau_1^{-1}}\nespp a_r.
$$

\smallskip

\end{itemize}

\begin{theorem}\label{d1 en homologia en YD} Let $E$ be as in item~(7) above and let $M$ be an $E$-bimodule. Let $K$ be a Yetter-Drinfeld subalgebra of $A$ (that is, $K$ is a $G$-graded subalgebra of $A$, which is closed under the action of $G$ over $A$). Assume that $\rho(H\ot_k K)\subseteq K$ and let $(\ov{X}_*(M),\ov{d}_*)$ be as in Theorem~\ref{th5.2}. Then the terms $\ov{d}^0$ and $\ov{d}^1$ of $\ov{d}$ are given by:
\begin{align*}
\ov{d}^0(\bx)& = [ma_1\ot \ba_{2r}]_H\ot_k \bh_{1s}\\
& + \sum_{i=1}^{r-1} (-1)^i [m\ot  \ba_{1,i-1}\ot a_ia_{i+1} \ot \ba_{i+1,r}]_H \ot_k \bh_{1s}\\
& + (-1)^r [{}^{\sigma_s^{-1}\cdots\sigma_1^{-1}}\nespp a_r m\ot \ba_{1,r-1}]_H\ot_k \bh_{1s}
\intertext{and}
\ov{d}^1(\bx)& = (-1)^r \bigl[m\ep(h_1)\ot \ba_{1r}\bigr]_H \ot_k \bh_{2s}\\
& + \sum_{i=1}^{s-1}(-1)^{r+i} \bigl[m\ot \ba_{1r}\bigr]_H \ot_k \bh_{1,i-1}\ot_k h_ih_{i+1} \ot_k \bh_{i+2,s}\\
& + \sum_{\tau_1,\tau_2,\tau_3 \in G \atop \tau_1\tau_2\tau_3 = \sigma_s} (-1)^{r+s}\bigl[\gamma(h_{s,\tau_3}^{(3)})m \gamma^{-1}({}^{\ov{\sigma}}\nesp h_{s,\tau_1}^{(1)})\ot {}^{\ov{\sigma}}\nesp h_{s,\tau_2}^{(2)}\triangleright \ba'_{1r}\bigr]_H \ot_k \bh_{1,s-1},
\end{align*}
in which
\begin{align*}
& \bx := [m\ot\ba_{1r}]_H\ot_k \bh_{1s}\quad\text{with $h_1\in H_{\sigma_1},\dots,h_s\in H_{\sigma_s}$,}\\
& \ov{\sigma} := \sigma_1\cdots \sigma_{s-1} \qquad \ba'_{1r}:= {}^{\ov{\sigma} \tau_1^{-1} \ov{\sigma}^{-1}}\nespp \ba_{1r}
\intertext{and}
& \sum_{\tau_1,\tau_2,\tau_3 \in G \atop \tau_1\tau_2\tau_3 = \sigma_s} h_{s,\tau_1}^{(1)} \ot_k h_{s,\tau_2}^{(2)}\ot_k h_{s,\tau_3}^{(3)} := \Delta^2(h_s),
\end{align*}
with $h_{s,\tau_1}^{(1)} \ot_k h_{s,\tau_2}^{(2)}\ot_k h_{s,\tau_3}^{(3)}\in H_{\tau_1}\ot_k H_{\tau_2}\ot_k H_{\tau_3}$.
\end{theorem}

\begin{proof} The formula for $\ov{d}^0$ follows immediately from Theorem~\ref{th5.2} and the fact that
$$
s^{-1}(a\ot_k \bh_{1s}) = \bh_{1s}\ot_k {}^{\sigma_s^{-1}\cdots\sigma_1^{-1}}\nesp a
$$
for each $h_1\in H_{\sigma_1},\dots,h_s\in H_{\sigma_s}$ and $a\in A$, while the formula for $\ov{d}^1$ follows from Theorem~\ref{th5.2} and the fact that
\begin{align*}
& c\bigl(\bh_{1,s-1}\ot_k h_s^{(1)}\bigr)\ot_k h_s^{(2)} = \sum_{\tau_1,\tau_2\in G\atop \tau_1\tau_2=\sigma_s} {}^{\ov{\sigma}}\nesp h_{s,\tau_1}^{(1)}\ot_k\bh_{1,s-1}\ot_k h_{s,\tau_2}^{(2)}
\intertext{and}
&\sum_{\tau_1,\tau_2,\tau_3 \in G \atop \tau_1\tau_2\tau_3 = \sigma_s} {}^{\ov{\sigma}}\nesp h_{s,\tau_1}^{(1)} \ot_k{}^{\ov{\sigma}}\nesp h_{s,\tau_2}^{(2)}\triangleright {}^{\ov{\sigma}(\tau_1)}\nespp \ba_{1r} \ot_k h_{s,\tau_3}^{(3)}  = \sum_{\tau_1,\tau_2 \in G \atop \tau_1\tau_2=\sigma_s} J\bigl(\ba_{1r}\ot_k {}^{\ov{\sigma}}\nesp h_{s,\tau_1}^{(1)}\bigr) \ot_k h_{s,\tau_2}^{(2)},
\end{align*}
for each $h_1\in H_{\sigma_1},\dots,h_s\in H_{\sigma_s}$ and $a_1,\dots,a_r\in A$.
\end{proof}

\begin{theorem}\label{d1 en cohomologia en YD} Let $E$, $M$ and $K$ be as in the previous theorem and let $(\ov{X}^*(M),\ov{d}^*)$ be as in Theorem~\ref{th7.1}. The terms $\ov{d}_0$ and $\ov{d}_1$ of $\ov{d}$ are given by:
\begin{align*}
\ov{d}_0(\beta)(\bx)& = a_1\beta\bigl(\ba_{2r}\ot_k \bh_{1s}\bigr)\\
& + \sum_{i=1}^{r-1} (-1)^i \beta\bigl(\ba_{1,i-1}\ot a_ia_{i+1} \ot \ba_{i+1,r}\ot_k \bh_{1s}\bigr)\\
& + (-1)^r\, {}^{\sigma_s^{-1}\cdots\sigma_1^{-1}}\nespp a_r\beta\bigl(\ba_{1,r-1}\ot_k \bh_{1s}\bigr)
\intertext{and}
\ov{d}_1(\beta)(\bx)& = (-1)^r \ep(h_1)\beta\bigl(\ba_{1r} \ot_k \bh_{2s}\bigr)\\
& + \sum_{i=1}^{s-1}(-1)^{r+i} \beta\bigl(\ba_{1r}\ot_k \bh_{1,i-1}\ot_k h_ih_{i+1} \ot_k \bh_{i+2,s}\bigr)\\
& + \sum_{\tau_1,\tau_2,\tau_3 \in G \atop \tau_1\tau_2\tau_3 = \sigma_s} (-1)^{r+s}\gamma(h_{s,\tau_3}^{(3)}) \beta\bigl({}^{\ov{\sigma}}\nesp h_{s,\tau_2}^{(2)}\triangleright \ba'_{1r}\ot_k \bh_{1,s-1}\bigr)\gamma^{-1}({}^{\ov{\sigma}}\nesp h_{s,\tau_1}^{(1)}),
\end{align*}
in which
\begin{align*}
& \bx := \ba_{1r}\ot_k \bh_{1s}\quad\text{with $h_1\in H_{\sigma_1},\dots,h_s\in H_{\sigma_s}$,}\\
& \ov{\sigma} := \sigma_1\cdots \sigma_{s-1},\qquad\ba'_{1r} = {}^{\ov{\sigma} \tau_1^{-1} \ov{\sigma}^{-1}}\nespp \ba_{1r}
\intertext{and}
& \sum_{\tau_1,\tau_2,\tau_3 \in G \atop \tau_1\tau_2\tau_3 = \sigma_s} h_{s,\tau_1}^{(1)} \ot_k h_{s,\tau_2}^{(2)}\ot_k h_{s,\tau_3}^{(3)} := \Delta^2(h_s),
\end{align*}
with $h_{s,\tau_1}^{(1)} \ot_k h_{s,\tau_2}^{(2)}\ot_k h_{s,\tau_3}^{(3)}\in H_{\tau_1}\ot_k H_{\tau_2}\ot_k H_{\tau_3}$.
\end{theorem}

\begin{proof} Mimic the proof of the previous theorem.
\end{proof}

\appendix
\section{}
This appendix is devoted to prove Propositions~\ref{formula para hat{phi} modulo la filtracion},~\ref{propiedades de hat{psi}} and~\ref{sobre la imagen de hat{w}}. Lemmas~\ref{leA.1}, \ref{leA.2}, \ref{leA.4} and~\ref{leA.6}, and Propositions~\ref{propA.5}, \ref{propA.7}  and~\ref{propA.9} generalize the corresponding results in~\cite{C-G-G}. Except for Propositions~\ref{propA.5} and~\ref{propA.8} we do not provide proofs, because the ones given in that paper work in our setting.

\smallskip

We will use the following notations:

\begin{enumerate}

\smallskip

\item We let $L_{rs}\subseteq U_{rs}$ denote the $K$-subbimodules of $X_{rs}$ generated by the simple tensors of the form
$$
1\ot_{\!A} \gamma_{\!A}(\bv_{1s})\ot\ba_{1r}\ot 1\quad \text{and}\quad 1\ot_{\!A}\gamma_{\!A}(\bv_{1s}) \ot \ba_{1r}\ot \gamma(v),
$$
respectively. Moreover we set
$$
\qquad L_n:=\bigoplus_{r+s=n} L_{rs}\quad\text{and}\quad  U_n:=\bigoplus_{r+s=n} U_{rs}.
$$

\smallskip

\item Given a subalgebra $R$ of $A$ we set
$$
X^{R1}_n:=\bigoplus_{r+s=n} X^{R1}_{rs}
$$
where $X^{R1}_{rs}$ is as in Notation~\ref{not2.1}.

\smallskip

\item We write $F^i_R(X_n):= F^i(X_n)\cap X^{R1}_n$.

\smallskip

\item We let $W_n$ denote the $K$-subbimodule of $E\ot\ov{E}^{\ot^n}\ot E$ generated by the simple tensors $1\ot \bx_{1n}\ot 1$ such that $x_i\in \ov{A}\cup \mathcal{V}_{\!K}$ for all~$i$.

\smallskip

\item We let $W'_n$ denote the $K$-subbimodule of $E\ot\ov{E}^{\ot^n}\ot E$ generated by the simple tensors $1\ot \bx_{1n}\ot 1$ such that $\#(\{j:x_j\notin \ov{A}\cup \mathcal{V}_{\!K}\})\le 1$.

\smallskip

\item Given a subalgebra $R$ of $A$, we let $C^R_n$ denote the $E$-subbimodule of $E\ot \ov{E}^{\ot^n}\ot E$ generated by all the simple tensors $1\ot \bx_{1n}\ot 1$ with some $x_i$ in~$\ov{R}$.

\smallskip

\item Let $R_i$ denote $F^i\bigl(E\ot \ov{E}^{\ot^n}\ot E\bigr)\setminus F^{i-1}\bigl(E\ot \ov{E}^{\ot^n}\ot E\bigr)$.

\medskip

\end{enumerate}
The identification $X_{rs}\simeq \bigr(E\ot_k\ov{V}^{\ot_k^s}\bigr)\ot \ov{A}^{\ot^r}\ot E$ induces identifications
$$
L_{rs} \simeq \bigl(K\ot_k\ov{V}^{\ot_k^s}\bigr)\ot \ov{A}^{\ot^r}\ot K\quad\text{and}\quad U_{rs} \simeq \bigl(K\ot_k\ov{V}^{\ot_k^s}\bigr)\ot \ov{A}^{\ot^r}\ot K\mathcal{V},
$$
where, as at the beginning of Section~\ref{SecRes}, $\mathcal{V}$ denotes the image, $k\ot V$, of $\gamma\colon V\to E$.

\begin{lemma}\label{leA.1} We have
$$
\ov{\sigma}_{n+1} = -  \sigma_{0,n+1}^0 \xcirc \sigma_{n+1}^{-1} \xcirc \nu_n + \sum_{r=0}^n \sum_{l=0}^{n-r} \sigma_{r+l+1,n-r-l}^l.
$$
\end{lemma}

\begin{lemma}\label{leA.2} The contracting homotopy $\ov{\sigma}$ satisfies $\ov{\sigma}\xcirc \ov{\sigma} = 0$.
\end{lemma}

\begin{remark}\label{reA.3} The previous lemma implies that
$$
\psi(\bx_{0n}\ot 1) =(-1)^n\ov{\sigma} \xcirc\psi(\bx_{0n})
$$
for all $n\ge 1$.
\end{remark}

\begin{lemma}\label{leA.4} It always holds that $d^l(L_{rs}) \subseteq U_{r+l-1,s-l}$, for each $l\ge 2$. Moreover
\begin{align*}
d^1(L_{rs}) & \subseteq EL_{r,s-1} + U_{r,s-1}.
\end{align*}
\end{lemma}

\begin{proposition}\label{propA.5} Let $R$ be a stable under $\chi$ subalgebra of $A$. If $\mathcal{F}$ takes its values in $R\ot_k V$, then
$$
\phi\bigl(1\ot_{\!A}\gamma_{\!A}(\bv_{1i})\ot\ba_{1,n-i}\ot 1\bigr)\equiv 1\ot \Sh(\bv_{1i}\ot_k \ba_{1,n-i})\ot 1
$$
module $F^{i-1}\bigl(E\ot\ov{E}^{\ot^n}\ot E\bigr)\cap W_n\cap C^R_n$.
\end{proposition}

\begin{proof} We proceed by induction on $n$. For $n=1$ this is trivial.  Assume that it is true for $n-1$. Let $\bx := 1\ot_{\!A} \gamma_{\!A}(\bv_{1i})\ot \ba_{1,n-i}\ot 1$. By item~(2) of Theorem~\ref{formula para d_1}, the fact that $d^l(\bx) \in U_{n-i+l-1,i-l}$ (by Lemma~\ref{leA.4}), the inductive hypothesis and the definition of $\xi$,
$$
\xi\xcirc \phi \xcirc d^l(\bx)\in F^{i-l+1}\bigl(E\ot\ov{E}^{\ot^n}\ot E\bigr)\cap W_n\cap C^R_n\quad\text{for all $l>1$.}
$$
So,
$$
\phi(\bx)=\xi\xcirc \phi \xcirc d^0(\bx) + \xi\xcirc \phi \xcirc d^1(\bx)\quad \pmod{F^{i-1} \bigl(E\ot\ov{E}^{\ot^n}\ot E\bigr)\cap W_n\cap C^R_n}.
$$
Moreover, by the definitions of $d^0$, $\phi$ and $\xi$,
$$
\xi\xcirc \phi \xcirc d^0(\bx)=(-1)^n \xi\xcirc \phi \bigl(1\ot_{\!A} \gamma_{\!A}(\bv_{1i})\ot \ba_{1,n-i}\bigr),
$$
and by Theorem~\ref{formula para d_1} and the definitions of $\phi$ and $\xi$,
$$
\xi\xcirc\phi\xcirc d^1(\bx) = \sum_l (-1)^i\xi\xcirc\phi\bigl(1\ot_{\!A} \gamma_{\!A}(\bv_{1,i-1}) \ot\ba^{(l)}_{1,n-i}\ot \gamma(v^{(l)}_i)\bigr),
$$
where $\sum_l \ba^{(l)}_{1,n-i}\ot_k v^{(l)}_i := \ov{\chi}(v_i\ot_k \ba_{1,n-i})$. The proof can be now easily finished using the inductive hypothesis.
\end{proof}

\begin{lemma}\label{leA.6} Consider a stable under $\chi$ subalgebra $R$ of $A$ such that $\mathcal{F}$ takes its values in $R\ot_k V$. The following facts hold:

\begin{enumerate}

\smallskip

\item Let $\bx := 1\ot_{\!A} \gamma_{\!A}(\bv_{1i})\ot \ba_{i+1,n}$. If $i<n$, then
$$
\ov{\sigma}(\bx) = \sigma^0(\bx) = (-1)^n \ot_{\!A} \gamma_{\!A}(\bv_{1i})\ot \ba_{i+1,n}\ot 1.
$$

\smallskip

\item If $\bz = 1\ot_{\!A}\gamma_{\!A}(\bv_{1,i-1})\ot\ba_{i,n-1}\ot a_n\gamma(v_n)$, then $\sigma^l(\bz) \in U_{n-i+l+1,i-1-l}$ for $l\ge 0$ and $\sigma^l(\bz) \in X^{R1}_n$ for $l\ge 1$.

\smallskip

\item If $\bz = 1\ot_{\!A} \gamma_{\!A}(\bv_{1,i-1})\ot\ba_{i,n-1}\ot\gamma(v_n)$, then $\sigma^l(\bz) = 0$ for $l\ge 0$.

\smallskip

\item If $\bz = 1\ot_{\!A} \gamma_{\!A}(\bv_{1,i-1})\ot\ba_{i,n-1}\ot a_n\gamma(v_n)$ and $i<n$, then $\ov{\sigma}(\bz) \equiv \sigma^0(\bz)$, module $F_R^{i-2}(X_n)\cap U_n$.

\smallskip

\item If $\bz = 1\ot_{\!A} \gamma_{\!A}(\bv_{1,n-1})\ot a_n\gamma(v_n)$, then $\ov{\sigma}(\bz) \equiv - \sigma^0\xcirc\sigma^{-1}\xcirc\nu(\bz) + \sigma^0(\bz)$, module $F_R^{i-2}(X_n) \cap U_n$.

\smallskip

\item If $\bz = 1\ot_{\!A} \gamma_{\!A}(\bv_{1,n-1})\ot\gamma(v_n)$, then $\ov{\sigma}(\bz) = - \sigma^0\xcirc\sigma^{-1}\xcirc\nu(\bz)$.

\smallskip

\item If $\bz = 1\ot_{\!A} \gamma_{\!A}(\bv_{1,i-1})\ot\ba_{i,n-1}\ot \gamma(v_n)$ and $i<n$, then $\ov{\sigma}(\bz) = 0$.

\end{enumerate}

\end{lemma}

\begin{proposition}\label{propA.7} Let $R$ be a stable under $\chi$ subalgebra of $A$ such that $\mathcal{F}$ takes its values in $R\ot_k V$. The following facts hold:

\begin{enumerate}

\smallskip

\item $\psi(1\ot\gamma(\bv_{1i})\ot\ba_{i+1,n}\ot 1) = 1\ot_{\!A} \gamma_{\!A}(\bv_{1i})\ot \ba_{i+1,n} \ot 1$.

\smallskip

\item If $\bx = 1\ot \bx_{1n}\ot 1\in R_i\cap W_n$ and there exists $1\le j\le i$ such that $x_j\in \ov{A}$, then $\psi(\bx) = 0$.

\smallskip

\item If $\bx = 1\ot \gamma(\bv_{1,i-1})\ot a_i\gamma(v_i)\ot\ba_{i+1,n} \ot 1$, then
\begin{align*}
\qquad\qquad \psi(\bx) & \equiv 1\ot_{\!A} \gamma_{\!A}(\bv_{1,i-1})\ot_{\!A} a_i\gamma(v_i)\ot \ba_{i+1,n}\ot 1\\
&+ \sum_l 1\ot_{\!A} \gamma_{\!A}(\bv_{1,i-1})\ot a_i\ot \ba^{(l)}_{i+1,n}\ot \gamma(v^{(l)}_i),
\end{align*}
module $F_R^{i-2}(X_n)\cap U_n$, where $\sum_l \ba^{(l)}_{i+1,n}\ot_k v^{(l)}_i := \ov{\chi}(v_i\ot_k\ba_{i+1,n})$.

\smallskip

\item If $\bx = 1\ot \gamma(\bv_{1,j-1})\ot a_j\gamma(v_j)\ot \gamma(\bv_{j+1,i})\ot \ba_{i+1,n}\ot 1$ with $j<i$, then
$$
\qquad\qquad \psi(\bx)\equiv 1\ot_{\!A} \gamma_{\!A}(\bv_{1,j-1})\ot_{\!A} a_j\gamma(v_j)\ot_{\!A} \gamma_{\!A}(\bv_{j+1,i}) \ot \ba_{i+1,n} \ot 1,
$$
module $F_R^{i-2}(X_n)\cap U_n$.

\smallskip

\item If $\bx = 1\ot \gamma(\bv_{1,i-1}) \ot \ba_{i,j-1}\ot a_j\gamma(v_j) \ot\ba_{j+1,n} \ot 1$ with $j>i$, then
$$
\qquad\qquad\psi(\bx)\equiv \sum_l 1\ot_{\!A} \gamma_{\!A}(\bv_{1,i-1})\ot \ba_{ij}\ot \ba^{(l)}_{j+1,n}\ot \gamma(v^{(l)}_j),
$$
module $F_R^{i-2}(X_n)\cap U_n$, where $\sum_l \ba^{(l)}_{j+1,n}\ot_k v^{(l)}_j := \ov{\chi}(v_j\ot_k\ba_{j+1,n})$.

\smallskip

\item If $\bx = 1\ot \bx_{1n}\ot 1\in R_i\cap W'_n$ and there exists $1\le j_1 < j_2\le n$ such that $x_{j_1}\in \ov{A}$ and $x_{j_2}\in \mathcal{V}_{\!K}$, then $\psi(\bx)\in  F_R^{i-2}(X_n) \cap U_n$.

\end{enumerate}

\end{proposition}

\begin{proof} For items~(1)--(5) the proofs given in~\cite{C-G-G} work. We next prove item~(6). Assume first that $x_n \notin i_{\ov{A}}(\ov{A})\cup \mathcal{V}$. Then, by Remark~\ref{reA.3} and item~(2),
$$
\psi(\bx) = (-1)^n \ov{\sigma}\xcirc \psi(1\ot\bx_{1n}) = (-1)^n \ov{\sigma}(0) = 0.
$$
Assume now that $x_n\in \ov{A}$. Then, by inductive hypothesis
$$
\psi(\bx)=(-1)^n \ov{\sigma}\xcirc \psi(1\ot\bx_{1n}) \in\ov{\sigma}\left(X^{R1}_{n-1} \cap\bigoplus_{l=0}^{i-2} U_{n-l-1,l}A\right),
$$
and the result follows from items~(1) and~(4) of Lemma~\ref{leA.6}. Finally, assume that $x_n\in \mathcal{V}$. Then, by inductive hypothesis or items~(3), (4) or~(5),
$$
\psi(\bx) = (-1)^n \ov{\sigma}\xcirc\psi(1\ot\bx_{1n}) \in\ov{\sigma}\left(AU_{n-i,i-1} \oplus \bigoplus_{l=0}^{i-2} U_{n-l-1,l}\mathcal{V}\right),
$$
and the result follows from items~(4) and~(7) of Lemma~\ref{leA.6}.
\end{proof}

\begin{proposition}\label{propA.8} The following facts hold:

\begin{enumerate}

\smallskip

\item If $\bx = 1\ot\gamma(\bv_{1i})\ot \ba_{1,n-i}\ot 1$, then
$$
\qquad\quad\phi\xcirc \psi(\bx) \equiv 1\ot \Sh(\bv_{1i}\ot_k\ba_{1,n-i})\ot 1
$$
module $F^{i-1}\bigl(E\ot \ov{E}^{\ot^n}\ot E\bigr)\cap W_n\cap C^R_n$.

\smallskip

\item If $\bx = 1\ot\bx_{1n}\ot 1\in R_i\cap W_n$ and there exists $1\le j\le i$ such that $x_j\in A$, then $\phi\xcirc \psi(\bx) = 0$.

\smallskip

\item If $\bx = 1\ot \gamma(\bv_{1,i-1})\ot a_i\gamma(v_i)\ot\ba_{i+1,n} \ot 1$, then
\begin{align*}
\qquad\qquad \phi\xcirc\psi(\bx) &\equiv \sum_l a_i^{(l)}\ot \Sh\bigl(\bv_{1,i-1}^{(l)} \ot_k v_i\ot_k \ba_{i+1,n}\bigr)\ot 1\\
& + \sum_l 1\ot \Sh\bigl(\bv_{1,i-1}\ot_k a_i\ot \ba_{i+1,n}^{(l)}\bigr)\ot \gamma(v_i^{(l)})
\end{align*}
module
$$
\qquad\qquad\Bigl(F^{i-1}\bigl(E\ot\ov{E}^{\ot^n}\ot E\bigr)\cap AW_n + F^{i-2}\bigl(E\ot \ov{E}^{\ot^n}\ot E\bigr)\cap W_n \mathcal{V}\Bigr)\cap C^R_n,
$$
where
$$
\qquad\qquad\!\!\sum_l a_i^{(l)}\ot_k \bv_{1,i-1}^{(l)}\! :=\! \chi(\bv_{1,i-1}\ot_k a_i) \quad\!\!\text{and}\quad\!\!\sum_l \ba_{i+1,n}^{(l)}\ot_k v_i^{(l)}\!:=\! \ov{\chi}(v_i\ot_k \ba_{i+1,n}).
$$

\smallskip

\item If $\bx = 1\ot \gamma(\bv_{1,j-1})\ot a_j\gamma(v_j)\ot \gamma(\bv_{j+1,i})\ot \ba_{i+1,n} \ot 1$ with $j<i$, then
$$
\qquad\qquad \phi\xcirc\psi(\bx)\equiv \sum_l a_j^{(l)}\ot \Sh\bigl(\bv_{1,j-1}^{(l)}\ot_k \bv_{ji}\ot_k \ba_{i+1,n}\bigr)\ot 1,
$$
module
$$
\qquad\qquad\Bigl(F^{i-1}\bigl(E\ot \ov{E}^{\ot^n}\ot E\bigr)\cap AW_n+ F^{i-2}\bigl(E\ot \ov{E}^{\ot^n}\ot E\bigr)\cap W_n\mathcal{V}\Bigr)\cap C^R_n,
$$
where $\sum_l a_j^{(l)}\ot_k \bv_{1,j-1}^{(l)} := \chi(\bv_{1,j-1}\ot_k a_j)$.

\smallskip

\item If $\bx = 1\ot \gamma(\bv_{1,i-1}) \ot \ba_{i,j-1}\ot a_j\gamma(v_j) \ot\ba_{j+1,n} \ot 1$ with $j>i$, then
$$
\qquad\qquad \phi\xcirc\psi(\bx)\equiv \sum_l 1\ot \Sh\bigl(\bv_{1,i-1}\ot_k \ba_{ij}\ot \ba^{(l)}_{j+1,n}\bigr)\ot \gamma(v^{(l)}_j)
$$
module $F^{i-2}\bigl(E\ot\ov{E}^{\ot^n}\ot E\bigr)\cap W_n\mathcal{V}\cap C^R_n$.

\smallskip

\item If $\bx = 1\ot \bx_{1n}\ot 1\in R_i\cap W'_n$ and there exists $1\le j_1 < j_2\le n$ such that $x_{j_1}\in \ov{A}$ and $x_{j_2}\in \mathcal{V}_{\!K}$, then
$$
\qquad\qquad \phi\xcirc\psi(\bx)\in F^{i-2}\bigl(E\ot\ov{E}^{\ot^n} \ot E\bigr) \cap W_n\mathcal{V} \cap C^R_n.
$$

\end{enumerate}

\end{proposition}

\begin{proof} (1)\enspace This follows from item~(1) of Proposition~\ref{propA.7} and
Proposition~\ref{propA.5}.

\smallskip

\noindent (2)\enspace This follows from item~(2) of Proposition~\ref{propA.7}.

\smallskip

\noindent (3)\enspace By item~(3) of Proposition~\ref{propA.7},
\begin{align*}
\qquad\qquad \phi \xcirc \psi(\bx) & \equiv\sum_l\phi\bigl(a_i^{(l)}\ot_{\!A} \gamma_{\!A}( \bv_{1,i-1}^{(l)} \ot_k v_i) \ot \ba_{i+1,n}\ot 1\bigr)\\
&+ \sum_l \phi\bigl(1\ot_{\!A} \gamma_{\!A}(\bv_{1,i-1})\ot a_i\ot \ba^{(l)}_{i+1,n}\ot \gamma(v^{(l)}_i)\bigr),
\end{align*}
module $\phi\bigl(F_R^{i-2}(X_n)\cap U_n\bigr)$, where
$$
\qquad\!\!\sum_l a_i^{(l)}\ot_k \bv_{1,i-1}^{(l)}\! := \!\chi(\bv_{1,i-1}\ot_k a_i) \quad\!\text{and}\quad \!\sum_l \ba_{i+1,n}^{(l)}\ot_k v_i^{(l)}\! :=\! \ov{\chi}(v_i\ot_k \ba_{i+1,n}).
$$
The desired result follows now from Proposition~\ref{propA.5}.
\smallskip

\noindent (4)\enspace By item~(4) of Proposition~\ref{propA.7},
$$
\phi\xcirc \psi(\bx)\equiv \sum_l \phi\bigl(a_j^{(l)}\ot_{\!A} \gamma_{\!A}(\bv_{1,j-1}^{(l)} \ot_k \bv_{ji}) \ot \ba_{i+1,n}\ot 1\bigr)
$$
module $F_R^{i-2}(X_n)\cap U_n$, where $\sum_l a_i^{(l)}\ot_k \bv_{1,i-1}^{(l)} := \chi(\bv_{1,j-1} \ot_k a_j)$. To conclude the proof of this item it suffices to apply Proposition~\ref{propA.5}.

\smallskip

\noindent (5)\enspace By item~(5) of Proposition~\ref{propA.7},
$$
\qquad\phi\xcirc \psi(\bx)\equiv \sum_l \phi\bigl(1\ot_{\!A} \gamma_{\!A}(\bv_{1,i-1})\ot \ba_{ij}\ot \ba^{(l)}_{j+1,n} \ot \gamma(v^{(l)}_j)\bigr),
$$
module $\phi\bigl(F_R^{i-2}(X_n)\cap U_n\bigr)$, where $\sum_l\ba^{(l)}_{j+1,n}\ot_k v^{(l)}_j := \ov{\chi} (v_j\ot_k\ba_{j+1,n})$. The result follows by applying Proposition~\ref{propA.5}.

\smallskip

\noindent (6)\enspace Proceed as in the proof of item~(5) but using item~(5) of Proposition~\ref{propA.7} instead of item~(5).
\end{proof}

\begin{proposition}\label{propA.9} If $\bx = 1\ot\bx_{1n}\ot 1\in R_i\cap W'_n$, then
$$
\omega(\bx) \in F^i\bigl(E\ot \ov{E}^{\ot^{n+1}}\ot E\bigr)\cap W_{n+1}.
$$
\end{proposition}

\section{}
The purpose of this appendix is to prove Proposition~\ref{inversa de theta}, Theorem~\ref{th5.2}, Propositions~\ref{cup producto en cleft} and~\ref{cap producto en cleft}, and Theorem~\ref{morfismo D}. We will freely use the notations introduced in the previous sections, and the properties established in Definitions~\ref{transposition}, \ref{accion debil} and~\ref{comodulo algebra}, and Remarks~\ref{transposicion inversa} and~\ref{propiedad de accion debil}. We will also use the diagrams introduced in~\eqref{s0}, \eqref{s1}, \eqref{s2}, \eqref{s3}, Definition~\ref{cociclo} and Remarks~\ref{notacion para wh{s}}, \ref{nuevas notaciones} and~\ref{ultima diagramatica}. Actually, in this appendix we will use them with a wider meaning. Finally we let $\ov{\gamma}$ denote the convolution inverse of $\gamma$.

\smallskip

Let $C_1$ and $C_2$ be two coalgebras. It is easy to see that if $c\colon C_1\ot_k C_2\to C_2\ot_k C_1$ is compatible with the coalgebra structures of $C_1$ and $C_2$, then $C_1\ot_k C_2$ is a coalgebra with counit $\varepsilon_{C_1}\ot_k \varepsilon_{C_2}$, via $\Delta := (C_1\ot_k c\ot_k C_2)\xcirc \bigl(\Delta_{C_1}\ot_k \Delta_{C_2}\bigr)$. We will denote this coalgebra by $C_1\ot_c C_2$.

\begin{lemma} Let $E$ be a $k$-algebra. If $u\colon C_1\to E$ and $v\colon C_2\to E$ are convolution invertible $k$-linear maps, then the map $\mu_E\xcirc (u\ot_k v)$ is also convolution invertible and its inverse is $\mu_E \xcirc (v^{-1}\ot_k u^{-1})\xcirc c$.
\end{lemma}

\begin{proof} Set $\ov{u}:=u^{-1}$, $\ov{v}:=v^{-1}$, $f:=\mu_E\xcirc (u\ot_k v)$ and $g:= \mu_E\xcirc (\ov{v}\ot_k \ov{u})\xcirc c$. We have
$$
\spreaddiagramcolumns{-1.6pc}\spreaddiagramrows{-1.6pc}
\objectmargin{0.0pc}\objectwidth{0.0pc}
\def\objectstyle{\sssize}
\def\labelstyle{\sssize}
\grow{\xymatrix@!0{
\\\\\\\\\\\\\\\\
\save\go+<0pt,0pt>\Drop{\txt{$f*g=$}}\restore }}
\grow{\xymatrix@!0{
\\\\
&\save\go+<0pt,3pt>\Drop{C_1}\restore  \ar@{-}[1,0]  &&&& \save\go+<0pt,3pt>\Drop{C_2}\restore \ar@{-}[1,0]\\
&\ar@{-}`l/4pt [1,-1] [1,-1] \ar@{-}`r [1,1] [1,1]&&&&\ar@{-}`l/4pt [1,-1] [1,-1] \ar@{-}`r [1,1] [1,1]&\\
\ar@{-}[4,0]&& \ar@{-}[1,1]+<-0.1pc,0.1pc> && \ar@{-}[2,-2]&&\ar@{-}[2,0]\\
&&&&&&\\
&&\ar@{-}[2,0]&&\ar@{-}[-1,-1]+<0.1pc,-0.1pc>\ar@{-}[1,1]+<-0.1pc,0.1pc> && \ar@{-}[2,-2]\\
&&&&&&\\
&&&&&&\ar@{-}[-1,-1]+<0.1pc,-0.1pc>\\
\save[]+<0pc,0.4pc> \Drop{} \ar@{-}[1,0]\restore &&\save[]+<0pc,0.4pc> \Drop{} \ar@{-}[1,0]\restore&&\save[]+<0pc,0.4pc> \Drop{} \ar@{-}[1,0]\restore&&\save[]+<0pc,0.4pc> \Drop{} \ar@{-}[1,0]\restore\\
*+[o]+<0.35pc>[F]{u}\ar@{-}[1,0]+<0pc,-0.2pc>&&*+[o]+<0.35pc>[F]{v}\ar@{-}[1,0]+<0pc,-0.2pc>
&&*+[o]+<0.35pc>[F]{\ov{v}}\ar@{-}[1,0]+<0pc,-0.2pc>&&*+[o]+<0.35pc>[F]{\ov{u}}\ar@{-}[1,0]+<0pc,-0.2pc>\\
\ar@{-}`d/4pt [1,1] `[0,2] [0,2]&&&&\ar@{-}`d/4pt [1,1] `[0,2] [0,2]&&\\
&\ar@{-}[1,0]&&&&\ar@{-}[1,0]&\\
&\ar@{-}`d/4pt [1,2] `[0,4] [0,4]&&&&&\\
&&&\ar@{-}[1,0]&&&\\
&&&&&&}}
\grow{\xymatrix@!0{
\\\\\\\\\\\\\\\\
\save\go+<0pt,0pt>\Drop{\txt{$=$}}\restore }}
\grow{\xymatrix@!0{
&&\save\go+<0pt,3pt>\Drop{C_1}\restore  \ar@{-}[1,0]  &&&& \save\go+<0pt,3pt>\Drop{C_2}\restore \ar@{-}[2,0]\\
&&\ar@{-}`l/4pt [1,-2] [1,-2] \ar@{-}`r [1,2] [1,2]&&&&\\
\ar@{-}[5,0]&&&&\ar@{-}[1,1]+<-0.1pc,0.1pc>&&\ar@{-}[3,-3]\\
&&&&&&\\
&&&&&&\ar@{-}[-1,-1]+<0.1pc,-0.1pc>\ar@{-}[3,0]\\
&&&\ar@{-}[1,0]&&&\\
&&&\ar@{-}`l/4pt [1,-1] [1,-1] \ar@{-}`r [1,1] [1,1]&&&\\
&&&&&&\\
\save[]+<0pc,0.4pc> \Drop{} \ar@{-}[1,0]\restore&&\save[]+<0pc,0.4pc> \Drop{} \ar@{-}[1,0]\restore&&\save[]+<0pc,0.4pc> \Drop{} \ar@{-}[1,0]\restore&&\save[]+<0pc,0.4pc> \Drop{} \ar@{-}[1,0]\restore\\
*+[o]+<0.35pc>[F]{u}\ar@{-}[1,0]+<0pc,-0.2pc>&&*+[o]+<0.35pc>[F]{v} \ar@{-}[1,0]+<0pc,-0.2pc>&&*+[o]+<0.35pc>[F]{\ov{v}}\ar@{-}[1,0]+<0pc,-0.2pc>&&*+[o]+<0.35pc>[F]{\ov{u}}\ar@{-}[1,0]+<0pc,-0.4pc>\\
\ar@{-}[2,0]&&\ar@{-}`d/4pt [1,1] `[0,2] [0,2]&&&&\\
&&&\ar@{-}[1,0]&&&\ar@{-}[3,0]+<-0.4pc,0pc>\\
\ar@{-}`d/4pt [1,1]+<0.2pc,0pc> `[0,3] [0,3]&&&&&&\\
&\save[]+<0.2pc,0pc> \Drop{}\ar@{-}[1,0]+<0.2pc,0pc>\restore&&&&&\\
&\save[]+<0.2pc,0pc> \Drop{} \ar@{-}`d/4pt [1,2] `[0,4] [0,4] \restore&&&&&\\
&&&\save[]+<0.1pc,0pc> \Drop{} \ar@{-}[1,0]+<0.1pc,0pc>\restore&&&\\
&&&\save[]+<0.1pc,0pc> \Drop{}\restore&&&
}}
\grow{\xymatrix@!0{
\\\\\\\\\\\\\\\\
\save\go+<0pt,0pt>\Drop{\txt{$= \eta_E\xcirc\epsilon_{C_1\ot_c C_2}$,}}\restore }}
$$
as desired. Similarly $g*f = \eta_E\xcirc\epsilon_{C_1\ot_c C_2}$.
\end{proof}

Let $E$ be a $k$-algebra. Recall that for all $s\in \mathds{N}$ we let $\mu_s\colon E^{\ot_k^s}\to E$ denote the map recursively defined by
$$
\mu_1 := \ide_E\quad\text{and}\quad \mu_{s+1}:= \mu_E\xcirc(\mu_s\ot_k E).
$$

\begin{lemma}\label{inversa de convolucion} Let $E$ be a $k$-algebra and let $H$ be a braided bialgebra. If $u\colon H\to E$ is a convolution invertible $k$-linear map, then for all $s\in \mathds{N}$, the map $\mu_s\xcirc u^{\ot_k^s}$, is also convolution invertible. Its inverse is $\mu_s \xcirc \ov{u}^{\ot_k^s}\xcirc \gc_s$, where $\ov{u}$ is the convolution inverse of $u$ and $\gc_s\colon H^{\ot_k^s}\to H^{\ot_k^s}$ is the map introduced at the beginning of Section~\ref{hom de inv}.
\end{lemma}

\begin{proof} We make the proof by induction on $s$. Case $s = 1$ is trivial. Assume that the result is valid for $s$. Let $C_1:= H^{\ot_c^s}$ and $C_2 = H$. By the previous lemma the $k$-linear map $\mu_E\xcirc\bigl((\mu_s \xcirc u^{\ot_k^s})\ot_k u\bigr)$ is convolution invertible and its convolution inverse is $\mu_E\xcirc\bigl((\mu_s\xcirc \ov{u}^{\ot_k^s}\xcirc \gc_s)\ot_k \ov{u}\bigr)\xcirc c_{s1}$. But, by \cite[Corollary~4.21]{G-G2}, we know that $H^{\ot_c^{s+1}} = C_1\ot_{c_{s1}} C_2$ and $\bigl((\ov{u}^{\ot_k^s}\xcirc \gc_s)\ot_k \ov{u}\bigr)\xcirc c_{s1} = \ov{u}^{\ot_k^{s+1}} \xcirc \gc_{s+1}$.
\end{proof}

\noindent\bf Proof of Proposition~\ref{inversa de theta}.\rm\enspace Let
$$
\wt{\theta}_{rs}\colon M\ot_{\!A} \ov{C} \ot_k D\to M \ot_k D \ot_k C \quad\text{and}\quad  \wt{\vartheta}_{rs}\colon M\ot_k D\ot_k C\to M\ot_{\!A} \ov{C}\ot_k D,
$$
be the $k$-linear maps diagrammatically defined by
$$
\spreaddiagramcolumns{-1.6pc}\spreaddiagramrows{-1.6pc}
\objectmargin{0.0pc}\objectwidth{0.0pc}
\def\objectstyle{\sssize}
\def\labelstyle{\sssize}
\grow{\xymatrix@!0{
\\\\\\\\
\save\go+<0pt,0pt>\Drop{\txt{$\widetilde{\theta} :=$}}\restore
}}
\grow{\xymatrix@!0{\\\\
\save\go+<0pt,3pt>\Drop{M}\restore \ar@{-}[7,0]&&\ar@{-}[2,0]+<0pc,-0.5pt> \save\go+<0pt,3pt>\Drop{\overline{C}}\restore &&&& \save\go+<0pt,3pt>\Drop{D}\restore \ar@{-}[2,0] \\
&&\ar@{-}`r/4pt [1,2] [1,2]&&&&\\
&&&&\ar@{-}[1,1]+<-0.125pc,0.0pc> \ar@{-}[1,1]+<0.0pc,0.125pc>&&\ar@{-}[2,-2]\\
&&*+[o]+<0.37pc>[F]{\ov{\mu}}\ar@{-}[2,0]+<0pc,0pc>&&&&\\
&&&&\ar@{-}[3,0]&&\ar@{-}[-1,-1]+<0.125pc,0.0pc>\ar@{-}[-1,-1]+<0.0pc,-0.125pc>\ar@{-}[3,0]\\
&&\ar@{-}`d/4pt [1,-2][1,-2]&&&&\\
&&&&&&\\
&&&&&&
}}
\quad
\grow{\xymatrix@!0{
\\\\\\\\
\save\go+<0pt,0pt>\Drop{\txt{and}}\restore
}}
\quad
\grow{\xymatrix@!0{
\\\\\\\\
\save\go+<0pt,0pt>\Drop{\txt{$\widetilde{\vartheta} :=$}}\restore
}}
\grow{\xymatrix@!0{
\save\go+<0pt,3pt>\Drop{M}\restore \ar@{-}[10,0] &&& \save\go+<0pt,3pt>\Drop{D}\restore \ar@{-}[2,2] &&  \save\go+<0pt,3pt>\Drop{C}\restore \ar@{-}[1,-1]+<0.125pc,0.0pc> \ar@{-}[1,-1]+<0.0pc,0.125pc>&\\
&&&&&&\\
&&&\ar@{-}[1,0]\ar@{-}[-1,1]+<-0.125pc,0.0pc> \ar@{-}[-1,1]+<0.0pc,-0.125pc>&&\ar@{-}[1,1]&\\
&&&\ar@{-}`l/4pt [1,-1] [1,-1] \ar@{-}`r [1,1] [1,1]&&&\ar@{-}[7,0]\\
&&\ar@{-}[1,0]+<0pc,-0.5pt>&&\ar@{-}[1,0]+<0pc,-0.5pt>&&\\
&&&&&&\\
&&*+[o]+<0.37pc>[F]{\ov{\mathbf{u}}}\ar@{-}[2,0]+<0pc,0pc>&&*+[o]+<0.37pc>[F]{\boldsymbol{\gamma}}\ar@{-}[4,0]&&\\
&&&&&&\\
&&\ar@{-}`d/4pt [1,-2][1,-2]&&&&\\
&&&&&&\\
&&&&&&
}}
\grow{\xymatrix@!0{
\\\\\\\\
\save\go+<0pt,0pt>\Drop{\txt{,}}\restore
}}
$$
where

\begin{itemize}

\smallskip

\item[-] $C := H^{\ot_c^s}$, $\ov{C} = E^{\ot_{\!A}^s}$ and $D:=A^{\ot_k^r}$,

\smallskip

\item[-] $\ov{\mu}$ is the map induced by $\mu_s\colon E^{\ot_k^s}\to E$,

\smallskip

\item[-] $\boldsymbol{\gamma}:= \gamma^{\ot_{\!A}^s}$ and $\ov{\mathbf{u}} := \mu_s\xcirc \ov{\gamma}^{\ot_k^s}\xcirc \gc_s$.

\end{itemize}
It is easy to see that $\theta_{rs}$ and $\vartheta_{rs}$ are induced by $(-1)^{rs}\wt{\theta}_{rs}$ and $(-1)^{rs}\wt{\vartheta}_{rs}$, respectively. Hence in order to finish the proof we must see that $\wt{\vartheta}_{rs}\xcirc \wt{\theta}_{rs} = \ide$ and $\wt{\theta}_{rs}\xcirc \wt{\vartheta}_{rs} = \ide$. Let
$$
L\colon \ov{C}\to E\ot_k C\quad\text{and}\quad \ov{L}\colon C\to E\ot_A \ov{C}
$$
be the $k$-linear maps defined by
$$
L:= (\ov{\mu}\ot_k D\ot_k C)\xcirc (\boldsymbol{\nu}_{\!A}\ot_k C)\quad\text{and}\quad \ov{L}:= (\mu_s \xcirc \ov{\gamma}^{\ot_k^s} \xcirc \gc_s\ot_k \gamma^{\ot_{\!A}^s})\xcirc \Delta_C,
$$
where $\boldsymbol{\nu}_{\!A}$ is the coaction introduced in Remark~\ref{acerca de nu}. Clearly
\begin{align*}
&\wt{\theta} := \bigl(M\ot_k s_{sr}\bigr)\xcirc \bigl(\wt{\rho}\ot_k C\ot_k D\bigr) \xcirc \bigl(M\ot_{\!A} L \ot_k D\bigr)
\intertext{and}
&\wt{\vartheta} :=\bigl(\wt{\rho}\ot_k C\ot_k D\bigr)\xcirc \bigl(M\ot_k\ov{L} \ot_k D\bigr)\xcirc \bigl(M\ot_k \s_{sr}^{-1}\bigr),
\end{align*}
where $\wt{\rho}$ denotes the right action of $E$ on $M$. We will prove that $\wt{\vartheta}_{rs}\xcirc \wt{\theta}_{rs} = \ide$ and we leave the task to prove that  $\wt{\theta}_{rs}\xcirc \wt{\vartheta}_{rs} = \ide$ to the reader. Let $\Gamma\colon M\ot_k C\to M\ot_A\ov{C}$ be the isomorphism given by $\Gamma(m\ot_k\bh_{1s}) = m\ot_A\gamma_A(\bh_{1s})$. Since
$$
\wt{\vartheta}_{rs}\xcirc \wt{\theta}_{rs} = (\wt{\rho}\ot_k C)\xcirc (M\ot_k\ov{L})\xcirc (\wt{\rho}\ot_k C) \xcirc (M\ot_A L)\ot_k D
$$
and
$$
(M\ot_A L)\xcirc \Gamma = (M\ot_A\mu_s\xcirc \gamma^{\ot_k^s}\ot_k C)\xcirc (M\ot_k \Delta_C),
$$
we have
$$
\spreaddiagramcolumns{-1.6pc}\spreaddiagramrows{-1.6pc}
\objectmargin{0.0pc}\objectwidth{0.0pc}
\def\objectstyle{\sssize}
\def\labelstyle{\sssize}
\grow{\xymatrix@!0{
\\\\\\\\\\
\save\go+<0pt,0pt>\Drop{\txt{$\Gamma^{-1}\xcirc(\wt{\rho}\ot_k C)\xcirc (M\ot_k\ov{L})\xcirc (\wt{\rho}\ot_k C) \xcirc (M\ot_A L)\xcirc \Gamma =$}}\restore }}
\grow{\xymatrix@!0{
\save\go+<0pt,3pt>\Drop{M}\restore\ar@{-}[10,0]&&&\save\go+<0pt,3pt>\Drop{C}\restore  \ar@{-}[1,0]  &\\
&&&\ar@{-}`l/4pt [1,-1] [1,-1] \ar@{-}`r [1,1] [1,1]&&\\
&&\save[]+<0pc,0pc> \Drop{} \ar@{-}[1,0]\restore&&\ar@{-}[3,0]\\
&&*+[o]+<0.35pc>[F]{v}\ar@{-}[1,0]+<0pc,0pc>&&\\
&&\ar@{-}`d/4pt [1,-2][1,-2]&&\\
&&&&\ar@{-}`l/4pt [1,-1] [1,-1] \ar@{-}`r [1,1] [1,1]\\
&&&\save[]+<0pc,0pc> \Drop{} \ar@{-}[1,0]\restore&&\ar@{-}[4,0]\\
&&&*+[o]+<0.35pc>[F]{\ov{v}}\ar@{-}[1,0]+<0pc,0pc>&&\\
&&&\ar@{-}`d/4pt [1,-2][1,-2]&&\\
&\ar@{-}[0,-1]&&&&\\
&&&&&}}
\grow{\xymatrix@!0{
\\\\\\\\\\
\save\go+<0pt,0pt>\Drop{\txt{$=$}}\restore }}
\grow{\xymatrix@!0{
\save\go+<0pt,3pt>\Drop{M}\restore\ar@{-}[10,0]&&&&\save\go+<0pt,3pt>\Drop{C}\restore  \ar@{-}[1,0]  &\\
&&&&\ar@{-}`l/4pt [1,-1] [1,-1] \ar@{-}`r [1,1] [1,1]&\\
&&&\ar@{-}[1,0]&&\ar@{-}[8,0]\\
&&&\ar@{-}`l/4pt [1,-1] [1,-1] \ar@{-}`r [1,1] [1,1]&&\\
&&\save[]+<0pc,0pc> \Drop{} \ar@{-}[1,0]&&\save[]+<0pc,0pc> \Drop{} \ar@{-}[1,0]\\
&&*+[o]+<0.35pc>[F]{v}\ar@{-}[1,0]+<0pc,0pc>\restore&&*+[o]+<0.35pc>[F]{\ov{v}}\ar@{-}[1,0]+<0pc,0pc>\restore\\
&&\ar@{-}`d/4pt [1,1] `[0,2] [0,2] &&\\
&&&\ar@{-}[1,0]&\\
&&& \ar@{-}`d/4pt [1,-2][1,-2]&\\
&\ar@{-}[0,-1]&&&&\\
&&&&&\\
&&&&&\\
&&&&&}}
$$
where $v := \mu_s\xcirc \gamma^{\ot_k^s}$ and $\ov{v}:= \mu_s\xcirc \ov{\gamma}^{\ot_k^s} \xcirc \gc_s$. To finish the proof it suffices to note that $\ov{v}$ is the convolution inverse of $v$, by Lemma~\ref{inversa de convolucion}.\qed

\medskip

\begin{lemma}\label{formula para chi} Let $s,r\in \mathds{N}$. For $C:= H^{\ot_c^s}$ and $D:=A^{\ot_k^r}$, the equality
$$
\spreaddiagramcolumns{-1.6pc}\spreaddiagramrows{-1.6pc}
\objectmargin{0.0pc}\objectwidth{0.0pc}
\def\objectstyle{\sssize}
\def\labelstyle{\sssize}
\grow{\xymatrix@!0{\\\\
\save\go+<0pt,3pt>\Drop{C}\restore \ar@/^0.1pc/ @{-}[2,2] \ar@/_0.1pc/ @{-}[2,2]&&\save\go+<0pt,3pt>\Drop{D}\restore \ar@/^0.1pc/ @{-}[2,-2]\ar@/_0.1pc/ @{-}[2,-2]\\
&&\\
&&
}}
\grow{\xymatrix@!0{
\\\\\\
\save\go+<0pt,0pt>\Drop{\txt{$=$}}\restore
}}
\grow{\xymatrix@!0{
&\save\go+<0pt,3pt>\Drop{C}\restore \ar@{-}[1,0]&&&\save\go+<0pt,3pt>\Drop{D}\restore \ar@{-}[2,0]\\
&\ar@{-}`l/4pt [1,-1] [1,-1] \ar@{-}`r [1,1] [1,1]&&&\\
\ar@{-}[2,0]&&\ar@{-}[1,1]+<-0.125pc,0.0pc> \ar@{-}[1,1]+<0.0pc,0.125pc>&&\ar@{-}[2,-2]\\
&&&&\\
\ar@{-}`d/4pt [1,2][1,2]&&\ar@{-}[2,0]&&\ar@{-}[-1,-1]+<0.125pc,0.0pc>\ar@{-}[-1,-1]+<0.0pc,-0.125pc> \ar@{-}[2,0]\\
&&&&\\
&&&&
}}
$$
is true.
\end{lemma}

\begin{proof} When $s=r=1$ the formula is true by definition. Assume that $r>1$ and that the formula is valid for $H$ and $D':= A^{\ot^{r-1}}$. Let $D:= A^{\ot^r}$. We have
$$
\spreaddiagramcolumns{-1.6pc}\spreaddiagramrows{-1.6pc}
\objectmargin{0.0pc}\objectwidth{0.0pc}
\def\objectstyle{\sssize}
\def\labelstyle{\sssize}
\grow{\xymatrix@!0{\\\\\\\\
\save\go+<0pt,3pt>\Drop{H}\restore \ar@/^0.1pc/ @{-}[2,2] \ar@/_0.1pc/ @{-}[2,2]&&\save\go+<0pt,3pt>\Drop{D}\restore \ar@/^0.1pc/ @{-}[2,-2]\ar@/_0.1pc/ @{-}[2,-2]\\
&&\\
&&
}}
\grow{\xymatrix@!0{
\\\\\\\\\\
\save\go+<0pt,0pt>\Drop{\txt{$=$}}\restore
}}
\grow{\xymatrix@!0{\\\\\\
\save\go+<0pt,3pt>\Drop{H}\restore \ar@/^0.1pc/ @{-}[2,2] \ar@/_0.1pc/ @{-}[2,2] &&\save\go+<0pt,3pt>\Drop{D'}\restore \ar@/^0.1pc/ @{-}[2,-2]\ar@/_0.1pc/ @{-}[2,-2]&&\save\go+<0pt,3pt>\Drop{A}\restore \ar@{-}[2,0]\\
&&&&\\
\ar@{-}[2,0]&&\ar@/^0.1pc/ @{-}[2,2] \ar@/_0.1pc/ @{-}[2,2]&& \ar@/^0.1pc/ @{-}[2,-2]\ar@/_0.1pc/ @{-}[2,-2]\\
&&&&\\
&&&&
}}
\grow{\xymatrix@!0{
\\\\\\\\\\
\save\go+<0pt,0pt>\Drop{\txt{$=$}}\restore
}}
\grow{\xymatrix@!0{
&\save\go+<0pt,3pt>\Drop{H}\restore \ar@{-}[1,0]&&&\save\go+<0pt,3pt>\Drop{D'}\restore \ar@{-}[2,0]&&& \save\go+<0pt,3pt>\Drop{A}\restore \ar@{-}[6,0]\\
&\ar@{-}`l/4pt [1,-1] [1,-1] \ar@{-}`r [1,1] [1,1]&&&&&&\\
\ar@{-}[2,0]&&\ar@{-}[1,1]+<-0.125pc,0.0pc> \ar@{-}[1,1]+<0.0pc,0.125pc>&&\ar@{-}[2,-2]&&&\\
&&&&&&&\\
\ar@{-}`d/4pt [1,2][1,2]&&\ar@{-}[7,0]&&\ar@{-}[-1,-1]+<0.125pc,0.0pc>\ar@{-}[-1,-1]+<0.0pc,-0.125pc> \ar@{-}[2,0]&&&\\
&&&& \ar@{-}[1,0]&&&\ar@{-}[2,0]\\
&&&&\ar@{-}`l/4pt [1,-1] [1,-1] \ar@{-}`r [1,1] [1,1]&&&\\
&&&\ar@{-}[2,0]&&\ar@{-}[1,1]+<-0.125pc,0.0pc> \ar@{-}[1,1]+<0.0pc,0.125pc>&&\ar@{-}[2,-2]\\
&&&&&&&\\
&&&\ar@{-}`d/4pt [1,2][1,2]&&\ar@{-}[2,0]&&\ar@{-}[-1,-1]+<0.125pc,0.0pc>\ar@{-}[-1,-1]+<0.0pc,-0.125pc> \ar@{-}[2,0]\\
&&&&&&&\\
&&&&&&&
}}
\grow{\xymatrix@!0{
\\\\\\\\\\
\save\go+<0pt,0pt>\Drop{\txt{$=$}}\restore
}}
\grow{\xymatrix@!0{
&\save[]+<0.2pc,3pt> \Drop{H}\ar@{-}[1,0]+<0.2pc,0pc>\restore&&&&&\save\go+<0pt,3pt>\Drop{D'}\restore \ar@{-}[4,0]&&\save\go+<0pt,3pt>\Drop{A}\restore \ar@{-}[6,0]\\
&\save[]+<0.2pc,0pc> \Drop{}\ar@{-}`l/4pt [1,-1] [1,-1] \ar@{-}`r [1,2] [1,2]\restore&&&&&&&\\
\ar@{-}[8,0]&&&\ar@{-}[1,0]&&&&&\\
&&&\ar@{-}`l/4pt [1,-1] [1,-1] \ar@{-}`r [1,1] [1,1]&&&&&\\
&&\ar@{-}[2,0]&&\ar@{-}[1,1]+<-0.125pc,0.0pc> \ar@{-}[1,1]+<0.0pc,0.125pc>&&\ar@{-}[2,-2]&&\\
&&&&&&&&\\
&&\ar@{-}[1,1]+<-0.125pc,0.0pc> \ar@{-}[1,1]+<0.0pc,0.125pc>&&\ar@{-}[2,-2]&& \ar@{-}[-1,-1]+<0.125pc,0.0pc>\ar@{-}[-1,-1]+<0.0pc,-0.125pc>\ar@{-}[1,1]+<-0.125pc,0.0pc> \ar@{-}[1,1]+<0.0pc,0.125pc>&&\ar@{-}[2,-2]\\
&&&&&&&&\\
&&\ar@{-}[2,0]\ar@{-}[2,0]&&\ar@{-}[-1,-1]+<0.125pc,0.0pc>\ar@{-}[-1,-1]+<0.0pc,-0.125pc>\ar@{-}`d/4pt [1,2][1,2]&&\ar@{-}[2,0]&& \ar@{-}[-1,-1]+<0.125pc,0.0pc>\ar@{-}[-1,-1]+<0.0pc,-0.125pc>\ar@{-}[2,0]\\
&&&&&&&&\\
&&&&&&&&
}}
\grow{\xymatrix@!0{
\\\\\\\\\\
\save\go+<0pt,0pt>\Drop{\txt{$=$}}\restore
}}
\grow{\xymatrix@!0{\\
&&\save[]+<0.2pc,3pt> \Drop{H}\ar@{-}[1,0]+<0.2pc,0pc>\restore&&&&\save\go+<0pt,3pt>\Drop{A}\restore \ar@{-}[2,0]&&\save\go+<0pt,3pt>\Drop{D'}\restore \ar@{-}[4,0]\\
&&\save[]+<0.2pc,0pc> \Drop{}\ar@{-}`l/4pt [1,-1] [1,-1] \ar@{-}`r [1,2] [1,2]\restore&&&&&&\\
&\ar@{-}[1,0]&&&\ar@{-}[1,1]+<-0.125pc,0.0pc> \ar@{-}[1,1]+<0.0pc,0.125pc>&&\ar@{-}[2,-2]&&\\
&\ar@{-}`l/4pt [1,-1] [1,-1] \ar@{-}`r [1,1] [1,1]&&&&&&&\\
\ar@{-}[2,0]&&\ar@{-}[1,1]+<-0.125pc,0.0pc> \ar@{-}[1,1]+<0.0pc,0.125pc>&&\ar@{-}[2,-2] &&\ar@{-}[-1,-1]+<0.125pc,0.0pc>\ar@{-}[-1,-1]+<0.0pc,-0.125pc>\ar@{-}[1,1]+<-0.125pc,0.0pc> \ar@{-}[1,1]+<0.0pc,0.125pc>&&\ar@{-}[2,-2]\\
&&&&&&&&\\
\ar@{-}`d/4pt [1,2][1,2]&&\ar@{-}[2,0]&&\ar@{-}[-1,-1]+<0.125pc,0.0pc>\ar@{-}[-1,-1]+<0.0pc,-0.125pc> \ar@{-}`d/4pt [1,2][1,2]&&\ar@{-}[2,0]&& \ar@{-}[-1,-1]+<0.125pc,0.0pc>\ar@{-}[-1,-1]+<0.0pc,-0.125pc> \ar@{-}[2,0]\\
&&&&&&&&\\
&&&&&&&&
}}
\grow{\xymatrix@!0{
\\\\\\\\\\
\save\go+<0pt,0pt>\Drop{\txt{$=$}}\restore
}}
\grow{\xymatrix@!0{\\\\
&\save\go+<0pt,3pt>\Drop{H}\restore \ar@{-}[1,0]&&&\save\go+<0pt,3pt>\Drop{D}\restore \ar@{-}[2,0]\\
&\ar@{-}`l/4pt [1,-1] [1,-1] \ar@{-}`r [1,1] [1,1]&&&\\
\ar@{-}[2,0]&&\ar@{-}[1,1]+<-0.125pc,0.0pc> \ar@{-}[1,1]+<0.0pc,0.125pc>&&\ar@{-}[2,-2]\\
&&&&\\
\ar@{-}`d/4pt [1,2][1,2]&&\ar@{-}[2,0]&&\ar@{-}[-1,-1]+<0.125pc,0.0pc>\ar@{-}[-1,-1]+<0.0pc,-0.125pc> \ar@{-}[2,0]\\
&&&&\\
&&&&
}}
\grow{\xymatrix@!0{
\\\\\\\\\\
\save\go+<0pt,0pt>\Drop{\txt{.}}\restore
}}
$$
Assume finally that $s>1$ and the formula is valid for $C':=H^{\ot_c^{s-1}}$ and $D:= A^{\ot^r}$. Then, we have
$$
\spreaddiagramcolumns{-1.6pc}\spreaddiagramrows{-1.6pc}
\objectmargin{0.0pc}\objectwidth{0.0pc}
\def\objectstyle{\sssize}
\def\labelstyle{\sssize}
\grow{\xymatrix@!0{\\\\\\\\
\save\go+<0pt,3pt>\Drop{C}\restore \ar@/^0.1pc/ @{-}[2,2] \ar@/_0.1pc/ @{-}[2,2]&&\save\go+<0pt,3pt>\Drop{D}\restore \ar@/^0.1pc/ @{-}[2,-2]\ar@/_0.1pc/ @{-}[2,-2]\\
&&\\
&&
}}
\grow{\xymatrix@!0{
\\\\\\\\\\
\save\go+<0pt,0pt>\Drop{\txt{$=$}}\restore
}}
\grow{\xymatrix@!0{\\\\\\
\save\go+<0pt,3pt>\Drop{H}\restore \ar@/^0.1pc/ @{-}[2,2] \ar@/_0.1pc/ @{-}[2,2] &&\save\go+<0pt,3pt>\Drop{C'}\restore \ar@/^0.1pc/ @{-}[2,-2]\ar@/_0.1pc/ @{-}[2,-2]&&\save\go+<0pt,3pt>\Drop{D}\restore \ar@{-}[2,0]\\
&&&&\\
\ar@{-}[2,0]&&\ar@/^0.1pc/ @{-}[2,2] \ar@/_0.1pc/ @{-}[2,2]&& \ar@/^0.1pc/ @{-}[2,-2]\ar@/_0.1pc/ @{-}[2,-2]\\
&&&&\\
&&&&
}}
\grow{\xymatrix@!0{
\\\\\\\\\\
\save\go+<0pt,0pt>\Drop{\txt{$=$}}\restore
}}
\grow{\xymatrix@!0{
&\save\go+<0pt,3pt>\Drop{H}\restore \ar@{-}[1,0] &&&& \save\go+<0pt,3pt>\Drop{C'}\restore\ar@{-}[1,0] &&&\save\go+<0pt,3pt>\Drop{D}\restore \ar@{-}[2,0]\\
&\ar@{-}`l/4pt [1,-1] [1,-1] \ar@{-}`r [1,1] [1,1]&&&&\ar@{-}`l/4pt [1,-1] [1,-1] \ar@{-}`r [1,1] [1,1]&&&\\
\ar@{-}[4,0]&&\ar@{-}[2,0]&&\ar@{-}[2,0]&&\ar@{-}[1,1]+<-0.125pc,0.0pc> \ar@{-}[1,1]+<0.0pc,0.125pc>&& \ar@{-}[2,-2]\\
&&&&&&&&\\
&&\ar@{-}[2,2]&&\ar@{-}`d/4pt [1,2][1,2] &&\ar@{-}[2,0]&& \ar@{-}[-1,-1]+<0.125pc,0.0pc> \ar@{-}[-1,-1]+<0.0pc,-0.125pc>\ar@{-}[6,0]\\
&&&&&&&&\\
\ar@{-}[2,2]&&&&\ar@{-}[1,1]+<-0.125pc,0.0pc> \ar@{-}[1,1]+<0.0pc,0.125pc>&&\ar@{-}[2,-2]&&\\
&&&&&&&&\\
&&\ar@{-}`d/4pt [1,2][1,2] &&\ar@{-}[2,0]&&\ar@{-}[2,0] \ar@{-}[-1,-1]+<0.125pc,0.0pc> \ar@{-}[-1,-1]+<0.0pc,-0.125pc>&&\\
&&&&&&&&\\
&&&&&&&&
}}
\grow{\xymatrix@!0{
\\\\\\\\\\
\save\go+<0pt,0pt>\Drop{\txt{$=$}}\restore
}}
\grow{\xymatrix@!0{\\
&\save\go+<0pt,3pt>\Drop{H}\restore \ar@{-}[1,0] &&&& \save\go+<0pt,3pt>\Drop{C'}\restore\ar@{-}[1,0] &&&\save\go+<0pt,3pt>\Drop{D}\restore \ar@{-}[2,0]\\
&\ar@{-}`l/4pt [1,-1] [1,-1] \ar@{-}`r [1,1] [1,1]&&&&\ar@{-}`l/4pt [1,-1] [1,-1] \ar@{-}`r [1,1] [1,1]&&&\\
\ar@{-}[5,0]&&\ar@{-}[1,1]+<-0.1pc,0.1pc> && \ar@{-}[2,-2]&& \ar@{-}[1,1]+<-0.125pc,0.0pc> \ar@{-}[1,1]+<0.0pc,0.125pc>&&\ar@{-}[2,-2]\\
&&&&&&&&\\
&&\ar@{-}[2,0]&&\ar@{-}[-1,-1]+<0.1pc,-0.1pc> \ar@{-}[1,1]+<-0.125pc,0.0pc> \ar@{-}[1,1]+<0.0pc,0.125pc> &&\ar@{-}[2,-2] &&\ar@{-}[-1,-1]+<0.125pc,0.0pc>\ar@{-}[-1,-1]+<0.0pc,-0.125pc> \ar@{-}[5,0]\\
&&&&&&&&\\
&&\ar@{-}`d/4pt [1,2][1,2]&&\ar@{-}[3,0]&&\ar@{-}[-1,-1]+<0.125pc,0.0pc>\ar@{-}[-1,-1]+<0.0pc,-0.125pc> \ar@{-}[3,0]&&\\
\ar@{-}`d/4pt [1,4][1,4]&&&&&&&&\\
&&&&&&&&\\
&&&&&&&&
}}
\grow{\xymatrix@!0{
\\\\\\\\\\
\save\go+<0pt,0pt>\Drop{\txt{$=$}}\restore
}}
\grow{\xymatrix@!0{\\\\\\º
&\save\go+<0pt,3pt>\Drop{C}\restore \ar@{-}[1,0]&&&\save\go+<0pt,3pt>\Drop{D}\restore \ar@{-}[2,0]\\
&\ar@{-}`l/4pt [1,-1] [1,-1] \ar@{-}`r [1,1] [1,1]&&&\\
\ar@{-}[2,0]&&\ar@{-}[1,1]+<-0.125pc,0.0pc> \ar@{-}[1,1]+<0.0pc,0.125pc>&&\ar@{-}[2,-2]\\
&&&&\\
\ar@{-}`d/4pt [1,2][1,2]&&\ar@{-}[2,0]&&\ar@{-}[-1,-1]+<0.125pc,0.0pc>\ar@{-}[-1,-1]+<0.0pc,-0.125pc> \ar@{-}[2,0]\\
&&&&\\
&&&&
}}
\grow{\xymatrix@!0{
\\\\\\\\\\
\save\go+<0pt,0pt>\Drop{\txt{,}}\restore
}}
$$
where $C:= H^{\ot_c^s}$.
\end{proof}

\begin{lemma}\label{para d1} Let $s,r\in \mathds{N}$. For $C:= H^{\ot_c^s}$ and $D:=A^{\ot_k^r}$, the equality
$$
\spreaddiagramcolumns{-1.6pc}\spreaddiagramrows{-1.6pc}
\objectmargin{0.0pc}\objectwidth{0.0pc}
\def\objectstyle{\sssize}
\def\labelstyle{\sssize}
\grow{\xymatrix@!0{
&\save\go+<0pt,3pt>\Drop{D}\restore \ar@{-}[2,2]&&\save\go+<0pt,3pt>\Drop{C}\restore \ar@{-}[1,-1]+<0.125pc,0.0pc> \ar@{-}[1,-1]+<0.0pc,0.125pc>&\\
&&&&\\
&\ar@{-}[-1,1]+<-0.125pc,0.0pc> \ar@{-}[-1,1]+<0.0pc,-0.125pc>\ar@{-}[1,0]&&\ar@{-}[2,1]&\\
& \ar@{-}`l/4pt [1,-1] [1,-1] \ar@{-}`r [1,1] [1,1]&&&\\
\ar@{-}[2,0]&& \ar@/^0.1pc/ @{-}[2,2] \ar@/_0.1pc/ @{-}[2,2]&& \ar@/^0.1pc/ @{-}[2,-2]\ar@/_0.1pc/ @{-}[2,-2]\\
&&&&\\
&&&&
}}
\grow{\xymatrix@!0{
\\\\\\
\save\go+<0pt,0pt>\Drop{\txt{$=$}}\restore }}
\grow{\xymatrix@!0{
& \save\go+<0pt,3pt>\Drop{D}\restore \ar@{-}[2,0]&&&\save\go+<0pt,3pt>\Drop{C}\restore \ar@{-}[1,0]&\\
&&&&\ar@{-}`l/4pt [1,-1] [1,-1] \ar@{-}`r [1,1] [1,1]&\\
&\ar@{-}[2,2]&&\ar@{-}[1,-1]+<0.125pc,0.0pc> \ar@{-}[1,-1]+<0.0pc,0.125pc>&&\ar@{-}[6,0]\\
&&&&&\\
&\ar@{-}[1,0]\ar@{-}[-1,1]+<-0.125pc,0.0pc> \ar@{-}[-1,1]+<0.0pc,-0.125pc>&&\ar@{-}[1,1]&&\\
&\ar@{-}`l/4pt [1,-1] [1,-1] \ar@{-}`r [1,1] [1,1]&&&\ar@{-}[3,0]&\\
\ar@{-}[2,0]&&\ar@{-}`d/4pt [1,2][1,2] &&&\\
&&&&&\\
&&&&&
}}
$$
is true.
\end{lemma}

\begin{proof} In fact, we have
$$
\spreaddiagramcolumns{-1.6pc}\spreaddiagramrows{-1.6pc}
\objectmargin{0.0pc}\objectwidth{0.0pc}
\def\objectstyle{\sssize}
\def\labelstyle{\sssize}
\grow{\xymatrix@!0{\\\\\\\\
&\save\go+<0pt,3pt>\Drop{D}\restore \ar@{-}[2,2]&&\save\go+<0pt,3pt>\Drop{C}\restore \ar@{-}[1,-1]+<0.125pc,0.0pc> \ar@{-}[1,-1]+<0.0pc,0.125pc>&\\
&&&&\\
&\ar@{-}[-1,1]+<-0.125pc,0.0pc> \ar@{-}[-1,1]+<0.0pc,-0.125pc>\ar@{-}[1,0]&&\ar@{-}[2,1]&\\
& \ar@{-}`l/4pt [1,-1] [1,-1] \ar@{-}`r [1,1] [1,1]&&&\\
\ar@{-}[2,0]&& \ar@/^0.1pc/ @{-}[2,2] \ar@/_0.1pc/ @{-}[2,2]&& \ar@/^0.1pc/ @{-}[2,-2]\ar@/_0.1pc/ @{-}[2,-2]\\
&&&&\\
&&&&\\
}}
\grow{\xymatrix@!0{
\\\\\\\\\\\\\\
\save\go+<0pt,0pt>\Drop{\txt{$=$}}\restore }}
\grow{\xymatrix@!0{\\\\
&\save\go+<0pt,3pt>\Drop{D}\restore \ar@{-}[4,4]&&\save\go+<0pt,3pt>\Drop{C}\restore \ar@{-}[1,-1]+<0.125pc,0.0pc> \ar@{-}[1,-1]+<0.0pc,0.125pc>&& \\
&&&&&\\
&\ar@{-}[-1,1]+<-0.125pc,0.0pc> \ar@{-}[-1,1]+<0.0pc,-0.125pc> \ar@{-}[1,0]&&&&\\
&\ar@{-}`l/4pt [1,-1] [1,-1] \ar@{-}`r [1,1] [1,1]&&&&\\
\ar@{-}[6,0]&&\ar@{-}[1,0]&&&\ar@{-}[2,0]\\
&&\ar@{-}`l/4pt [1,-1] [1,-1] \ar@{-}`r [1,1] [1,1]&&&\\
&\ar@{-}[2,0]&&\ar@{-}[1,1]+<-0.125pc,0.0pc> \ar@{-}[1,1]+<0.0pc,0.125pc>&&\ar@{-}[2,-2]\\
&&&&&\\
&\ar@{-}`d/4pt [1,2][1,2] &&\ar@{-}[2,0]&&\ar@{-}[-1,-1]+<0.125pc,0.0pc> \ar@{-}[-1,-1]+<0.0pc,-0.125pc>\ar@{-}[2,0]\ar@{-}[2,0]\\
&&&&&\\
&&&&&}
}
\grow{\xymatrix@!0{
\\\\\\\\\\\\\\
\save\go+<0pt,0pt>\Drop{\txt{$=$}}\restore }}
\grow{\xymatrix@!0{
\save\go+<0pt,3pt>\Drop{D}\restore \ar@{-}[2,0] &&&\save\go+<0.2pc,3pt>\Drop{C}\restore \save[]+<0.2pc,0pc> \Drop{}\ar@{-}[1,0]+<0.2pc,0pc>\restore&&& \\
&&&\save[]+<0.2pc,0pc> \Drop{}\ar@{-}`l/4pt [1,-1] [1,-1] \ar@{-}`r [1,2] [1,2]\restore &&&\\
\ar@{-}[2,2]&&\ar@{-}[1,-1]+<0.125pc,0.0pc> \ar@{-}[1,-1]+<0.0pc,0.125pc>&&&\ar@{-}[1,0]&\\
&&&&&\ar@{-}`l/4pt [1,-1] [1,-1] \ar@{-}`r [1,1] [1,1]&\\
\ar@{-}[9,0]\ar@{-}[-1,1]+<-0.125pc,0.0pc> \ar@{-}[-1,1]+<0.0pc,-0.125pc> &&\ar@{-}[2,2]&&\ar@{-}[1,-1]+<0.125pc,0.0pc> \ar@{-}[1,-1]+<0.0pc,0.125pc>&&\ar@{-}[2,0]\\
&&&&&&\\
&&\ar@{-}[5,0]\ar@{-}[-1,1]+<-0.125pc,0.0pc> \ar@{-}[-1,1]+<0.0pc,-0.125pc>&& \ar@{-}[2,2]&&\ar@{-}[1,-1]+<0.125pc,0.0pc> \ar@{-}[1,-1]+<0.0pc,0.125pc>\\
&&&&&&\\
&&&&\ar@{-}[1,0]\ar@{-}[-1,1]+<-0.125pc,0.0pc> \ar@{-}[-1,1]+<0.0pc,-0.125pc>&&\ar@{-}[1,0]\\
&&&&\ar@{-}[1,1]+<-0.125pc,0.0pc> \ar@{-}[1,1]+<0.0pc,0.125pc>&&\ar@{-}[2,-2]\\
&&&&&&\\
&&\ar@{-}`d/4pt [1,2][1,2]&&\ar@{-}[2,0]&&\ar@{-}[-1,-1]+<0.125pc,0.0pc> \ar@{-}[-1,-1]+<0.0pc,-0.125pc>\ar@{-}[2,0]\\
&&&&&&\\
&&&&&&
}}
\grow{\xymatrix@!0{
\\\\\\\\\\\\\\
\save\go+<0pt,0pt>\Drop{\txt{$=$}}\restore }}
\grow{\xymatrix@!0{\\\\
\save\go+<0pt,3pt>\Drop{D}\restore \ar@{-}[2,0] &&&\save\go+<0.2pc,3pt>\Drop{C}\restore \save[]+<0.2pc,0pc> \Drop{}\ar@{-}[1,0]+<0.2pc,0pc>\restore&&& \\
&&&\save[]+<0.2pc,0pc> \Drop{}\ar@{-}`l/4pt [1,-1] [1,-1] \ar@{-}`r [1,2] [1,2]\restore &&&\\
\ar@{-}[2,2]&&\ar@{-}[1,-1]+<0.125pc,0.0pc> \ar@{-}[1,-1]+<0.0pc,0.125pc>&&&\ar@{-}[1,0]&\\
&&&&&\ar@{-}`l/4pt [1,-1] [1,-1] \ar@{-}`r [1,1] [1,1]&\\
\ar@{-}[-1,1]+<-0.125pc,0.0pc> \ar@{-}[-1,1]+<0.0pc,-0.125pc>\ar@{-}[5,0] &&\ar@{-}[2,2]&&\ar@{-}[1,-1]+<0.125pc,0.0pc> \ar@{-}[1,-1]+<0.0pc,0.125pc>&&\ar@{-}[5,0]\\
&&&&&&\\
&&\ar@{-}[-1,1]+<-0.125pc,0.0pc> \ar@{-}[-1,1]+<0.0pc,-0.125pc>\ar@{-}[1,0]&&\ar@{-}[3,0]&&\\
&&\ar@{-}`d/4pt [1,2][1,2]  &&&&\\
&&&&&&\\
&&&&&&}}
\grow{\xymatrix@!0{
\\\\\\\\\\\\\\
\save\go+<0pt,0pt>\Drop{\txt{$=$}}\restore }}
\grow{\xymatrix@!0{\\\\\\\\
& \save\go+<0pt,3pt>\Drop{D}\restore \ar@{-}[2,0]&&&\save\go+<0pt,3pt>\Drop{C}\restore \ar@{-}[1,0]&\\
&&&&\ar@{-}`l/4pt [1,-1] [1,-1] \ar@{-}`r [1,1] [1,1]&\\
&\ar@{-}[2,2]&&\ar@{-}[1,-1]+<0.125pc,0.0pc> \ar@{-}[1,-1]+<0.0pc,0.125pc>&&\ar@{-}[6,0]\\
&&&&&\\
&\ar@{-}[1,0]\ar@{-}[-1,1]+<-0.125pc,0.0pc> \ar@{-}[-1,1]+<0.0pc,-0.125pc>&&\ar@{-}[1,1]&&\\
&\ar@{-}`l/4pt [1,-1] [1,-1] \ar@{-}`r [1,1] [1,1]&&&\ar@{-}[3,0]&\\
\ar@{-}[2,0]&&\ar@{-}`d/4pt [1,2][1,2] &&&\\
&&&&&\\
&&&&&
}}
\grow{\xymatrix@!0{
\\\\\\\\\\\\\\
\save\go+<0pt,0pt>\Drop{\txt{,}}\restore}}
$$
where the first equality follows from Lemma~\ref{formula para chi}.
\end{proof}

\noindent\bf Proof of Theorem~\ref{th5.2}.\rm\enspace By Remark~\ref{re def theta}, the map
$$
\theta_*\colon (\wh{X}_*(M),\wh{d}_*)\to (\ov{X}_*(M),\ov{d}_*)
$$
is an isomorphism of chain complexes. Hence, by the discussion at the beginning of Section~\ref{Hochschild-Brzezinski}, the homology of $(\ov{X}_*(M),\ov{d}_*)$ is the Hochschild homology of the $K$-algebra $E$ with coefficients in $M$. In order to complete the proof we must compute $\ov{d}^0$ and $\ov{d}^1$. First we consider the map $\ov{d}^0$. Let
$$
\wt{\nu}_i\colon M\ot_{\!A} E^{\ot_{\!A}^s}\ot_k A^{\ot_k^r} \to M\ot_{\!A} E^{\ot_{\!A}^s}\ot_k A^{\ot_k^{r-1}}\qquad\text{($0\le i\le r$)}
$$
be the morphisms defined by
$$
\wt{\nu}_i(m\ot_{\!A}\ov{\bx}_{1s}\ot\ba_{1r}) := \begin{cases} m\ot_{\!A}\ov{\bx}_{1,s-1} \ot_{\!A} x_sa\ot\ba_{2r}  &\text{if $i = 0$,}\\ m\ot_{\!A}\ov{\bx}_{1s}\ot\ba_{1,i-1}\ot a_ia_{i+1}\ot \ba_{i+2,r} &\text{if $0< i< r$,}\\ a_rm\ot_{\!A}\ov{\bx}_{1s}\ot\ba_{1,r-1}  &\text{if $i = r$.}\end{cases}
$$
For $0\le i\le r$, set $\ov{\nu}_i := \wt{\theta} \xcirc \wt{\nu}_i\xcirc \wt{\vartheta}$, where $\wt{\theta}$ and $\wt{\vartheta}$ are as in the proof of Proposition~\ref{inversa de theta}. By item~(1) of Theorem~\ref{formula para wh{d}_1} we know that $\wh{d}^0$ is induced by $\sum_{i=0}^r (-1)^i \wt{\nu}_i$. Hence, $\ov{d}^0$ is induced by $\sum_{i=0}^r (-1)^{s+i} \ov{\nu}_i$. So, in order to complete the computation of $\ov{d}^0$ it is enough to calculate the $\ov{\nu}_i$'s. We begin with the computation of $\ov{\nu}_0$. Let $C:= H^{\ot_c^{s-1}}$, $D:= A^{\ot_k^{r-1}}$, $\boldsymbol{\gamma}:= \gamma^{\ot_k^{s-1}}$, $\mu:=\mu_{s-1}$, $\mathbf{u} := \mu\xcirc \boldsymbol{\gamma}$ and $\ov{\mathbf{u}}:=\mu\xcirc \ov{\gamma}^{\ot_k^{s-1}}\xcirc \gc_{s-1}$. Since, by Lemma~\ref{inversa de convolucion},
\spreaddiagramcolumns{-1.6pc}\spreaddiagramrows{-1.6pc}
\objectmargin{0.0pc}\objectwidth{0.0pc}
\def\objectstyle{\sssize}
\def\labelstyle{\sssize}
\allowdisplaybreaks
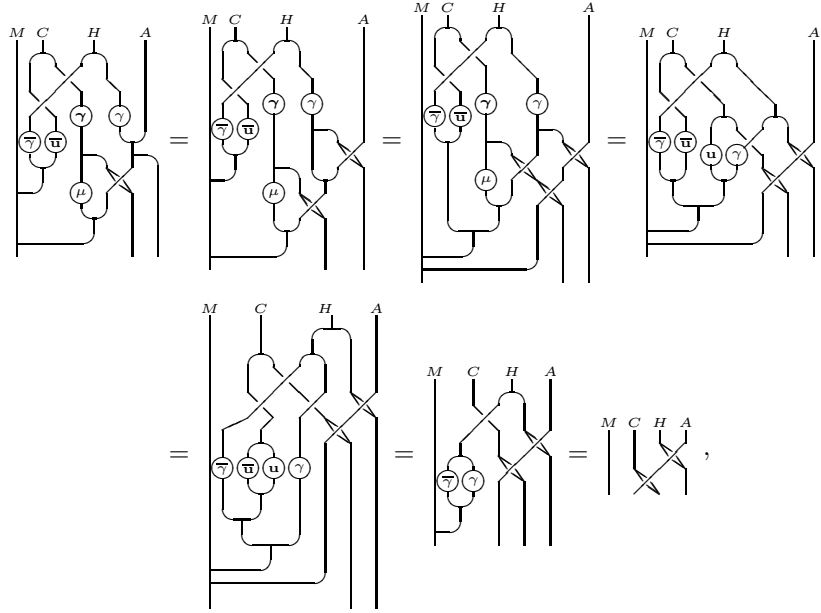
\begin{align*}
\grow{\xymatrix@!0{\\\\
\save\go+<0pt,3pt>\Drop{M}\restore \ar@{-}[17,0]&&\save\go+<0pt,3pt>\Drop{C}\restore \ar@{-}[1,0]&&&& \save\go+<0pt,3pt>\Drop{H}\restore \ar@{-}[1,0]&&&& \save\go+<0pt,3pt>\Drop{A}\restore\ar@{-}[4,0]\\
&&\ar@{-}`l/4pt [1,-1] [1,-1] \ar@{-}`r [1,1] [1,1]&&&&\ar@{-}`l/4pt [1,-1] [1,-1] \ar@{-}`r [1,1] [1,1]&&&&&\\
&\ar@{-}[2,0]&&\ar@{-}[1,1]+<-0.1pc,0.1pc> && \ar@{-}[2,-2]&&\ar@{-}[1,0]&&&&\\
&&&&&&&\ar@{-}[1,1]&&&&\\
&\ar@{-}[1,1]+<-0.1pc,0.1pc> && \ar@{-}[2,-2]&& \ar@{-}[-1,-1]+<0.1pc,-0.1pc> \ar@{-}[2,0]+<0pc,4pt>&&& \ar@{-}[2,0]+<0pc,4pt>&&\ar@{-}[3,0]\\
&&&&&&&&&&&\\
&\ar@{-}[1,0]+<0pc,-0.5pt>&&\ar@{-}[1,0]+<0pc,-0.5pt>\ar@{-}[-1,-1]+<0.1pc,-0.1pc>&&*+[o]+<0.36pc>[F]{\boldsymbol{\gamma}}\ar@{-}[5,0]+<0pc,-0.5pt>&&& *+[o]+<0.37pc>[F]{\gamma}\ar@{-}[1,0]+<0pc,0pc>&&&\\
&&&&&&&&\ar@{-}`d/4pt [1,1] `[0,2] [0,2]&&&\\
&*+[o]+<0.36pc>[F]{\ov{\gamma}}\ar@{-}[1,0]+<0pc,0pc>&&*+[o]+<0.37pc>[F]{\ov{\mathbf{u}}}\ar@{-}[1,0]+<0pc,0pc>&&&&&&\ar@{-}[2,0]&&\\
&\ar@{-}`d/4pt [1,1] `[0,2] [0,2]&&&&\ar@{-}`r/4pt [1,2] [1,2]&&&&\ar@{-}`r/4pt [1,2][1,2]&&\\
&&\ar@{-}[1,0]&&&&&\ar@{-}[1,1]+<-0.125pc,0.0pc>\ar@{-}[1,1]+<0.0pc,0.125pc>&&\ar@{-}[2,-2]&&\ar@{-}[5,0]\\
&&\ar@{-}`d/4pt [1,-2][1,-2]&&&&&&&&&\\
&&&&&*+[o]+<0.36pc>[F]{\mu}\ar@{-}[1,0]+<0pc,0pc>&&\ar@{-}[1,0]&&\ar@{-}[-1,-1]+<0.125pc,0.0pc>\ar@{-}[-1,-1]+<0.0pc,-0.125pc>\ar@{-}[3,0]&\\
&&&&&\ar@{-}`d/4pt [1,1] `[0,2] [0,2]&&&&&&\\
&&&&&&\ar@{-}[1,0]&&&&&\\
&&&&&&\ar@{-}`d/4pt [1,-6][1,-6]&&&\ar@{-}[2,0]&&\ar@{-}[2,0]\\
&&&&&&&&&&&\\
&&&&&&&&&&&}}
& \grow{\xymatrix@!0{
\\\\\\\\\\\\\\\\\\\\
\save\go+<0pt,0pt>\Drop{\txt{$=$}}\restore }}
\grow{\xymatrix@!0{\\
\save\go+<0pt,3pt>\Drop{M}\restore \ar@{-}[19,0]&&\save\go+<0pt,3pt>\Drop{C}\restore \ar@{-}[1,0]&&&& \save\go+<0pt,3pt>\Drop{H}\restore \ar@{-}[1,0]&&&&&& \save\go+<0pt,3pt>\Drop{A}\restore\ar@{-}[4,0]\\
&&\ar@{-}`l/4pt [1,-1] [1,-1] \ar@{-}`r [1,1] [1,1]&&&&\ar@{-}`l/4pt [1,-1] [1,-1] \ar@{-}`r [1,1] [1,1]&&&&&\\
&\ar@{-}[2,0]&&\ar@{-}[1,1]+<-0.1pc,0.1pc> && \ar@{-}[2,-2]&&\ar@{-}[1,0]&&&&\\
&&&&&&&\ar@{-}[1,1]&&&&\\
&\ar@{-}[1,1]+<-0.1pc,0.1pc> && \ar@{-}[2,-2]&& \ar@{-}[-1,-1]+<0.1pc,-0.1pc> \ar@{-}[2,0]+<0pc,4pt>&&&\ar@{-}[2,0]+<0pc,4pt>&&&&\ar@{-}[5,0]\\
&&&&&&&&&&&&\\
&\ar@{-}[1,0]+<0pc,-0.5pt>&&\ar@{-}[1,0]+<0pc,-0.5pt>\ar@{-}[-1,-1]+<0.1pc,-0.1pc>&&
*+[o]+<0.36pc>[F]{\boldsymbol{\gamma}}\ar@{-}[6,0]+<0pc,-0.5pt>&&&*+[o]+<0.37pc>[F]{\gamma}\ar@{-}[5,0]&&&&\\
&&&&&&&&&&&&\\
&*+[o]+<0.35pc>[F]{\ov{\gamma}}\ar@{-}[1,0]+<0pc,0pc>&&*+[o]+<0.37pc>[F]{\ov{\mathbf{u}}}\ar@{-}[1,0]+<0pc,0pc>&&&&& \ar@{-}`r/4pt [1,2] [1,2]&&&&\\
&\ar@{-}`d/4pt [1,1] `[0,2] [0,2]&&&&&&&&&\ar@{-}[1,1]+<-0.125pc,0.0pc> \ar@{-}[1,1]+<0.0pc,0.125pc>&&\ar@{-}[2,-2]\\
&&\ar@{-}[1,0]&&&&&&&&&&\\
&&\ar@{-}`d/4pt [1,-2][1,-2]&&&\ar@{-}`r/4pt [1,2] [1,2]&&&\ar@{-}`d/4pt [1,1] `[0,2] [0,2]&&&& \ar@{-}[-1,-1]+<0.125pc,0.0pc> \ar@{-}[-1,-1]+<0.0pc,-0.125pc>\ar@{-}[8,0]\\
&&&&&&&\ar@{-}[1,0]&&\ar@{-}[1,0]&&&\\
&&&&&*+[o]+<0.36pc>[F]{\mu}\ar@{-}[2,0]+<0pc,0pc>&&\ar@{-}[1,1]+<-0.125pc,0.0pc> \ar@{-}[1,1]+<0.0pc,0.125pc>&&\ar@{-}[2,-2]&&&\\
&&&&&&&&&&&&\\
&&&&&\ar@{-}`d/4pt [1,1] `[0,2] [0,2]&&&&\ar@{-}[4,0]\ar@{-}[-1,-1]+<0.125pc,0.0pc> \ar@{-}[-1,-1]+<0.0pc,-0.125pc>&&&\\
&&&&&&\ar@{-}[1,0]&&&&&&\\
&&&&&&\ar@{-}`d/4pt [1,-6][1,-6]&&&&&&\\
&&&&&&&&&&&&\\
&&&&&&&&&&&&}}
\grow{\xymatrix@!0{
\\\\\\\\\\\\\\\\\\\\
\save\go+<0pt,0pt>\Drop{\txt{$=$}}\restore }}
\grow{\xymatrix@!0{
\save\go+<0pt,3pt>\Drop{M}\restore \ar@{-}[21,0]&&\save\go+<0pt,3pt>\Drop{C}\restore \ar@{-}[1,0]&&&& \save\go+<0pt,3pt>\Drop{H}\restore \ar@{-}[1,0]&&&&&&& \save\go+<0pt,3pt>\Drop{A}\restore\ar@{-}[4,0]\\
&&\ar@{-}`l/4pt [1,-1] [1,-1] \ar@{-}`r [1,1] [1,1]&&&&\ar@{-}`l/4pt [1,-1] [1,-1] \ar@{-}`r [1,1] [1,1]&&&&&&\\
&\ar@{-}[2,0]&&\ar@{-}[1,1]+<-0.1pc,0.1pc> && \ar@{-}[2,-2]&&\ar@{-}[1,0]&&&&&\\
&&&&&&&\ar@{-}[2,2]&&&&&&\\
&\ar@{-}[1,1]+<-0.1pc,0.1pc> && \ar@{-}[2,-2]&& \ar@{-}[-1,-1]+<0.1pc,-0.1pc> \ar@{-}[2,0]+<0pc,-0.5pt>&&&&&&&&\ar@{-}[6,0]\\
&&&&& &&&&\ar@{-}[2,0]+<0pc,4pt>&&&&\\
&\ar@{-}[1,0]+<0pc,-0.5pt>&&\ar@{-}[1,0]+<0pc,-0.5pt>\ar@{-}[-1,-1]+<0.1pc,-0.1pc>&&&&&&&&&&\\
&&&&&*+[o]+<0.36pc>[F]{\boldsymbol{\gamma}}\ar@{-}[5,0]+<0pc,-0.5pt>&&&& *+[o]+<0.37pc>[F]{\gamma}\ar@{-}[4,0]+<0pc,0pc>&&&&\\
&*+[o]+<0.34pc>[F]{\ov{\gamma}}\ar@{-}[1,0]+<0pc,0pc>&&*+[o]+<0.37pc>[F]{\ov{\mathbf{u}}}\ar@{-}[1,0]+<0pc,0pc>&&&&&&&&&&\\
&\ar@{-}`d/4pt [1,1] `[0,2] [0,2]&&&&&&&&\ar@{-}`r/4pt [1,2] [1,2]&&&&\\
&&\ar@{-}[6,0]&&&\ar@{-}`r/4pt [1,2] [1,2]&&&&&&\ar@{-}[1,1]+<-0.125pc,0.0pc> \ar@{-}[1,1]+<0.0pc,0.125pc>&&\ar@{-}[2,-2]\\
&&&&&&&\ar@{-}[1,1]+<-0.125pc,0.0pc> \ar@{-}[1,1]+<0.0pc,0.125pc>&&\ar@{-}[2,-2]&&&&\\
&&&&&&&&&&&\ar@{-}[1,0]&&\ar@{-}[-1,-1]+<0.125pc,0.0pc> \ar@{-}[-1,-1]+<0.0pc,-0.125pc>\ar@{-}[7,0]\\
&&&&&*+[o]+<0.36pc>[F]{\mu}\ar@{-}[1,0]+<0pc,0pc>&&\ar@{-}[1,0]&&\ar@{-}[-1,-1]+<0.125pc,0.0pc>\ar@{-}[-1,-1]+<0.0pc,-0.125pc> \ar@{-}[1,1]+<-0.125pc,0.0pc> \ar@{-}[1,1]+<0.0pc,0.125pc>&&\ar@{-}[2,-2]&&\\
&&&&&\ar@{-}`d/4pt [1,1] `[0,2] [0,2] &&&&&&&&\\
&&&&&&\ar@{-}[1,0]&&&\ar@{-}[4,0]&&\ar@{-}[-1,-1]+<0.125pc,0.0pc> \ar@{-}[-1,-1]+<0.0pc,-0.125pc>\ar@{-}[4,0]&&\\
&&\ar@{-}`d/4pt [1,2] `[0,4] [0,4]&&&&&&&&&&&\\
&&&&\ar@{-}[1,0]&&&&&&&&&\\
&&&&\ar@{-}`d/4pt [1,-4][1,-4]&&&&&&&&&\\
&&&&&&&&&\ar@{-}`d/4pt [1,-9][1,-9]&&\ar@{-}[2,0]&&\ar@{-}[2,0]\\
&&&&&&&&&&&&&\\
&&&&&&&&&&&&&}}
\grow{\xymatrix@!0{
\\\\\\\\\\\\\\\\\\\\
\save\go+<0pt,0pt>\Drop{\txt{$=$}}\restore }}
\grow{\xymatrix@!0{\\\\
\save\go+<0pt,3pt>\Drop{M}\restore \ar@{-}[17,0]&&\save\go+<0pt,3pt>\Drop{C}\restore \ar@{-}[1,0]&&&& \save\go+<0pt,3pt>\Drop{H}\restore \ar@{-}[1,0]&&&&&&& \save\go+<0pt,3pt>\Drop{A}\restore\ar@{-}[8,0]\\
&&\ar@{-}`l/4pt [1,-1] [1,-1] \ar@{-}`r [1,1] [1,1]&&&&\ar@{-}`l/4pt [1,-1] [1,-1] \ar@{-}`r [1,1] [1,1]&&&&&&\\
&\ar@{-}[2,0]&&\ar@{-}[1,1]+<-0.1pc,0.1pc> && \ar@{-}[2,-2]&&\ar@{-}[3,3]&&&&&\\
&&&&&&&&&&&&&\\
&\ar@{-}[1,1]+<-0.1pc,0.1pc> && \ar@{-}[2,-2]&& \ar@{-}[-1,-1]+<0.1pc,-0.1pc> \ar@{-}[1,1]&&&&&&&&\\
&&&&&&\ar@{-}[1,0]&&&&\ar@{-}[1,0]&&&\\
&\ar@{-}[1,0]+<0pc,-0.5pt>&&\ar@{-}[1,0]+<0pc,-0.5pt>\ar@{-}[-1,-1]+<0.1pc,-0.1pc>&&& \ar@{-}`l/4pt [1,-1] [1,-1] \ar@{-}`r [1,1] [1,1] &&&&\ar@{-}`l/4pt [1,-1] [1,-1] \ar@{-}`r [1,1] [1,1] &&&\\
&&&&&\ar@{-}[1,0]+<0pc,-0.5pt>&& \ar@{-}[1,1]+<-0.1pc,0.1pc> && \ar@{-}[2,-2] &&\ar@{-}[1,0]+<0pc,0.0pt>&&\\\
&*+[o]+<0.34pc>[F]{\ov{\gamma}}\ar@{-}[2,0]+<0pc,0pc>&&*+[o]+<0.37pc>[F]{\ov{\mathbf{u}}}\ar@{-}[2,0]+<0pc,0pc> &&&&&&&& \ar@{-}[1,1]+<-0.125pc,0.0pc> \ar@{-}[1,1]+<0.0pc,0.125pc>&&\ar@{-}[2,-2] \\
&&&&&*+[o]+<0.40pc>[F]{\mathbf{u}}\ar@{-}[1,0]+<0pc,0pc>&&*+[o]+<0.37pc>[F]{\gamma} \ar@{-}[1,0]+<0pc,0pc>&&\ar@{-}[-1,-1]+<0.1pc,-0.1pc>\ar@{-}[1,0]&&&&\\
&\ar@{-}`d/4pt [1,1] `[0,2] [0,2]&&&&\ar@{-}`d/4pt [1,1] `[0,2] [0,2]&&&&\ar@{-}[1,1]+<-0.125pc,0.0pc> \ar@{-}[1,1]+<0.0pc,0.125pc>&&\ar@{-}[2,-2]&& \ar@{-}[-1,-1]+<0.125pc,0.0pc> \ar@{-}[-1,-1]+<0.0pc,-0.125pc>\ar@{-}[7,0]\\
&&\ar@{-}[1,0]&&&&\ar@{-}[1,0]&&&&&&&\\
&&\ar@{-}`d/4pt [1,2] `[0,4] [0,4]&&&&&&&\ar@{-}[3,0]&&\ar@{-}[-1,-1]+<0.125pc,0.0pc> \ar@{-}[-1,-1]+<0.0pc,-0.125pc> \ar@{-}[5,0]&&\\
&&&&\ar@{-}[1,0]&&&&&&&&&\\
&&&&\ar@{-}`d/4pt [1,-4][1,-4]&&&&&&&&&\\
&&&&&&&&&\ar@{-}`d/4pt [1,-9][1,-9]&&&&\\
&&&&&&&&&&&&&\\
&&&&&&&&&&&&&}}\\
& \grow{\xymatrix@!0{
\\\\\\\\\\\\\\\\\\\\\\
\save\go+<0pt,0pt>\Drop{\txt{$=$}}\restore }}
\grow{\xymatrix@!0{
\save\go+<0pt,3pt>\Drop{M}\restore \ar@{-}[23,0]&&&&\save\go+<0pt,3pt>\Drop{C}\restore \ar@{-}[3,0]&&&&& \save\go+<0pt,3pt>\Drop{H}\restore \save[]+<0.2pc,0pc> \Drop{}\ar@{-}[1,0]+<0.2pc,0pc>\restore &&&& \save\go+<0pt,3pt>\Drop{A}\restore\ar@{-}[6,0]\\
&&&&&&&&&\save[]+<0.2pc,0pc> \Drop{}\ar@{-}`l/4pt [1,-1] [1,-1] \ar@{-}`r [1,2] [1,2]\restore &&&&\\
&&&&&&&&\ar@{-}[1,0]&&&\ar@{-}[4,0] &&\\
&&&&\ar@{-}`l/4pt [1,-1] [1,-1] \ar@{-}`r [1,1] [1,1] &&&&\ar@{-}`l/4pt [1,-1] [1,-1] \ar@{-}`r [1,1] [1,1] &&&&&\\
&&&\ar@{-}[2,0]&&\ar@{-}[1,1]+<-0.1pc,0.1pc> && \ar@{-}[2,-2]&&\ar@{-}[2,0]&&&&\\
&&&&&&&&&&&&&\\
&&&\ar@{-}[1,1]+<-0.1pc,0.1pc> && \ar@{-}[2,-2]&&\ar@{-}[-1,-1]+<0.1pc,-0.1pc> \ar@{-}[1,1]+<-0.1pc,0.1pc> && \ar@{-}[2,-2]&&\ar@{-}[1,1]+<-0.125pc,0.0pc> \ar@{-}[1,1]+<0.0pc,0.125pc>&&\ar@{-}[2,-2]\\
&&&&&&&&&&&&&\\
&&&\ar@{-}[1,-2]  &&\ar@{-}[1,-1] \ar@{-}[-1,-1]+<0.1pc,-0.1pc>&&\ar@{-}[1,0] && \ar@{-}[-1,-1]+<0.1pc,-0.1pc> \ar@{-}[1,1]+<-0.125pc,0.0pc> \ar@{-}[1,1]+<0.0pc,0.125pc>&&\ar@{-}[2,-2] &&\ar@{-}[-1,-1]+<0.125pc,0.0pc>\ar@{-}[-1,-1]+<0.0pc,-0.125pc>\ar@{-}[15,0]\\
&\ar@{-}[2,0]&&&\ar@{-}[1,0] &&&\ar@{-}[2,0]&&&&&&\\
&&&&\ar@{-}`l/4pt [1,-1] [1,-1] \ar@{-}`r [1,1] [1,1] &&&&&\ar@{-}[10,0] &&\ar@{-}[13,0]\ar@{-}[-1,-1]+<0.125pc,0.0pc> \ar@{-}[-1,-1]+<0.0pc,-0.125pc> &&\\
&\ar@{-}[0,0]+<0pc,-0.5pt>&&\ar@{-}[0,0]+<0pc,-0.5pt>&&\ar@{-}[0,0]+<0pc,-0.5pt>&&\ar@{-}[0,0]+<0pc,-0.5pt>&&&&&\\
&*+[o]+<0.34pc>[F]{\ov{\gamma}}\ar@{-}[3,0]+<0pc,0pc>&&*+[o]+<0.37pc>[F]{\ov{\mathbf{u}}}\ar@{-}[1,0]+<0pc,0pc> &&*+[o]+<0.40pc>[F]{\mathbf{u}}\ar@{-}[1,0]+<0pc,0pc>&&*+[o]+<0.37pc>[F]{\gamma}\ar@{-}[5,0]+<0pc,0pc>&&&&&\\
&&&\ar@{-}`d/4pt [1,1] `[0,2] [0,2]  &&&&&&&&&\\
&&&&\ar@{-}[1,0]&&&&&&&&\\
&\ar@{-}`d/4pt [1,1]+<0.2pc,0pc> `[0,3] [0,3]&&&&&&&&&&&\\
&&\save[]+<0.2pc,0pc> \Drop{}\ar@{-}[1,0]+<0.2pc,0pc>\restore &&&&&&&&&&\\
&&\save[]+<0.2pc,0pc> \Drop{} \ar@{-}`d/4pt [1,2]+<0.2pc,0pc> `[0,4]+<0.4pc,0pc> [0,4]+<0.4pc,0pc> \restore&&&&&&&&&&\\
&&&&\save[]+<0.3pc,0pc> \Drop{}\ar@{-}[1,0]+<0.3pc,0pc>\restore &&&&&&&&\\
&&&&\save[]+<0.3pc,0pc> \Drop{} \ar@{-}`d/4pt [1,-4][1,-4]\restore&&&&&&&&\\
&&&&&&&&&\ar@{-}`d/4pt [1,-9][1,-9]&&&&\\
&&&&&&&&&&&&&\\
&&&&&&&&&&&&&\\
&&&&&&&&&&&&&}}
\grow{\xymatrix@!0{
\\\\\\\\\\\\\\\\\\\\\\
\save\go+<0pt,0pt>\Drop{\txt{$=$}}\restore }}
\grow{\xymatrix@!0{\\\\\\\\\\
\save\go+<0pt,3pt>\Drop{M}\restore \ar@{-}[13,0]&&&\save\go+<0pt,3pt>\Drop{C}\restore \ar@{-}[2,0]&&& \save\go+<0pt,3pt>\Drop{H}\restore \ar@{-}[1,0]&&& \save\go+<0pt,3pt>\Drop{A}\restore\ar@{-}[4,0]\\
&&&&&&\ar@{-}`l/4pt [1,-1] [1,-1] \ar@{-}`r [1,1] [1,1]&&&\\
&&& \ar@{-}[1,1]+<-0.1pc,0.1pc> && \ar@{-}[2,-2]&&\ar@{-}[2,0]&&\\
&&&&&&&&&\\
&&&\ar@{-}[1,-1]&&\ar@{-}[-1,-1]+<0.1pc,-0.1pc>\ar@{-}[2,0] &&\ar@{-}[1,1]+<-0.125pc,0.0pc> \ar@{-}[1,1]+<0.0pc,0.125pc>&&\ar@{-}[2,-2]\\
&&\ar@{-}[1,0]&&&&&&&\\
&&\ar@{-}`l/4pt [1,-1] [1,-1] \ar@{-}`r [1,1] [1,1] &&&\ar@{-}[1,1]+<-0.125pc,0.0pc> \ar@{-}[1,1]+<0.0pc,0.125pc>&&\ar@{-}[2,-2]&&\ar@{-}[-1,-1]+<0.125pc,0.0pc>\ar@{-}[-1,-1]+<0.0pc,-0.125pc>\ar@{-}[2,0]\\
&\ar@{-}[0,0]+<0pc,-0.5pt>&&\ar@{-}[0,0]+<0pc,-0.5pt>&&&&&&\\
&*+[o]+<0.34pc>[F]{\ov{\gamma}}\ar@{-}[1,0]+<0pc,0pc>&&*+[o]+<0.37pc>[F]{\gamma} \ar@{-}[1,0]+<0pc,0pc>&&\ar@{-}[5,0]&&\ar@{-}[5,0]\ar@{-}[-1,-1]+<0.125pc,0.0pc>\ar@{-}[-1,-1]+<0.0pc,-0.125pc>&&\ar@{-}[5,0]\\
&\ar@{-}`d/4pt [1,1] `[0,2] [0,2]&&&&&&&&\\
&&\ar@{-}[1,0]&&&&&&&\\
&&\ar@{-}`d/4pt [1,-2][1,-2]&&&&&&&\\
&&&&&&&&&\\
&&&&&&&&&\\
}}
\grow{\xymatrix@!0{
\\\\\\\\\\\\\\\\\\\\\\
\save\go+<0pt,0pt>\Drop{\txt{$=$}}\restore }}
\grow{\xymatrix@!0{\\\\\\\\\\\\\\\\\\
\save\go+<0pt,3pt>\Drop{M}\restore \ar@{-}[5,0]&&\save\go+<0pt,3pt>\Drop{C}\restore \ar@{-}[3,0]&& \save\go+<0pt,3pt>\Drop{H}\restore \ar@{-}[1,0]&& \save\go+<0pt,3pt>\Drop{A}\restore\ar@{-}[1,0]\\
&&&& \ar@{-}[1,1]+<-0.125pc,0.0pc> \ar@{-}[1,1]+<0.0pc,0.125pc>&&\ar@{-}[2,-2]\\
&&&&&&\\
&& \ar@{-}[1,1]+<-0.125pc,0.0pc> \ar@{-}[1,1]+<0.0pc,0.125pc>&&\ar@{-}[2,-2] &&\ar@{-}[-1,-1]+<0.125pc,0.0pc>\ar@{-}[-1,-1]+<0.0pc,-0.125pc>\ar@{-}[2,0]\\
&&&&&&\\
&&&&\ar@{-}[-1,-1]+<0.125pc,0.0pc>\ar@{-}[-1,-1]+<0.0pc,-0.125pc>&&}}
\grow{\xymatrix@!0{
\\\\\\\\\\\\\\\\\\\\\\
\save\go+<0pt,0pt>\Drop{\txt{,}}\restore }}
\end{align*}
we have
$$
%
%
%
\grow{\xymatrix@!0{
\\\\\\\\\\\\\\\\\\\\\\\\\\\\
\save\go+<0pt,0pt>\Drop{\txt{$\ov{\nu}_0=$}}\restore }}
\grow{\xymatrix@!0{
&&\save\go+<0pt,3pt>\Drop{M}\restore \ar@{-}[4,0] &&\save\go+<0pt,3pt>\Drop{A}\restore \ar@{-}[2,0]&&\save\go+<0pt,3pt>\Drop{D}\restore \ar@{-}[2,2]&&\ar@{-}[1,-1]+<0.125pc,0.0pc> \ar@{-}[1,-1]+<0.0pc,0.125pc> \save\go+<0pt,3pt>\Drop{C}\restore&&\save\go+<0pt,3pt>\Drop{H}\restore \ar@{-}[2,0]\\
&&&&&&&&&&&&\\
&&&& \ar@{-}[2,2]&&\ar@{-}[1,-1]+<0.125pc,0.0pc> \ar@{-}[1,-1]+<0.0pc,0.125pc>\ar@{-}[-1,1]+<-0.125pc,0.0pc> \ar@{-}[-1,1]+<0.0pc,-0.125pc>&& \ar@{-}[2,2]&&\ar@{-}[1,-1]+<0.125pc,0.0pc> \ar@{-}[1,-1]+<0.0pc,0.125pc>&&\\
&&&&&&&&&&&&\\
&&\ar@{-}[2,-2]&&\ar@{-}[2,-2]\ar@{-}[-1,1]+<-0.125pc,0.0pc> \ar@{-}[-1,1]+<0.0pc,-0.125pc>&&\ar@{-}[2,2]&&\ar@{-}[1,-1]+<0.125pc,0.0pc> \ar@{-}[1,-1]+<0.0pc,0.125pc>\ar@{-}[-1,1]+<-0.125pc,0.0pc> \ar@{-}[-1,1]+<0.0pc,-0.125pc>&&\ar@{-}[3,3]&&\\
&&&&&&&&&&&&\\
\ar@{-}[20,0]&&\ar@{-}[3,0]&&&&\ar@{-}[-1,1]+<-0.125pc,0.0pc> \ar@{-}[-1,1]+<0.0pc,-0.125pc> \ar@{-}[3,0]&&\ar@{-}[2,2]&&&&\\
&&&&&&&&&&&&&\ar@{-}[15,0]\\
&&&&&&&&&&\ar@{-}[4,0]&&&\\
&&\ar@{-}`l/4pt [1,-1] [1,-1] \ar@{-}`r [1,1] [1,1]&&&&\ar@{-}`l/4pt [1,-1] [1,-1] \ar@{-}`r [1,1] [1,1]&&&&&&&\\
&\ar@{-}[2,0]&&\ar@{-}[1,1]+<-0.1pc,0.1pc> && \ar@{-}[2,-2]&&\ar@{-}[1,0]&&&&&&\\
&&&&&&&\ar@{-}[1,1]&&&&&&\\
&\ar@{-}[1,1]+<-0.1pc,0.1pc> && \ar@{-}[2,-2]&& \ar@{-}[-1,-1]+<0.1pc,-0.1pc> \ar@{-}[2,0]+<0pc,4pt>&&& \ar@{-}[2,0]+<0pc,4pt>&&\ar@{-}[3,0]&&&\\
&&&&&&&&&&&&&\\
&\ar@{-}[1,0]+<0pc,-0.5pt>&&\ar@{-}[1,0]+<0pc,-0.5pt>\ar@{-}[-1,-1]+<0.1pc,-0.1pc>& &*+[o]+<0.36pc>[F]{\boldsymbol{\gamma}}\ar@{-}[5,0]+<0pc,-0.5pt>&&& *+[o]+<0.37pc>[F]{\gamma}\ar@{-}[1,0]+<0pc,0pc>&&&&&\\
&&&&&&&&\ar@{-}`d/4pt [1,1] `[0,2] [0,2]&&&&&\\
&*+[o]+<0.35pc>[F]{\ov{\gamma}}\ar@{-}[1,0]+<0pc,0pc>&&*+[o]+<0.37pc>[F]{\ov{\mathbf{u}}}\ar@{-}[1,0]+<0pc,0pc>&&&&&&\ar@{-}[2,0]&&&&\\
&\ar@{-}`d/4pt [1,1] `[0,2] [0,2]&&&&\ar@{-}`r/4pt [1,2] [1,2]&&&&\ar@{-}`r/4pt [1,2][1,2]&&&&\\
&&\ar@{-}[1,0]&&&&&\ar@{-}[1,1]+<-0.125pc,0.0pc>\ar@{-}[1,1]+<0.0pc,0.125pc>&&\ar@{-}[2,-2]&&\ar@{-}[4,0]\\
&&\ar@{-}`d/4pt [1,-2][1,-2]&&&&&&&&&&&\\
&&&&&*+[o]+<0.37pc>[F]{\mu}\ar@{-}[1,0]+<0pc,0pc>&&\ar@{-}[1,0]&&\ar@{-}[-1,-1]+<0.125pc,0.0pc>\ar@{-}[-1,-1]+<0.0pc,-0.125pc>\ar@{-}[3,0]&&&&\\
&&&&&\ar@{-}`d/4pt [1,1] `[0,2] [0,2]&&&&&&&&\\
&&&&&&\ar@{-}[1,0]&&&&&\ar@{-}[1,1]+<-0.125pc,0.0pc> \ar@{-}[1,1]+<0.0pc,0.125pc>&&\ar@{-}[2,-2]\\
&&&&&&\ar@{-}`d/4pt [1,-6][1,-6]&&&\ar@{-}[1,0]&&&&\\
&&&&&&&&&\ar@{-}[1,1]+<-0.125pc,0.0pc> \ar@{-}[1,1]+<0.0pc,0.125pc>&&\ar@{-}[2,-2]&&\ar@{-}[-1,-1]+<0.125pc,0.0pc>\ar@{-}[-1,-1]+<0.0pc,-0.125pc>\ar@{-}[3,0]\\
\ar@{-}[2,0]&&&&&&&&&&&\\
&&&&&&&&&\ar@{-}[1,0]&&\ar@{-}[-1,-1]+<0.125pc,0.0pc>\ar@{-}[-1,-1]+<0.0pc,-0.125pc>\ar@{-}[1,0]&&\\
&&&&&&&&&&&&&}}
\grow{\xymatrix@!0{
\\\\\\\\\\\\\\\\\\\\\\\\\\\\
\save\go+<0pt,0pt>\Drop{\txt{$=$}}\restore }}
\grow{\xymatrix@!0{\\\\\\\\\\\\\\\\
\save\go+<0pt,3pt>\Drop{M}\restore \ar@{-}[13,0] &&\save\go+<0pt,3pt>\Drop{A}\restore \ar@{-}[2,0]&&\save\go+<0pt,3pt>\Drop{D}\restore \ar@{-}[2,2]&&\ar@{-}[1,-1]+<0.125pc,0.0pc> \ar@{-}[1,-1]+<0.0pc,0.125pc> \save\go+<0pt,3pt>\Drop{C}\restore&& \save\go+<0pt,3pt>\Drop{H}\restore \ar@{-}[2,0]\\
&&&&&&&&\\
&& \ar@{-}[2,2]&&\ar@{-}[1,-1]+<0.125pc,0.0pc> \ar@{-}[1,-1]+<0.0pc,0.125pc>\ar@{-}[-1,1]+<-0.125pc,0.0pc> \ar@{-}[-1,1]+<0.0pc,-0.125pc>&& \ar@{-}[2,2]&&\ar@{-}[1,-1]+<0.125pc,0.0pc> \ar@{-}[1,-1]+<0.0pc,0.125pc>\\
&&&&&&&&\\
&&\ar@{-}[5,0]\ar@{-}[-1,1]+<-0.125pc,0.0pc> \ar@{-}[-1,1]+<0.0pc,-0.125pc>&&\ar@{-}[2,2]&& \ar@{-}[1,-1]+<0.125pc,0.0pc> \ar@{-}[1,-1]+<0.0pc,0.125pc> \ar@{-}[-1,1]+<-0.125pc,0.0pc> \ar@{-}[-1,1]+<0.0pc,-0.125pc>&&\ar@{-}[5,0]\\
&&&&&&&&\\
&&&&\ar@{-}[1,0]\ar@{-}[-1,1]+<-0.125pc,0.0pc> \ar@{-}[-1,1]+<0.0pc,-0.125pc> &&\ar@{-}[1,0]&&\\
&&&&\ar@{-}[1,1]+<-0.125pc,0.0pc> \ar@{-}[1,1]+<0.0pc,0.125pc>&&\ar@{-}[2,-2]&&\\
&&&&&&&&\\
&&\ar@{-}[1,1]+<-0.125pc,0.0pc> \ar@{-}[1,1]+<0.0pc,0.125pc>&&\ar@{-}[2,-2]&& \ar@{-}[-1,-1]+<0.125pc,0.0pc>\ar@{-}[-1,-1]+<0.0pc,-0.125pc>\ar@{-}[1,1]+<-0.125pc,0.0pc> \ar@{-}[1,1]+<0.0pc,0.125pc>&&\ar@{-}[2,-2]\\
&&&&&&&&\\
&&\ar@{-}`d/4pt [1,-2][1,-2]&&\ar@{-}[-1,-1]+<0.125pc,0.0pc>\ar@{-}[-1,-1]+<0.0pc,-0.125pc> \ar@{-}[1,1]+<-0.125pc,0.0pc> \ar@{-}[1,1]+<0.0pc,0.125pc>&&\ar@{-}[2,-2] &&\ar@{-}[-1,-1]+<0.125pc,0.0pc>\ar@{-}[-1,-1]+<0.0pc,-0.125pc>\ar@{-}[2,0]\\
&&&&&&&&\\
&&&&&&\ar@{-}[-1,-1]+<0.125pc,0.0pc>\ar@{-}[-1,-1]+<0.0pc,-0.125pc>&&
}}
\grow{\xymatrix@!0{
\\\\\\\\\\\\\\\\\\\\\\\\\\\\
\save\go+<0pt,0pt>\Drop{\txt{$=$}}\restore }}
\grow{\xymatrix@!0{\\\\\\\\\\\\\\\\\\\\\\\\\\
\save\go+<0pt,3pt>\Drop{M}\restore \ar@{-}[2,0] &&\save\go+<0pt,3pt>\Drop{A}\restore \ar@{-}`d/4pt [1,-2][1,-2]&&\save\go+<0pt,3pt>\Drop{D}\restore \ar@{-}[2,0]&&\ar@{-}[2,0] \save\go+<0pt,3pt>\Drop{C}\restore&& \save\go+<0pt,3pt>\Drop{H}\restore \ar@{-}[2,0]\\
&&&&&&&&\\
&&&&&&&&
}}
\grow{\xymatrix@!0{
\\\\\\\\\\\\\\\\\\\\\\\\\\\\
\save\go+<0pt,0pt>\Drop{\txt{.}}\restore }}
$$
Now, we compute $\ov{\nu}_i$ for $0< i < r$. Let $D_1:=A^{\ot_k^{i-1}}$, $D_2:=A^{\ot_k^{r-i-1}}$, $C:=H^{\ot_c^s}$, $\boldsymbol{\gamma}:= \gamma^{\ot_k^s}$, $\mu:=\mu_s$, $\mathbf{u} := \mu\xcirc \boldsymbol{\gamma}$ and $\ov{\mathbf{u}}:=\mu\xcirc \ov{\gamma}^{\ot_k^s}\xcirc \gc_s$.
%
%
By Lemma~\ref{inversa de convolucion}
$$
%
%
\grow{\xymatrix@!0{
\\\\\\\\\\\\\\\\\\\\\\
\save\go+<0pt,0pt>\Drop{\txt{$\ov{\nu}_i=$}}\restore }}
\grow{\xymatrix@!0{
\save\go+<0pt,3pt>\Drop{M}\restore \ar@{-}[23,0] &&&&&\save\go+<0pt,3pt>\Drop{D_1}\restore \ar@{-}[6,0]&&\save\go+<0pt,3pt>\Drop{A}\restore \ar@{-}[4,0]&&\save\go+<0pt,3pt>\Drop{A}\restore \ar@{-}[2,0] && \save\go+<0pt,3pt>\Drop{D_2}\restore \ar@{-}[2,2]&&\ar@{-}[1,-1]+<0.125pc,0.0pc> \ar@{-}[1,-1]+<0.0pc,0.125pc> \save\go+<0pt,3pt>\Drop{C}\restore\\
&&&&&&&&&&&&&\\
&&&&&&&&&\ar@{-}[2,2]&&\ar@{-}[1,-1]+<0.125pc,0.0pc> \ar@{-}[1,-1]+<0.0pc,0.125pc>\ar@{-}[-1,1]+<-0.125pc,0.0pc> \ar@{-}[-1,1]+<0.0pc,-0.125pc>&&\ar@{-}[6,0]\\
&&&&&&&&&&&&&\\
&&&&&&&\ar@{-}[2,2]&&\ar@{-}[1,-1]+<0.125pc,0.0pc> \ar@{-}[1,-1]+<0.0pc,0.125pc> \ar@{-}[-1,1]+<-0.125pc,0.0pc> \ar@{-}[-1,1]+<0.0pc,-0.125pc>&&\ar@{-}[3,0]&&\\
&&&&&&&&&&&&&\\
&&&&&\ar@{-}[2,2]&&\ar@{-}[1,-1]+<0.125pc,0.0pc> \ar@{-}[1,-1]+<0.0pc,0.125pc> \ar@{-}[-1,1]+<-0.125pc,0.0pc> \ar@{-}[-1,1]+<0.0pc,-0.125pc>&&\ar@{-}[1,0]&&&&\\
&&&&&&&&&\ar@{-}`d/4pt [1,1] `[0,2] [0,2]&&&&\\
&&&&&\ar@{-}[2,-2]\ar@{-}[-1,1]+<-0.125pc,0.0pc> \ar@{-}[-1,1]+<0.0pc,-0.125pc>&&\ar@{-}[1,1] &&&\ar@{-}[11,0]&&&\ar@{-}[1,-1]\\
&&&&&&&&\ar@{-}[8,0]&&&&\ar@{-}[12,0]\\
&&&\ar@{-}[1,0]&&&&&&&&&&\\
&&&\ar@{-}`l/4pt [1,-1] [1,-1] \ar@{-}`r [1,1] [1,1]&&&&&&&&\\
&&\ar@{-}[1,0]+<0pc,-0.5pt>&&\ar@{-}[1,0]+<0pc,-0.5pt>&&&&&&&\\
&&&&&&&&&&&\\
&&*+[o]+<0.37pc>[F]{\ov{\mathbf{u}}}\ar@{-}[1,0]+<0pc,0pc>&&*+[o]+<0.37pc>[F]{\boldsymbol{\gamma}}\ar@{-}[1,0]+<0pc,0pc>&&&&&&&\\
&&\ar@{-}`d/4pt [1,-2][1,-2]&&\ar@{-}[2,0]+<0pc,-0.5pt>&&&&&&&\\
&&&&\ar@{-}`r/4pt [1,2] [1,2]&&&&&&&\\
&&&&&&\ar@{-}[1,1]+<-0.125pc,0.0pc> \ar@{-}[1,1]+<0.0pc,0.125pc>&&\ar@{-}[2,-2]&&&&&\\
&&&&*+[o]+<0.37pc>[F]{\mu}\ar@{-}[1,0]+<0pc,0pc>&&&&&&&\\
&&&&\ar@{-}`d/4pt [1,-4][1,-4]&&\ar@{-}[5,0]&& \ar@{-}[-1,-1]+<0.125pc,0.0pc> \ar@{-}[-1,-1]+<0.0pc,-0.125pc> \ar@{-}[1,1]+<-0.125pc,0.0pc> \ar@{-}[1,1]+<0.0pc,0.125pc>&&\ar@{-}[2,-2]&&&\\
&&&&&&&&&&&&&\\
\ar@{-}[3,0]&&&&&&&&\ar@{-}[3,0] &&\ar@{-}[-1,-1]+<0.125pc,0.0pc>\ar@{-}[-1,-1]+<0.0pc,-0.125pc> \ar@{-}[1,1]+<-0.125pc,0.0pc> \ar@{-}[1,1]+<0.0pc,0.125pc>&&\ar@{-}[2,-2]&\\
&&&&&&&&&&&&&\\
&&&&&&&&&&\ar@{-}[1,0]&&\ar@{-}[-1,-1]+<0.125pc,0.0pc>\ar@{-}[-1,-1]+<0.0pc,-0.125pc>\ar@{-}[1,0]&\\
&&&&&&&&&&&&&
}}
\grow{\xymatrix@!0{
\\\\\\\\\\\\\\\\\\\\\\
\save\go+<0pt,0pt>\Drop{\txt{$=$}}\restore }}
\grow{\xymatrix@!0{\\
\save\go+<0pt,3pt>\Drop{M}\restore \ar@{-}[22,0] &&&\save\go+<0pt,3pt>\Drop{D_1}\restore \ar@{-}[6,0]&&\save\go+<0pt,3pt>\Drop{A}\restore \ar@{-}[4,0]&&\save\go+<0pt,3pt>\Drop{A}\restore \ar@{-}[2,0] && \save\go+<0pt,3pt>\Drop{D_2}\restore \ar@{-}[2,2]&&\ar@{-}[1,-1]+<0.125pc,0.0pc> \ar@{-}[1,-1]+<0.0pc,0.125pc> \save\go+<0pt,3pt>\Drop{C}\restore\\
&&&&&&&&&&&\\
&&&&&&&\ar@{-}[2,2]&&\ar@{-}[1,-1]+<0.125pc,0.0pc> \ar@{-}[1,-1]+<0.0pc,0.125pc>\ar@{-}[-1,1]+<-0.125pc,0.0pc> \ar@{-}[-1,1]+<0.0pc,-0.125pc>&&\ar@{-}[2,2]\\
&&&&&&&&&&&\\
&&&&&\ar@{-}[2,2]&&\ar@{-}[1,-1]+<0.125pc,0.0pc> \ar@{-}[1,-1]+<0.0pc,0.125pc> \ar@{-}[-1,1]+<-0.125pc,0.0pc> \ar@{-}[-1,1]+<0.0pc,-0.125pc>&&\ar@{-}[2,2]&&&&\ar@{-}[16,0]\\
&&&&&&&&&&&&&\\
&&&\ar@{-}[2,2]&&\ar@{-}[1,-1]+<0.125pc,0.0pc> \ar@{-}[1,-1]+<0.0pc,0.125pc> \ar@{-}[-1,1]+<-0.125pc,0.0pc> \ar@{-}[-1,1]+<0.0pc,-0.125pc>&&\ar@{-}[2,2]&&&&\ar@{-}[12,0]&&\\
&&&&&&&&&&&&&\\
&&&\ar@{-}[1,0]\ar@{-}[-1,1]+<-0.125pc,0.0pc> \ar@{-}[-1,1]+<0.0pc,-0.125pc>&&\ar@{-}[2,2]&&&& \ar@{-}[8,0]&&&&\\
&&&\ar@{-}`l/4pt [1,-1] [1,-1] \ar@{-}`r [1,1] [1,1]&&&&&&&&&&\\
&&\ar@{-}[1,0]+<0pc,-0.5pt>&&\ar@{-}[3,0]&&&\ar@{-}[4,0]&&&&&&\\
&&&&&&&&&&&&&\\
&&*+[o]+<0.37pc>[F]{\ov{\mathbf{u}}}\ar@{-}[1,0]+<0pc,0pc>&&&&&&&&&&&\\
&&\ar@{-}`d/4pt [1,-2][1,-2]&&\ar@{-}`l/4pt [1,-1] [1,-1] \ar@{-}`r [1,1] [1,1]&&&&&&&&&\\
&&&\ar@{-}[1,0]+<0pc,-0.5pt>&&\ar@{-}[1,1]+<-0.125pc,0.0pc> \ar@{-}[1,1]+<0.0pc,0.125pc>&&\ar@{-}[2,-2]&&&&&&\\
&&&&&&&&&&&&&\\
&&&*+[o]+<0.4pc>[F]{\mathbf{u}}\ar@{-}[1,0]+<0pc,0pc>&&\ar@{-}[6,0]&&\ar@{-}[-1,-1]+<0.125pc,0.0pc>\ar@{-}[-1,-1]+<0.0pc,-0.125pc> \ar@{-}[1,1]+<-0.125pc,0.0pc> \ar@{-}[1,1]+<0.0pc,0.125pc>&&\ar@{-}[2,-2]&&&&\\
&&&\ar@{-}`d/4pt [1,-3][1,-3]&&&&&&&&&&\\
&&&&&&&\ar@{-}[4,0]&&\ar@{-}[-1,-1]+<0.125pc,0.0pc>\ar@{-}[-1,-1]+<0.0pc,-0.125pc>
\ar@{-}[1,1]+<-0.125pc,0.0pc> \ar@{-}[1,1]+<0.0pc,0.125pc>&&\ar@{-}[2,-2]&&\\
&&&&&&&&&&&&&\\
&&&&&&&&&\ar@{-}[2,0]&&\ar@{-}[-1,-1]+<0.125pc,0.0pc>\ar@{-}[-1,-1]+<0.0pc,-0.125pc>\ar@{-}[1,1]+<-0.125pc,0.0pc> \ar@{-}[1,1]+<0.0pc,0.125pc>&&\ar@{-}[2,-2]\\
&&&&&&&&&&&&&\\
&&&&&&&&&&&&&\ar@{-}[-1,-1]+<0.125pc,0.0pc>\ar@{-}[-1,-1]+<0.0pc,-0.125pc>
}}
\grow{\xymatrix@!0{
\\\\\\\\\\\\\\\\\\\\\\
\save\go+<0pt,0pt>\Drop{\txt{$=$}}\restore }}
\grow{\xymatrix@!0{\\
\save\go+<0pt,3pt>\Drop{M}\restore \ar@{-}[21,0] &&&&\save\go+<0pt,3pt>\Drop{D_1}\restore \ar@{-}[6,0]&&\save\go+<0pt,3pt>\Drop{A}\restore \ar@{-}[4,0]&&\save\go+<0pt,3pt>\Drop{A}\restore \ar@{-}[2,0] && \save\go+<0pt,3pt>\Drop{D_2}\restore \ar@{-}[2,2]&&\ar@{-}[1,-1]+<0.125pc,0.0pc> \ar@{-}[1,-1]+<0.0pc,0.125pc> \save\go+<0pt,3pt>\Drop{C}\restore\\
&&&&&&&&&&&&\\
&&&&&&&&\ar@{-}[2,2]&&\ar@{-}[1,-1]+<0.125pc,0.0pc> \ar@{-}[1,-1]+<0.0pc,0.125pc>\ar@{-}[-1,1]+<-0.125pc,0.0pc> \ar@{-}[-1,1]+<0.0pc,-0.125pc>&&\ar@{-}[1,1]\\
&&&&&&&&&&&&&\ar@{-}[15,0]\\
&&&&&&\ar@{-}[2,2]&&\ar@{-}[1,-1]+<0.125pc,0.0pc> \ar@{-}[1,-1]+<0.0pc,0.125pc> \ar@{-}[-1,1]+<-0.125pc,0.0pc> \ar@{-}[-1,1]+<0.0pc,-0.125pc>&&\ar@{-}[1,1]&&&\\
&&&&&&&&&&&\ar@{-}[11,0]&&\\
&&&&\ar@{-}[2,2]&&\ar@{-}[1,-1]+<0.125pc,0.0pc> \ar@{-}[1,-1]+<0.0pc,0.125pc> \ar@{-}[-1,1]+<-0.125pc,0.0pc> \ar@{-}[-1,1]+<0.0pc,-0.125pc>&&\ar@{-}[1,1]&&&&&\\
&&&&&&&&&\ar@{-}[7,0]&&&\\
&&&&\ar@{-}[1,0]\ar@{-}[-1,1]+<-0.125pc,0.0pc> \ar@{-}[-1,1]+<0.0pc,-0.125pc>&& \ar@{-}[1,1]&&&&&&&\\
&&&&\ar@{-}`l/4pt [1,-1] [1,-1] \ar@{-}`r [1,1] [1,1]&&&\ar@{-}[3,0]&&&&&&\\
&&&\ar@{-}[1,0]&&\ar@{-}[2,0]&&&&&&&&\\
&&&\ar@{-}`l/4pt [1,-1] [1,-1] \ar@{-}`r [1,1] [1,1] &&&&&&&&&&\\
&&\ar@{-}[0,0]+<0pc,-0.5pt>&&\ar@{-}[0,0]+<0pc,-0.5pt>&\ar@{-}[1,1]+<-0.125pc,0.0pc> \ar@{-}[1,1]+<0.0pc,0.125pc>&&\ar@{-}[2,-2]&&&&&\\
&&*+[o]+<0.37pc>[F]{\ov{\mathbf{u}}}\ar@{-}[1,0]+<0pc,0pc>&&*+[o]+<0.40pc>[F]{\mathbf{u}}\ar@{-}[1,0]+<0pc,0pc>&&&&&&&&&\\
&&\ar@{-}`d/4pt [1,1] `[0,2] [0,2]&&&\ar@{-}[5,0]&& \ar@{-}[-1,-1]+<0.125pc,0.0pc>\ar@{-}[-1,-1]+<0.0pc,-0.125pc>\ar@{-}[1,1]+<-0.125pc,0.0pc> \ar@{-}[1,1]+<0.0pc,0.125pc>&&\ar@{-}[2,-2]&&&&\\
&&&\ar@{-}[1,0]&&&&&&&&&&\\
&&&\ar@{-}`d/4pt [1,-3][1,-3]&&&&\ar@{-}[3,0]&& \ar@{-}[-1,-1]+<0.125pc,0.0pc>\ar@{-}[-1,-1]+<0.0pc,-0.125pc> \ar@{-}[1,1]+<-0.125pc,0.0pc> \ar@{-}[1,1]+<0.0pc,0.125pc>&&\ar@{-}[2,-2]&&\\
&&&&&&&&&&&&&\\
&&&&&&&&&\ar@{-}[1,0]&&\ar@{-}[-1,-1]+<0.125pc,0.0pc>\ar@{-}[-1,-1]+<0.0pc,-0.125pc> \ar@{-}[1,1]+<-0.125pc,0.0pc> \ar@{-}[1,1]+<0.0pc,0.125pc>&&\ar@{-}[2,-2]\\
&&&&&\ar@{-}[2,0]&&\ar@{-}`d/4pt [1,1] `[0,2] [0,2]&&&&&&\\
&&&&&&&&\ar@{-}[1,0]&&&\ar@{-}[1,0]&&\ar@{-}[1,0]\ar@{-}[-1,-1]+<0.125pc,0.0pc>\ar@{-}[-1,-1]+<0.0pc,-0.125pc>\\
&&&&&&&&&&&&&
}}
\grow{\xymatrix@!0{
\\\\\\\\\\\\\\\\\\\\\\
\save\go+<0pt,0pt>\Drop{\txt{$=$}}\restore }}
\grow{\xymatrix@!0{\\\\\\\\\\\\\\\\\\\\
\save\go+<0pt,3pt>\Drop{M}\restore \ar@{-}[2,0]&&\save\go+<0pt,3pt>\Drop{D_1}\restore \ar@{-}[2,0]&&\save\go+<0pt,3pt>\Drop{A}\restore \ar@{-}`d/4pt [1,1] `[0,2] [0,2]&&\save\go+<0pt,3pt>\Drop{A}\restore  && \save\go+<0pt,3pt>\Drop{D_2}\restore \ar@{-}[2,0]&& \save\go+<0pt,3pt>\Drop{C}\restore\ar@{-}[2,0]\\
&&&&&\ar@{-}[1,0]&&&&&\\
&&&&&&&&&&
}}
\grow{\xymatrix@!0{
\\\\\\\\\\\\\\\\\\\\\\
\save\go+<0pt,0pt>\Drop{\txt{.}}\restore }}
$$
It remains to compute $\ov{\nu}_r$. Let $D:=A^{\ot_k^{r-1}}$, $C:=H^{\ot_c^s}$, $\boldsymbol{\gamma} := \gamma^{\ot_k^s}$, $\mu:=\mu_s$, $\mathbf{u} := \mu\xcirc \boldsymbol{\gamma}$ and $\ov{\mathbf{u}}:=\mu\xcirc \ov{\gamma}^{\ot_k^s}\xcirc \gc_s$. By Lemma~\ref{inversa de convolucion},
$$
%
%
\grow{\xymatrix@!0{
\\\\\\\\\\\\\\\\\\\\
\save\go+<0pt,0pt>\Drop{\txt{$\ov{\nu}_r=$}}\restore }}
\grow{\xymatrix@!0{\\\\
\save\go+<0pt,3pt>\Drop{M}\restore \ar@{-}[11,0]&&&&\save\go+<0pt,3pt>\Drop{D}\restore \ar@{-}[2,0]&&\save\go+<0pt,3pt>\Drop{A}\restore \ar@{-}[2,2]&&\ar@{-}[1,-1]+<0.125pc,0.0pc> \ar@{-}[1,-1]+<0.0pc,0.125pc> \save\go+<0pt,3pt>\Drop{C}\restore\\
&&&&&&&&\\
&&&& \ar@{-}[2,2]&&\ar@{-}[1,-1]+<0.125pc,0.0pc> \ar@{-}[1,-1]+<0.0pc,0.125pc>\ar@{-}[-1,1]+<-0.125pc,0.0pc> \ar@{-}[-1,1]+<0.0pc,-0.125pc>&&\ar@{-}[6,0]\\
&&&&&&&&\\
&&&&\ar@{-}[1,-1]\ar@{-}[-1,1]+<-0.125pc,0.0pc> \ar@{-}[-1,1]+<0.0pc,-0.125pc>&&\ar@{-}[4,0]&&\\
&&&\ar@{-}[1,0]&&&&&\\
&&&\ar@{-}`l/4pt [1,-1] [1,-1] \ar@{-}`r [1,1] [1,1]&&&&&\\
&&\ar@{-}[0,0]+<0pc,-0.5pt>&&\ar@{-}[0,0]+<0pc,-0.5pt>&&&&\\
&&*+[o]+<0.37pc>[F]{\ov{\mathbf{u}}}\ar@{-}[1,0]+<0pc,0pc>&&*+[o]+<0.37pc>[F]{\boldsymbol{\gamma}}\ar@{-}[2,0]+<0pc,0pc>&& \ar@{-}[2,2]&&\ar@{-}[2,-2]\\
&&\ar@{-}`d/4pt [1,-2][1,-2]&&&&&&\\
&&&&\ar@{-}[2,2]&&\ar@{-}[2,-2]&&\ar@{-}[2,2]\\
\ar@{-}[1,2]&&&&&&&&\\
&&\ar@{-}[2,2]&&\ar@{-}[2,-2]&&\ar@{-}[2,0]+<0pc,-0.5pt>&&&&\ar@{-}[2,0]\\
&&&&&&\ar@{-}`r/4pt [1,2] [1,2]&&&&\\
&&\ar@{-}[1,0]&&\ar@{-}[5,0]&&&&\ar@{-}[2,2]&&\ar@{-}[2,-2]\\
&&\ar@{-}`d/4pt [1,2][1,2]&&&&*+[o]+<0.37pc>[F]{\mu}\ar@{-}[1,0]+<0pc,0pc>&&&&\\
&&&&&&\ar@{-}[1,0]&&\ar@{-}[3,0]&&\ar@{-}[3,0]\\
&&&&&&\ar@{-}`d/4pt [1,-2][1,-2]&&&&\\
&&&&&&&&&&\\
&&&&&&&&&&
}}
\grow{\xymatrix@!0{
\\\\\\\\\\\\\\\\\\\\
\save\go+<0pt,0pt>\Drop{\txt{$=$}}\restore }}
\grow{\xymatrix@!0{\\\\\\
\save\go+<0pt,3pt>\Drop{M}\restore \ar@{-}[11,0]&&&&\save\go+<0pt,3pt>\Drop{D}\restore \ar@{-}[2,0]&&\save\go+<0pt,3pt>\Drop{A}\restore \ar@{-}[2,2]&&\ar@{-}[1,-1]+<0.125pc,0.0pc> \ar@{-}[1,-1]+<0.0pc,0.125pc> \save\go+<0pt,3pt>\Drop{C}\restore\\
&&&&&&&&\\
&&&& \ar@{-}[2,2]&&\ar@{-}[1,-1]+<0.125pc,0.0pc> \ar@{-}[1,-1]+<0.0pc,0.125pc>\ar@{-}[-1,1]+<-0.125pc,0.0pc> \ar@{-}[-1,1]+<0.0pc,-0.125pc>&&\ar@{-}[3,0]\\
&&&&&&&&\\
&&&&\ar@{-}[1,-1]\ar@{-}[-1,1]+<-0.125pc,0.0pc> \ar@{-}[-1,1]+<0.0pc,-0.125pc>&&\ar@{-}[1,0]&&\\
&&&\ar@{-}[1,0]&&&\ar@{-}[2,2]&&\ar@{-}[2,-2]\\
&&&\ar@{-}`l/4pt [1,-1] [1,-1] \ar@{-}`r [1,1] [1,1]&&&&&\\
&&\ar@{-}[0,0]+<0pc,-0.5pt>&&\ar@{-}[2,2]&&\ar@{-}[2,-2]&&\ar@{-}[1,1]\\
&&*+[o]+<0.37pc>[F]{\ov{\mathbf{u}}}\ar@{-}[1,0]+<0pc,0pc>&&&&&&&\ar@{-}[3,0]\\
&&\ar@{-}`d/4pt [1,-2][1,-2]&&\ar@{-}[1,-1]&&\ar@{-}[1,0]&&&\\
&&&\ar@{-}[3,-3]&&&\ar@{-}`l/4pt [1,-1] [1,-1] \ar@{-}`r [1,1] [1,1]&&&\\
\ar@{-}[2,2]&&&&&\ar@{-}[0,0]+<0pc,-0.5pt>&&\ar@{-}[2,2]&&\ar@{-}[2,-2]\\
&&&&&*+[o]+<0.40pc>[F]{\mathbf{u}}\ar@{-}[1,0]+<0pc,0pc>&&&&\\
\ar@{-}[1,0]&&\ar@{-}[3,0]&&&\ar@{-}`d/4pt [1,-3][1,-3]&&\ar@{-}[3,0]&&\ar@{-}[3,0]\\
\ar@{-}`d/4pt [1,2][1,2]&&&&&&&&&\\
&&&&&&&&&\\
&&&&&&&&&
}}
\grow{\xymatrix@!0{
\\\\\\\\\\\\\\\\\\\\
\save\go+<0pt,0pt>\Drop{\txt{$=$}}\restore }}
\grow{\xymatrix@!0{
\save\go+<0pt,3pt>\Drop{M}\restore \ar@{-}[15,0]&&&&&\save\go+<0pt,3pt>\Drop{D}\restore \ar@{-}[2,0]&&\save\go+<0pt,3pt>\Drop{A}\restore \ar@{-}[2,2]&&\ar@{-}[1,-1]+<0.125pc,0.0pc> \ar@{-}[1,-1]+<0.0pc,0.125pc> \save\go+<0pt,3pt>\Drop{C}\restore\\
&&&&&&&&&\\
&&&&& \ar@{-}[2,2]&&\ar@{-}[1,-1]+<0.125pc,0.0pc> \ar@{-}[1,-1]+<0.0pc,0.125pc>\ar@{-}[-1,1]+<-0.125pc,0.0pc> \ar@{-}[-1,1]+<0.0pc,-0.125pc>&&\ar@{-}[11,0]\\
&&&&&&&&&\\
&&&&&\ar@{-}[1,-1]\ar@{-}[-1,1]+<-0.125pc,0.0pc> \ar@{-}[-1,1]+<0.0pc,-0.125pc>&&\ar@{-}[5,0]&&\\
&&&&\ar@{-}[1,0]&&&&&\\
&&&&\ar@{-}`l/4pt [1,-1] [1,-1] \ar@{-}`r [1,1] [1,1]&&&&&\\
&&&\ar@{-}[1,0]&&\ar@{-}[2,0]&&&&\\
&&&\ar@{-}`l/4pt [1,-1] [1,-1] \ar@{-}`r [1,1] [1,1]&&&&&&\\
&&\ar@{-}[0,0]+<0pc,-0.5pt>&&\ar@{-}[0,0]+<0pc,-0.5pt>&\ar@{-}[1,1]+<-0.125pc,0.0pc> \ar@{-}[1,1]+<0.0pc,0.125pc>&&\ar@{-}[2,-2]&\\
&&*+[o]+<0.37pc>[F]{\ov{\mathbf{u}}}\ar@{-}[1,0]+<0pc,0pc>&&*+[o]+<0.40pc>[F]{\mathbf{u}}\ar@{-}[1,0]+<0pc,0pc>&&&&\\
&&\ar@{-}`d/4pt [1,1] `[0,2] [0,2] &&&\ar@{-}[4,0]&&\ar@{-}[2,0] \ar@{-}[-1,-1]+<0.125pc,0.0pc>\ar@{-}[-1,-1]+<0.0pc,-0.125pc> &&\\
&&&\ar@{-}[1,0]&&&&&\\
&&&\ar@{-}`d/4pt [1,-3][1,-3]&&&&\ar@{-}[2,2]&&\ar@{-}[2,-2]\\
&&&&&&&&&\\
\ar@{-}[2,3]&&&&&\ar@{-}[2,2]&&\ar@{-}[2,-2]&&\ar@{-}[8,0]\\
&&&&&&&&&\\
&&&\ar@{-}[2,2]&&\ar@{-}[2,-2]&&\ar@{-}[6,0]&&\\
&&&&&&&&&\\
&&&\ar@{-}[1,0]&&\ar@{-}[4,0]&&&&\\
&&&\ar@{-}`d/4pt [1,2][1,2]&&&&&&\\
&&&&&&&&&\\
&&&&&&&&&\\
&&&&&&&&&
}}
\grow{\xymatrix@!0{
\\\\\\\\\\\\\\\\\\\\
\save\go+<0pt,0pt>\Drop{\txt{$=$}}\restore }}
\grow{\xymatrix@!0{\\\\
\save\go+<0pt,3pt>\Drop{M}\restore \ar@{-}[12,0]&&&\save\go+<0pt,3pt>\Drop{D}\restore \ar@{-}[2,0]&&\save\go+<0pt,3pt>\Drop{A}\restore \ar@{-}[2,2]&&\ar@{-}[1,-1]+<0.125pc,0.0pc> \ar@{-}[1,-1]+<0.0pc,0.125pc> \save\go+<0pt,3pt>\Drop{C}\restore\\
&&&&&&&\\
&&& \ar@{-}[2,2]&&\ar@{-}[1,-1]+<0.125pc,0.0pc> \ar@{-}[1,-1]+<0.0pc,0.125pc>\ar@{-}[-1,1]+<-0.125pc,0.0pc> \ar@{-}[-1,1]+<0.0pc,-0.125pc>&&\ar@{-}[7,0]\\
&&&&&&&\\
&&&\ar@{-}[1,-1]\ar@{-}[-1,1]+<-0.125pc,0.0pc> \ar@{-}[-1,1]+<0.0pc,-0.125pc>&&\ar@{-}[3,0]&&\\
&&\ar@{-}[1,0]&&&&&\\
&&\ar@{-}`l/4pt [1,-1] [1,-1] \ar@{-}`r [1,1] [1,1]&&&&&\\
&\ar@{-}[1,0]&&\ar@{-}[1,1]+<-0.125pc,0.0pc> \ar@{-}[1,1]+<0.0pc,0.125pc>&&\ar@{-}[2,-2]&&\\
&\ar@{-}[1,0]+<0pt,2.5pt>&&&&&&\\
&\save\go+<0pt,1.6pt>\Drop{\circ}\restore&&\ar@{-}[2,0]&& \ar@{-}[-1,-1]+<0.125pc,0.0pc> \ar@{-}[-1,-1]+<0.0pc,-0.125pc> \ar@{-}[2,2]&&\ar@{-}[2,-2]\\
&&&&&&&\\
&&&\ar@{-}[2,2]&&\ar@{-}[2,-2]&&\ar@{-}[7,0]\\
\ar@{-}[1,1]&&&&&&&\\
&\ar@{-}[2,2]&&\ar@{-}[2,-2]&&\ar@{-}[5,0]&&\\
&&&&&&&\\
&\ar@{-}[1,0]&&\ar@{-}[3,0]&&&&\\
&\ar@{-}`d/4pt [1,2][1,2] &&&&&&\\
&&&&&&&\\
&&&&&&&
}}
\grow{\xymatrix@!0{
\\\\\\\\\\\\\\\\\\\\
\save\go+<0pt,0pt>\Drop{\txt{$=$}}\restore }}
\grow{\xymatrix@!0{\\\\\\\\\\
\save\go+<0pt,3pt>\Drop{M}\restore \ar@{-}[7,0]&&\save\go+<0pt,3pt>\Drop{D}\restore \ar@{-}[5,0]&&\save\go+<0pt,3pt>\Drop{A}\restore \ar@{-}[2,2]&&\ar@{-}[1,-1]+<0.125pc,0.0pc> \ar@{-}[1,-1]+<0.0pc,0.125pc> \save\go+<0pt,3pt>\Drop{C}\restore\\
&&&&&&\\
&&&&\ar@{-}[1,0]\ar@{-}[-1,1]+<-0.125pc,0.0pc> \ar@{-}[-1,1]+<0.0pc,-0.125pc>&&\ar@{-}[1,0]\\
&&&&\ar@{-}[2,2]&&\ar@{-}[2,-2]\\
&&&&&&\\
&&\ar@{-}[2,2]&&\ar@{-}[2,-2]&&\ar@{-}[7,0]\\
&&&&&&\\
\ar@{-}[2,2]&&\ar@{-}[2,-2]&&\ar@{-}[5,0]&&\\
&&&&&&\\
\ar@{-}[1,0]&&\ar@{-}[3,0]&&&&\\
\ar@{-}`d/4pt [1,2][1,2]&&&&&&\\
&&&&&&\\
&&&&&&
}}
\grow{\xymatrix@!0{
\\\\\\\\\\\\\\\\\\\\
\save\go+<0pt,0pt>\Drop{\txt{.}}\restore }}
$$
We next compute $\ov{d}^1$. Let
$$
\wt{u}_i\colon M\ot_{\!A} E^{\ot_{\!A}^s}\ot_k A^{\ot_k^r}\to M\ot_{\!A} E^{\ot_{\!A}^{s-1}} \ot_k A^{\ot_k^r}\qquad\text{($0\le i\le s$)}
$$
be as above of Notation~\ref{not3.1}. We set
$$
\ov{u}_i := \wt{\theta} \xcirc \wt{u}_i\xcirc \wt{\vartheta}\qquad\text{for $0\le i\le s$}.
$$
By item~(2) of Theorem~\ref{formula para wh{d}_1} we know that $\wh{d}^1$ is induced by $\sum_{i=0}^s (-1)^i \wt{u}_i$. Hence, $\ov{d}^1$ is induced by $\sum_{i=0}^s (-1)^{r+i} \ov{u}_i$. So, in order to complete the computation of $\ov{d}^1$ we must calculate the $\ov{u}_i$'s. We begin with $\ov{u}_0$. Let $D:= A^{\ot_k^r}$, $C:= H^{\ot_c^{s-1}}$, $\mu:=\mu_{s-1}$, $\mathbf{u}:= \mu\xcirc \gamma^{\ot_k^{s-1}}$ and $\ov{\mathbf{u}}:=\mu\xcirc \ov{\gamma}^{\ot_k^{s-1}}\xcirc\gc_{s-1}$. Again by Lemma~\ref{inversa de convolucion},
$$
%
%
\grow{\xymatrix@!0{
\\\\\\\\\\\\\\\\\\
\save\go+<0pt,0pt>\Drop{\txt{$\ov{u}_0=$}}\restore }}
\grow{\xymatrix@!0{
\save\go+<0pt,3pt>\Drop{M}\restore \ar@{-}[19,0]&&&&&\save\go+<0pt,3pt>\Drop{D}\restore \ar@{-}[2,2]&&\ar@{-}[1,-1]+<0.125pc,0.0pc> \ar@{-}[1,-1]+<0.0pc,0.125pc> \save\go+<0pt,3pt>\Drop{H}\restore &&\ar@{-}[2,0]\save\go+<0pt,3pt>\Drop{C}\restore\\
&&&&&&&&&\\
&&&&&\ar@{-}[-1,1]+<-0.125pc,0.0pc> \ar@{-}[-1,1]+<0.0pc,-0.125pc>\ar@{-}[2,-2]&&\ar@{-}[3,3] &&\ar@{-}[1,-1]+<0.125pc,0.0pc> \ar@{-}[1,-1]+<0.0pc,0.125pc>\\
&&&&&&&&&\\
&&&\ar@{-}[1,0]&&&&\ar@{-}[1,0]\ar@{-}[-1,1]+<-0.125pc,0.0pc> \ar@{-}[-1,1]+<0.0pc,-0.125pc>&&&\\
&&& \ar@{-}`l/4pt [1,-1] [1,-1] \ar@{-}`r [1,1] [1,1]&&&& \ar@{-}`l/4pt [1,-1] [1,-1] \ar@{-}`r [1,1] [1,1]&&&\ar@{-}[2,2]\\
&&\ar@{-}[2,0]&&\ar@{-}[1,1]+<-0.1pc,0.1pc> && \ar@{-}[2,-2]&&\ar@{-}[2,0]&&\\
&&&&&&&&&&&&\ar@{-}[6,0]\\
&&\ar@{-}[1,1]+<-0.1pc,0.1pc> && \ar@{-}[2,-2]&&\ar@{-}[-1,-1]+<0.1pc,-0.1pc>\ar@{-}[1,0] &&\ar@{-}[1,0]&&&&\\
&&&&&&\ar@{-}[0,0]+<0pc,-0.5pt>&&\ar@{-}[0,0]+<0pc,-0.5pt>&&\\
&&\ar@{-}[1,0]+<0pc,-0.5pt>&&\ar@{-}[1,0]+<0pc,-0.5pt>\ar@{-}[-1,-1]+<0.1pc,-0.1pc>&&*+[o]+<0.40pc>[F]{\gamma}\ar@{-}[6,0]+<0pc,0pc>&& *+[o]+<0.40pc>[F]{\boldsymbol{\gamma}}\ar@{-}[3,0]+<0pc,0pc>&&\\
&&&&&&&&&&\\
&&*+[o]+<0.37pc>[F]{\ov{\mathbf{u}}}\ar@{-}[1,0]+<0pc,0pc>&&*+[o]+<0.37pc>[F]{\ov{\gamma}}\ar@{-}[1,0]+<0pc,0pc>&&&&
\ar@{-}`r/4pt [1,2] [1,2]&&&&\\
&&\ar@{-}`d/4pt [1,1] `[0,2] [0,2]&&&&&&\ar@{-}[0,0]+<0pc,-0.5pt>&& \ar@{-}[1,1]+<-0.125pc,0.0pc> \ar@{-}[1,1]+<0.0pc,0.125pc>&&\ar@{-}[2,-2]\\
&&&\ar@{-}[1,0]&&&&& *+[o]+<0.37pc>[F]{\mu}\ar@{-}[3,0]+<0pc,0pc>&&&&\\
&&&\ar@{-}`d/4pt [1,-3][1,-3]&&&&&&&\ar@{-}[4,0]&&\ar@{-}[-1,-1]+<0.125pc,0.0pc> \ar@{-}[-1,-1]+<0.0pc,-0.125pc>\ar@{-}[4,0]\\
&&&&&&\ar@{-}`d/4pt [1,-6][1,-6]&&&&&&\\
&&&&&&&&\ar@{-}`d/4pt [1,-8][1,-8]&&&&\\
&&&&&&&&&&&&\\
&&&&&&&&&&&&}}
\grow{\xymatrix@!0{
\\\\\\\\\\\\\\\\\\
\save\go+<0pt,0pt>\Drop{\txt{$=$}}\restore }}
\grow{\xymatrix@!0{\\
\save\go+<0pt,3pt>\Drop{M}\restore \ar@{-}[17,0]&&&&&\save\go+<0pt,3pt>\Drop{D}\restore \ar@{-}[2,2]&&\ar@{-}[1,-1]+<0.125pc,0.0pc> \ar@{-}[1,-1]+<0.0pc,0.125pc> \save\go+<0pt,3pt>\Drop{H}\restore &&\ar@{-}[2,0]\save\go+<0pt,3pt>\Drop{C}\restore\\
&&&&&&&&&\\
&&&&&\ar@{-}[-1,1]+<-0.125pc,0.0pc> \ar@{-}[-1,1]+<0.0pc,-0.125pc>\ar@{-}[2,-2]&&\ar@{-}[3,3] &&\ar@{-}[1,-1]+<0.125pc,0.0pc> \ar@{-}[1,-1]+<0.0pc,0.125pc>\\
&&&&&&&&&\\
&&&\ar@{-}[2,0]&&&&\ar@{-}[1,0]\ar@{-}[-1,1]+<-0.125pc,0.0pc> \ar@{-}[-1,1]+<0.0pc,-0.125pc>&&&\\
&&&&&&&\ar@{-}`l/4pt [1,-2] [1,-2] \ar@{-}`r [1,2] [1,2]&&&\ar@{-}[2,2]\\
&&&\ar@{-}[1,1]+<-0.1pc,0.1pc> && \ar@{-}[3,-3]&&&&\ar@{-}[1,0]&\\
&&&&&&&&&\ar@{-}`l/4pt [1,-1] [1,-1] \ar@{-}`r [1,1] [1,1]&&&\ar@{-}[1,0]\\
&&&&&\ar@{-}[1,0]\ar@{-}[-1,-1]+<0.1pc,-0.1pc>&&&\ar@{-}[1,0]+<0pc,-0.5pt>&&\ar@{-}[1,1]+<-0.125pc,0.0pc> \ar@{-}[1,1]+<0.0pc,0.125pc>&&\ar@{-}[2,-2]\\
&&\ar@{-}[1,0]&&&\ar@{-}`l/4pt [1,-1] [1,-1] \ar@{-}`r [1,1] [1,1]&&&&&&&\\
&&\ar@{-}[0,0]+<0pc,-0.5pt>&&\ar@{-}[0,0]+<0pc,-0.5pt>&&\ar@{-}[0,0]+<0pc,-0.5pt>&&*+[o]+<0.40pc>[F]{\mathbf{u}}\ar@{-}[5,0]+<0pc,0pc>&& \ar@{-}[7,0]&&\ar@{-}[-1,-1]+<0.125pc,0.0pc>\ar@{-}[-1,-1]+<0.0pc,-0.125pc>\ar@{-}[7,0]\\
&&*+[o]+<0.37pc>[F]{\ov{\mathbf{u}}}\ar@{-}[1,0]+<0pc,0pc>&&*+[o]+<0.37pc>[F]{\ov{\gamma}}\ar@{-}[1,0]+<0pc,0pc> &&*+[o]+<0.40pc>[F]{\gamma}\ar@{-}[1,0]+<0pc,0pc>&&&&\\
&&\ar@{-}`d/4pt [1,-2][1,-2]&&\ar@{-}`d/4pt [1,1] `[0,2] [0,2]&&&&&&\\
&&&&&\ar@{-}[1,0]&&&&&\\
&&&&&\ar@{-}`d/4pt [1,-5][1,-5]&&&&&\\
&&&&&&&&\ar@{-}`d/4pt [1,-8][1,-8]&&&&\\
&&&&&&&&&&&&\\
&&&&&&&&&&&&}}
\grow{\xymatrix@!0{
\\\\\\\\\\\\\\\\\\
\save\go+<0pt,0pt>\Drop{\txt{$=$}}\restore }}
\grow{\xymatrix@!0{\\\\\\
\save\go+<0pt,3pt>\Drop{M}\restore \ar@{-}[14,0]&&\save\go+<0pt,3pt>\Drop{D}\restore \ar@{-}[2,2]&&\ar@{-}[1,-1]+<0.125pc,0.0pc> \ar@{-}[1,-1]+<0.0pc,0.125pc> \save\go+<0pt,3pt>\Drop{H}\restore &&\ar@{-}[2,0]\save\go+<0pt,3pt>\Drop{C}\restore\\
&&&&&&\\
&&\ar@{-}[-1,1]+<-0.125pc,0.0pc> \ar@{-}[-1,1]+<0.0pc,-0.125pc>\ar@{-}[1,0]&&\ar@{-}[2,2] &&\ar@{-}[1,-1]+<0.125pc,0.0pc> \ar@{-}[1,-1]+<0.0pc,0.125pc>\\
&&\ar@{-}[1,0]+<0pt,2.5pt>&&&&\\
&&\save\go+<0pt,1.6pt>\Drop{\circ}\restore &&\ar@{-}[1,0]\ar@{-}[-1,1]+<-0.125pc,0.0pc> \ar@{-}[-1,1]+<0.0pc,-0.125pc>&&\ar@{-}[1,1]&\\
&&&&\ar@{-}`l/4pt [1,-1] [1,-1] \ar@{-}`r [1,1] [1,1]&&&\ar@{-}[3,0]\\
&&&\ar@{-}[1,0]&&\ar@{-}[2,0]&&\\
&&&\ar@{-}`l/4pt [1,-1] [1,-1] \ar@{-}`r [1,1] [1,1]&&&&\\
&&\ar@{-}[0,0]+<0pc,-0.5pt>&&\ar@{-}[0,0]+<0pc,-0.5pt>&\ar@{-}[1,1]+<-0.125pc,0.0pc> \ar@{-}[1,1]+<0.0pc,0.125pc>&&\ar@{-}[2,-2]\\
&&*+[o]+<0.37pc>[F]{\ov{\mathbf{u}}}\ar@{-}[1,0]+<0pc,0pc>&&*+[o]+<0.40pc>[F]{\mathbf{u}}\ar@{-}[1,0]+<0pc,0pc>&&&\\
&&\ar@{-}`d/4pt [1,1] `[0,2] [0,2]&&&\ar@{-}[4,0]&&\ar@{-}[-1,-1]+<0.125pc,0.0pc> \ar@{-}[-1,-1]+<0.0pc,-0.125pc>\ar@{-}[4,0]\\
&&&\ar@{-}[1,0]&&&&\\
&&&\ar@{-}`d/4pt [1,-3][1,-3]&&&&\\
&&&&&&&\\
&&&&&&&}}
\grow{\xymatrix@!0{
\\\\\\\\\\\\\\\\\\
\save\go+<0pt,0pt>\Drop{\txt{$=$}}\restore }}
\grow{\xymatrix@!0{\\\\\\\\\\\\\\
\save\go+<0pt,3pt>\Drop{M}\restore \ar@{-}[4,0]&&\save\go+<0pt,3pt>\Drop{D}\restore \ar@{-}[2,2]&&\ar@{-}[1,-1]+<0.125pc,0.0pc> \ar@{-}[1,-1]+<0.0pc,0.125pc> \save\go+<0pt,3pt>\Drop{H}\restore &&\ar@{-}[4,0]\save\go+<0pt,3pt>\Drop{C}\restore\\
&&&&&&\\
&&\ar@{-}[-1,1]+<-0.125pc,0.0pc> \ar@{-}[-1,1]+<0.0pc,-0.125pc>\ar@{-}[1,0]&&\ar@{-}[2,0]&&\\
&&\ar@{-}[1,0]+<0pt,2.5pt>&&&&\\
&&\save\go+<0pt,1.6pt>\Drop{\circ}\restore &&&&}}
\grow{\xymatrix@!0{
\\\\\\\\\\\\\\\\\\
\save\go+<0pt,0pt>\Drop{\txt{.}}\restore }}
$$
Now, we compute $\ov{u}_i$ for $0<i<s$. Let $C_1:=H^{\ot_c^{i-1}}$ and $C_2:=H^{\ot_c^{s-i-1}}$. Consider the map
$$
\Phi\colon H^{\ot_c^s}\longrightarrow A\ot_k H^{\ot_c^{s-1}},
$$
diagrammatically defined by
$$
%
%
\grow{\xymatrix@!0{
\\\\\\\\\\\\\\
\save\go+<0pt,0pt>\Drop{\txt{$\Phi:=$}}\restore }}
\grow{\xymatrix@!0{
\save\go+<0pt,3pt>\Drop{C_1}\restore \ar@{-}[1,0]+<0pc,-0.5pt> &&& \save\go+<0pt,3pt>\Drop{H}\restore \ar@{-}[1,0]+<0pc,-0.5pt> && \save\go+<0pt,3pt>\Drop{H}\restore \ar@{-}[1,0]+<0pc,-0.5pt> &&&  \save\go+<0pt,3pt>\Drop{C_2}\restore \ar@{-}[1,0]+<0pc,-0.5pt>\\
&&&&&&&&&&\\
*+[o]+<0.37pc>[F]{\boldsymbol{\gamma}}\ar@{-}[4,0]+<0pc,0pc>&&&*+[o]+<0.37pc>[F]{\gamma}\ar@{-}[1,0]+<0pc,0pc> &&*+[o]+<0.37pc>[F]{\gamma}\ar@{-}[1,0]+<0pc,0pc>&&&*+[o]+<0.37pc>[F]{\boldsymbol{\gamma}}\ar@{-}[4,0]+<0pc,0pc>&&\\
&&&\ar@{-}`d/4pt [1,1] `[0,2] [0,2]&&&&&&&\\
&&&&\ar@{-}[2,0]&&&&&&\\
\ar@{-}`r/4pt [1,2] [1,2]&&&&\ar@{-}`r/4pt [1,2] [1,2]&&&&\ar@{-}`r/4pt [1,2] [1,2]&&\\
\ar@{-}[0,0]+<0pc,-0.5pt>&&\ar@{-}[1,1]+<-0.125pc,0.0pc> \ar@{-}[1,1]+<0.0pc,0.125pc> &&\ar@{-}[2,-2]&&\ar@{-}[1,1]+<-0.125pc,0.0pc> \ar@{-}[1,1]+<0.0pc,0.125pc> &&\ar@{-}[2,-2]&&\ar@{-}[9,0]\\
*+[o]+<0.40pc>[F]{\mu}\ar@{-}[1,0]+<0pc,0pc>&&&&&&&&&&\\
\ar@{-}`d/4pt [1,1]`[0,2][0,2]&&&&\ar@{-}[-1,-1]+<0.125pc,0.0pc> \ar@{-}[-1,-1]+<0.0pc,-0.125pc>\ar@{-}[1,1]+<-0.125pc,0.0pc> \ar@{-}[1,1]+<0.0pc,0.125pc>&& \ar@{-}[2,-2]&&\ar@{-}[-1,-1]+<0.125pc,0.0pc>\ar@{-}[-1,-1]+<0.0pc,-0.125pc>\ar@{-}[7,0]&&\\
&\ar@{-}[4,0]&&&&&&&&&\\
&&&&\ar@{-}[1,0]+<0pc,-0.5pt>&&\ar@{-}[-1,-1]+<0.125pc,0.0pc>\ar@{-}[-1,-1]+<0.0pc,-0.125pc> \ar@{-}[5,0]&&&&\\
&&&&&&&&&&\\
&&&&*+[o]+<0.40pc>[F]{\mu}\ar@{-}[1,0]+<0pc,0pc>&&&&&&\\
&\ar@{-}`d/4pt [1,1]+<0.2pc,0pc> `[0,3] [0,3] &&&&& &&&&\\
&&\save[]+<0.2pc,0pc> \Drop{}\ar@{-}[1,0]+<0.2pc,0pc>\restore&&&&&&&&\\
&&\save[]+<0.2pc,0pc> \Drop{}\restore&&&&&&&&
}}
\grow{\xymatrix@!0{
\\\\\\\\\\\\\\
\save\go+<0pt,0pt>\Drop{\txt{,}}\restore }}
$$
where $\boldsymbol{\gamma}$ denotes both $\gamma^{\ot_k^{i-1}}$ and $\gamma^{\ot_k^{s-i-1}}$, and $\mu$ denotes both $\mu_{i-1}$ and $\mu_{s-i-1}$. Since
$$
%
%
\grow{\xymatrix@!0{
\\\\\\\\\\\\\\\\\\\\
\save\go+<0pt,0pt>\Drop{\txt{$\Phi=$}}\restore }}
\grow{\xymatrix@!0{
\save\go+<0pt,3pt>\Drop{C_1}\restore \ar@{-}[1,0]+<0pc,-0.5pt> &&&
\save\go+<0pt,3pt>\Drop{H}\restore \ar@{-}[1,0]+<0pc,-0.5pt> &&&&
\save\go+<0pt,3pt>\Drop{H}\restore \ar@{-}[1,0]+<0pc,-0.5pt> &&&
\save\go+<0pt,3pt>\Drop{C_2}\restore \ar@{-}[1,0]+<0pc,-0.5pt>\\
&&&&&&&&&&\\
*+[o]+<0.37pc>[F]{\boldsymbol{\gamma}}\ar@{-}[6,0]+<0pc,0pc>&&&*+[o]+<0.37pc>[F]{\gamma}\ar@{-}[5,0]+<0pc,0pc>
&&&&*+[o]+<0.37pc>[F]{\gamma}\ar@{-}[3,0]+<0pc,0pc>&&&*+[o]+<0.37pc>[F]{\boldsymbol{\gamma}}\ar@{-}[6,0]+<0pc,0pc>&&\\
&&&&&&&&&&&&\\
&&&\ar@{-}`r/4pt [1,2] [1,2]&&&&\ar@{-}`r/4pt [1,2] [1,2]&&&&&\\
&&&&&\ar@{-}[1,1]+<-0.125pc,0.0pc> \ar@{-}[1,1]+<0.0pc,0.125pc>&&\ar@{-}[2,-2]
&&\ar@{-}[2,0]&&&\\
&&&&&&&&&&&&\\
&&&\ar@{-}`d/4pt [1,1] `[0,2] [0,2]&&&&\ar@{-}[-1,-1]+<0.125pc,0.0pc>
\ar@{-}[-1,-1]+<0.0pc,-0.125pc>\ar@{-}`d/4pt [1,1] `[0,2] [0,2]&&&&&\\
\ar@{-}[1,1]&&&&\ar@{-}[1,0]&&&&\ar@{-}[1,0]&&\ar@{-}[1,-1]&&\\
&\ar@{-}[2,0]&&&\ar@{-}[1,1]&&&&\ar@{-}[1,-1]&\ar@{-}[2,0]&&&\\
&\ar@{-}`r/4pt [1,2] [1,2]&&&&\ar@{-}[1,0]&&\ar@{-}[1,0]&&\ar@{-}`r/4pt [1,2]
[1,2]&&&\\
&\ar@{-}[0,0]+<0pc,-0.5pt>&&\ar@{-}[1,1]+<-0.125pc,0.0pc>
\ar@{-}[1,1]+<0.0pc,0.125pc> &&\ar@{-}[2,-2]&&\ar@{-}[1,1]+<-0.125pc,0.0pc>
\ar@{-}[1,1]+<0.0pc,0.125pc> &&\ar@{-}[2,-2]&&\ar@{-}[9,0]\\
&*+[o]+<0.40pc>[F]{\mu}\ar@{-}[1,0]+<0pc,0pc>&&&&&&&&&&\\
&\ar@{-}`d/4pt [1,1]`[0,2][0,2]&&&&\ar@{-}[-1,-1]+<0.125pc,0.0pc>
\ar@{-}[-1,-1]+<0.0pc,-0.125pc>\ar@{-}[1,1]+<-0.125pc,0.0pc>
\ar@{-}[1,1]+<0.0pc,0.125pc>&&
\ar@{-}[2,-2]&&\ar@{-}[-1,-1]+<0.125pc,0.0pc>\ar@{-}[-1,-1]+<0.0pc,-0.125pc>\ar@{-}[7,0]&&\\
&&\ar@{-}[4,0]&&&&&&&&&\\
&&&&&\ar@{-}[1,0]+<0pc,-0.5pt>&&\ar@{-}[-1,-1]+<0.125pc,0.0pc>\ar@{-}[-1,-1]+<0.0pc,-0.125pc>
\ar@{-}[5,0]&&&&\\
&&&&&&&&&&&\\
&&&&&*+[o]+<0.40pc>[F]{\mu}\ar@{-}[1,0]+<0pc,0pc>&&&&&&\\
&&\ar@{-}`d/4pt [1,1]+<0.2pc,0pc> `[0,3] [0,3] &&&&& &&&&\\
&&&\save[]+<0.2pc,0pc> \Drop{}\ar@{-}[1,0]+<0.2pc,0pc>\restore&&&&&&&&\\
&&&\save[]+<0.2pc,0pc> \Drop{}\restore&&&&&&&&
}}
\grow{\xymatrix@!0{
\\\\\\\\\\\\\\\\\\\\
\save\go+<0pt,0pt>\Drop{\txt{$=$}}\restore }}
\grow{\xymatrix@!0{\\\\
&\save\go+<0pt,3pt>\Drop{C_1}\restore \ar@{-}[1,0] &&&&
\save\go+<0pt,3pt>\Drop{H}\restore \ar@{-}[1,0] &&&&
\save\go+<0pt,3pt>\Drop{H}\restore \ar@{-}[1,0] &&&&
\save\go+<0pt,3pt>\Drop{C_2}\restore \ar@{-}[1,0]\\
&\ar@{-}`l/4pt [1,-1] [1,-1] \ar@{-}`r [1,1] [1,1]&&&&\ar@{-}`l/4pt [1,-1] [1,-1]
\ar@{-}`r [1,1] [1,1]&&&&\ar@{-}`l/4pt [1,-1] [1,-1] \ar@{-}`r [1,1]
[1,1]&&&&\ar@{-}`l/4pt [1,-1] [1,-1] \ar@{-}`r [1,1] [1,1]&\\
\ar@{-}[0,0]+<0pc,-0.5pt>&&\ar@{-}[2,0]&&\ar@{-}[0,0]+<0pc,-0.5pt>&&\ar@{-}[2,0]&&
\ar@{-}[0,0]+<0pc,-0.5pt>&&\ar@{-}[2,0]&&\ar@{-}[0,0]+<0pc,-0.5pt>&&\ar@{-}[2,0]\\
*+[o]+<0.40pc>[F]{\mathbf{u}}\ar@{-}[1,0]+<0pc,0pc>&&&&*+[o]+<0.40pc>[F]{\gamma}\ar@{-}[1,0]+<0pc,0pc>
&&&&*+[o]+<0.40pc>[F]{\gamma}\ar@{-}[1,0]+<0pc,0pc>&&&&*+[o]+<0.40pc>[F]{\mathbf{u}}\ar@{-}[1,0]+<0pc,0pc>&&\\
\ar@{-}[8,0]&&\ar@{-}[1,1]+<-0.125pc,0.0pc> \ar@{-}[1,1]+<0.0pc,0.125pc>&&
\ar@{-}[2,-2]&&\ar@{-}[1,1]+<-0.125pc,0.0pc> \ar@{-}[1,1]+<0.0pc,0.125pc> &&
\ar@{-}[2,-2]&&\ar@{-}[1,1]+<-0.125pc,0.0pc> \ar@{-}[1,1]+<0.0pc,0.125pc>
&&\ar@{-}[2,-2]&&\ar@{-}[12,0]\\
&&&&&&&&&&&&&&\\
&&\ar@{-}[4,0]&&\ar@{-}[-1,-1]+<0.125pc,0.0pc>\ar@{-}[-1,-1]+<0.0pc,-0.125pc>\ar@{-}[1,1]+<-0.125pc,0.0pc>
\ar@{-}[1,1]+<0.0pc,0.125pc>&&\ar@{-}[2,-2]&&\ar@{-}[-1,-1]+<0.125pc,0.0pc>
\ar@{-}[-1,-1]+<0.0pc,-0.125pc>\ar@{-}[1,1]+<-0.125pc,0.0pc>
\ar@{-}[1,1]+<0.0pc,0.125pc>&&\ar@{-}[2,-2]&&\ar@{-}[-1,-1]+<0.125pc,0.0pc>\ar@{-}[-1,-1]+<0.0pc,-0.125pc>\ar@{-}[4,0]&&\\
&&&&&&&&&&&&&&\\
&&&&\ar@{-}[2,0]&&\ar@{-}[-1,-1]+<0.125pc,0.0pc>\ar@{-}[-1,-1]+<0.0pc,-0.125pc>\ar@{-}[1,1]+<-0.125pc,0.0pc>
\ar@{-}[1,1]+<0.0pc,0.125pc>&&\ar@{-}[2,-2]&&\ar@{-}[2,0]\ar@{-}[-1,-1]+<0.125pc,0.0pc>\ar@{-}[-1,-1]+<0.0pc,-0.125pc>&&&&\\
&&&&&&&&&&&&&&\\
&&\ar@{-}`d/4pt [1,1] `[0,2][0,2]&&&&\ar@{-}[4,0]+<-0.4pc,0pc>
&&\ar@{-}[-1,-1]+<0.125pc,0.0pc> \ar@{-}[-1,-1]+<0.0pc,-0.125pc> \ar@{-}[6,0]
&&\ar@{-}`d/4pt [1,1] `[0,2] [0,2]&&&&\\
&&&\ar@{-}[1,0]&&&&&&&&\ar@{-}[5,0]&&&\\
\ar@{-}`d/4pt [1,1]+<0.2pc,0pc> `[0,3] [0,3] &&&&&&&&&&&&&&\\
&\save[]+<0.2pc,0pc> \Drop{}\ar@{-}[1,0]+<0.2pc,0pc>\restore&&&&&&&&&&&&&\\
&\save[]+<0.2pc,0pc> \Drop{} \ar@{-}`d/4pt [1,2] `[0,4] [0,4]\restore&&&&&&&&&&&&&\\
&&&\save[]+<0.2pc,0pc> \Drop{}\ar@{-}[1,0]+<0.2pc,0pc>\restore&&&&&&&&&&&\\
&&&\save[]+<0.2pc,0pc> \Drop{}\restore&&&&&&&&&&&}}\\
\grow{\xymatrix@!0{
\\\\\\\\\\\\\\\\\\\\
\save\go+<0pt,0pt>\Drop{\txt{$=$}}\restore }}
\grow{\xymatrix@!0{\\\\
&\save\go+<0pt,3pt>\Drop{C_1}\restore \ar@{-}[1,0] &&&&
\save\go+<0pt,3pt>\Drop{H}\restore \ar@{-}[1,0] &&&&
\save\go+<0pt,3pt>\Drop{H}\restore \ar@{-}[1,0] &&&&
\save\go+<0pt,3pt>\Drop{C_2}\restore \ar@{-}[1,0]\\
&\ar@{-}`l/4pt [1,-1] [1,-1] \ar@{-}`r [1,1] [1,1]&&&&\ar@{-}`l/4pt [1,-1] [1,-1]
\ar@{-}`r [1,1] [1,1]&&&&\ar@{-}`l/4pt [1,-1] [1,-1] \ar@{-}`r [1,1]
[1,1]&&&&\ar@{-}`l/4pt [1,-1] [1,-1] \ar@{-}`r [1,1] [1,1]&\\
\ar@{-}[6,0]&&\ar@{-}[1,1]+<-0.1pc,0.1pc> &&
\ar@{-}[2,-2]&&\ar@{-}[1,1]+<-0.1pc,0.1pc> &&
\ar@{-}[2,-2]&&\ar@{-}[1,1]+<-0.1pc,0.1pc> && \ar@{-}[2,-2]&&\ar@{-}[15,0]\\
&&&&&&&&&&&&&&\\
&&\ar@{-}[4,0]&&\ar@{-}[-1,-1]+<0.1pc,-0.1pc>\ar@{-}[1,1]+<-0.1pc,0.1pc> &&
\ar@{-}[2,-2]&&\ar@{-}[-1,-1]+<0.1pc,-0.1pc>\ar@{-}[1,1]+<-0.1pc,0.1pc> &&
\ar@{-}[2,-2]&&\ar@{-}[-1,-1]+<0.1pc,-0.1pc>\ar@{-}[13,0]&&\\
&&&&&&&&&&&&&&\\
&&&&\ar@{-}[2,0]&&\ar@{-}[-1,-1]+<0.1pc,-0.1pc>\ar@{-}[1,1]+<-0.1pc,0.1pc> &&
\ar@{-}[2,-2]&&\ar@{-}[-1,-1]+<0.1pc,-0.1pc>\ar@{-}[6,0]&&&&\\
&&&&&&&&&&&&&&\\
\ar@{-}[1,0]+<0pc,-0.5pt>&&\ar@{-}[1,0]+<0pc,-0.5pt>&&\ar@{-}[1,0]+<0pc,-0.5pt>&&
\ar@{-}[1,0]+<0pc,-0.5pt>&&\ar@{-}[-1,-1]+<0.1pc,-0.1pc>\ar@{-}[4,0]&&&&&&\\
&&&&&&&&&&&&&&\\
*+[o]+<0.40pc>[F]{\mathbf{u}}\ar@{-}[1,0]+<0pc,0pc>&&
*+[o]+<0.40pc>[F]{\gamma}\ar@{-}[1,0]+<0pc,0pc>&&
*+[o]+<0.40pc>[F]{\gamma}\ar@{-}[1,0]+<0pc,0pc>&&
*+[o]+<0.40pc>[F]{\mathbf{u}}\ar@{-}[1,0]+<0pc,0pc>&&&&&&&&\\
\ar@{-}[2,0]&&\ar@{-}`d/4pt [1,1] `[0,2] [0,2]&&&&\ar@{-}[4,-1] &&&&&&&&\\
&&&\ar@{-}[1,0]&&&&&\ar@{-}`d/4pt [1,1] `[0,2] [0,2]&&&&&&\\
\ar@{-}`d/4pt [1,1]+<0.2pc,0pc> `[0,3] [0,3] &&&&&&&&&\ar@{-}[4,0]&&&&&\\
&\save[]+<0.2pc,0pc> \Drop{}\ar@{-}[1,0]+<0.2pc,0pc>\restore&&&&&&&&&&&&&\\
&\save[]+<0.2pc,0pc> \Drop{} \ar@{-}`d/4pt [1,2] `[0,4] [0,4]\restore&&&&&&&&&&&&&\\
&&&\save[]+<0.2pc,0pc> \Drop{}\ar@{-}[1,0]+<0.2pc,0pc>\restore&&&&&&&&&&&\\
&&&\save[]+<0.2pc,0pc> \Drop{}\restore&&&&&&&&&&&}}
 \grow{\xymatrix@!0{
\\\\\\\\\\\\\\\\\\\\
\save\go+<0pt,0pt>\Drop{\txt{$=$}}\restore }}
\grow{\xymatrix@!0{\\\\\\\\\\\\\\
&\save\go+<0pt,3pt>\Drop{C}\restore \ar@{-}[1,0]\\
&\ar@{-}`l/4pt [1,-1] [1,-1] \ar@{-}`r [1,1] [1,1]&\\
\ar@{-}[1,0]+<0pc,-0.5pt>&&\ar@{-}[1,0]+<0pc,-0.5pt>\\
&&\\
*+[o]+<0.40pc>[F]{\mathbf{u}}\ar@{-}[2,0]+<0pc,0pc>&&*+[o]+<0.31pc>[F]{\mathfrak{u}_{\hspace{-0.5pt}
i}} \ar@{-}[2,0]+<0pc,0pc>\\
&&\\
&&}}
\grow{\xymatrix@!0{
\\\\\\\\\\\\\\\\\\\\
\save\go+<0pt,0pt>\Drop{\txt{,}}\restore }}
$$
where $C:=H^{\ot_c^s}$, $\mathfrak{u}_i\colon H^{\ot_c^s}\longrightarrow H^{\ot_c^{s-1}}$ is the map given by
$$
\mathfrak{u}_i(\bh_{1s}):=\bh_{1,i-1}\ot_k h_ih_{i+1}\ot_k\bh_{i+2,s}
$$
and $\mathbf{u}$ denotes $\mu_{i-1}\xcirc\gamma^{\ot_k^{i-1}}$, $\mu_{s-i-1}\xcirc \gamma^{\ot_k^{s-i-1}}$ and $\mu_s \xcirc\gamma^{\ot_k^s}$, we have
$$
%
%
\grow{\xymatrix@!0{
\\\\\\\\\\\\
\save\go+<0pt,0pt>\Drop{\txt{$\ov{u}_i=$}}\restore }}
\grow{\xymatrix@!0{\\
\save\go+<0pt,3pt>\Drop{M}\restore \ar@{-}[10,0] &&&&\save\go+<0pt,3pt>\Drop{D}\restore \ar@{-}[5,5]&&\ar@{-}[1,-1]+<0.125pc,0.0pc> \ar@{-}[1,-1]+<0.0pc,0.125pc> \save\go+<0pt,3pt>\Drop{C}\restore \\
&&&&&&&&&\\
&&&&\ar@{-}[1,0]\ar@{-}[-1,1]+<-0.125pc,0.0pc> \ar@{-}[-1,1]+<0.0pc,-0.125pc>&&\\
&&&&\ar@{-}`l/4pt [1,-2] [1,-2] \ar@{-}`r [1,2] [1,2]&&\\
&&\ar@{-}[1,0]+<0pc,-0.5pt>&&&&\ar@{-}[2,0]&&\\
&&&&&&&&&\ar@{-}[2,0]\\
&&*+[o]+<0.37pc>[F]{\ov{\mathbf{u}}}\ar@{-}[1,0]+<0pc,0pc> &&&&*+<0.1pc>[F]{\,\,\,\Phi\,\,\,} \ar@{-}@<0.40pc>[1,0] \ar@{-}@<-0.40pc>[2,0] \\
&&\ar@{-}`d/4pt [1,-2][1,-2]&&&&&\ar@{-}[1,1]+<-0.125pc,0.0pc> \ar@{-}[1,1]+<0.0pc,0.125pc>&&\ar@{-}[2,-2]\\
&&&&&\ar@{-}`d/4pt [1,-5][1,-5]&&&&\\
&&&&&&&\ar@{-}[1,0]&&\ar@{-}[-1,-1]+<0.125pc,0.0pc> \ar@{-}[-1,-1]+<0.0pc,-0.125pc>\ar@{-}[1,0]\\
&&&&&&&&&}}
\grow{\xymatrix@!0{
\\\\\\\\\\\\
\save\go+<0pt,0pt>\Drop{\txt{$=$}}\restore }}
\grow{\xymatrix@!0{\\
\save\go+<0pt,3pt>\Drop{M}\restore \ar@{-}[11,0] &&&\save\go+<0pt,3pt>\Drop{D}\restore \ar@{-}[4,4]&&\ar@{-}[1,-1]+<0.125pc,0.0pc> \ar@{-}[1,-1]+<0.0pc,0.125pc> \save\go+<0pt,3pt>\Drop{C}\restore \\
&&&&&&&\\
&&&\ar@{-}[1,0]\ar@{-}[-1,1]+<-0.125pc,0.0pc> \ar@{-}[-1,1]+<0.0pc,-0.125pc>&&\\
&&&\ar@{-}`l/4pt [1,-1] [1,-1] \ar@{-}`r [1,1] [1,1]&&&&\\
&&\ar@{-}[0,0]+<0pc,-0.5pt>&&\ar@{-}[1,0]&&&\ar@{-}[5,0]\\
&&*+[o]+<0.37pc>[F]{\ov{\mathbf{u}}}\ar@{-}[1,0]+<0pc,0pc>&&\ar@{-}`l/4pt [1,-1] [1,-1] \ar@{-}`r [1,1] [1,1] &&&\\
&&\ar@{-}`d/4pt [1,-2][1,-2]&\ar@{-}[0,0]+<0pc,-0.5pt> &&\ar@{-}[0,0]+<0pc,-0.5pt>&&\\
&&&*+[o]+<0.40pc>[F]{\mathbf{u}}\ar@{-}[1,0]+<0pc,0pc>&&*+[o]+<0.31pc>[F]{\mathfrak{u}_{\hspace{-0.5pt} i}}\ar@{-}[2,0]+<0pc,0pc>&&\\
&&&\ar@{-}`d/4pt [1,-3][1,-3]&&&&\\
&&&&&\ar@{-}[1,1]+<-0.125pc,0.0pc> \ar@{-}[1,1]+<0.0pc,0.125pc>&&\ar@{-}[2,-2]\\
&&&&&&&\\
&&&&&&&\ar@{-}[-1,-1]+<0.125pc,0.0pc>\ar@{-}[-1,-1]+<0.0pc,-0.125pc>}}
\grow{\xymatrix@!0{
\\\\\\\\\\\\
\save\go+<0pt,0pt>\Drop{\txt{$=$}}\restore }}
\grow{\xymatrix@!0{
\save\go+<0pt,3pt>\Drop{M}\restore \ar@{-}[12,0] &&&&\save\go+<0pt,3pt>\Drop{D}\restore \ar@{-}[3,3]&&\ar@{-}[1,-1]+<0.125pc,0.0pc> \ar@{-}[1,-1]+<0.0pc,0.125pc> \save\go+<0pt,3pt>\Drop{C}\restore \\
&&&&&&&\\
&&&&\ar@{-}[1,0]\ar@{-}[-1,1]+<-0.125pc,0.0pc> \ar@{-}[-1,1]+<0.0pc,-0.125pc>&&&\\
&&&&\ar@{-}`l/4pt [1,-1] [1,-1] \ar@{-}`r [1,1] [1,1]&&&\ar@{-}[3,0]\\
&&&\ar@{-}[1,0]&&\ar@{-}[2,0]&&\\
&&&\ar@{-}`l/4pt [1,-1] [1,-1] \ar@{-}`r [1,1] [1,1]&&&&\\
&&\ar@{-}[0,0]+<0pc,-0.5pt>&&\ar@{-}[0,0]+<0pc,-0.5pt>&\ar@{-}[1,1]+<-0.125pc,0.0pc> \ar@{-}[1,1]+<0.0pc,0.125pc>&&\ar@{-}[2,-2]\\
&&*+[o]+<0.37pc>[F]{\ov{\mathbf{u}}}\ar@{-}[1,0]+<0pc,0pc>&&*+[o]+<0.40pc>[F]{\mathbf{u}}\ar@{-}[1,0]+<0pc,0pc>&&&\\
&&\ar@{-}`d/4pt [1,1]`[0,2][0,2]&&&\ar@{-}[4,0]&&\ar@{-}[-1,-1]+<0.125pc,0.0pc> \ar@{-}[-1,-1]+<0.0pc,-0.125pc>\ar@{-}[1,0]\\
&&&\ar@{-}[3,0]&&&&\ar@{-}[0,0]+<0pc,-0.5pt>\\
&&&&&&&*+[o]+<0.31pc>[F]{\mathfrak{u}_{\hspace{-0.5pt} i}}\ar@{-}[2,0]+<0pc,0pc>\\
&&&&&&&\\
&&&&&&&}}
 \grow{\xymatrix@!0{
\\\\\\\\\\\\
\save\go+<0pt,0pt>\Drop{\txt{$=$}}\restore }}
\grow{\xymatrix@!0{\\
\save\go+<0pt,3pt>\Drop{M}\restore \ar@{-}[10,0] &&\save\go+<0pt,3pt>\Drop{D}\restore \ar@{-}[3,3]&&\ar@{-}[1,-1]+<0.125pc,0.0pc> \ar@{-}[1,-1]+<0.0pc,0.125pc> \save\go+<0pt,3pt>\Drop{C}\restore \\
&&&&&\\
&&\ar@{-}[1,0]\ar@{-}[-1,1]+<-0.125pc,0.0pc> \ar@{-}[-1,1]+<0.0pc,-0.125pc>&&&\\
&&\ar@{-}`l/4pt [1,-1] [1,-1] \ar@{-}`r [1,1] [1,1]&&&\ar@{-}[1,0]\\
&\ar@{-}[1,0]+<0pt,2.5pt>&& \ar@{-}[1,1]+<-0.125pc,0.0pc> \ar@{-}[1,1]+<0.0pc,0.125pc>&&\ar@{-}[2,-2]\\
&\save\go+<0pt,1.6pt>\Drop{\circ}\restore&&&&\\
&&&\ar@{-}[4,0]&&\ar@{-}[-1,-1]+<0.125pc,0.0pc>\ar@{-}[-1,-1]+<0.0pc,-0.125pc>\ar@{-}[1,0]+<0pc,-0.5pt>\\
&&&&&\\
&&&&&*+[o]+<0.31pc>[F]{\mathfrak{u}_{\hspace{-0.5pt} i}}\ar@{-}[2,0]+<0pc,0pc>\\
&&&&&\\
&&&&&}}
\grow{\xymatrix@!0{
\\\\\\\\\\\\
\save\go+<0pt,0pt>\Drop{\txt{$=$}}\restore }}
\grow{\xymatrix@!0{\\\\\\\\
\save\go+<0pt,3pt>\Drop{M}\restore \ar@{-}[4,0] &&\save\go+<0pt,3pt>\Drop{D}\restore \ar@{-}[4,0]&&\save\go+<0pt,3pt>\Drop{C}\restore \ar@{-}[1,0]+<0pc,-0.5pt>\\
&&&&\\
&&&&*+[o]+<0.31pc>[F]{\mathfrak{u}_{\hspace{-0.5pt} i}}\ar@{-}[2,0]+<0pc,0pc>\\
&&&&\\
&&&&}}
\grow{\xymatrix@!0{
\\\\\\\\\\\\
\save\go+<0pt,0pt>\Drop{\txt{,}}\restore }}
$$
where $D := A^{\ot_k^r}$ and $\ov{\mathbf{u}}:=\mu_s\xcirc \ov{\gamma}^{\ot_k^s}\xcirc\gc_s$. Again by Lemma~\ref{inversa de convolucion}. Finally, we compute~$\ov{u}_s$. Let $C:= H^{\ot_k^{s-1}}$, $D:=A^{\ot_k^r}$, $\mathbf{u}:= \mu_{s-1}\xcirc\gamma^{\ot_k^{s-1}}$ and $\ov{\mathbf{u}}:=\mu_{s-1}\xcirc \ov{\gamma}^{\ot_k^{s-1}}\xcirc\gc_{s-1}$. Again by Lemmas~\ref{inversa de convolucion} and~\ref{para d1},
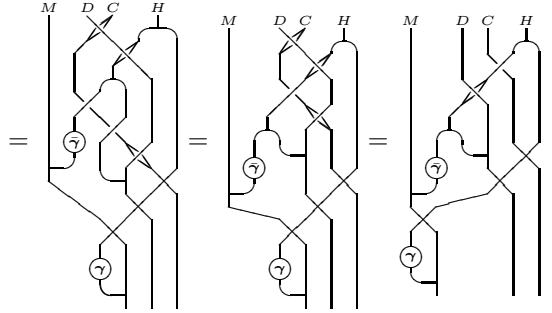
\begin{align*}
%
%
\grow{\xymatrix@!0{
\\\\\\\\\\\\\\\\\\\\\\\\\\
\save\go+<0pt,0pt>\Drop{\txt{$\ov{u}_s$}}\restore
}}
& \grow{\xymatrix@!0{
\\\\\\\\\\\\\\\\\\\\\\\\\\
\save\go+<0pt,0pt>\Drop{\txt{$=$}}\restore
}}
\grow{\xymatrix@!0{\\
\save\go+<0pt,3pt>\Drop{M}\restore \ar@{-}[18,0] &&&& \save\go+<0pt,3pt>\Drop{D}\restore \ar@{-}[2,2]&&\ar@{-}[1,-1]+<0.125pc,0.0pc> \ar@{-}[1,-1]+<0.0pc,0.125pc> \save\go+<0pt,3pt>\Drop{C}\restore && \save\go+<0pt,3pt>\Drop{H}\restore\ar@{-}[2,0]\\
&&&&&&&&&\\
&&&&\ar@{-}[2,-2]\ar@{-}[-1,1]+<-0.125pc,0.0pc> \ar@{-}[-1,1]+<0.0pc,-0.125pc> &&\ar@{-}[2,2]&&\ar@{-}[1,-1]+<0.125pc,0.0pc> \ar@{-}[1,-1]+<0.0pc,0.125pc>&\\
&&&&&&&&\\
&&\ar@{-}[1,0]&&&&\ar@{-}[1,0]\ar@{-}[-1,1]+<-0.125pc,0.0pc> \ar@{-}[-1,1]+<0.0pc,-0.125pc>&&\ar@{-}[1,1]\\
&&\ar@{-}`l/4pt [1,-1] [1,-1] \ar@{-}`r [1,1] [1,1]&&&&\ar@{-}`l/4pt [1,-1] [1,-1] \ar@{-}`r [1,1] [1,1]&&&\ar@{-}[5,0]\\
&\ar@{-}[2,0]&&\ar@{-}[1,1]+<-0.1pc,0.1pc> && \ar@{-}[2,-2]&&\ar@{-}[4,0]&&\\
&&&&&&&&&\\
&\ar@{-}[1,1]+<-0.1pc,0.1pc> && \ar@{-}[2,-2]&&\ar@{-}[-1,-1]+<0.1pc,-0.1pc> \ar@{-}[1,0]+<0pc,-0.5pt>&&&&\\
&&&&&&&&&\\
&\ar@{-}[1,0]+<0pc,-0.5pt> &&\ar@{-}[1,0]+<0pc,-0.5pt>\ar@{-}[-1,-1]+<0.1pc,-0.1pc>&& *+[o]+<0.37pc>[F]{\boldsymbol{\gamma}}\ar@{-}[5,0]+<0pc,0pc>&& \ar@/^0.1pc/ @{-}[2,2] \ar@/_0.1pc/ @{-}[2,2]&& \ar@/^0.1pc/ @{-}[2,-2]\ar@/_0.1pc/ @{-}[2,-2]\\
&&&&&&&&&\\
&*+[o]+<0.37pc>[F]{\ov{\gamma}}\ar@{-}[2,0]+<0pc,0pc>&&*+[o]+<0.37pc>[F]{\overline{\mathbf{u}}} \ar@{-}[2,0]+<0pc,0pc>&&&&\ar@{-}[1,0]&&\ar@{-}[1,0]\\
&&&&&&&\ar@{-}[2,2]&&\ar@{-}[2,-2]\\
&\ar@{-}`d/4pt [1,1] `[0,2] [0,2]&&&&&&&&\\
&&\ar@{-}[1,0]&&&\ar@{-}[2,2]&&\ar@{-}[2,-2]&&\ar@{-}[2,2]\\
&&\ar@{-}`d/4pt [1,-2][1,-2]&&&&&&&\\
&&&&&\ar@{-}[1,-3]&&\ar@{-}[2,0]&&&&\ar@{-}[2,0]\\
\ar@{-}[2,2]&&\ar@{-}[2,-2]&&&&&\ar@{-}`r/4pt [1,2] [1,2]&&&&\\
&&&&&&&\ar@{-}[1,0]+<0pc,-0.5pt>&&\ar@{-}[1,1]+<-0.125pc,0.0pc> \ar@{-}[1,1]+<0.0pc,0.125pc>&&\ar@{-}[2,-2]\\
\ar@{-}[1,0]+<0pc,-0.5pt>&&\ar@{-}[6,0]&&&&&&&&&\\
&&&&&&&*+[o]+<0.40pc>[F]{\mu}\ar@{-}[3,0]+<0pc,0pc>&&\ar@{-}[5,0] && \ar@{-}[-1,-1]+<0.125pc,0.0pc>\ar@{-}[-1,-1]+<0.0pc,-0.125pc>\ar@{-}[5,0]\\
*+[o]+<0.40pc>[F]{\gamma}\ar@{-}[1,0]+<0pc,0pc>&&&&&&&&&&&\\
\ar@{-}`d/4pt [1,2][1,2] &&&&&&&&&&&\\
&&&&&&&\ar@{-}`d/4pt [1,-5][1,-5]&&&&\\
&&&&&&&&&&&\\
&&&&&&&&&&&}}
\grow{\xymatrix@!0{
\\\\\\\\\\\\\\\\\\\\\\\\\\\\
\save\go+<0pt,0pt>\Drop{\txt{$=$}}\restore
}}
\grow{\xymatrix@!0{\\
\save\go+<0pt,3pt>\Drop{M}\restore \ar@{-}[20,0] &&&& \save\go+<0pt,3pt>\Drop{D}\restore \ar@{-}[2,2]&&\ar@{-}[1,-1]+<0.125pc,0.0pc> \ar@{-}[1,-1]+<0.0pc,0.125pc> \save\go+<0pt,3pt>\Drop{C}\restore && \save\go+<0pt,3pt>\Drop{H}\restore\ar@{-}[2,0]\\
&&&&&&&&&\\
&&&&\ar@{-}[2,-2]\ar@{-}[-1,1]+<-0.125pc,0.0pc> \ar@{-}[-1,1]+<0.0pc,-0.125pc> &&\ar@{-}[2,2]&&\ar@{-}[1,-1]+<0.125pc,0.0pc> \ar@{-}[1,-1]+<0.0pc,0.125pc>&\\
&&&&&&&&\\
&&\ar@{-}[1,0]&&&&\ar@{-}[1,0]\ar@{-}[-1,1]+<-0.125pc,0.0pc> \ar@{-}[-1,1]+<0.0pc,-0.125pc>&&\ar@{-}[1,1]\\
&&\ar@{-}`l/4pt [1,-1] [1,-1] \ar@{-}`r [1,1] [1,1]&&&&\ar@{-}`l/4pt [1,-1] [1,-1] \ar@{-}`r [1,1] [1,1]&&&\ar@{-}[2,2]\\
&\ar@{-}[2,0]&&\ar@{-}[1,1]+<-0.1pc,0.1pc> && \ar@{-}[2,-2]&&\ar@{-}[2,2]&&\\
&&&&&&&&&&&\ar@{-}[3,0]\\
&\ar@{-}[1,1]+<-0.1pc,0.1pc> && \ar@{-}[2,-2] && \ar@{-}[-1,-1]+<0.1pc,-0.1pc>\ar@{-}[1,1]&&&&\ar@{-}[2,0]&&\\
&&&&&&\ar@{-}[1,0]&&&&&\\
&\ar@{-}[1,0]+<0pc,-0.5pt>&&\ar@{-}[-1,-1]+<0.1pc,-0.1pc>\ar@{-}[1,0]+<0pc,-0.5pt> &&&\ar@{-}`l/4pt [1,-1] [1,-1] \ar@{-}`r [1,1] [1,1] &&&\ar@/^0.1pc/ @{-}[2,2] \ar@/_0.1pc/ @{-}[2,2]&& \ar@/^0.1pc/ @{-}[2,-2]\ar@/_0.1pc/ @{-}[2,-2]\\
&&&&&\ar@{-}[1,0]+<0pc,-0.5pt>&&\ar@{-}[4,0]&&&&\\
&*+[o]+<0.37pc>[F]{\ov{\gamma}}\ar@{-}[2,0]+<0pc,0pc>&&*+[o]+<0.37pc>[F]{\overline{\mathbf{u}}} \ar@{-}[2,0]+<0pc,0pc> &&&&&&\ar@{-}[1,0]&&\ar@{-}[1,0]\\
&&&&&*+[o]+<0.4pc>[F]{\mathbf{u}} \ar@{-}[4,0]+<0pc,0pc>&&&&\ar@{-}[2,2]&&\ar@{-}[2,-2]\\
&\ar@{-}`d/4pt [1,1] `[0,2] [0,2] &&&&&&&&&&\\
&&\ar@{-}[1,0]&&&&&\ar@{-}[2,2]&&\ar@{-}[4,-4] &&\ar@{-}[3,0]\\
&& \ar@{-}`d/4pt [1,-2][1,-2]&&&&&&&&&\\
&&&&&\ar@{-}`d/4pt [1,-5][1,-5]&&&&\ar@{-}[1,0]&&\\
&&&&&&&&&\ar@{-}[1,1]+<-0.125pc,0.0pc> \ar@{-}[1,1]+<0.0pc,0.125pc>&&\ar@{-}[2,-2]\\
&&&&&\ar@{-}[1,-3]&&&&&&\\
\ar@{-}[2,2]&&\ar@{-}[2,-2]&&&&&&&\ar@{-}[7,0] &&\ar@{-}[-1,-1]+<0.125pc,0.0pc>\ar@{-}[-1,-1]+<0.0pc,-0.125pc>\ar@{-}[7,0]\\
&&&&&&&&&&&\\
\ar@{-}[1,0]+<0pc,-0.5pt>&&\ar@{-}[5,0]&&&&&&&&&\\
&&&&&&&&&&&\\
*+[o]+<0.40pc>[F]{\gamma}\ar@{-}[1,0]+<0pc,0pc>&&&&&&&&&&&\\
\ar@{-}`d/4pt [1,2][1,2]&&&&&&&&&&&\\
&&&&&&&&&&&\\
&&&&&&&&&&&
}}
\grow{\xymatrix@!0{
\\\\\\\\\\\\\\\\\\\\\\\\\\\\
\save\go+<0pt,0pt>\Drop{\txt{$=$}}\restore
}}
\grow{\xymatrix@!0{
\save\go+<0pt,3pt>\Drop{M}\restore \ar@{-}[21,0] &&&& \save\go+<0pt,3pt>\Drop{D}\restore \ar@{-}[2,2]&&\ar@{-}[1,-1]+<0.125pc,0.0pc> \ar@{-}[1,-1]+<0.0pc,0.125pc> \save\go+<0pt,3pt>\Drop{C}\restore && \save\go+<0pt,3pt>\Drop{H}\restore\ar@{-}[2,0]\\
&&&&&&&&&\\
&&&&\ar@{-}[2,-2]\ar@{-}[-1,1]+<-0.125pc,0.0pc> \ar@{-}[-1,1]+<0.0pc,-0.125pc> &&\ar@{-}[2,2]&&\ar@{-}[1,-1]+<0.125pc,0.0pc> \ar@{-}[1,-1]+<0.0pc,0.125pc>&\\
&&&&&&&&\\
&&\ar@{-}[2,0]&&&&\ar@{-}[1,0]\ar@{-}[-1,1]+<-0.125pc,0.0pc> \ar@{-}[-1,1]+<0.0pc,-0.125pc>&&\ar@{-}[4,4]\\
&&&&&&\ar@{-}`l/4pt [1,-2] [1,-2] \ar@{-}`r [1,2] [1,2]&&&\\
&&\ar@{-}[1,1]+<-0.1pc,0.1pc> && \ar@{-}[2,-2]&&&&\ar@{-}[1,0]&\\
&&&&&&&&\ar@{-}[2,2]&\\
&&\ar@{-}[1,0]+<0pc,-0.5pt> &&\ar@{-}[-1,-1]+<0.1pc,-0.1pc>\ar@{-}[1,2]+<0.2pc,0pc>&&&&&&&&\ar@{-}[2,0]\\
&&&&&&\save[]+<0.2pc,0pc> \Drop{}\ar@{-}[1,0]+<0.2pc,0pc>\restore&&&&\ar@{-}[1,0]&&\\
&&*+[o]+<0.37pc>[F]{\ov{\gamma}}\ar@{-}[1,0]+<0pc,0pc>&&&&\save[]+<0.2pc,0pc> \Drop{}\ar@{-}`l/4pt [1,-1] [1,-1] \ar@{-}`r [1,2] [1,2]\restore &&&& \ar@/^0.1pc/ @{-}[2,2] \ar@/_0.1pc/ @{-}[2,2]&& \ar@/^0.1pc/ @{-}[2,-2]\ar@/_0.1pc/ @{-}[2,-2]\\
&&\ar@{-}`d/4pt [1,-2][1,-2]&&&\ar@{-}[1,0]&&&\ar@{-}[4,0]&\\
&&&&&\ar@{-}`l/4pt [1,-1] [1,-1] \ar@{-}`r [1,1] [1,1]&&&&&\ar@{-}[1,0]&&\ar@{-}[1,0]\\
&&&&\ar@{-}[1,0]+<0pc,-0.5pt> &&\ar@{-}[1,0]+<0pc,-0.5pt> &&&&\ar@{-}[2,2]&&\ar@{-}[2,-2]\\
&&&&&&&&&\\
&&&&*+[o]+<0.37pc>[F]{\overline{\mathbf{u}}} \ar@{-}[2,0]+<0pc,0pc>&& *+[o]+<0.4pc>[F]{\mathbf{u}} \ar@{-}[2,0]+<0pc,0pc>&&\ar@{-}[2,2] &&\ar@{-}[2,-2]&&\ar@{-}[2,0]\\
&&&&&&&&&\\
&&&&\ar@{-}`d/4pt [1,1] `[0,2] [0,2] &&&&\ar@{-}[3,-2]&&
\ar@{-}[1,1]+<-0.125pc,0.0pc> \ar@{-}[1,1]+<0.0pc,0.125pc>&&\ar@{-}[2,-2]\\
&&&&&\ar@{-}[1,0]&&&&&&&\\
&&&&&\ar@{-}`d/4pt [1,-5][1,-5]&&&&&\ar@{-}[9,0]&& \ar@{-}[-1,-1]+<0.125pc,0.0pc>\ar@{-}[-1,-1]+<0.0pc,-0.125pc>\ar@{-}[9,0]\\
&&&&&&\ar@{-}[1,-4] &&&\\
\ar@{-}[2,2]&&\ar@{-}[2,-2]&&&&&&&\\
&&&&&&&&&\\
\ar@{-}[1,0]+<0pc,-0.5pt>&&\ar@{-}[5,0]&&&&&&&\\
&&&&&&&&&\\
*+[o]+<0.40pc>[F]{\gamma}\ar@{-}[1,0]+<0pc,0pc>&&&&&&&&&\\
\ar@{-}`d/4pt [1,2][1,2]&&&&&&&&&\\
&&&&&&&&&\\
&&&&&&&&&&&&
}}
\grow{\xymatrix@!0{
\\\\\\\\\\\\\\\\\\\\\\\\\\\\
\save\go+<0pt,0pt>\Drop{\txt{$=$}}\restore
}}
\grow{\xymatrix@!0{
\save\go+<0pt,3pt>\Drop{M}\restore \ar@{-}[13,0]&&&\save\go+<0pt,3pt>\Drop{D}\restore \ar@{-}[2,2] &&\save\go+<0pt,3pt>\Drop{C}\restore \ar@{-}[1,-1]+<0.125pc,0.0pc> \ar@{-}[1,-1]+<0.0pc,0.125pc> &&& \save\go+<0.2pc,3pt>\Drop{H}\restore \save\go+<0.2pc,0pt>\Drop{}\ar@{-}[1,0]+<0.2pc,0pc>\restore &&\\
&&&&&&&&\save[]+<0.2pc,0pc> \Drop{}\ar@{-}`l/4pt [1,-1] [1,-1] \ar@{-}`r [1,2] [1,2]\restore&&\\
&&&\ar@{-}[1,-1]\ar@{-}[-1,1]+<-0.125pc,0.0pc> \ar@{-}[-1,1]+<0.0pc,-0.125pc>&&\ar@{-}[2,2]&& \ar@{-}[1,-1]+<0.125pc,0.0pc> \ar@{-}[1,-1]+<0.0pc,0.125pc>&&&\ar@{-}[6,0]\\
&&\ar@{-}[3,0]&&&&&&&&\\
&&&&&\ar@{-}[1,0]\ar@{-}[-1,1]+<-0.125pc,0.0pc> \ar@{-}[-1,1]+<0.0pc,-0.125pc>&&\ar@{-}[1,1] &&&\\
&&&&&\ar@{-}`l/4pt [1,-1] [1,-1] \ar@{-}`r [1,1] [1,1]&&&\ar@{-}[3,0]&&\\
&&\ar@{-}[1,1]+<-0.1pc,0.1pc> && \ar@{-}[2,-2]&&\ar@{-}`d/4pt [1,2][1,2] &&&&\\
&&&&&&&&&&\\
&&\ar@{-}[1,0]+<0pc,-0.5pt>&&\ar@{-}[-1,-1]+<0.1pc,-0.1pc>\ar@{-}[4,4]&&&&\ar@{-}[2,2]&&\ar@{-}[2,-2]\\
&&&&&&&&&&\\
&&*+[o]+<0.37pc>[F]{\bar{\boldsymbol{\gamma}}}\ar@{-}[1,0]&&&&&&\ar@{-}[2,-2]&&\ar@{-}[2,0]\\
&&\ar@{-}`d/4pt [1,-2][1,-2]&&&&&&&&\\
&&&&&&\ar@{-}[1,-4]&&\ar@{-}[1,1]+<-0.125pc,0.0pc> \ar@{-}[1,1]+<0.0pc,0.125pc>&&\ar@{-}[2,-2]\\
\ar@{-}[2,2]&&\ar@{-}[2,-2]&&&&&&&&\\
&&&&&&&&\ar@{-}[6,0]&&\ar@{-}[-1,-1]+<0.125pc,0.0pc> \ar@{-}[-1,-1]+<0.0pc,-0.125pc> \ar@{-}[6,0]\\
\ar@{-}[1,0]+<0pc,-0.5pt>&&\ar@{-}[5,0]&&&&&&&&\\
&&&&&&&&&&\\
*+[o]+<0.37pc>[F]{\boldsymbol{\gamma}}\ar@{-}[1,0]&&&&&&&&&&\\
\ar@{-}`d/4pt [1,2][1,2]&&&&&&&&&&\\
&&&&&&&&&&\\
&&&&&&&&&&
}}\\
&
\grow{\xymatrix@!0{
\\\\\\\\\\\\\\\\\\\\
\save\go+<0pt,0pt>\Drop{\txt{$=$}}\restore
}}
\grow{\xymatrix@!0{
\save\go+<0pt,3pt>\Drop{M}\restore \ar@{-}[13,0]&&&\save\go+<0pt,3pt>\Drop{D}\restore \ar@{-}[2,2] &&\save\go+<0pt,3pt>\Drop{C}\restore \ar@{-}[1,-1]+<0.125pc,0.0pc> \ar@{-}[1,-1]+<0.0pc,0.125pc> &&& \save\go+<0.2pc,3pt>\Drop{H}\restore \save\go+<0.2pc,0pt>\Drop{}\ar@{-}[1,0]+<0.2pc,0pc>\restore &&\\
&&&&&&&&\save[]+<0.2pc,0pc> \Drop{}\ar@{-}`l/4pt [1,-1] [1,-1] \ar@{-}`r [1,2] [1,2]\restore&&\\
&&&\ar@{-}[1,-1]\ar@{-}[-1,1]+<-0.125pc,0.0pc> \ar@{-}[-1,1]+<0.0pc,-0.125pc>&&\ar@{-}[2,2]&& \ar@{-}[1,-1]+<0.125pc,0.0pc> \ar@{-}[1,-1]+<0.0pc,0.125pc>&&&\ar@{-}[10,0]\\
&&\ar@{-}[3,0]&&&&&&&&\\
&&&&&\ar@{-}[1,0]\ar@{-}[-1,1]+<-0.125pc,0.0pc> \ar@{-}[-1,1]+<0.0pc,-0.125pc>&&\ar@{-}[1,1]&&&\\
&&&&&\ar@{-}`l/4pt [1,-1] [1,-1] \ar@{-}`r [1,1] [1,1]&&&\ar@{-}[5,0]&&\\
&&\ar@{-}[1,1]+<-0.1pc,0.1pc> && \ar@{-}[2,-2]&&\ar@{-}[2,0]&&&&\\
&&&&&&&&&&\\
&&\ar@{-}[1,0]+<0pc,-0.5pt>&&\ar@{-}[-1,-1]+<0.1pc,-0.1pc>\ar@{-}[1,1]+<-0.1pc,0.1pc> && \ar@{-}[2,-2]&&&&\\
&&&&&&&&&&\\
&&*+[o]+<0.37pc>[F]{\bar{\boldsymbol{\gamma}}} \ar@{-}[1,0]&&\ar@{-}[2,0]&&\ar@{-}[-1,-1]+<0.1pc,-0.1pc> \ar@{-}[1,1]+<-0.125pc,0.0pc> \ar@{-}[1,1]+<0.0pc,0.125pc>&&\ar@{-}[2,-2] &&\\
&&\ar@{-}`d/4pt [1,-2][1,-2]&&&&&&&&\\
&&&&\ar@{-}`d/4pt [1,2][1,2]&&\ar@{-}[2,0]&&\ar@{-}[-1,-1]+<0.125pc,0.0pc>\ar@{-}[-1,-1]+<0.0pc,-0.125pc> \ar@{-}[2,2]&&\ar@{-}[2,-2]\\
\ar@{-}[3,4]&&&&&&&&&&\\
&&&&&&\ar@{-}[2,2]&&\ar@{-}[2,-2]&&\ar@{-}[9,0]\\
&&&&&&&&&&\\
&&&&\ar@{-}[2,2]&&\ar@{-}[2,-2]&&\ar@{-}[7,0]&&\\
&&&&&&&&&&\\
&&&&\ar@{-}[1,0]+<0pc,-0.5pt>&&\ar@{-}[5,0]&&&&\\
&&&&&&&&&&\\
&&&&*+[o]+<0.37pc>[F]{\boldsymbol{\gamma}}\ar@{-}[1,0]&&&&&&\\
&&&&\ar@{-}`d/4pt [1,2][1,2]\ar@{-}`d/4pt [1,2][1,2]&&&&&&\\
&&&&&&&&&&\\
&&&&&&&&&&
}}
\grow{\xymatrix@!0{
\\\\\\\\\\\\\\\\\\\\
\save\go+<0pt,0pt>\Drop{\txt{$=$}}\restore
}}
\grow{\xymatrix@!0{\\
\save\go+<0pt,3pt>\Drop{M}\restore \ar@{-}[14,0]&&&&\save\go+<0pt,3pt>\Drop{D}\restore  \ar@{-}[2,2]&&\save\go+<0pt,3pt>\Drop{C}\restore \ar@{-}[1,-1]+<0.125pc,0.0pc> \ar@{-}[1,-1]+<0.0pc,0.125pc> &&&\save\go+<0pt,3pt>\Drop{H}\restore \ar@{-}[1,0]&\\
&&&&&&&&&\ar@{-}`l/4pt [1,-1] [1,-1] \ar@{-}`r [1,1] [1,1]&\\
&&&&\ar@{-}[2,0]\ar@{-}[-1,1]+<-0.125pc,0.0pc> \ar@{-}[-1,1]+<0.0pc,-0.125pc>&&\ar@{-}[2,2]&& \ar@{-}[1,-1]+<0.125pc,0.0pc> \ar@{-}[1,-1]+<0.0pc,0.125pc>&&\ar@{-}[9,0]\\
&&&&&&&&&&\\
&&&&\ar@{-}[1,1]+<-0.1pc,0.1pc> && \ar@{-}[2,-2]\ar@{-}[-1,1]+<-0.125pc,0.0pc> \ar@{-}[-1,1]+<0.0pc,-0.125pc>&&\ar@{-}[2,0]&&\\
&&&&&&&&&&\\
&&&&\ar@{-}[1,-1]&&\ar@{-}[-1,-1]+<0.1pc,-0.1pc>\ar@{-}[1,1]+<-0.125pc,0.0pc> \ar@{-}[1,1]+<0.0pc,0.125pc>&&\ar@{-}[2,-2]&&\\
&&&\ar@{-}[1,0]&&&&&&&\\
&&&\ar@{-}`l/4pt [1,-1] [1,-1] \ar@{-}`r [1,1] [1,1]&&&\ar@{-}[5,0]&&\ar@{-}[3,0] \ar@{-}[-1,-1]+<0.125pc,0.0pc>\ar@{-}[-1,-1]+<0.0pc,-0.125pc>&&\\
&&\ar@{-}[1,0]+<0pc,-0.5pt>&&\ar@{-}`d/4pt [1,2][1,2]&&&&&&\\
&&&&&&&&&&\\
&&*+[o]+<0.37pc>[F]{\bar{\boldsymbol{\gamma}}} \ar@{-}[1,0]&&&&&&\ar@{-}[2,2]&&\ar@{-}[6,-6]\\
&&\ar@{-}`d/4pt [1,-2][1,-2]&&&&&&&&\\
&&&&&&\ar@{-}[2,2]&&&&\ar@{-}[9,0]\\
\ar@{-}[1,4]&&&&&&&&&&\\
&&&&\ar@{-}[2,2]&&&&\ar@{-}[7,0]&&\\
&&&&&&&&&&\\
&&&&\ar@{-}[1,0]+<0pc,-0.5pt>&&\ar@{-}[5,0]&&&&\\
&&&&&&&&&&\\
&&&&*+[o]+<0.37pc>[F]{\boldsymbol{\gamma}}\ar@{-}[1,0]&&&&&&\\
&&&&\ar@{-}`d/4pt [1,2][1,2]&&&&&&\\
&&&&&&&&&&\\
&&&&&&&&&&
}}
\grow{\xymatrix@!0{
\\\\\\\\\\\\\\\\\\\\
\save\go+<0pt,0pt>\Drop{\txt{$=$}}\restore
}}
\grow{\xymatrix@!0{\\
\save\go+<0pt,3pt>\Drop{M}\restore \ar@{-}[14,0]&&&&\save\go+<0pt,3pt>\Drop{D}\restore  \ar@{-}[4,0]&&\save\go+<0pt,3pt>\Drop{C}\restore \ar@{-}[2,0]&&&\save\go+<0pt,3pt>\Drop{H}\restore \ar@{-}[1,0]&\\
&&&&&&&&&\ar@{-}`l/4pt [1,-1] [1,-1] \ar@{-}`r [1,1] [1,1]&\\
&&&&&&\ar@{-}[1,1]+<-0.1pc,0.1pc> && \ar@{-}[2,-2]&&\ar@{-}[7,0]\\
&&&&&&&&&&\\
&&&&\ar@{-}[2,2]&&\ar@{-}[1,-1]+<0.125pc,0.0pc> \ar@{-}[1,-1]+<0.0pc,0.125pc> &&\ar@{-}[-1,-1]+<0.1pc,-0.1pc>\ar@{-}[5,0]&&\\
&&&&&&&&&&\\
&&&&\ar@{-}[1,-1]\ar@{-}[-1,1]+<-0.125pc,0.0pc> \ar@{-}[-1,1]+<0.0pc,-0.125pc>&&\ar@{-}[5,0]&&&&\\
&&&\ar@{-}[1,0]&&&&&&&\\
&&&\ar@{-}`l/4pt [1,-1] [1,-1] \ar@{-}`r [1,1] [1,1]&&&&&&&\\
&&\ar@{-}[1,0]+<0pc,-0.5pt>&&\ar@{-}`d/4pt [1,2][1,2]&&&&\ar@{-}[2,2]&&\ar@{-}[4,-4]\\
&&&&&&&&&&\\
&&*+[o]+<0.37pc>[F]{\bar{\boldsymbol{\gamma}}} \ar@{-}[1,0]&&&&\ar@{-}[2,2]&&&&\ar@{-}[10,0]\\
&&\ar@{-}`d/4pt [1,-2][1,-2]&&&&&&&&\\
&&&&&&\ar@{-}[1,-4]&&\ar@{-}[8,0]&&\\
\ar@{-}[2,2]&&\ar@{-}[2,-2]&&&&&&&&\\
&&&&&&&&&&\\
\ar@{-}[1,0]+<0pc,-0.5pt>&&\ar@{-}[5,0]&&&&&&&&\\
&&&&&&&&&&\\
*+[o]+<0.37pc>[F]{\boldsymbol{\gamma}}\ar@{-}[1,0]&&&&&&&&&&\\
\ar@{-}`d/4pt [1,2][1,2]&&&&&&&&&&\\
&&&&&&&&&&\\
&&&&&&&&&&
}}
\grow{\xymatrix@!0{
\\\\\\\\\\\\\\\\\\\\
\save\go+<0pt,0pt>\Drop{\txt{,}}\restore
}}
\end{align*}
which finish the proof.\qed

\begin{lemma}\label{mengano1} We have
$$
%
%
\grow{\xymatrix@!0{\\
\save\go+<0pt,3pt>\Drop{C}\restore \ar@{-}[2,2] && \save\go+<0pt,3pt>\Drop{C'}\restore \ar@{-}`d/4pt [1,2][1,2]  && \save\go+<0pt,3pt>\Drop{D}\restore \ar@{-}[2,0]\\
&&&&\\
&&\ar@{-}[1,1]+<-0.125pc,0.0pc> \ar@{-}[1,1]+<0.0pc,0.125pc>&&\ar@{-}[2,-2]\\
&&&&\\
&&&&\ar@{-}[-1,-1]+<0.125pc,0.0pc>\ar@{-}[-1,-1]+<0.0pc,-0.125pc>
}}
\grow{\xymatrix@!0{
\\\\
\save\go+<0pt,0pt>\Drop{\txt{$=$}}\restore
}}
\grow{\xymatrix@!0{
\save\go+<0pt,3pt>\Drop{C}\restore  \ar@{-}[1,1]+<-0.1pc,0.1pc> && \save\go+<0pt,3pt>\Drop{C'}\restore \ar@{-}[2,-2]  && \save\go+<0pt,3pt>\Drop{D}\restore \ar@{-}[2,0]\\
&&&&\\
\ar@{-}[2,0] &&\ar@{-}[-1,-1]+<0.1pc,-0.1pc> \ar@{-}[1,1]+<-0.125pc,0.0pc> \ar@{-}[1,1]+<0.0pc,0.125pc>&&\ar@{-}[2,-2]\\
&&&&\\
\ar@{-}`d/4pt [1,2][1,2] &&\ar@{-}[2,0]&&\ar@{-}[-1,-1]+<0.125pc,0.0pc>\ar@{-}[-1,-1]+<0.0pc,-0.125pc>\ar@{-}[2,0]\\
&&&&\\
&&&&
}}
\grow{\xymatrix@!0{
\\\\
\save\go+<0pt,0pt>\Drop{\txt{,}}\restore
}}
$$
where $C:= H^{\ot_c^s}$, $C':= H^{\ot_c^{s'}}$ and $D:=A^{\ot^r}$.
\end{lemma}

\begin{proof} For $s=s'=r=1$ the result is valid by definition. An inductive argument using
$$
\grow{\xymatrix@!0{\\
\save\go+<0pt,3pt>\Drop{H}\restore \ar@{-}[1,0] && \save\go+<0pt,3pt>\Drop{C'}\restore \ar@{-}[1,0] && \save\go+<0pt,3pt>\Drop{H}\restore \ar@{-}`d/4pt [1,2][1,2] &&\save\go+<0pt,3pt>\Drop{A}\restore  \ar@{-}[4,0]\\
\ar@{-}[3,4]&&\ar@{-}`d/4pt [1,4][1,4]&&&&\\
&&&&&&\\
&&&&&&\\
&&&&\ar@{-}[1,1]+<-0.125pc,0.0pc> \ar@{-}[1,1]+<0.0pc,0.125pc>&&\ar@{-}[2,-2]\\
&&&&&&\\
&&&&&&\ar@{-}[-1,-1]+<0.125pc,0.0pc>\ar@{-}[-1,-1]+<0.0pc,-0.125pc>
}}
\grow{\xymatrix@!0{
\\\\\\\\
\save\go+<0pt,0pt>\Drop{\txt{=}}\restore
}}
\grow{\xymatrix@!0{\\
\save\go+<0pt,3pt>\Drop{H}\restore \ar@{-}[1,1]+<-0.1pc,0.1pc>&& \save\go+<0pt,3pt>\Drop{C'}\restore \ar@{-}[2,-2]&& \save\go+<0pt,3pt>\Drop{H}\restore \ar@{-}`d/4pt [1,2][1,2] &&\save\go+<0pt,3pt>\Drop{A}\restore  \ar@{-}[2,0]\\
&&&&&&\\
\ar@{-}[3,0]&&\ar@{-}[-1,-1]+<0.1pc,-0.1pc>\ar@{-}[1,1]&&&&\ar@{-}[1,-1]\\
&&&\ar@{-}[1,1]+<-0.125pc,0.0pc> \ar@{-}[1,1]+<0.0pc,0.125pc>&&\ar@{-}[2,-2]&\\
&&&&&&\\
\ar@{-}`d/4pt [1,3][1,3]&&&\ar@{-}[2,0]&&\ar@{-}[-1,-1]+<0.125pc,0.0pc>\ar@{-}[-1,-1]+<0.0pc,-0.125pc>\ar@{-}[2,0]&\\
&&&&&&\\
&&&&&&
}}
\grow{\xymatrix@!0{
\\\\\\\\
\save\go+<0pt,0pt>\Drop{\txt{=}}\restore
}}
\grow{\xymatrix@!0{
\save\go+<0pt,3pt>\Drop{H}\restore \ar@{-}[1,1]+<-0.1pc,0.1pc>&& \save\go+<0pt,3pt>\Drop{C'}\restore \ar@{-}[2,-2]&& \save\go+<0pt,3pt>\Drop{H}\restore \ar@{-}[2,0] &&\save\go+<0pt,3pt>\Drop{A}\restore  \ar@{-}[4,0]\\
&&&&&&\\
\ar@{-}[5,0]&&\ar@{-}[-1,-1]+<0.1pc,-0.1pc> \ar@{-}[1,1]+<-0.1pc,0.1pc> && \ar@{-}[2,-2]&&\\
&&&&&&\\
&&\ar@{-}[2,0]&&\ar@{-}[-1,-1]+<0.1pc,-0.1pc>\ar@{-}[1,1]+<-0.125pc,0.0pc> \ar@{-}[1,1]+<0.0pc,0.125pc>&& \ar@{-}[2,-2]\\
&&&&&&\\
&&\ar@{-}`d/4pt [1,2][1,2]&&\ar@{-}[3,0]&&\ar@{-}[-1,-1]+<0.125pc,0.0pc>\ar@{-}[-1,-1]+<0.0pc,-0.125pc>\ar@{-}[3,0]\\
\ar@{-}`d/4pt [1,4][1,4]&&&&&&\\
&&&&&&\\
&&&&&&
}}
\grow{\xymatrix@!0{
\\\\\\\\
\save\go+<0pt,0pt>\Drop{\txt{,}}\restore
}}
$$
shows that the result is valid when $s=r=1$ and $s'\in \mathds{N}$. A similar argument using the equality
$$
\grow{\xymatrix@!0{\\
\save\go+<0pt,3pt>\Drop{H}\restore \ar@{-}[2,0] && \save\go+<0pt,3pt>\Drop{C}\restore \ar@{-}[1,0] && \save\go+<0pt,3pt>\Drop{C'}\restore \ar@{-}`d/4pt [1,2][1,2] &&\save\go+<0pt,3pt>\Drop{A}\restore  \ar@{-}[2,0]\\
&&\ar@{-}[1,2]&&&&\\
\ar@{-}[2,2]&&&&\ar@{-}[1,1]+<-0.125pc,0.0pc> \ar@{-}[1,1]+<0.0pc,0.125pc>&&\ar@{-}[2,-2]\\
&&&&&&\\
&&\ar@{-}[1,1]+<-0.125pc,0.0pc> \ar@{-}[1,1]+<0.0pc,0.125pc>&&\ar@{-}[2,-2]&&\ar@{-}[-1,-1]+<0.125pc,0.0pc>\ar@{-}[-1,-1]+<0.0pc,-0.125pc>\ar@{-}[2,0]\\
&&&&&&\\
&&&&\ar@{-}[-1,-1]+<0.125pc,0.0pc>\ar@{-}[-1,-1]+<0.0pc,-0.125pc>&&
}}
\grow{\xymatrix@!0{
\\\\\\
\save\go+<0pt,0pt>\Drop{\txt{$=$}}\restore
}}
\grow{\xymatrix@!0{
\save\go+<0pt,3pt>\Drop{H}\restore \ar@{-}[2,0] && \save\go+<0pt,3pt>\Drop{C}\restore \ar@{-}[1,1]+<-0.1pc,0.1pc> &&  \save\go+<0pt,3pt>\Drop{C'}\restore  \ar@{-}[2,-2] &&\save\go+<0pt,3pt>\Drop{A}\restore  \ar@{-}[2,0]\\
&&&&&&\\
\ar@{-}[1,1]+<-0.1pc,0.1pc> && \ar@{-}[2,-2]&&\ar@{-}[-1,-1]+<0.1pc,-0.1pc>\ar@{-}[1,1]+<-0.125pc,0.0pc> \ar@{-}[1,1]+<0.0pc,0.125pc>&&\ar@{-}[2,-2]\\
&&&&&&\\
\ar@{-}[2,0]&&\ar@{-}[-1,-1]+<0.1pc,-0.1pc>\ar@{-}[1,1]+<-0.125pc,0.0pc> \ar@{-}[1,1]+<0.0pc,0.125pc>&&\ar@{-}[2,-2] &&\ar@{-}[-1,-1]+<0.125pc,0.0pc>\ar@{-}[-1,-1]+<0.0pc,-0.125pc>\ar@{-}[4,0]\\
&&&&&&\\
\ar@{-}`d/4pt [1,2][1,2]&&\ar@{-}[2,0]&&\ar@{-}[-1,-1]+<0.125pc,0.0pc>\ar@{-}[-1,-1]+<0.0pc,-0.125pc>\ar@{-}[2,0]&&\\
&&&&&&\\
&&&&&&
}}
\grow{\xymatrix@!0{
\\\\\\
\save\go+<0pt,0pt>\Drop{\txt{,}}\restore
}}
$$
shows that the result is valid when $r=1$ and $s,s'\in \mathds{N}$. Finally, again an inductive argument using
$$
\grow{\xymatrix@!0{\\
\save\go+<0pt,3pt>\Drop{C}\restore \ar@{-}[5,0] && \save\go+<0pt,3pt>\Drop{C'}\restore \ar@{-}[1,0] &&&  \save\go+<0pt,3pt>\Drop{D}\restore  \ar@{-}[2,0] &&\save\go+<0pt,3pt>\Drop{A}\restore  \ar@{-}[6,0]\\
&&\ar@{-}`l/4pt [1,-1] [1,-1] \ar@{-}`r [1,1] [1,1]&&&&&\\
&\ar@{-}[2,0]&&\ar@{-}[1,1]+<-0.125pc,0.0pc> \ar@{-}[1,1]+<0.0pc,0.125pc>&&\ar@{-}[2,-2]&&\\
&&&&&&&\\
&\ar@{-}`d/4pt [1,2][1,2]&&\ar@{-}[2,0]&&\ar@{-}[-1,-1]+<0.125pc,0.0pc>\ar@{-}[-1,-1]+<0.0pc,-0.125pc> \ar@{-}`d/4pt [1,2][1,2]&&\\
\ar@{-}[1,1]&&&&&&&\\
&\ar@{-}[1,1]+<-0.125pc,0.0pc> \ar@{-}[1,1]+<0.0pc,0.125pc>&&\ar@{-}[2,-2]&&&&\ar@{-}[4,-4]\\
&&&&&&&\\
&\ar@{-}[2,0]&&\ar@{-}[-1,-1]+<0.125pc,0.0pc>\ar@{-}[-1,-1]+<0.0pc,-0.125pc>\ar@{-}[1,1]+<-0.125pc,0.0pc> \ar@{-}[1,1]+<0.0pc,0.125pc>&&&&\\
&&&&&&&\\
&&&&&\ar@{-}[-1,-1]+<0.125pc,0.0pc>\ar@{-}[-1,-1]+<0.0pc,-0.125pc>&&
}}
\grow{\xymatrix@!0{
\\\\\\\\\\
\save\go+<0pt,0pt>\Drop{\txt{$=$}}\restore
}}
\grow{\xymatrix@!0{\\
\save\go+<0pt,3pt>\Drop{C}\restore \ar@{-}[2,0] &&& \save\go+<0pt,3pt>\Drop{C'}\restore \ar@{-}[1,0] &&&  \save\go+<0pt,3pt>\Drop{D}\restore  \ar@{-}[2,0] &&\save\go+<0pt,3pt>\Drop{A}\restore  \ar@{-}[6,0]\\
&&&\ar@{-}`l/4pt [1,-1] [1,-1] \ar@{-}`r [1,1] [1,1]&&&&&\\
\ar@{-}[1,1]+<-0.1pc,0.1pc> && \ar@{-}[2,-2]&&\ar@{-}[1,1]+<-0.125pc,0.0pc> \ar@{-}[1,1]+<0.0pc,0.125pc>&&\ar@{-}[2,-2]&&\\
&&&&&&&&\\
\ar@{-}[2,0]&&\ar@{-}[-1,-1]+<0.1pc,-0.1pc>\ar@{-}[1,1]+<-0.125pc,0.0pc> \ar@{-}[1,1]+<0.0pc,0.125pc>&&\ar@{-}[2,-2] &&\ar@{-}[-1,-1]+<0.125pc,0.0pc>\ar@{-}[-1,-1]+<0.0pc,-0.125pc>\ar@{-}`d/4pt [1,2][1,2]&&\\
&&&&&&&&\\
\ar@{-}`d/4pt [1,2][1,2]&&\ar@{-}[3,0]&&\ar@{-}[-1,-1]+<0.125pc,0.0pc>\ar@{-}[-1,-1]+<0.0pc,-0.125pc> \ar@{-}[1,1]&&&&\ar@{-}[1,-1]\\
&&&&&\ar@{-}[1,1]+<-0.125pc,0.0pc> \ar@{-}[1,1]+<0.0pc,0.125pc>&&\ar@{-}[2,-2]&\\
&&&&&&&&\\
&&&&&&&\ar@{-}[-1,-1]+<0.125pc,0.0pc>\ar@{-}[-1,-1]+<0.0pc,-0.125pc>&
}}
\grow{\xymatrix@!0{
\\\\\\\\\\
\save\go+<0pt,0pt>\Drop{\txt{$=$}}\restore
}}
\grow{\xymatrix@!0{
\save\go+<0pt,3pt>\Drop{C}\restore \ar@{-}[2,0] &&& \save\go+<0pt,3pt>\Drop{C'}\restore \ar@{-}[1,0] &&&  \save\go+<0pt,3pt>\Drop{D}\restore  \ar@{-}[2,0] &&\save\go+<0pt,3pt>\Drop{A}\restore  \ar@{-}[8,0]\\
&&&\ar@{-}`l/4pt [1,-1] [1,-1] \ar@{-}`r [1,1] [1,1]&&&&&\\
\ar@{-}[1,1]+<-0.1pc,0.1pc> && \ar@{-}[2,-2]&&\ar@{-}[1,1]+<-0.125pc,0.0pc> \ar@{-}[1,1]+<0.0pc,0.125pc>&&\ar@{-}[2,-2]&&\\
&&&&&&&&\\
\ar@{-}[2,0]&&\ar@{-}[-1,-1]+<0.1pc,-0.1pc>\ar@{-}[1,1]+<-0.125pc,0.0pc> \ar@{-}[1,1]+<0.0pc,0.125pc>&&\ar@{-}[2,-2] &&\ar@{-}[-1,-1]+<0.125pc,0.0pc>\ar@{-}[-1,-1]+<0.0pc,-0.125pc>\ar@{-}[2,0]&&\\
&&&&&&&&\\
\ar@{-}`d/4pt [1,2][1,2]&&\ar@{-}[6,0]&&\ar@{-}[-1,-1]+<0.125pc,0.0pc>\ar@{-}[-1,-1]+<0.0pc,-0.125pc> \ar@{-}[1,1]+<-0.1pc,0.1pc> && \ar@{-}[2,-2]&&\\
&&&&&&&&\\
&&&&\ar@{-}[2,0]&&\ar@{-}[-1,-1]+<0.1pc,-0.1pc> \ar@{-}[1,1]+<-0.125pc,0.0pc> \ar@{-}[1,1]+<0.0pc,0.125pc>&&\ar@{-}[2,-2]\\
&&&&&&&&\\
&&&&\ar@{-}`d/4pt [1,2][1,2]&&\ar@{-}[2,0]&&\ar@{-}[-1,-1]+<0.125pc,0.0pc>\ar@{-}[-1,-1]+<0.0pc,-0.125pc> \ar@{-}[2,0]\\
&&&&&&&&\\
&&&&&&&&
}}
\grow{\xymatrix@!0{
\\\\\\\\\\
\save\go+<0pt,0pt>\Drop{\txt{$=$}}\restore
}}
\grow{\xymatrix@!0{
\save\go+<0pt,3pt>\Drop{C}\restore \ar@{-}[2,0] &&& \save\go+<0pt,3pt>\Drop{C'}\restore \ar@{-}[1,0] &&&  \save\go+<0pt,3pt>\Drop{D}\restore  \ar@{-}[6,0] &&\save\go+<0pt,3pt>\Drop{A}\restore  \ar@{-}[8,0]\\
&&&\ar@{-}`l/4pt [1,-1] [1,-1] \ar@{-}`r [1,1] [1,1]&&&&&\\
\ar@{-}[1,1]+<-0.1pc,0.1pc> && \ar@{-}[2,-2]&&\ar@{-}[2,0]&&&&\\
&&&&&&&&\\
\ar@{-}[6,0]&&\ar@{-}[-1,-1]+<0.1pc,-0.1pc>\ar@{-}[1,1]+<-0.1pc,0.1pc> && \ar@{-}[2,-2]&&&&\\
&&&&&&&&\\
&&\ar@{-}[2,0]&&\ar@{-}[-1,-1]+<0.1pc,-0.1pc>\ar@{-}[1,1]+<-0.125pc,0.0pc> \ar@{-}[1,1]+<0.0pc,0.125pc>&&\ar@{-}[2,-2]&&\\
&&&&&&&&\\
&&\ar@{-}[1,1]+<-0.125pc,0.0pc> \ar@{-}[1,1]+<0.0pc,0.125pc>&&\ar@{-}[2,-2]&& \ar@{-}[-1,-1]+<0.125pc,0.0pc>\ar@{-}[-1,-1]+<0.0pc,-0.125pc> \ar@{-}[1,1]+<-0.125pc,0.0pc> \ar@{-}[1,1]+<0.0pc,0.125pc>&&\ar@{-}[2,-2]\\
&&&&&&&&\\
\ar@{-}`d/4pt [1,2][1,2]&&\ar@{-}[2,0]&&\ar@{-}[-1,-1]+<0.125pc,0.0pc>\ar@{-}[-1,-1]+<0.0pc,-0.125pc>\ar@{-}`d/4pt [1,2][1,2]&&\ar@{-}[2,0]&& \ar@{-}[-1,-1]+<0.125pc,0.0pc>\ar@{-}[-1,-1]+<0.0pc,-0.125pc>\ar@{-}[2,0]\\
&&&&&&&&\\
&&&&&&&&
}}
\grow{\xymatrix@!0{
\\\\\\\\\\
\save\go+<0pt,0pt>\Drop{\txt{$=$}}\restore
}}
\grow{\xymatrix@!0{\\
&&\save\go+<0pt,3pt>\Drop{C}\restore \ar@{-}[1,1]+<-0.1pc,0.1pc> && \save\go+<0pt,3pt>\Drop{C'}\restore  \ar@{-}[2,-2]&&  \save\go+<0pt,3pt>\Drop{D}\restore \ar@{-}[2,0]&&\save\go+<0pt,3pt>\Drop{A}\restore \ar@{-}[4,0]\\
&&&&&&&&\\
&&\ar@{-}[1,-1]&&\ar@{-}[-1,-1]+<0.1pc,-0.1pc>\ar@{-}[1,1]+<-0.125pc,0.0pc> \ar@{-}[1,1]+<0.0pc,0.125pc>&&\ar@{-}[2,-2]&&\\
&\ar@{-}[1,0]&&&&&&&\\
&\ar@{-}`l/4pt [1,-1] [1,-1] \ar@{-}`r [1,1] [1,1]&&&\ar@{-}[1,0]&& \ar@{-}[-1,-1]+<0.125pc,0.0pc> \ar@{-}[-1,-1]+<0.0pc,-0.125pc> \ar@{-}[1,1]+<-0.125pc,0.0pc> \ar@{-}[1,1]+<0.0pc,0.125pc>&&\ar@{-}[2,-2]\\
\ar@{-}[2,0]&&\ar@{-}[1,1]+<-0.125pc,0.0pc> \ar@{-}[1,1]+<0.0pc,0.125pc>&&\ar@{-}[2,-2]&&&&\\
&&&&&&\ar@{-}[3,0]&&\ar@{-}[-1,-1]+<0.125pc,0.0pc>\ar@{-}[-1,-1]+<0.0pc,-0.125pc>\ar@{-}[3,0]\\
\ar@{-}`d/4pt [1,2][1,2]&&\ar@{-}[2,0]&&\ar@{-}[-1,-1]+<0.125pc,0.0pc>\ar@{-}[-1,-1]+<0.0pc,-0.125pc>\ar@{-}`d/4pt [1,2][1,2]&&&&\\
&&&&&&&&\\
&&&&&&&&
}}
\grow{\xymatrix@!0{
\\\\\\\\\\
\save\go+<0pt,0pt>\Drop{\txt{,}}\restore
}}
$$
completes the proof.
\end{proof}

\begin{lemma}\label{mengano3} The following equality holds:
$$
%
%
\grow{\xymatrix@!0{\\
\save\go+<0pt,3pt>\Drop{C}\restore \ar@/^0.1pc/ @{-}[2,2] \ar@/_0.1pc/ @{-}[2,2] && \save\go+<0pt,3pt>\Drop{D}\restore \ar@/^0.1pc/ @{-}[2,-2]\ar@/_0.1pc/ @{-}[2,-2]\\
&&&\\
\ar@{-}[2,0]&&\ar@{-}[1,0]&\\
&&\ar@{-}`l/4pt [1,-1] [1,-1] \ar@{-}`r [1,1] [1,1]&\\
&&&
}}
\grow{\xymatrix@!0{
\\\\\\
\save\go+<0pt,0pt>\Drop{\txt{$=$}}\restore
}}
\grow{\xymatrix@!0{
&\save\go+<0pt,3pt>\Drop{C}\restore \ar@{-}[1,0] &&& \save\go+<0pt,3pt>\Drop{D}\restore \ar@{-}[2,0]\\
&\ar@{-}`l/4pt [1,-1] [1,-1] \ar@{-}`r [1,1] [1,1]&&&\\
\ar@{-}[2,0]&&\ar@{-}[1,1]+<-0.125pc,0.0pc> \ar@{-}[1,1]+<0.0pc,0.125pc>&&\ar@{-}[2,-2]\\
&&&&\\
\ar@/^0.1pc/ @{-}[2,2] \ar@/_0.1pc/ @{-}[2,2]&& \ar@/^0.1pc/ @{-}[2,-2]\ar@/_0.1pc/ @{-}[2,-2]&&\ar@{-}[-1,-1]+<0.125pc,0.0pc>\ar@{-}[-1,-1]+<0.0pc,-0.125pc>\ar@{-}[2,0]\\
&&&&\\
&&&&
}}
\grow{\xymatrix@!0{
\\\\\\
\save\go+<0pt,0pt>\Drop{\txt{,}}\restore
}}
$$
where $C:=H^{\ot_c^n}$ and $D:=A^{\ot_k^r}$.
\end{lemma}

\begin{proof} In fact,
$$
%
%
\grow{\xymatrix@!0{\\\\
\save\go+<0pt,3pt>\Drop{C}\restore \ar@/^0.1pc/ @{-}[2,2] \ar@/_0.1pc/ @{-}[2,2] && \save\go+<0pt,3pt>\Drop{D}\restore \ar@/^0.1pc/ @{-}[2,-2]\ar@/_0.1pc/ @{-}[2,-2]\\
&&&\\
\ar@{-}[2,0]&&\ar@{-}[1,0]&\\
&&\ar@{-}`l/4pt [1,-1] [1,-1] \ar@{-}`r [1,1] [1,1]&\\
&&&
}}
\grow{\xymatrix@!0{
\\\\\\\\
\save\go+<0pt,0pt>\Drop{\txt{$=$}}\restore
}}
\grow{\xymatrix@!0{\\
&\save\go+<0pt,3pt>\Drop{C}\restore \ar@{-}[1,0] &&& \save\go+<0pt,3pt>\Drop{D}\restore \ar@{-}[2,0]\\
&\ar@{-}`l/4pt [1,-1] [1,-1] \ar@{-}`r [1,1] [1,1]&&&\\
\ar@{-}[2,0] &&\ar@{-}[1,1]+<-0.125pc,0.0pc> \ar@{-}[1,1]+<0.0pc,0.125pc>&&\ar@{-}[2,-2]\\
&&&&\\
\ar@{-}`d/4pt [1,2][1,2]&&\ar@{-}[2,0] &&\ar@{-}[-1,-1]+<0.125pc,0.0pc>\ar@{-}[-1,-1]+<0.0pc,-0.125pc>\ar@{-}[1,0]\\
&&&&\ar@{-}`l/4pt [1,-1] [1,-1] \ar@{-}`r [1,1] [1,1]\\
&&&&&
}}
\grow{\xymatrix@!0{
\\\\\\\\
\save\go+<0pt,0pt>\Drop{\txt{$=$}}\restore
}}
\grow{\xymatrix@!0{
&&\save\go+<0.2pc,3pt>\Drop{C}\restore \save[]+<0.2pc,0pc> \Drop{}\ar@{-}[1,0]+<0.2pc,0pc>\restore &&&& \save\go+<0pt,3pt>\Drop{D}\restore \ar@{-}[2,0]\\
&&\save[]+<0.2pc,0pc> \Drop{}\ar@{-}`l/4pt [1,-1] [1,-1] \ar@{-}`r [1,2] [1,2]\restore&&&&&\\
&\ar@{-}[1,0]&&&\ar@{-}[1,1]+<-0.125pc,0.0pc> \ar@{-}[1,1]+<0.0pc,0.125pc>&&\ar@{-}[2,-2]\\
&\ar@{-}`l/4pt [1,-1] [1,-1] \ar@{-}`r [1,1] [1,1]&&&&&\\
\ar@{-}[2,0]&&\ar@{-}[1,1]+<-0.125pc,0.0pc> \ar@{-}[1,1]+<0.0pc,0.125pc>&&\ar@{-}[2,-2]&&\ar@{-}[-1,-1]+<0.125pc,0.0pc>\ar@{-}[-1,-1]+<0.0pc,-0.125pc>\ar@{-}[4,0]\\
&&&&&&\\
\ar@{-}`d/4pt [1,2][1,2]&&\ar@{-}[2,0]&&\ar@{-}[-1,-1]+<0.125pc,0.0pc>\ar@{-}[-1,-1]+<0.0pc,-0.125pc>\ar@{-}[2,0]&&\\
&&&&&&\\
&&&&&&
}}
\grow{\xymatrix@!0{
\\\\\\\\
\save\go+<0pt,0pt>\Drop{\txt{$=$}}\restore
}}
\grow{\xymatrix@!0{\\
&\save\go+<0pt,3pt>\Drop{C}\restore \ar@{-}[1,0] &&& \save\go+<0pt,3pt>\Drop{D}\restore \ar@{-}[2,0]\\
&\ar@{-}`l/4pt [1,-1] [1,-1] \ar@{-}`r [1,1] [1,1]&&&\\
\ar@{-}[2,0]&&\ar@{-}[1,1]+<-0.125pc,0.0pc> \ar@{-}[1,1]+<0.0pc,0.125pc>&&\ar@{-}[2,-2]\\
&&&&\\
\ar@/^0.1pc/ @{-}[2,2] \ar@/_0.1pc/ @{-}[2,2]&& \ar@/^0.1pc/ @{-}[2,-2]\ar@/_0.1pc/ @{-}[2,-2]&&\ar@{-}[-1,-1]+<0.125pc,0.0pc>\ar@{-}[-1,-1]+<0.0pc,-0.125pc>\ar@{-}[2,0]\\
&&&&\\
&&&&
}}
\grow{\xymatrix@!0{
\\\\\\\\
\save\go+<0pt,0pt>\Drop{\txt{,}}\restore
}}
$$
where the first and last equalities follows from Lemma~\ref{formula para chi}.
\end{proof}

\begin{lemma}\label{mengano4} We have
$$
\grow{\xymatrix@!0{
&\save\go+<0pt,3pt>\Drop{D}\restore \ar@{-}[2,2]&& \save\go+<0pt,3pt>\Drop{C}\restore \ar@{-}[1,-1]+<0.125pc,0.0pc> \ar@{-}[1,-1]+<0.0pc,0.125pc>&& \save\go+<0pt,3pt>\Drop{C'}\restore \ar@{-}[2,0]\\
&&&&&&\\
& \ar@{-}[-1,1]+<-0.125pc,0.0pc> \ar@{-}[-1,1]+<0.0pc,-0.125pc>\ar@ {-}[2,-1]&&\ar@{-}[2,2] &&\ar@{-}[1,-1]+<0.125pc,0.0pc> \ar@{-}[1,-1]+<0.0pc,0.125pc>&\\
&&&&&&\\
\ar@{-}[2,0]&&&\ar@{-}[1,0]\ar@{-}[-1,1]+<-0.125pc,0.0pc> \ar@{-}[-1,1]+<0.0pc,-0.125pc>&&\ar@{-}[2,1]&\\
&&&\ar@{-}`l/4pt [1,-1] [1,-1] \ar@{-}`r [1,1] [1,1]&&&\\
\ar@{-}[1,1]+<-0.1pc,0.1pc> && \ar@{-}[2,-2]&&\ar@/^0.1pc/ @{-}[2,2] \ar@/_0.1pc/ @{-}[2,2]&& \ar@/^0.1pc/ @{-}[2,-2]\ar@/_0.1pc/ @{-}[2,-2]\\
&&&&&&\\
\ar@{-}[2,0]&&\ar@{-}[1,1]+<-0.125pc,0.0pc> \ar@{-}[1,1]+<0.0pc,0.125pc>\ar@{-}[-1,-1]+<0.1pc,-0.1pc>&& \ar@{-}[2,-2]&&\ar@{-}[2,0]\\
&&&&&&\\
&&&&\ar@{-}[-1,-1]+<0.125pc,0.0pc>\ar@{-}[-1,-1]+<0.0pc,-0.125pc>&&
}}
\grow{\xymatrix@!0{
\\\\\\\\\\
\save\go+<0pt,0pt>\Drop{\txt{$=$}}\restore
}}
\grow{\xymatrix@!0{
&\ar@{-}[4,0]\save\go+<0pt,3pt>\Drop{D}\restore &&\save\go+<0pt,3pt>\Drop{C}\restore \ar@{-}[2,0] &&&\save\go+<0pt,3pt>\Drop{C'}\restore \ar@{-}[1,0]\\
&&&&&&\ar@{-}`l/4pt [1,-1] [1,-1] \ar@{-}`r [1,1] [1,1]&\\
&&&\ar@{-}[1,1]+<-0.1pc,0.1pc> && \ar@{-}[2,-2]&&\ar@{-}[1,1]\\
&&&&&&&&\ar@{-}[7,0]\\
&\ar@{-}[2,2]&&\ar@{-}[1,-1]+<0.125pc,0.0pc> \ar@{-}[1,-1]+<0.0pc,0.125pc> &&\ar@{-}[-1,-1]+<0.1pc,-0.1pc>\ar@{-}[1,1]&&&\\
&&&&&&\ar@{-}[5,0]&&\\
&\ar@{-}[1,0]\ar@{-}[-1,1]+<-0.125pc,0.0pc> \ar@{-}[-1,1]+<0.0pc,-0.125pc>&&\ar@{-}[1,1]&&&&&\\
&\ar@{-}`l/4pt [1,-1] [1,-1] \ar@{-}`r [1,1] [1,1]&&&\ar@{-}[3,0]&&&&\\
\ar@{-}[2,0]&&\ar@{-}`d/4pt [1,2][1,2] &&&&&&\\
&&&&&&&&\\
&&&&&&&&
}}
\grow{\xymatrix@!0{
\\\\\\\\\\
\save\go+<0pt,0pt>\Drop{\txt{,}}\restore
}}
$$
where $C:=H^{\ot_c^s}$, $C':=H^{\ot_c^{s'}}$ and $D:=A^{\ot_k^r}$.
\end{lemma}

\begin{proof} In fact, we have
$$
\grow{\xymatrix@!0{\\\\
&\save\go+<0pt,3pt>\Drop{D}\restore \ar@{-}[2,2]&& \save\go+<0pt,3pt>\Drop{C}\restore \ar@{-}[1,-1]+<0.125pc,0.0pc> \ar@{-}[1,-1]+<0.0pc,0.125pc>&& \save\go+<0pt,3pt>\Drop{C'}\restore \ar@{-}[2,0]\\
&&&&&&\\
& \ar@{-}[-1,1]+<-0.125pc,0.0pc> \ar@{-}[-1,1]+<0.0pc,-0.125pc>\ar@ {-}[2,-1]&&\ar@{-}[2,2] &&\ar@{-}[1,-1]+<0.125pc,0.0pc> \ar@{-}[1,-1]+<0.0pc,0.125pc>&\\
&&&&&&\\
\ar@{-}[2,0]&&&\ar@{-}[1,0]\ar@{-}[-1,1]+<-0.125pc,0.0pc> \ar@{-}[-1,1]+<0.0pc,-0.125pc>&&\ar@{-}[2,1]&\\
&&&\ar@{-}`l/4pt [1,-1] [1,-1] \ar@{-}`r [1,1] [1,1]&&&\\
\ar@{-}[1,1]+<-0.1pc,0.1pc> && \ar@{-}[2,-2]&&\ar@/^0.1pc/ @{-}[2,2] \ar@/_0.1pc/ @{-}[2,2]&& \ar@/^0.1pc/ @{-}[2,-2]\ar@/_0.1pc/ @{-}[2,-2]\\
&&&&&&\\
\ar@{-}[2,0]&&\ar@{-}[1,1]+<-0.125pc,0.0pc> \ar@{-}[1,1]+<0.0pc,0.125pc>\ar@{-}[-1,-1]+<0.1pc,-0.1pc>&& \ar@{-}[2,-2]&&\ar@{-}[2,0]\\
&&&&&&\\
&&&&\ar@{-}[-1,-1]+<0.125pc,0.0pc>\ar@{-}[-1,-1]+<0.0pc,-0.125pc>&&
}}
\grow{\xymatrix@!0{
\\\\\\\\\\\\\\
\save\go+<0pt,0pt>\Drop{\txt{$=$}}\restore
}}
\grow{\xymatrix@!0{
&\save\go+<0pt,3pt>\Drop{D}\restore \ar@{-}[2,2]&&\save\go+<0pt,3pt>\Drop{C}\restore \ar@{-}[1,-1]+<0.125pc,0.0pc> \ar@{-}[1,-1]+<0.0pc,0.125pc> &&\save\go+<0pt,3pt>\Drop{C'}\restore \ar@{-}[2,0]&\\
&&&&&&\\
&\ar@{-}[4,-1]\ar@{-}[-1,1]+<-0.125pc,0.0pc> \ar@{-}[-1,1]+<0.0pc,-0.125pc>&&\ar@{-}[2,2]&& \ar@{-}[1,-1]+<0.125pc,0.0pc> \ar@{-}[1,-1]+<0.0pc,0.125pc>&\\
&&&&&&\\
&&&\ar@{-}[1,0]\ar@{-}[-1,1]+<-0.125pc,0.0pc> \ar@{-}[-1,1]+<0.0pc,-0.125pc>&&\ar@{-}[2,2]&&\\
&&&\ar@{-}`l/4pt [1,-1] [1,-1] \ar@{-}`r [1,1] [1,1]&&&&\\
\ar@{-}[1,1]+<-0.1pc,0.1pc> && \ar@{-}[2,-2]&&\ar@{-}[1,0]&&&\ar@{-}[2,0]\\
&&&&\ar@{-}`l/4pt [1,-1] [1,-1] \ar@{-}`r [1,1] [1,1]&&&\\
\ar@{-}[6,0]&&\ar@{-}[-1,-1]+<0.1pc,-0.1pc>\ar@{-}[3,0]&\ar@{-}[2,0]&&\ar@{-}[1,1]+<-0.125pc,0.0pc> \ar@{-}[1,1]+<0.0pc,0.125pc>&&\ar@{-}[2,-2]\\
&&&&&&&\\
&&&\ar@{-}`d/4pt [1,2][1,2]&&\ar@{-}[2,0]&&\ar@{-}[-1,-1]+<0.125pc,0.0pc> \ar@{-}[-1,-1]+<0.0pc,-0.125pc>\ar@{-}[4,0]\\
&&\ar@{-}[1,1]&&&&&\\
&&&\ar@{-}[1,1]+<-0.125pc,0.0pc> \ar@{-}[1,1]+<0.0pc,0.125pc>&&\ar@{-}[2,-2]&&\\
&&&&&&&\\
&&&&&\ar@{-}[-1,-1]+<0.125pc,0.0pc>\ar@{-}[-1,-1]+<0.0pc,-0.125pc>&&
}}
\grow{\xymatrix@!0{
\\\\\\\\\\\\\\
\save\go+<0pt,0pt>\Drop{\txt{$=$}}\restore
}}
\grow{\xymatrix@!0{\\
&\save\go+<0pt,3pt>\Drop{D}\restore \ar@{-}[2,2]&& \save\go+<0pt,3pt>\Drop{C}\restore \ar@{-}[1,-1]+<0.125pc,0.0pc> \ar@{-}[1,-1]+<0.0pc,0.125pc> &&&\save\go+<0pt,3pt>\Drop{C'}\restore\ar@{-}[1,0] \\
&&&&&&\ar@{-}`l/4pt [1,-1] [1,-1] \ar@{-}`r [1,1] [1,1]&\\
&\ar@{-}[4,-1]\ar@{-}[-1,1]+<-0.125pc,0.0pc> \ar@{-}[-1,1]+<0.0pc,-0.125pc>&& \ar@{-}[3,3]&&\ar@{-}[1,-1]+<0.125pc,0.0pc> \ar@{-}[1,-1]+<0.0pc,0.125pc>&&\ar@{-}[10,0]\\
&&&&&&&\\
&&&\ar@{-}[1,0]\ar@{-}[-1,1]+<-0.125pc,0.0pc> \ar@{-}[-1,1]+<0.0pc,-0.125pc>&&&&\\
&&&\ar@{-}`l/4pt [1,-1] [1,-1] \ar@{-}`r [1,1] [1,1]&&&\ar@{-}[5,0]&\\
\ar@{-}[1,1]+<-0.1pc,0.1pc> && \ar@{-}[2,-2]&&\ar@{-}[1,0]&&&\\
&&&&\ar@{-}`d/4pt [1,2][1,2]&&&\\
\ar@{-}[4,0]&&\ar@{-}[-1,-1]+<0.1pc,-0.1pc>\ar@{-}[2,2]&&&&&\\
&&&&&&&\\
&&&&\ar@{-}[1,1]+<-0.125pc,0.0pc> \ar@{-}[1,1]+<0.0pc,0.125pc>&&\ar@{-}[2,-2]&\\
&&&&&&&\\
&&&&&&\ar@{-}[-1,-1]+<0.125pc,0.0pc>\ar@{-}[-1,-1]+<0.0pc,-0.125pc>&
}}
\grow{\xymatrix@!0{
\\\\\\\\\\\\\\
\save\go+<0pt,0pt>\Drop{\txt{$=$}}\restore
}}
\grow{\xymatrix@!0{\\
& \save\go+<0pt,3pt>\Drop{D}\restore \ar@{-}[2,2]&&\save\go+<0pt,3pt>\Drop{C}\restore \ar@{-}[1,-1]+<0.125pc,0.0pc> \ar@{-}[1,-1]+<0.0pc,0.125pc>&&&\save\go+<0pt,3pt>\Drop{C'}\restore \ar@{-}[1,0]&\\
&&&&&&\ar@{-}`l/4pt [1,-1] [1,-1] \ar@{-}`r [1,1] [1,1]&\\
&\ar@{-}[2,0]\ar@{-}[-1,1]+<-0.125pc,0.0pc> \ar@{-}[-1,1]+<0.0pc,-0.125pc>&& \ar@{-}[3,3]&&\ar@{-}[1,-1]+<0.125pc,0.0pc> \ar@{-}[1,-1]+<0.0pc,0.125pc>&&\ar@{-}[9,0]\\
&&&&&&&\\
&\ar@{-}[1,1]+<-0.1pc,0.1pc> && \ar@{-}[2,-2]\ar@{-}[-1,1]+<-0.125pc,0.0pc> \ar@{-}[-1,1]+<0.0pc,-0.125pc> &&&&\\
&&&&&&\ar@{-}[2,0]&\\
&\ar@{-}[1,0]&&\ar@{-}[-1,-1]+<0.1pc,-0.1pc>\ar@{-}[1,1]&&&&\\
&\ar@{-}`l/4pt [1,-1] [1,-1] \ar@{-}`r [1,1] [1,1]&&&\ar@{-}[1,1]+<-0.125pc,0.0pc> \ar@{-}[1,1]+<0.0pc,0.125pc>&&\ar@{-}[2,-2]&\\
\ar@{-}[3,0]&&\ar@{-}[1,0]&&&&&\\
&&\ar@{-}`d/4pt [1,2][1,2]&&\ar@{-}[2,0]&&\ar@{-}[2,0]\ar@{-}[-1,-1]+<0.125pc,0.0pc>\ar@{-}[-1,-1]+<0.0pc,-0.125pc>&\\
&&&&&&&\\
&&&&&&&
}}
\grow{\xymatrix@!0{
\\\\\\\\\\\\\\
\save\go+<0pt,0pt>\Drop{\txt{$=$}}\restore
}}
\grow{\xymatrix@!0{
& \save\go+<0pt,3pt>\Drop{D}\restore \ar@{-}[4,0]&& \save\go+<0pt,3pt>\Drop{C}\restore \ar@{-}[2,0]&&& \save\go+<0pt,3pt>\Drop{C'}\restore \ar@{-}[1,0]\\
&&&&&&\ar@{-}`l/4pt [1,-1] [1,-1] \ar@{-}`r [1,1] [1,1]&\\
&&&\ar@{-}[1,1]+<-0.1pc,0.1pc> && \ar@{-}[2,-2]&&\ar@{-}[12,0]\\
&&&&&&&\\
&\ar@{-}[2,2]&&\ar@{-}[1,-1]+<0.125pc,0.0pc> \ar@{-}[1,-1]+<0.0pc,0.125pc>&&\ar@{-}[-1,-1]+<0.1pc,-0.1pc> \ar@{-}[1,1]&&\\
&&&&&&\ar@{-}[2,0]&\\
&\ar@{-}[4,0]\ar@{-}[-1,1]+<-0.125pc,0.0pc> \ar@{-}[-1,1]+<0.0pc,-0.125pc>&&\ar@{-}[1,1]&&&&\\
&&&&\ar@{-}[2,2]&&\ar@{-}[1,-1]+<0.125pc,0.0pc> \ar@{-}[1,-1]+<0.0pc,0.125pc>&\\
&&&&&&&\\
&&&&\ar@{-}[1,0]\ar@{-}[-1,1]+<-0.125pc,0.0pc> \ar@{-}[-1,1]+<0.0pc,-0.125pc>&&\ar@{-}[1,0]&\\
&\ar@{-}`l/4pt [1,-1] [1,-1] \ar@{-}`r [1,1] [1,1]&&&\ar@{-}[1,1]+<-0.125pc,0.0pc> \ar@{-}[1,1]+<0.0pc,0.125pc>&&\ar@{-}[2,-2]&\\
\ar@{-}[3,0]&&\ar@{-}[1,0]&&&&&\\
&&\ar@{-}`d/4pt [1,2][1,2]&&\ar@{-}[2,0]&&\ar@{-}[-1,-1]+<0.125pc,0.0pc>\ar@{-}[-1,-1]+<0.0pc,-0.125pc>\ar@{-}[2,0]&\\
&&&&&&&\\
&&&&&&&
}}
\grow{\xymatrix@!0{
\\\\\\\\\\\\\\
\save\go+<0pt,0pt>\Drop{\txt{$=$}}\restore
}}
\grow{\xymatrix@!0{\\\\
&\ar@{-}[4,0]\save\go+<0pt,3pt>\Drop{D}\restore &&\save\go+<0pt,3pt>\Drop{C}\restore \ar@{-}[2,0] &&&\save\go+<0pt,3pt>\Drop{C'}\restore \ar@{-}[1,0]\\
&&&&&&\ar@{-}`l/4pt [1,-1] [1,-1] \ar@{-}`r [1,1] [1,1]&\\
&&&\ar@{-}[1,1]+<-0.1pc,0.1pc> && \ar@{-}[2,-2]&&\ar@{-}[1,1]\\
&&&&&&&&\ar@{-}[7,0]\\
&\ar@{-}[2,2]&&\ar@{-}[1,-1]+<0.125pc,0.0pc> \ar@{-}[1,-1]+<0.0pc,0.125pc> &&\ar@{-}[-1,-1]+<0.1pc,-0.1pc>\ar@{-}[1,1]&&&\\
&&&&&&\ar@{-}[5,0]&&\\
&\ar@{-}[1,0]\ar@{-}[-1,1]+<-0.125pc,0.0pc> \ar@{-}[-1,1]+<0.0pc,-0.125pc>&&\ar@{-}[1,1]&&&&&\\
&\ar@{-}`l/4pt [1,-1] [1,-1] \ar@{-}`r [1,1] [1,1]&&&\ar@{-}[3,0]&&&&\\
\ar@{-}[2,0]&&\ar@{-}`d/4pt [1,2][1,2] &&&&&&\\
&&&&&&&&\\
&&&&&&&&
}}
\grow{\xymatrix@!0{
\\\\\\\\\\\\\\
\save\go+<0pt,0pt>\Drop{\txt{,}}\restore
}}
$$
where the first equality follows from Lemma~\ref{formula para chi}, and the third one follows from Lemma~\ref{mengano1}.
\end{proof}

\noindent Let
\begin{align*}
& \wt{\vartheta}^{rs}\colon  \Hom_{(A,E)}\bigl(E^{\ot_{\!A}^s}\ot A^{\ot^r},E\bigr)\to \Hom_{K^e}\bigl(A^{\ot^r}\ot_k H^{\ot_k^s},E\bigr)
\intertext{and}
& \wt{\theta}^{rs}\colon \Hom_{K^e}\bigl(A^{\ot^r}\ot_k H^{\ot_k^s},E\bigr)\to \Hom_{(A,E)}\bigl(E^{\ot_{\!A}^s}\ot A^{\ot^r},E\bigr)
\end{align*}
be the $k$-linear maps diagrammatically defined by
$$
%
\grow{\xymatrix@!0{
\\\\\\\\
\save\go+<0pt,0pt>\Drop{\txt{$\widetilde{\theta}(\beta) :=$}}\restore
}}
\grow{\xymatrix@!0{\\
\save\go+<0pt,3pt>\Drop{\overline{C}}\restore \ar@{-}[3,0]+<0pc,-0.5pt> &&&& \save\go+<0pt,3pt>\Drop{D}\restore \ar@{-}[2,0]\\
\ar@{-}`r/4pt [1,2] [1,2]&&&&\\
&&\ar@{-}[1,1]+<-0.125pc,0.0pc> \ar@{-}[1,1]+<0.0pc,0.125pc>&&\ar@{-}[2,-2]\\
&&&&\\
*+[o]+<0.37pc>[F]{\ov{\mu}}\ar@{-}[3,0]+<0pc,0pc>&&\ar@{-}[1,0]+<0pc,-1.5pt>&&\ar@{-}[-1,-1]+<0.125pc,0.0pc>\ar@{-}[-1,-1]+<0.0pc,-0.125pc> \ar@{-}[1,0]+<0pc,-1.5pt>\\
&&&&\\
&&&*+<0.1pc>[F]{\,\,\,\beta\,\,\,}\ar@{-}[1,0]&\\
\ar@{-}`d/4pt [1,1]+<0.2pc,0pc> `[0,3] [0,3]&&&&\\
&\save[]+<0.2pc,0pc> \Drop{}\ar@{-}[1,0]+<0.2pc,0pc>\restore&&&\\
&\save[]+<0.2pc,0pc> \Drop{}\restore&&&
}}
\quad
\grow{\xymatrix@!0{
\\\\\\\\
\save\go+<0pt,0pt>\Drop{\txt{and}}\restore
}}
\quad
\grow{\xymatrix@!0{
\\\\\\\\
\save\go+<0pt,0pt>\Drop{\txt{$\widetilde{\vartheta}(\alpha) :=$}}\restore
}}
\grow{\xymatrix@!0{
&\save\go+<0pt,3pt>\Drop{\overline{D}}\restore \ar@{-}[2,2] && \save\go+<0pt,3pt>\Drop{C}\restore \ar@{-}[1,-1]+<0.125pc,0.0pc> \ar@{-}[1,-1]+<0.0pc,0.125pc> &\\
&&&&\\
&\ar@{-}[1,0]\ar@{-}[-1,1]+<-0.125pc,0.0pc> \ar@{-}[-1,1]+<0.0pc,-0.125pc>&&\ar@{-}[1,1]&\\
&\ar@{-}`l/4pt [1,-1] [1,-1] \ar@{-}`r [1,1] [1,1]&&&\ar@{-}[4,0]+<0pc,-1.5pt>\\
\ar@{-}[1,0]+<0pc,-0.5pt>&&\ar@{-}[1,0]+<0pc,-0.5pt>&&\\
&&&&\\
*+[o]+<0.37pc>[F]{\ov{\mathbf{u}}}\ar@{-}[3,0]+<0pc,0pc>&&*+[o]+<0.37pc>[F]{\boldsymbol{\gamma}} \ar@{-}[1,0]+<0pc,-1.5pt>&&\\
&&&&\\
&&&*+<0.3pc>[F]{\,\,\,\alpha\,\,\,}\ar@{-}[1,0]&\\
\ar@{-}`d/4pt [1,1]+<0.2pc,0pc> `[0,3] [0,3]&&&&\\
&\save[]+<0.2pc,0pc> \Drop{}\ar@{-}[1,0]+<0.2pc,0pc>\restore&&&\\
&\save[]+<0.2pc,0pc> \Drop{}\restore&&&
}}
\grow{\xymatrix@!0{
\\\\\\\\
\save\go+<0pt,0pt>\Drop{\txt{,}}\restore
}}
$$
where

\begin{itemize}

\smallskip

\item[-] $C := \ov{H}^{\ot_c^s}$, $\ov{C} = (E/A)^{\ot_{\!A}^s}$ and $D:=\ov{A}^{\ot_k^r}$,

\smallskip

\item[-] $\ov{\mu}$ is the map induced by $\mu_s\colon E^{\ot_k^s}\to E$,

\smallskip

\item[-] $\boldsymbol{\gamma}:= \gamma^{\ot_{\!A}^s}$ and $\ov{\mathbf{u}} := \mu_s\xcirc \ov{\gamma}^{\ot_k^s}\xcirc \gc_s$.

\end{itemize}
It is easy to see that $\theta^{rs}$ and $\vartheta^{rs}$ are induced by $(-1)^{rs}\wt{\theta}^{rs}$ and $(-1)^{rs}\wt{\vartheta}^{rs}$, respectively.

\begin{definition} For
$$
\alpha\in \Hom_{(A,E)}\bigl(E^{\ot_{\!A}^s}\ot A^{\ot^r},E\bigr)\quad\text{and}\quad \alpha'\in \Hom_{(A,E)}\bigl(E^{\ot_{\!A}^{s'}}\ot A^{\ot^{r'}},E\bigr)
$$
we define
$$
\alpha\wt{\bullet} \alpha'\in \Hom_{(A,E)}\bigl(E^{\ot_{\!A}^{s''}}\ot A^{\ot^{r''}},E\bigr)
$$
by
$$
(\alpha\wt{\bullet}\alpha')\bigl(\gamma_A(\bv_{1s''})\ot \ba_{1r''}\bigr) := \sum_i  \alpha \bigl(\gamma_A(\bv_{1s})\ot \ba_{1r}^{(i)}\bigr) \alpha'\bigl(\gamma_A(\bv_{s+1,s''}^{(i)}) \ot \ba_{r+1,r''}\bigr),
$$
where $r'' := r+r'$, $s'':=s+s'$ and $\sum_i \ba_{1r}^{(i)}\ot_k \bv_{s+1,s''}^{(i)} := \ov{\chi}\bigl(\bv_{s+1,s''}\ot \ba_{1r}\bigr)$.
\end{definition}

\begin{lemma}\label{mengano5} Let $C:=H^{\ot_c^s}$, $C':=H^{\ot_c^{s'}}$, $D:=A^{\ot_k^r}$ and $D':=A^{\ot_k^{r'}}$. We have
$$
\grow{\xymatrix@!0{\\\\
&\save\go+<0pt,3pt>\Drop{C}\restore \ar@{-}[1,0]&&&& \save\go+<0pt,3pt>\Drop{C'}\restore \ar@{-}[1,0]&&& \save\go+<0pt,3pt>\Drop{D}\restore \ar@{-}[8,0] && \save\go+<0pt,3pt>\Drop{D'}\restore \ar@{-}[8,0]\\
&\ar@{-}`l/4pt [1,-1] [1,-1] \ar@{-}`r [1,1] [1,1]&&&&\ar@{-}`l/4pt [1,-1] [1,-1] \ar@{-}`r [1,1] [1,1]&&&&&\\
\ar@{-}[2,0]&&\ar@{-}[1,1]+<-0.1pc,0.1pc> && \ar@{-}[2,-2]&&\ar@{-}[2,0]&&&&\\
&&&&&&&&&&\\
\ar@{-}[1,1]+<-0.1pc,0.1pc> && \ar@{-}[2,-2]&&\ar@{-}[-1,-1]+<0.1pc,-0.1pc>\ar@{-}[1,0]&&\ar@{-}[1,0]&&&&\\
&&&&&&&&&&\\
\ar@{-}[1,0]+<0pc,-0.5pt>&&\ar@{-}[1,0]+<0pc,-0.5pt>\ar@{-}[-1,-1]+<0.1pc,-0.1pc>&& *+[o]+<0.37pc>[F]{\boldsymbol{\gamma}}\ar@{-}[2,0]+<0pc,0pc>&& *+[o]+<0.37pc>[F]{\boldsymbol{\gamma}}\ar@{-}[2,0]+<0pc,0pc>&&&&\\
&&&&&&&&&&\\
*+[o]+<0.37pc>[F]{\ov{\mathbf{u}}}\ar@{-}[2,0]+<0pc,0pc>&&*+[o]+<0.37pc>[F]{\ov{\mathbf{u}}}\ar@{-}[2,0]+<0pc,0pc>&&&&&&&&\\
&&&&&&&*+<0.1pc>[F]{\wt{\theta}(\wt{\beta})\wt{\bullet} \wt{\theta}(\wt{\beta}')}\ar@{-}[2,0]&&&\\
\ar@{-}`d/4pt [1,1] `[0,2] [0,2]&&&&&&&&&&\\
&\ar@{-}[1,0]&&&&&&\ar@{-}[1,-2]&&&\\
&\ar@{-}`d/4pt [1,2] `[0,4] [0,4]&&&&&&&&&\\
&&&\ar@{-}[1,0]&&&&&&&\\
&&&&&&&&&&
}}
\grow{\xymatrix@!0{
\\\\\\\\\\\\\\\\\\
\save\go+<0pt,0pt>\Drop{\txt{$=$}}\restore
}}
\grow{\xymatrix@!0{
\save\go+<0pt,3pt>\Drop{C}\restore \ar@{-}[2,0] &&&\save[]+<0.2pc,3pt> \Drop{C'}\ar@{-}[1,0]+<0.2pc,0pc>\restore&&&&& \save\go+<0pt,3pt>\Drop{D}\restore \ar@{-}[4,0] && \save\go+<0pt,3pt>\Drop{D'}\restore \ar@{-}[6,0]\\
&&&\save[]+<0.2pc,0pc> \Drop{}\ar@{-}`l/4pt [1,-1] [1,-1] \ar@{-}`r [1,2] [1,2]\restore&&&&&&&\\
\ar@{-}[1,1]+<-0.1pc,0.1pc> &&\ar@{-}[2,-2]&&&\ar@{-}[1,0]&&&&&&\\
&&&&&\ar@{-}`l/4pt [1,-1] [1,-1] \ar@{-}`r [1,1] [1,1]&&&&&&\\
\ar@{-}[4,0]+<0pc,-0.5pt>&&\ar@{-}[-1,-1]+<0.1pc,-0.1pc>\ar@{-}[4,0]&&\ar@{-}[2,0]&&\ar@{-}[1,1]+<-0.125pc,0.0pc> \ar@{-}[1,1]+<0.0pc,0.125pc>&&\ar@{-}[2,-2]&&\\
&&&&&&&&&&\\
&&&&\ar@/^0.1pc/ @{-}[2,2] \ar@/_0.1pc/ @{-}[2,2]&& \ar@/^0.1pc/ @{-}[2,-2]\ar@/_0.1pc/ @{-}[2,-2] &&\ar@{-}[-1,-1]+<0.125pc,0.0pc> \ar@{-}[-1,-1]+<0.0pc,-0.125pc> \ar@{-}[1,1]+<-0.125pc,0.0pc> \ar@{-}[1,1]+<0.0pc,0.125pc>&&\ar@{-}[2,-2]\\
&&&&&&&&&&\\
&&\ar@{-}[1,1]+<-0.125pc,0.0pc> \ar@{-}[1,1]+<0.0pc,0.125pc>&&\ar@{-}[2,-2]&&\ar@{-}[1,0]+<0pc,-0.5pt> &&\ar@{-}[1,0]+<0pc,-1pt>&&\ar@{-}[1,0]+<0pc,-1pt> \ar@{-}[-1,-1]+<0.125pc,0.0pc> \ar@{-}[-1,-1]+<0.0pc,-0.125pc>\\
*+[o]+<0.37pc>[F]{\ov{\mathbf{u}}}\ar@{-}[4,0]+<0pc,0pc>&&&&&&&&&&\\
&&\ar@{-}[1,0]+<0pc,-2pt>&&\ar@{-}[1,0]+<0pc,-2pt>\ar@{-}[-1,-1]+<0.125pc,0.0pc>\ar@{-}[-1,-1]+<0.0pc,-0.125pc>&& *+[o]+<0.40pc>[F]{\mathbf{u}}\ar@{-}[2,0]+<0pc,0pc> &&&*+<0.1pc>[F]{\,\,\,\wt{\beta}'\,\,\,}\ar@{-}[1,0]&\\
&&&&&&&&&\ar@{-}[1,-1]&\\
&&&*+<0.1pc>[F]{\,\,\,\wt{\beta}\,\,\,}\ar@{-}[2,0]&&&\ar@{-}`d/4pt [1,1] `[0,2] [0,2]&&&&\\
\ar@{-}[3,1]&&&&&&&\ar@{-}[1,0]&&&\\
&&&\ar@{-}`d/4pt [1,2] `[0,4] [0,4]&&&&&&&\\
&&&&&\ar@{-}[1,0]&&&&&\\
&\ar@{-}`d/4pt [1,2] `[0,4] [0,4]&&&&&&&&&\\
&&&\ar@{-}[1,0]&&&&&&&\\
&&&&&&&&&&
}}
\grow{\xymatrix@!0{
\\\\\\\\\\\\\\\\\\\\
\save\go+<0pt,0pt>\Drop{\txt{,}}\restore
}}
$$
where

\begin{itemize}

\smallskip

\item[-] $\boldsymbol{\gamma}$ denotes the maps $\gamma^{\ot_k^s}$, $\gamma^{\ot_k^{s'}}$, $\gamma^{\ot_{\!A}^s}$ and $\gamma^{\ot_{\!A}^{s'}}$,

\smallskip

\item[-] $\mu$ denotes the maps $\mu_s$ and $\mu_{s'}$,

\smallskip

\item[-] $\mathbf{u}:=\mu_{s'}\xcirc \gamma^{\ot_k^{s'}}$ and $\ov{\mathbf{u}}$ denotes both the maps $\mu_s\xcirc \ov{\gamma}^{\ot_k^s}\xcirc \gc_s$ and $\mu_{s'}\xcirc \ov{\gamma}^{\ot_k^{s'}}\xcirc \gc_{s'}$.

\smallskip

\end{itemize}

\end{lemma}

\begin{proof} In fact,
$$
\grow{\xymatrix@!0{\\\\\\\\\\
&\save\go+<0pt,3pt>\Drop{C}\restore \ar@{-}[1,0]&&&& \save\go+<0pt,3pt>\Drop{C'}\restore \ar@{-}[1,0]&&& \save\go+<0pt,3pt>\Drop{D}\restore \ar@{-}[8,0] && \save\go+<0pt,3pt>\Drop{D'}\restore \ar@{-}[8,0]\\
&\ar@{-}`l/4pt [1,-1] [1,-1] \ar@{-}`r [1,1] [1,1]&&&&\ar@{-}`l/4pt [1,-1] [1,-1] \ar@{-}`r [1,1] [1,1]&&&&&\\
\ar@{-}[2,0]&&\ar@{-}[1,1]+<-0.1pc,0.1pc> && \ar@{-}[2,-2]&&\ar@{-}[2,0]&&&&\\
&&&&&&&&&&\\
\ar@{-}[1,1]+<-0.1pc,0.1pc> && \ar@{-}[2,-2]&&\ar@{-}[-1,-1]+<0.1pc,-0.1pc>\ar@{-}[1,0]&&\ar@{-}[1,0]&&&&\\
&&&&&&&&&&\\
\ar@{-}[1,0]+<0pc,-0.5pt>&&\ar@{-}[1,0]+<0pc,-0.5pt>\ar@{-}[-1,-1]+<0.1pc,-0.1pc>&& *+[o]+<0.37pc>[F]{\boldsymbol{\gamma}}\ar@{-}[2,0]+<0pc,0pc>&& *+[o]+<0.37pc>[F]{\boldsymbol{\gamma}}\ar@{-}[2,0]+<0pc,0pc>&&&&\\
&&&&&&&&&&\\
*+[o]+<0.37pc>[F]{\ov{\mathbf{u}}}\ar@{-}[2,0]+<0pc,0pc>&&*+[o]+<0.37pc>[F]{\ov{\mathbf{u}}}\ar@{-}[2,0]+<0pc,0pc>&&&&&&&&\\
&&&&&&&*+<0.1pc>[F]{\wt{\theta}(\wt{\beta})\wt{\bullet}\wt{\theta}(\beta')}\ar@{-}[2,0]&&&\\
\ar@{-}`d/4pt [1,1] `[0,2] [0,2]&&&&&&&&&&\\
&\ar@{-}[1,0]&&&&&&\ar@{-}[1,-2]&&&\\
&\ar@{-}`d/4pt [1,2] `[0,4] [0,4]&&&&&&&&&\\
&&&\ar@{-}[1,0]&&&&&&&\\
&&&&&&&&&&
}}
\grow{\xymatrix@!0{
\\\\\\\\\\\\\\\\\\\\\\
\save\go+<0pt,0pt>\Drop{\txt{$=$}}\restore
}}
\grow{\xymatrix@!0{
&\save\go+<0pt,3pt>\Drop{C}\restore \ar@{-}[2,0]&&&\save[]+<0.2pc,3pt> \Drop{C'} \ar@{-}[1,0]+<0.2pc,0pc>\restore&&&&&&\ar@{-}[4,0]\save\go+<0pt,3pt>\Drop{D}\restore &&&& \save\go+<0pt,3pt>\Drop{D'}\restore \ar@{-}[11,0]\\
&&&&\save[]+<0.2pc,0pc> \Drop{}\ar@{-}`l/4pt [1,-1] [1,-1] \ar@{-}`r [1,2] [1,2]\restore&&&&&&&&&\\
&\ar@{-}[1,1]+<-0.1pc,0.1pc> && \ar@{-}[2,-2]&&&\ar@{-}[2,2]&&&&&&&&\\
&&&&&&&&&&&&&&\\
&\ar@{-}[1,-1]&&\ar@{-}[-1,-1]+<0.1pc,-0.1pc>\ar@{-}[1,0]&&&&&\ar@/^0.1pc/ @{-}[2,2] \ar@/_0.1pc/ @{-}[2,2]&& \ar@/^0.1pc/ @{-}[2,-2]\ar@/_0.1pc/ @{-}[2,-2]&&&&\\
\ar@{-}[5,0]+<0pc,-0.5pt>&&&\ar@{-}`l/4pt [1,-1] [1,-1] \ar@{-}`r [1,1] [1,1]&&&&&&&&&&&\\
&&\ar@{-}[4,0]+<0pc,-0.5pt>&&\ar@{-}[1,0]+<0pc,-0.5pt>&&&&\ar@{-}[5,0]&&\ar@{-}[1,0]+<0pc,-0.5pt>&&&&\\
&&&&&&&&&&&&&&\\
&&&& *+[o]+<0.37pc>[F]{\boldsymbol{\gamma}}\ar@{-}[3,0]+<0pc,-0.5pt>&&&&&& *+[o]+<0.37pc>[F]{\boldsymbol{\gamma}}\ar@{-}[3,0]+<0pc,-0.5pt>&&&&\\
&&&&&&&&&&&&&&\\
&&&&\ar@{-}[1,0]  \ar@{-}`r/4pt [1,2] [1,2]&&&&&&\ar@{-}[1,0]  \ar@{-}`r/4pt [1,2] [1,2]&&&&\\
*+[o]+<0.37pc>[F]{\ov{\mathbf{u}}}\ar@{-}[4,0]+<0pc,0pc>&& *+[o]+<0.37pc>[F]{\ov{\mathbf{u}}}\ar@{-}[4,0]+<0pc,0pc>&&&&\ar@{-}[1,1]+<-0.125pc,0.0pc> \ar@{-}[1,1]+<0.0pc,0.125pc>&&\ar@{-}[2,-2]&&&&\ar@{-}[1,1]+<-0.125pc,0.0pc> \ar@{-}[1,1]+<0.0pc,0.125pc>&&\ar@{-}[2,-2]\\
&&&&*+[o]+<0.40pc>[F]{\mu}\ar@{-}[4,0]&&&&&&*+[o]+<0.40pc>[F]{\mu}\ar@{-}[5,0]&&&&\\
&&&&&&\ar@{-}[1,0]+<0pc,-2pt>&&\ar@{-}[-1,-1]+<0.125pc,0.0pc> \ar@{-}[-1,-1]+<0.0pc,-0.125pc> \ar@{-}[1,0]+<0pc,-2pt>&&&&\ar@{-}[1,0]+<0pc,-1pt> && \ar@{-}[-1,-1]+<0.125pc,0.0pc> \ar@{-}[-1,-1]+<0.0pc,-0.125pc>\ar@{-}[1,0]+<0pc,-1pt>\\
&&&&&&&&
&&&&&&\\
\ar@{-}`d/4pt [1,1] `[0,2] [0,2]&&&&&&&*+<0.1pc>[F]{\,\,\,\wt{\beta}\,\,\,}\ar@{-}[2,0] &&&&&&*+<0.1pc>[F]{\,\,\,\wt{\beta}'\,\,\,}\ar@{-}[1,0]&\\
&\ar@{-}[1,0]&&&\ar@{-}[1,1]&&&&&&&&&\ar@{-}[1,-1]&\\
&\ar@{-}[5,3]&&&&\ar@{-}`d/4pt [1,1] `[0,2] [0,2]&&&&&\ar@{-}`d/4pt [1,1] `[0,2] [0,2]&&&&\\
&&&&&&\ar@{-}[2,0]&&&&&\ar@{-}[1,0]&&&\\
&&&&&&&&&&&\ar@{-}[1,-1]&&&\\
&&&&&&\ar@{-}`d/4pt [1,2] `[0,4] [0,4]&&&&&&&&\\
&&&&&&&&\ar@{-}[1,0]&&&&&&\\
&&&&\ar@{-}`d/4pt [1,2] `[0,4] [0,4]&&&&&&&&&&\\
&&&&&&\ar@{-}[1,0]&&&&&&&&\\
&&&&&&&&&&&&&&
}}
\grow{\xymatrix@!0{
\\\\\\\\\\\\\\\\\\\\\\
\save\go+<0pt,0pt>\Drop{\txt{$=$}}\restore
}}
\grow{\xymatrix@!0{
&&\save\go+<0pt,3pt>\Drop{C}\restore \ar@{-}[2,0] &&&&\save[]+<0.2pc,3pt> \Drop{C'}\restore \ar@{-}[1,0]&&&&\save\go+<0pt,3pt>\Drop{D}\restore \ar@{-}[2,0] &&&& \save\go+<0pt,3pt>\Drop{D'}\restore  \ar@{-}[7,0]\\
&&&&&&\ar@{-}`l/4pt [1,-2] [1,-2] \ar@{-}`r [1,2] [1,2]&&&&&&&&\\
&&\ar@{-}[1,1]+<-0.1pc,0.1pc> && \ar@{-}[2,-2]&&&&\ar@/^0.1pc/ @{-}[2,2] \ar@/_0.1pc/ @{-}[2,2]&& \ar@/^0.1pc/ @{-}[2,-2]\ar@/_0.1pc/ @{-}[2,-2]&&&&\\
&&&&&&&&&&&&&&\\
&&\ar@{-}[2,-2]&&\ar@{-}[-1,-1]+<0.1pc,-0.1pc>\ar@{-}[1,0]+<0.2pc,-1pt>&&&&\ar@{-}[3,0]&&\ar@{-}[1,1]&&\\
&&&&\save[]+<0.2pc,0pc> \Drop{}\ar@{-}[1,0]+<0.2pc,0pc>\restore&&&&&&&\ar@{-}[1,0]&&&\\
\ar@{-}[3,0]&&&&\save[]+<0.2pc,0pc> \Drop{}\ar@{-}`l/4pt [1,-1] [1,-1] \ar@{-}`r [1,2] [1,2]\restore&&&&&&&\ar@{-}`l/4pt [1,-1] [1,-1] \ar@{-}`r [1,1] [1,1]&&&\\
&&&\ar@{-}[1,0]&&&\ar@{-}[1,1]+<-0.125pc,0.0pc> \ar@{-}[1,1]+<0.0pc,0.125pc>&&\ar@{-}[2,-2]&& \ar@{-}[2,0]+<0pc,-0.5pt>&&\ar@{-}[1,1]+<-0.125pc,0.0pc> \ar@{-}[1,1]+<0.0pc,0.125pc>&&\ar@{-}[2,-2]\\
&&&\ar@{-}`l/4pt [1,-1] [1,-1] \ar@{-}`r [1,1] [1,1]&&&&&&&&&&&\\
&&&&&&\ar@{-}[1,0]+<0pc,-2pt>&&\ar@{-}[-1,-1]+<0.125pc,0.0pc>
\ar@{-}[-1,-1]+<0.0pc,-0.125pc>\ar@{-}[1,0]+<0pc,-2pt>&&&&\ar@{-}[1,0]+<0pc,-1pt>&&\ar@{-}[1,0]+<0pc,-1pt> \ar@{-}[-1,-1]+<0.125pc,0.0pc>\ar@{-}[-1,-1]+<0.0pc,-0.125pc>\\
*+[o]+<0.37pc>[F]{\ov{\mathbf{u}}}\ar@{-}[3,0]+<0pc,0pc>&&*+[o]+<0.37pc>[F]{\ov{\mathbf{u}}}\ar@{-}[1,0]+<0pc,0pc>&& *+[o]+<0.40pc>[F]{\mathbf{u}}\ar@{-}[1,0]+<0pc,0pc>&&&&&&*+[o]+<0.40pc>[F]{\mathbf{u}}\ar@{-}[3,0]+<0pc,0pc>&&&&\\
&&\ar@{-}`d/4pt [1,1] `[0,2] [0,2] &&&&&*+<0.1pc>[F]{\,\,\,\wt{\beta}\,\,\,}\ar@{-}[2,0]&&&&&&*+<0.1pc>[F]{\,\,\,\wt{\beta}'\,\,\,}\ar@{-}[2,0]&\\
&&&\ar@{-}[1,0]&&&&&&&&&&&\\
\ar@{-}`d/4pt [1,1]+<0.2pc,0pc> `[0,3] [0,3]&&&&&&&\ar@{-}[3,0]  &&&\ar@{-}`d/4pt [1,1]+<0.2pc,0pc> `[0,3] [0,3]&&&&\\
&\save[]+<0.2pc,0pc> \Drop{}\ar@{-}[6,0]+<0.2pc,0pc>\restore&&&&&&&&&&\save[]+<0.2pc,0pc> \Drop{}\ar@{-}[1,0]+<0.2pc,0pc>\restore&&&\\
&&&&&&&&&&&\save[]+<0.2pc,0pc> \Drop{}\ar@{-}[1,0]\restore&&&\\
&&&&&&&\ar@{-}`d/4pt [1,2] `[0,4] [0,4]&&&&&&&\\
&&&&&&&&&\ar@{-}[1,0]&&&&&\\
&&&&&&&&&\ar@{-}[2,-3]+<-0.4pc,0pc>&&&&&\\
&&&&&&&&&&&&&&\\
&\save[]+<0.2pc,0pc> \Drop{}\ar@{-}`d/4pt [1,2]+<0.1pc,0pc> `[0,4] [0,4]+<0.2pc,0pc>\restore&&&&&&&&&&&&&\\
&&&\save[]+<0.1pc,0pc> \Drop{}\ar@{-}[1,0]+<0.1pc,0pc>\restore&&&&&&&&&&&\\
&&&&&&&&&&&&&&\\
&&&\save[]+<0.1pc,0pc> \Drop{}\restore &&&&&&&&&&&
}}
\grow{\xymatrix@!0{
\\\\\\\\\\\\\\\\\\\\\\
\save\go+<0pt,0pt>\Drop{\txt{$=$}}\restore
}}
\grow{\xymatrix@!0{\\\\
\save\go+<0pt,3pt>\Drop{C}\restore \ar@{-}[2,0] &&&\save[]+<0.2pc,3pt> \Drop{C'}\ar@{-}[1,0]+<0.2pc,0pc>\restore&&&&& \save\go+<0pt,3pt>\Drop{D}\restore \ar@{-}[4,0] && \save\go+<0pt,3pt>\Drop{D'}\restore \ar@{-}[6,0]\\
&&&\save[]+<0.2pc,0pc> \Drop{}\ar@{-}`l/4pt [1,-1] [1,-1] \ar@{-}`r [1,2] [1,2]\restore&&&&&&&\\
\ar@{-}[1,1]+<-0.1pc,0.1pc> &&\ar@{-}[2,-2]&&&\ar@{-}[1,0]&&&&&&\\
&&&&&\ar@{-}`l/4pt [1,-1] [1,-1] \ar@{-}`r [1,1] [1,1]&&&&&&\\
\ar@{-}[4,0]+<0pc,-0.5pt>&&\ar@{-}[-1,-1]+<0.1pc,-0.1pc>\ar@{-}[4,0]&&\ar@{-}[2,0]&&\ar@{-}[1,1]+<-0.125pc,0.0pc> \ar@{-}[1,1]+<0.0pc,0.125pc>&&\ar@{-}[2,-2]&&\\
&&&&&&&&&&\\
&&&&\ar@/^0.1pc/ @{-}[2,2] \ar@/_0.1pc/ @{-}[2,2]&& \ar@/^0.1pc/ @{-}[2,-2]\ar@/_0.1pc/ @{-}[2,-2] &&\ar@{-}[-1,-1]+<0.125pc,0.0pc> \ar@{-}[-1,-1]+<0.0pc,-0.125pc> \ar@{-}[1,1]+<-0.125pc,0.0pc> \ar@{-}[1,1]+<0.0pc,0.125pc>&&\ar@{-}[2,-2]\\
&&&&&&&&&&\\
&&\ar@{-}[1,1]+<-0.125pc,0.0pc> \ar@{-}[1,1]+<0.0pc,0.125pc>&&\ar@{-}[2,-2]&&\ar@{-}[1,0]+<0pc,-0.5pt> &&\ar@{-}[1,0]+<0pc,-1pt>&&\ar@{-}[1,0]+<0pc,-1pt> \ar@{-}[-1,-1]+<0.125pc,0.0pc> \ar@{-}[-1,-1]+<0.0pc,-0.125pc>\\
*+[o]+<0.37pc>[F]{\ov{\mathbf{u}}}\ar@{-}[4,0]+<0pc,0pc>&&&&&&&&&&\\
&&\ar@{-}[1,0]+<0pc,-2pt>&&\ar@{-}[1,0]+<0pc,-2pt>\ar@{-}[-1,-1]+<0.125pc,0.0pc>\ar@{-}[-1,-1]+<0.0pc,-0.125pc>&& *+[o]+<0.40pc>[F]{\mathbf{u}}\ar@{-}[2,0]+<0pc,0pc> &&&*+<0.1pc>[F]{\,\,\,\wt{\beta}'\,\,\,}\ar@{-}[1,0]&\\
&&&&&&&&&\ar@{-}[1,-1]&\\
&&&*+<0.1pc>[F]{\,\,\,\wt{\beta}\,\,\,}\ar@{-}[2,0]&&&\ar@{-}`d/4pt [1,1] `[0,2] [0,2]&&&&\\
\ar@{-}[3,1]&&&&&&&\ar@{-}[1,0]&&&\\
&&&\ar@{-}`d/4pt [1,2] `[0,4] [0,4]&&&&&&&\\
&&&&&\ar@{-}[1,0]&&&&&\\
&\ar@{-}`d/4pt [1,2] `[0,4] [0,4]&&&&&&&&&\\
&&&\ar@{-}[1,0]&&&&&&&\\
&&&&&&&&&&
}}
\grow{\xymatrix@!0{
\\\\\\\\\\\\\\\\\\\\\\
\save\go+<0pt,0pt>\Drop{\txt{,}}\restore
}}
$$
where the first equality follows from the definition of $\wt{\theta}(\wt{\beta})\wt{\bullet}\wt{\theta}(\wt{\beta}')$, the second one, from Lemma~\ref{mengano3}, and the third one, from Lemma~\ref{inversa de convolucion}.
\end{proof}

\noindent\bf Proof of Proposition~\ref{cup producto en cleft}.\rm\enspace Let $C:=H^{\ot_c^s}$, $C':=H^{\ot_c^{s'}}$, $D:=A^{\ot_k^r}$ and $D':=A^{\ot_k^{r'}}$ and let
$$
\wt{\beta}\colon D\ot_k C\to E \quad\text{and}\quad \wt{\beta}'\colon D'\ot_k C'\to E
$$
be the maps induced by $\beta$ and $\beta'$, respectively. Let

\begin{itemize}

\smallskip

\item[-] $\boldsymbol{\gamma}$ denote both the maps $\gamma^{\ot_{\!A}^s}$ and $\gamma^{\ot_{\!A}^{s'}}$,

\smallskip

\item[-] $\mu$ denote both the maps $\mu_s$ and $\mu_{s'}$,

\smallskip

\item[-] $\mathbf{u}$ denote the map $\mu_{s'}\xcirc \gamma^{\ot_k^{s'}}$,

\smallskip

\item[-] $\ov{\mathbf{u}}$ denote both the maps $\mu_s\xcirc \ov{\gamma}^{\ot_k^s}\xcirc \gc_s$ and $\mu_{s'}\xcirc \ov{\gamma}^{\ot_k^{s'}}\xcirc \gc_{s'}$.

\smallskip

\end{itemize}
We have
$$
\grow{\xymatrix@!0{\\\\
&&&\save\go+<0pt,3pt>\Drop{D}\restore \ar@{-}[2,0]&&\save\go+<0pt,3pt>\Drop{D'}\restore \ar@{-}[2,2] &&\save\go+<0pt,3pt>\Drop{C}\restore \ar@{-}[1,-1]+<0.125pc,0.0pc> \ar@{-}[1,-1]+<0.0pc,0.125pc> &&\save\go+<0pt,3pt>\Drop{C'}\restore \ar@{-}[2,0]&\\
&&&&&&&&&&\\
&&&\ar@{-}[2,2]&&\ar@{-}[1,-1]+<0.125pc,0.0pc> \ar@{-}[1,-1]+<0.0pc,0.125pc> \ar@{-}[-1,1]+<-0.125pc,0.0pc> \ar@{-}[-1,1]+<0.0pc,-0.125pc>&&\ar@{-}[2,2]&&\ar@{-}[1,-1]+<0.125pc,0.0pc> \ar@{-}[1,-1]+<0.0pc,0.125pc>&\\
&&&&&&&&&&\\
&&&\ar@{-}[3,-2]\ar@{-}[-1,1]+<-0.125pc,0.0pc> \ar@{-}[-1,1]+<0.0pc,-0.125pc>&& \ar@{-}[2,2] &&\ar@{-}[1,-1]+<0.125pc,0.0pc> \ar@{-}[1,-1]+<0.0pc,0.125pc> \ar@{-}[-1,1]+<-0.125pc,0.0pc> \ar@{-}[-1,1]+<0.0pc,-0.125pc>&&\ar@{-}[3,1]&\\
&&&&&&&&&&\\
&&&&&\ar@{-}[1,0]\ar@{-}[-1,1]+<-0.125pc,0.0pc> \ar@{-}[-1,1]+<0.0pc,-0.125pc>&&\ar@{-}[1,1]&&&\\
&\ar@{-}[1,0]&&&& \ar@{-}[1,0]&&& \ar@{-}[8,0] && \ar@{-}[8,0]\\
&\ar@{-}`l/4pt [1,-1] [1,-1] \ar@{-}`r [1,1] [1,1]&&&&\ar@{-}`l/4pt [1,-1] [1,-1] \ar@{-}`r [1,1] [1,1]&&&&&\\
\ar@{-}[2,0]&&\ar@{-}[1,1]+<-0.1pc,0.1pc> && \ar@{-}[2,-2]&&\ar@{-}[2,0]&&&&\\
&&&&&&&&&&\\
\ar@{-}[1,1]+<-0.1pc,0.1pc> && \ar@{-}[2,-2]&&\ar@{-}[-1,-1]+<0.1pc,-0.1pc>\ar@{-}[1,0]&&\ar@{-}[1,0]&&&&\\
&&&&&&&&&&\\
\ar@{-}[1,0]+<0pc,-0.5pt>&&\ar@{-}[1,0]+<0pc,-0.5pt>\ar@{-}[-1,-1]+<0.1pc,-0.1pc>&& *+[o]+<0.37pc>[F]{\boldsymbol{\gamma}}\ar@{-}[2,0]+<0pc,0pc>&& *+[o]+<0.37pc>[F]{\boldsymbol{\gamma}}\ar@{-}[2,0]+<0pc,0pc>&&&&\\
&&&&&&&&&&\\
*+[o]+<0.37pc>[F]{\ov{\mathbf{u}}}\ar@{-}[2,0]+<0pc,0pc>&&*+[o]+<0.37pc>[F]{\ov{\mathbf{u}}}\ar@{-}[2,0]+<0pc,0pc>&&&&&&&&\\
&&&&&&&*+<0.1pc>[F]{\wt{\theta}(\wt{\beta})\wt{\bullet}\wt{\theta}(\wt{\beta}')}\ar@{-}[2,0]&&&\\
\ar@{-}`d/4pt [1,1] `[0,2] [0,2]&&&&&&&&&&\\
&\ar@{-}[1,0]&&&&&&\ar@{-}[1,-2]&&&\\
&\ar@{-}`d/4pt [1,2] `[0,4] [0,4]&&&&&&&&&\\
&&&\ar@{-}[1,0]&&&&&&&\\
&&&&&&&&&&
}}
\grow{\xymatrix@!0{
\\\\\\\\\\\\\\\\\\\\\\\\
\save\go+<0pt,0pt>\Drop{\txt{$=$}}\restore
}}
\grow{\xymatrix@!0{
&&\save\go+<0pt,3pt>\Drop{D}\restore \ar@{-}[2,0]&&\save\go+<0pt,3pt>\Drop{D'}\restore \ar@{-}[2,2] &&\save\go+<0pt,3pt>\Drop{C}\restore \ar@{-}[1,-1]+<0.125pc,0.0pc> \ar@{-}[1,-1]+<0.0pc,0.125pc> &&\save\go+<0pt,3pt>\Drop{C'}\restore \ar@{-}[2,0]&\\
&&&&&&&&&\\
&&\ar@{-}[2,2]&&\ar@{-}[1,-1]+<0.125pc,0.0pc> \ar@{-}[1,-1]+<0.0pc,0.125pc> \ar@{-}[-1,1]+<-0.125pc,0.0pc> \ar@{-}[-1,1]+<0.0pc,-0.125pc>&&\ar@{-}[2,2]&&\ar@{-}[1,-1]+<0.125pc,0.0pc> \ar@{-}[1,-1]+<0.0pc,0.125pc>&\\
&&&&&&&&&\\
&&\ar@{-}[3,-2]\ar@{-}[-1,1]+<-0.125pc,0.0pc> \ar@{-}[-1,1]+<0.0pc,-0.125pc>&& \ar@{-}[2,2] &&\ar@{-}[1,-1]+<0.125pc,0.0pc> \ar@{-}[1,-1]+<0.0pc,0.125pc> \ar@{-}[-1,1]+<-0.125pc,0.0pc> \ar@{-}[-1,1]+<0.0pc,-0.125pc>&&\ar@{-}[3,2]&\\
&&&&&&&&&\\
&&&&\ar@{-}[1,-1]+<0.2pc,0pt> \ar@{-}[-1,1]+<-0.125pc,0.0pc> \ar@{-}[-1,1]+<0.0pc,-0.125pc>&&\ar@{-}[1,2]&&&\\
\ar@{-}[2,0] &&&\save[]+<0.2pc,0pt> \Drop{}\ar@{-}[1,0]+<0.2pc,0pc>\restore&&&&&\ar@{-}[4,0] &&\ar@{-}[6,0]\\
&&&\save[]+<0.2pc,0pc> \Drop{}\ar@{-}`l/4pt [1,-1] [1,-1] \ar@{-}`r [1,2] [1,2]\restore&&&&&&&\\
\ar@{-}[1,1]+<-0.1pc,0.1pc> &&\ar@{-}[2,-2]&&&\ar@{-}[1,0]&&&&&&\\
&&&&&\ar@{-}`l/4pt [1,-1] [1,-1] \ar@{-}`r [1,1] [1,1]&&&&&&\\
\ar@{-}[1,0]+<0pc,-0.5pt>&&\ar@{-}[-1,-1]+<0.1pc,-0.1pc>\ar@{-}[4,0]&&\ar@{-}[2,0]&&\ar@{-}[1,1]+<-0.125pc,0.0pc> \ar@{-}[1,1]+<0.0pc,0.125pc>&&\ar@{-}[2,-2]&&\\
&&&&&&&&&&\\
*+[o]+<0.37pc>[F]{\ov{\mathbf{u}}}\ar@{-}[2,0]+<0pc,0pc>&&&&\ar@/^0.1pc/ @{-}[2,2] \ar@/_0.1pc/ @{-}[2,2]&& \ar@/^0.1pc/ @{-}[2,-2]\ar@/_0.1pc/ @{-}[2,-2] &&\ar@{-}[-1,-1]+<0.125pc,0.0pc> \ar@{-}[-1,-1]+<0.0pc,-0.125pc> \ar@{-}[1,1]+<-0.125pc,0.0pc> \ar@{-}[1,1]+<0.0pc,0.125pc>&&\ar@{-}[2,-2]\\
&&&&&&&&&&\\
\ar@{-}[8,1]&&\ar@{-}[1,1]+<-0.125pc,0.0pc> \ar@{-}[1,1]+<0.0pc,0.125pc>&&\ar@{-}[2,-2]&&\ar@{-}[1,0]+<0pc,-0.5pt> &&\ar@{-}[1,0]+<0pc,-1pt>&&\ar@{-}[1,0]+<0pc,-1pt> \ar@{-}[-1,-1]+<0.125pc,0.0pc> \ar@{-}[-1,-1]+<0.0pc,-0.125pc>\\
&&&&&&&&&&\\
&&\ar@{-}[1,0]+<0pc,-2pt>&&\ar@{-}[1,0]+<0pc,-2pt>\ar@{-}[-1,-1]+<0.125pc,0.0pc>\ar@{-}[-1,-1]+<0.0pc,-0.125pc>& &*+[o]+<0.40pc>[F]{\mathbf{u}}\ar@{-}[2,0]+<0pc,0pc> &&&*+<0.1pc>[F]{\,\,\,\wt{\beta}'\,\,\,}\ar@{-}[1,0]&\\
&&&&&&&&&\ar@{-}[1,-1]&\\
&&&*+<0.1pc>[F]{\,\,\,\wt{\beta}\,\,\,}\ar@{-}[2,0]&&&\ar@{-}`d/4pt [1,1] `[0,2] [0,2]&&&&\\
&&&&&&&\ar@{-}[1,0]&&&\\
&&&\ar@{-}`d/4pt [1,2] `[0,4] [0,4]&&&&&&&\\
&&&&&\ar@{-}[1,0]&&&&&\\
&\ar@{-}`d/4pt [1,2] `[0,4] [0,4]&&&&&&&&&\\
&&&\ar@{-}[1,0]&&&&&&&\\
&&&&&&&&&&
}}
\grow{\xymatrix@!0{
\\\\\\\\\\\\\\\\\\\\\\\\
\save\go+<0pt,0pt>\Drop{\txt{$=$}}\restore
}}
\grow{\xymatrix@!0{
&&\save\go+<0pt,3pt>\Drop{D}\restore \ar@{-}[2,0]&&\save\go+<0pt,3pt>\Drop{D'}\restore \ar@{-}[2,2] &&\save\go+<0pt,3pt>\Drop{C}\restore \ar@{-}[1,-1]+<0.125pc,0.0pc> \ar@{-}[1,-1]+<0.0pc,0.125pc> &&&&\save\go+<0pt,3pt>\Drop{C'}\restore \ar@{-}[3,0]\\
&&&&&&&&&&\\
&&\ar@{-}[2,2]&&\ar@{-}[1,-1]+<0.125pc,0.0pc> \ar@{-}[1,-1]+<0.0pc,0.125pc>\ar@{-}[-1,1]+<-0.125pc,0.0pc> \ar@{-}[-1,1]+<0.0pc,-0.125pc>&&\ar@{-}[1,2]&&&&\\
&&&&&&&&\ar@{-}[2,2]&&\ar@{-}[1,-1]+<0.125pc,0.0pc> \ar@{-}[1,-1]+<0.0pc,0.125pc>\\
&&\ar@{-}[2,-2]\ar@{-}[-1,1]+<-0.125pc,0.0pc> \ar@{-}[-1,1]+<0.0pc,-0.125pc>&&\ar@{-}[4,0] &&&&&&\\
&&&&&&&&\ar@{-}[1,-1]\ar@{-}[-1,1]+<-0.125pc,0.0pc> \ar@{-}[-1,1]+<0.0pc,-0.125pc>&&\ar@{-}[4,0]\\
\ar@{-}[7,0]&&&&&&&\ar@{-}[1,0]&&&\\
&&&&&&&\ar@{-}`l/4pt [1,-1] [1,-1] \ar@{-}`r [1,1] [1,1]&&&\\
&&&&\ar@{-}[2,2]&&\ar@{-}[1,-1]+<0.125pc,0.0pc> \ar@{-}[1,-1]+<0.0pc,0.125pc>&&\ar@{-}[1,0]&&\\
&&&&&&&&\ar@{-}[1,1]+<-0.125pc,0.0pc> \ar@{-}[1,1]+<0.0pc,0.125pc>&&\ar@{-}[2,-2]\\
&&&&\ar@{-}[1,-1]\ar@{-}[-1,1]+<-0.125pc,0.0pc> \ar@{-}[-1,1]+<0.0pc,-0.125pc>&&\ar@{-}[3,0]&&&&\\
&&&\ar@{-}[1,0]&&&&&\ar@{-}[1,0]+<0pc,-1pt>&&\ar@{-}[1,0]+<0pc,-1pt> \ar@{-}[-1,-1]+<0.125pc,0.0pc>\ar@{-}[-1,-1]+<0.0pc,-0.125pc>\\
&&&\ar@{-}`l/4pt [1,-1] [1,-1] \ar@{-}`r [1,1] [1,1]&&&&&&&\\
\ar@{-}[1,1]+<-0.1pc,0.1pc> && \ar@{-}[2,-2]&&\ar@/^0.1pc/ @{-}[2,2] \ar@/_0.1pc/ @{-}[2,2]&& \ar@/^0.1pc/ @{-}[2,-2]\ar@/_0.1pc/ @{-}[2,-2]&&&*+<0.1pc>[F]{\,\,\,\wt{\beta}'\,\,\,}\ar@{-}[1,0]&\\
&&&&&&&&&\ar@{-}[5,-1]&\\
\ar@{-}[1,0]+<0pc,-0.5pt>&&\ar@{-}[-1,-1]+<0.1pc,-0.1pc> \ar@{-}[1,1]+<-0.125pc,0.0pc> \ar@{-}[1,1]+<0.0pc,0.125pc>&&\ar@{-}[2,-2]&&\ar@{-}[1,0]+<0pc,-0.5pt>&&&&\\
&&&&&&&&&&\\
*+[o]+<0.37pc>[F]{\ov{\mathbf{u}}}\ar@{-}[2,0]+<0pc,0pc>&&\ar@{-}[1,0]+<0pc,-2pt>&& \ar@{-}[1,0]+<0pc,-2pt> \ar@{-}[-1,-1]+<0.125pc,0.0pc>\ar@{-}[-1,-1]+<0.0pc,-0.125pc>&&*+[o]+<0.40pc>[F]{\mathbf{u}}\ar@{-}[2,0]+<0pc,0pc>&&&&\\
&&&&&&&&&&\\
\ar@{-}[4,1]&&&*+<0.1pc>[F]{\,\,\,\wt{\beta}\,\,\,}\ar@{-}[2,0]&&&\ar@{-}`d/4pt [1,1] `[0,2] [0,2]&&&&\\
&&&&&&&\ar@{-}[1,0]&&&\\
&&&\ar@{-}`d/4pt [1,2] `[0,4] [0,4]&&&&&&&\\
&&&&&\ar@{-}[1,0]&&&&&\\
&\ar@{-}`d/4pt [1,2] `[0,4] [0,4]&&&&&&&&&\\
&&&\ar@{-}[1,0]&&&&&&&\\
&&&&&&&&&&
}}
\grow{\xymatrix@!0{
\\\\\\\\\\\\\\\\\\\\\\\\
\save\go+<0pt,0pt>\Drop{\txt{$=$}}\restore
}}
\grow{\xymatrix@!0{\\\\
&&\save\go+<0pt,3pt>\Drop{D}\restore \ar@{-}[2,0] && \save\go+<0pt,3pt>\Drop{D'}\restore \ar@{-}[2,2] &&\save\go+<0pt,3pt>\Drop{C}\restore \ar@{-}[1,-1]+<0.125pc,0.0pc> \ar@{-}[1,-1]+<0.0pc,0.125pc>&&& \save\go+<0pt,3pt>\Drop{C'}\restore \ar@{-}[1,0]&\\
&&&&&&&&&\ar@{-}`l/4pt [1,-1] [1,-1] \ar@{-}`r [1,1] [1,1]&\\
&&\ar@{-}[2,2]&&\ar@{-}[1,-1]+<0.125pc,0.0pc> \ar@{-}[1,-1]+<0.0pc,0.125pc>\ar@{-}[-1,1]+<-0.125pc,0.0pc> \ar@{-}[-1,1]+<0.0pc,-0.125pc>&&\ar@{-}[2,2]&&\ar@{-}[1,-1]+<0.125pc,0.0pc> \ar@{-}[1,-1]+<0.0pc,0.125pc>&&\ar@{-}[3,0]+<0pc,-1pt>\\
&&&&&&&&&&\\
&&\ar@{-}[2,-2]\ar@{-}[-1,1]+<-0.125pc,0.0pc> \ar@{-}[-1,1]+<0.0pc,-0.125pc>&&\ar@{-}[2,2]&&\ar@{-}[1,-1]+<0.125pc,0.0pc> \ar@{-}[1,-1]+<0.0pc,0.125pc>\ar@{-}[-1,1]+<-0.125pc,0.0pc> \ar@{-}[-1,1]+<0.0pc,-0.125pc>&&\ar@{-}[1,0]+<0pc,-1pt>&&\\
&&&&&&&&&&\\
\ar@{-}[3,0]&&&&\ar@{-}[1,-1]\ar@{-}[-1,1]+<-0.125pc,0.0pc> \ar@{-}[-1,1]+<0.0pc,-0.125pc>&&\ar@{-}[3,0]&&&*+<0.1pc>[F]{\,\,\,\wt{\beta}'\,\,\,}\ar@{-}[2,0]&\\
&&&\ar@{-}[1,0]&&&&&&&\\
&&&\ar@{-}`l/4pt [1,-1] [1,-1] \ar@{-}`r [1,1] [1,1]&&&&&&\ar@{-}[7,-1]&\\
\ar@{-}[1,1]+<-0.1pc,0.1pc> && \ar@{-}[2,-2]&&\ar@/^0.1pc/ @{-}[2,2] \ar@/_0.1pc/ @{-}[2,2]&& \ar@/^0.1pc/ @{-}[2,-2]\ar@/_0.1pc/ @{-}[2,-2]&&&&\\
&&&&&&&&&&\\
\ar@{-}[1,0]+<0pc,-0.5pt>&&\ar@{-}[-1,-1]+<0.1pc,-0.1pc>\ar@{-}[1,1]+<-0.125pc,0.0pc> \ar@{-}[1,1]+<0.0pc,0.125pc>&&\ar@{-}[2,-2]&&\ar@{-}[1,0]+<0pc,-0.5pt>&&&&\\
&&&&&&&&&&\\
*+[o]+<0.40pc>[F]{\mathbf{u}}\ar@{-}[2,0]+<0pc,0pc>&&\ar@{-}[1,0]+<0pc,-2pt>&&\ar@{-}[-1,-1]+<0.125pc,0.0pc> \ar@{-}[-1,-1]+<0.0pc,-0.125pc>\ar@{-}[1,0]+<0pc,-2pt>&&*+[o]+<0.37pc>[F]{\boldsymbol{\gamma}}\ar@{-}[2,0]+<0pc,-0.5pt>&&&&\\
&&&&&&&&&&\\
\ar@{-}[4,1]&&&*+<0.1pc>[F]{\,\,\,\wt{\beta}\,\,\,}\ar@{-}[2,0]&&&\ar@{-}`d/4pt [1,1] `[0,2] [0,2]&&&&\\
&&&&&&&\ar@{-}[1,0]&&&\\
&&&\ar@{-}`d/4pt [1,2] `[0,4] [0,4]&&&&&&&\\
&&&&&\ar@{-}[1,0]&&&&&\\
&\ar@{-}`d/4pt [1,2] `[0,4] [0,4]&&&&&&&&&\\
&&&\ar@{-}[1,0]&&&&&&&\\
&&&&&&&&&&
}}
\grow{\xymatrix@!0{
\\\\\\\\\\\\\\\\\\\\\\\\
\save\go+<0pt,0pt>\Drop{\txt{$=$}}\restore
}}
\grow{\xymatrix@!0{
&&\save\go+<0pt,3pt>\Drop{D}\restore \ar@{-}[1,0] &&&\save\go+<0pt,3pt>\Drop{D'}\restore \ar@{-}[2,2]&& \save\go+<0pt,3pt>\Drop{C}\restore \ar@{-}[1,-1]+<0.125pc,0.0pc> \ar@{-}[1,-1]+<0.0pc,0.125pc>&&&\save\go+<0pt,3pt>\Drop{C'}\restore  \ar@{-}[1,0]&\\
&&\ar@{-}[1,-1]&&&&&&&&\ar@{-}`l/4pt [1,-1] [1,-1] \ar@{-}`r [1,1] [1,1]&\\
&\ar@{-}[7,0]&&&&\ar@{-}[2,-2]\ar@{-}[-1,1]+<-0.125pc,0.0pc> \ar@{-}[-1,1]+<0.0pc,-0.125pc>&&\ar@{-}[2,2]&& \ar@{-}[1,-1]+<0.125pc,0.0pc> \ar@{-}[1,-1]+<0.0pc,0.125pc>&&\ar@{-}[3,0]+<0pc,-1pt>\\
&&&&&&&&&&&\\
&&&\ar@{-}[3,0]&&&&\ar@{-}[1,-1]\ar@{-}[-1,1]+<-0.125pc,0.0pc> \ar@{-}[-1,1]+<0.0pc,-0.125pc>&& \ar@{-}[1,0]+<0pc,-1pt>&&\\
&&&&&&\ar@{-}[1,0]&&&&&\\
&&&&&&\ar@{-}`l/4pt [1,-1] [1,-1] \ar@{-}`r [1,1] [1,1]&&&&*+<0.1pc>[F]{\,\,\,\wt{\beta}'\,\,\,}\ar@{-}[8,0]&\\
&&&\ar@{-}[1,1]+<-0.1pc,0.1pc> && \ar@{-}[2,-2]&&\ar@{-}[1,1]&&&&\\
&&&&&&&&\ar@{-}[2,0]+<0pc,-0.5pt>&&&\\
&\ar@{-}[2,2]&&\ar@{-}[1,-1]+<0.125pc,0.0pc> \ar@{-}[1,-1]+<0.0pc,0.125pc> &&\ar@{-}[-1,-1]+<0.1pc,-0.1pc>\ar@{-}[1,1]&&&&&&\\
&&&&&&\ar@{-}[5,0]+<0pc,-2pt>&&&&&\\
&\ar@{-}[1,0]\ar@{-}[-1,1]+<-0.125pc,0.0pc> \ar@{-}[-1,1]+<0.0pc,-0.125pc>&&\ar@{-}[1,1]&&&&&*+[o]+<0.40pc>[F]{\mathbf{u}}\ar@{-}[3,0]+<0pc,0pc>&&&\\
&\ar@{-}`l/4pt [1,-1] [1,-1] \ar@{-}`r [1,1] [1,1]&&&\ar@{-}[3,0]+<0pc,-2pt>&&&&&&&\\
\ar@{-}[1,0]+<0pc,-0.5pt>&&\ar@{-}`d/4pt [1,2][1,2]&&&&&&&&&\\
&&&&&&&&\ar@{-}`d/4pt [1,1] `[0,2] [0,2]&&&\\
*+[o]+<0.40pc>[F]{\mathbf{u}}\ar@{-}[7,0]+<0pc,0pc>&&&&&&&&&\ar@{-}[3,0]&&\\
&&&&&*+<0.1pc>[F]{\,\,\,\wt{\beta}\,\,\,}\ar@{-}[2,0]&&&&&&\\
&&&&&&&&&&&\\
&&&&&\ar@{-}`d/4pt [1,2] `[0,4] [0,4]&&&&&&\\
&&&&&&&\ar@{-}[1,0]&&&&\\
&&&&&&&\ar@{-}[2,-3]&&&&\\
&&&&&&&&&&&\\
\ar@{-}`d/4pt [1,2] `[0,4] [0,4]&&&&&&&&&&&\\
&&\ar@{-}[1,0]&&&&&&&&&\\
&&&&&&&&&&&
}}
\grow{\xymatrix@!0{
\\\\\\\\\\\\\\\\\\\\\\\\
\save\go+<0pt,0pt>\Drop{\txt{,}}\restore
}}
$$
where the first equality follows from Lemma~\ref{mengano5}, the second and third ones are easy to check (and left to the reader), and the last one follows from Lemma~\ref{mengano4}. So, in order to finish the proof it suffices to note that the first diagram represents $\wt{\vartheta}\bigl(\wt{\theta}(\wt{\beta})\wt{\bullet} \wt{\theta}(\wt{\beta}')\bigr)$ and that this map induces $(-1)^{rs'}\vartheta \bigl(\theta(\beta)\bullet \theta(\beta') \bigr)$.\qed

\begin{lemma}\label{mengano6} Let $C:=H^{\ot_c^s}$, $C':=H^{\ot_c^{s'}}$ and $D:=A^{\ot_k^r}$. We have:
$$
%
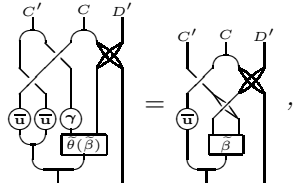
\grow{\xymatrix@!0{
&\save\go+<0pt,3pt>\Drop{C'}\restore \ar@{-}[1,0]&&&& \save\go+<0pt,3pt>\Drop{C}\restore \ar@{-}[1,0] &&& \save\go+<0pt,3pt>\Drop{D'}\restore \ar@{-}[2,0]\\
&\ar@{-}`l/4pt [1,-1] [1,-1] \ar@{-}`r [1,1] [1,1]&&&&\ar@{-}`l/4pt [1,-1] [1,-1] \ar@{-}`r [1,1] [1,1]&&&\\
\ar@{-}[2,0]&&\ar@{-}[1,1]+<-0.1pc,0.1pc> && \ar@{-}[2,-2]&&\ar@/^0.1pc/ @{-}[2,2] \ar@/_0.1pc/ @{-}[2,2]&& \ar@/^0.1pc/ @{-}[2,-2]\ar@/_0.1pc/ @{-}[2,-2]\\
&&&&&&&&\\
\ar@{-}[1,1]+<-0.1pc,0.1pc> && \ar@{-}[2,-2]&&\ar@{-}[-1,-1]+<0.1pc,-0.1pc>\ar@{-}[2,0]&&\ar@{-}[5,0]+<0pc,-1pt> && \ar@{-}[9,0]\\
&&&&&&&&\\
\ar@{-}[1,0]+<0pc,-0.5pt>&&\ar@{-}[-1,-1]+<0.1pc,-0.1pc>\ar@{-}[1,0]+<0pc,-0.5pt>&&\ar@{-}[1,0]+<0pc,-0.5pt>&&&&\\
&&&&&&&&\\
*+[o]+<0.37pc>[F]{\ov{\mathbf{u}}}\ar@{-}[1,0]+<0pc,-1pt>&&*+[o]+<0.37pc>[F]{\ov{\mathbf{u}}}\ar@{-}[1,0] +<0pc,-1pt> &&*+[o]+<0.37pc>[F]{\boldsymbol{\gamma}} \ar@{-}[1,0]+<0pc,-1pt>&&&&\\
\ar@{-}`d/4pt [1,1] `[0,2] [0,2]&&&&&&&&\\
&\ar@{-}[1,0]&&&&*+<0.1pc>[F]{\,\widetilde{\theta}(\widetilde{\beta})\,}\ar@{-}[1,0]&&&\\
&\ar@{-}`d/4pt [1,2] `[0,4] [0,4]&&&&&&&\\
&&&\ar@{-}[1,0]&&&&&\\
&&&&&&&&
}}
\grow{\xymatrix@!0{
\\\\\\\\\\\\\\
\save\go+<0pt,0pt>\Drop{\txt{$=$}}\restore
}}
\grow{\xymatrix@!0{\\\\
\save\go+<0pt,3pt>\Drop{C'}\restore \ar@{-}[2,0]&&&\save\go+<0pt,3pt>\Drop{C}\restore \ar@{-}[1,0] &&& \save\go+<0pt,3pt>\Drop{D'}\restore \ar@{-}[2,0]\\
&&&\ar@{-}`l/4pt [1,-1] [1,-1] \ar@{-}`r [1,1] [1,1]&&&\\
\ar@{-}[1,1]+<-0.1pc,0.1pc> && \ar@{-}[2,-2]&&\ar@/^0.1pc/ @{-}[2,2] \ar@/_0.1pc/ @{-}[2,2]&& \ar@/^0.1pc/ @{-}[2,-2]\ar@/_0.1pc/ @{-}[2,-2]\\
&&&&&&\\
\ar@{-}[1,0]+<0pc,-0.5pt>&&\ar@{-}[-1,-1]+<0.1pc,-0.1pc>\ar@{-}[1,1]+<-0.125pc,0.0pc> \ar@{-}[1,1]+<0.0pc,0.125pc>&&\ar@{-}[2,-2]&&\ar@{-}[7,0]\\
&&&&&&\\
*+[o]+<0.37pc>[F]{\ov{\mathbf{u}}}\ar@{-}[3,0]&&\ar@{-}[1,0]+<0pc,-1pt>&& \ar@{-}[-1,-1]+<0.125pc,0.0pc>\ar@{-}[-1,-1]+<0.0pc,-0.125pc>\ar@{-}[1,0]+<0pc,-1pt>&&\\
&&&&&&\\
&&&*+<0.1pc>[F]{\,\,\,\widetilde{\beta}\,\,\,}\ar@{-}[1,0]&&&\\
\ar@{-}`d/4pt [1,1]+<0.2pc,0pc> `[0,3] [0,3]&&&&&&\\
&\save[]+<0.2pc,0pc> \Drop{}\ar@{-}[1,0]+<0.2pc,0pc>\restore&&&&&\\
&\save[]+<0.2pc,0pc> \Drop{}\restore&&&&&
}}
\grow{\xymatrix@!0{
\\\\\\\\\\\\\\
\save\go+<0pt,0pt>\Drop{\txt{,}}\restore
}}
$$
where

\begin{itemize}

\smallskip

\item[-] $\boldsymbol{\gamma}$ denotes the maps $\gamma^{\ot_k^{s'}}$ and $\gamma^{\ot_{\!A}^{s'}}$,

\smallskip

\item[-] $\mu$ denotes the map $\mu_{s'}$,

\smallskip

\item[-] $\mathbf{u}:=\mu_{s'}\xcirc \gamma^{\ot_k^{s'}}$ and $\ov{\mathbf{u}}$ denotes both the maps $\mu_s\xcirc \ov{\gamma}^{\ot_k^s}\xcirc \gc_s$ and $\mu_{s'}\xcirc \ov{\gamma}^{\ot_k^{s'}}\xcirc \gc_{s'}$.

\smallskip

\end{itemize}

\end{lemma}

\begin{proof} By the definition of $\wt{\theta}$ and Lemma~\ref{inversa de convolucion},
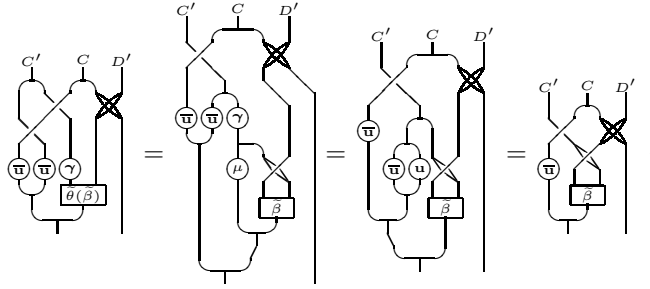
$$
%
\grow{\xymatrix@!0{\\\\\\\\
&\save\go+<0pt,3pt>\Drop{C'}\restore \ar@{-}[1,0]&&&& \save\go+<0pt,3pt>\Drop{C}\restore \ar@{-}[1,0] &&& \save\go+<0pt,3pt>\Drop{D'}\restore \ar@{-}[2,0]\\
&\ar@{-}`l/4pt [1,-1] [1,-1] \ar@{-}`r [1,1] [1,1]&&&&\ar@{-}`l/4pt [1,-1] [1,-1] \ar@{-}`r [1,1] [1,1]&&&\\
\ar@{-}[2,0]&&\ar@{-}[1,1]+<-0.1pc,0.1pc> && \ar@{-}[2,-2]&&\ar@/^0.1pc/ @{-}[2,2] \ar@/_0.1pc/ @{-}[2,2]&& \ar@/^0.1pc/ @{-}[2,-2]\ar@/_0.1pc/ @{-}[2,-2]\\
&&&&&&&&\\
\ar@{-}[1,1]+<-0.1pc,0.1pc> && \ar@{-}[2,-2]&&\ar@{-}[-1,-1]+<0.1pc,-0.1pc>\ar@{-}[2,0]&&\ar@{-}[5,0]+<0pc,-1pt> && \ar@{-}[9,0]\\
&&&&&&&&\\
\ar@{-}[1,0]+<0pc,-0.5pt>&&\ar@{-}[-1,-1]+<0.1pc,-0.1pc>\ar@{-}[1,0]+<0pc,-0.5pt>&&\ar@{-}[1,0]+<0pc,-0.5pt>&&&&\\
&&&&&&&&\\
*+[o]+<0.37pc>[F]{\ov{\mathbf{u}}}\ar@{-}[1,0]+<0pc,-1pt>&&*+[o]+<0.37pc>[F]{\ov{\mathbf{u}}}\ar@{-}[1,0] +<0pc,-1pt> &&*+[o]+<0.37pc>[F]{\boldsymbol{\gamma}} \ar@{-}[1,0]+<0pc,-1pt>&&&&\\
\ar@{-}`d/4pt [1,1] `[0,2] [0,2]&&&&&&&&\\
&\ar@{-}[1,0]&&&&*+<0.1pc>[F]{\,\widetilde{\theta}(\widetilde{\beta})\,}\ar@{-}[1,0]&&&\\
&\ar@{-}`d/4pt [1,2] `[0,4] [0,4]&&&&&&&\\
&&&\ar@{-}[1,0]&&&&&\\
&&&&&&&&
}}
\grow{\xymatrix@!0{
\\\\\\\\\\\\\\\\\\\\\\
\save\go+<0pt,0pt>\Drop{\txt{$=$}}\restore
}}
\grow{\xymatrix@!0{
\save\go+<0pt,3pt>\Drop{C'}\restore \ar@{-}[2,0]&&&&\ar@{-}[1,0]\save\go+<0pt,3pt>\Drop{C}\restore \ar@{-}[1,0] &&&& \save\go+<0pt,3pt>\Drop{D'}\restore \ar@{-}[2,0]\\
&&&&\ar@{-}`l/4pt [1,-2] [1,-2] \ar@{-}`r [1,2] [1,2]&&&&\\
\ar@{-}[1,1]+<-0.1pc,0.1pc> && \ar@{-}[2,-2]&&&&\ar@/^0.1pc/ @{-}[2,2] \ar@/_0.1pc/ @{-}[2,2]&& \ar@/^0.1pc/ @{-}[2,-2]\ar@/_0.1pc/ @{-}[2,-2]\\
&&&&&&&&\\
\ar@{-}[3,0]+<0pc,-0.5pt>&&\ar@{-}[-1,-1]+<0.1pc,-0.1pc>\ar@{-}[1,1]&&&&\ar@{-}[1,0]&&\ar@{-}[2,2]\\
&&&\ar@{-}[1,0]&&&\ar@{-}[2,2]&&&&\\
&&&\ar@{-}`l/4pt [1,-1] [1,-1] \ar@{-}`r [1,1] [1,1]&&&&&&&\ar@{-}[15,0]\\
\ar@{-}[0,0]+<0pc,-0.5pt>&&\ar@{-}[0,0]+<0pc,-0.5pt>&&\ar@{-}[0,0]+<0pc,-0.5pt>&&&&\ar@{-}[4,0]&&\\
*+[o]+<0.37pc>[F]{\ov{\mathbf{u}}}\ar@{-}[1,0]&&*+[o]+<0.37pc>[F]{\ov{\mathbf{u}}}\ar@{-}[1,0]&& *+[o]+<0.37pc>[F]{\boldsymbol{\gamma}} \ar@{-}[3,0]+<0pc,-0.5pt>&&&&&&\\
\ar@{-}`d/4pt [1,1] `[0,2] [0,2]&&&&&&&&&&\\
&\ar@{-}[9,0]&&&\ar@{-}`r/4pt [1,2] [1,2]&&&&&&\\
&&&&&& \ar@{-}[1,1]+<-0.125pc,0.0pc> \ar@{-}[1,1]+<0.0pc,0.125pc>&&\ar@{-}[2,-2]&&\\
&&&&*+[o]+<0.40pc>[F]{\mu}\ar@{-}[4,0]+<0pc,0pc>&&&&&&\\
&&&&&&\ar@{-}[1,0]+<0pc,-1pt>&&\ar@{-}[-1,-1]+<0.125pc,0.0pc>\ar@{-}[-1,-1]+<0.0pc,-0.125pc> \ar@{-}[1,0]+<0pc,-1pt>&&\\
&&&&&&&&&&\\
&&&&&&&*+<0.1pc>[F]{\,\,\,\widetilde{\beta}\,\,\,}\ar@{-}[1,0]&&&\\
&&&&\ar@{-}`d/4pt [1,1]+<0.2pc,0pc> `[0,3] [0,3]&&&&&&\\
&&&&&\save[]+<0.2pc,0pc> \Drop{}\ar@{-}[1,0]+<0.2pc,0pc>\restore&&&&&\\
&&&&&\save[]+<0.2pc,0pc> \Drop{}\ar@{-}[1,0]\restore&&&&&\\
&\ar@{-}`d/4pt [1,2] `[0,4] [0,4]&&&&&&&&&\\
&&&\ar@{-}[1,0]&&&&&&&\\
&&&&&&&&&&
}}
\grow{\xymatrix@!0{
\\\\\\\\\\\\\\\\\\\\\\
\save\go+<0pt,0pt>\Drop{\txt{$=$}}\restore
}}
\grow{\xymatrix@!0{\\\\
&\save\go+<0pt,3pt>\Drop{C'}\restore \ar@{-}[2,0]&&&&\save\go+<0pt,3pt>\Drop{C}\restore \ar@{-}[1,0] &&&& \save\go+<0pt,3pt>\Drop{D'}\restore \ar@{-}[2,0]\\
&&&&&\ar@{-}`l/4pt [1,-2] [1,-2] \ar@{-}`r [1,2] [1,2]&&&&\\
&\ar@{-}[1,1]+<-0.1pc,0.1pc> && \ar@{-}[3,-3]&&&&\ar@/^0.1pc/ @{-}[2,2] \ar@/_0.1pc/ @{-}[2,2]&& \ar@/^0.1pc/ @{-}[2,-2]\ar@/_0.1pc/ @{-}[2,-2]\\
&&&&&&&&&\\
&&&\ar@{-}[-1,-1]+<0.1pc,-0.1pc>\ar@{-}[1,1]&&&&\ar@{-}[5,0]&&\ar@{-}[14,0]\\
\ar@{-}[1,0]+<0pc,-0.5pt>&&&&\ar@{-}[1,0]&&&&&\\
&&&&\ar@{-}`l/4pt [1,-1] [1,-1] \ar@{-}`r [1,1] [1,1]&&&&&\\
*+[o]+<0.37pc>[F]{\ov{\mathbf{u}}}\ar@{-}[6,0]&&&\ar@{-}[1,0] &&\ar@{-}[2,0]+<0pc,-0.5pt>&&&&\\
&&&\ar@{-}`l/4pt [1,-1] [1,-1] \ar@{-}`r [1,1] [1,1]&&&&&&\\
&&\ar@{-}[0,0]+<0pc,-0.5pt>&&\ar@{-}[0,0]+<0pc,-0.5pt>&\ar@{-}[1,1]+<-0.125pc,0.0pc> \ar@{-}[1,1]+<0.0pc,0.125pc>&&\ar@{-}[2,-2]&&\\
&&*+[o]+<0.37pc>[F]{\ov{\mathbf{u}}}\ar@{-}[1,0]&&*+[o]+<0.4pc>[F]{\mathbf{u}}\ar@{-}[1,0]&&&&&\\
&&\ar@{-}`d/4pt [1,1] `[0,2] [0,2]&&&\ar@{-}[1,0]+<0pc,-1pt> &&\ar@{-}[-1,-1]+<0.125pc,0.0pc>\ar@{-}[-1,-1]+<0.0pc,-0.125pc>\ar@{-}[1,0]+<0pc,-1pt>&&\\
&&&\ar@{-}[1,0]&&&&&&\\
\ar@{-}`d/4pt [1,1]+<0.2pc,0pc> `[0,3] [0,3]&&&&&&*+<0.1pc>[F]{\,\,\,\widetilde{\beta}\,\,\,}\ar@{-}[3,0]&&&\\
&\save[]+<0.2pc,0pc> \Drop{}\ar@{-}[1,0]+<0.2pc,0pc>\restore&&&&&&&&\\
&\save[]+<0.2pc,0pc> \Drop{}\ar@{-}[1,1]\restore&&&&&&&&\\
&&\ar@{-}`d/4pt [1,2] `[0,4] [0,4]&&&&&&&\\
&&&&\ar@{-}[1,0]&&&&&\\
&&&&&&&&&
}}
\grow{\xymatrix@!0{
\\\\\\\\\\\\\\\\\\\\\\
\save\go+<0pt,0pt>\Drop{\txt{$=$}}\restore
}}
\grow{\xymatrix@!0{\\\\\\\\\\\\
\save\go+<0pt,3pt>\Drop{C'}\restore \ar@{-}[2,0]&&&\save\go+<0pt,3pt>\Drop{C}\restore \ar@{-}[1,0] &&& \save\go+<0pt,3pt>\Drop{D'}\restore \ar@{-}[2,0]\\
&&&\ar@{-}`l/4pt [1,-1] [1,-1] \ar@{-}`r [1,1] [1,1]&&&\\
\ar@{-}[1,1]+<-0.1pc,0.1pc> && \ar@{-}[2,-2]&&\ar@/^0.1pc/ @{-}[2,2] \ar@/_0.1pc/ @{-}[2,2]&& \ar@/^0.1pc/ @{-}[2,-2]\ar@/_0.1pc/ @{-}[2,-2]\\
&&&&&&\\
\ar@{-}[1,0]+<0pc,-0.5pt>&&\ar@{-}[-1,-1]+<0.1pc,-0.1pc>\ar@{-}[1,1]+<-0.125pc,0.0pc> \ar@{-}[1,1]+<0.0pc,0.125pc>&&\ar@{-}[2,-2]&&\ar@{-}[7,0]\\
&&&&&&\\
*+[o]+<0.37pc>[F]{\ov{\mathbf{u}}}\ar@{-}[3,0]&&\ar@{-}[1,0]+<0pc,-1pt>&& \ar@{-}[-1,-1]+<0.125pc,0.0pc>\ar@{-}[-1,-1]+<0.0pc,-0.125pc>\ar@{-}[1,0]+<0pc,-1pt>&&\\
&&&&&&\\
&&&*+<0.1pc>[F]{\,\,\,\widetilde{\beta}\,\,\,}\ar@{-}[1,0]&&&\\
\ar@{-}`d/4pt [1,1]+<0.2pc,0pc> `[0,3] [0,3]&&&&&&\\
&\save[]+<0.2pc,0pc> \Drop{}\ar@{-}[1,0]+<0.2pc,0pc>\restore&&&&&\\
&\save[]+<0.2pc,0pc> \Drop{}\restore&&&&&
}}
\grow{\xymatrix@!0{
\\\\\\\\\\\\\\\\\\\\\\
\save\go+<0pt,0pt>\Drop{\txt{,}}\restore
}}
$$
as desired
\end{proof}

\begin{definition} Let $r'\le r$ and $s'\le s$. For
$$
m\ot_A \ov{\bx}_{1s}\ot \ba_{1r} \in M\ot_A E^{\ot_{\!A}^s}\ot A^{\ot^r}\quad\text{and}\quad \alpha\in \Hom_{(A,E)}\bigl(E^{\ot_{\!A}^{s'}}\ot A^{\ot^{r'}},E\bigr),
$$
we define
$$
\bigl(m_{\!A}\gamma_A(\bv_{1s})\ot \ba_{1r}\bigr) \wt{\bullet} \alpha \in M\ot_A E^{\ot_{\!A}^{s-s'}}\ot A^{\ot^{r-r'}}
$$
by
$$
(m\ot_{\!A}\gamma_A(\bv_{1s})\ot \ba_{1r})\wt{\bullet} \alpha\! :=\! \sum_i m \alpha\bigl(\gamma_A(\bv_{1s'})\ot \ba_{1r'}^{(i)}\bigr)\ot_{\!A}\gamma_A(\bv_{s'+1,s}^{(i)}) \ot \ba_{r'+1,r},
$$
where $\sum_i \ba_{1r'}^{(i)}\ot_k \bv_{s'+1,s}^{(i)} := \ov{\chi}\bigl(\bv_{s'+1,s}\ot \ba_{1r'}\bigr)$.
\end{definition}

\noindent\bf Proof of Proposition~\ref{cap producto en cleft}.\rm\enspace The case $s<s'$ or $r<r'$ is trivial. Assume that $s'\le s$ and $r'\le r$. Let $C:=H^{\ot_c^s}$, $C':=H^{\ot_c^{s'}}$, $D:=A^{\ot_k^r}$ and $D':=A^{\ot_k^{r'}}$ and let
$$
\wt{\beta}\colon D\ot_k C\to E
$$
be the map induced by $\beta$. Let

\begin{itemize}

\smallskip

\item[-] $\boldsymbol{\gamma}$ denote both the maps $\gamma^{\ot_k^s}$ and $\gamma^{\ot_{\!A}^{s'}}$,

\smallskip

\item[-] $\mu$ denote the map $\mu_s$,

\smallskip

\item[-] $\mathbf{u}$ denote the map $\mu_s\xcirc \gamma^{\ot_k^s}$,

\smallskip

\item[-] $\ov{\mathbf{u}}$ denote both the maps $\mu_s\xcirc \ov{\gamma}^{\ot_k^s}\xcirc \gc_s$ and $\mu_{s'}\xcirc \ov{\gamma}^{\ot_k^{s'}}\xcirc \gc_{s'}$.

\smallskip

\end{itemize}
A direct computation shows that
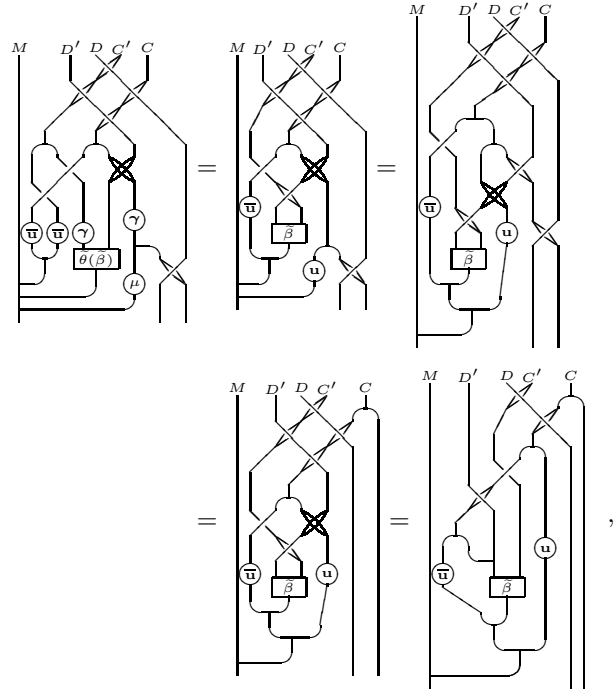
\begin{align*}
\grow{\xymatrix@!0{\\\\\\
\save\go+<0pt,3pt>\Drop{M}\restore \ar@{-}[21,0]&&&&\save\go+<0pt,3pt>\Drop{D'}\restore \ar@{-}[2,0]&& \save\go+<0pt,3pt>\Drop{D}\restore \ar@{-}[2,2] && \save\go+<0pt,3pt>\Drop{C'}\restore \ar@{-}[1,-1]+<0.125pc,0.0pc> \ar@{-}[1,-1]+<0.0pc,0.125pc> && \save\go+<0pt,3pt>\Drop{C}\restore \ar@{-}[2,0]\\
&&&&&&&&&&&&&\\
&&&&\ar@{-}[2,2]&& \ar@{-}[1,-1]+<0.125pc,0.0pc> \ar@{-}[1,-1]+<0.0pc,0.125pc>\ar@{-}[-1,1]+<-0.125pc,0.0pc> \ar@{-}[-1,1]+<0.0pc,-0.125pc>&&\ar@{-}[2,2]&& \ar@{-}[1,-1]+<0.125pc,0.0pc> \ar@{-}[1,-1]+<0.0pc,0.125pc>&&&\\
&&&&&&&&&&&&&\\
&&&&\ar@{-}[2,-2]\ar@{-}[-1,1]+<-0.125pc,0.0pc> \ar@{-}[-1,1]+<0.0pc,-0.125pc>&&\ar@{-}[2,2]&& \ar@{-}[1,-1]+<0.125pc,0.0pc> \ar@{-}[1,-1]+<0.0pc,0.125pc>\ar@{-}[-1,1]+<-0.125pc,0.0pc> \ar@{-}[-1,1]+<0.0pc,-0.125pc>&&\ar@{-}[3,3]&&&\\
&&&&&&&&&&&&&\\
&&\ar@{-}[1,0] &&&&\ar@{-}[1,0] \ar@{-}[-1,1]+<-0.125pc,0.0pc> \ar@{-}[-1,1]+<0.0pc,-0.125pc>&&\ar@{-}[1,1]&&&&&\\
&&\ar@{-}`l/4pt [1,-1] [1,-1] \ar@{-}`r [1,1] [1,1]&&&&\ar@{-}`l/4pt [1,-1] [1,-1] \ar@{-}`r [1,1] [1,1]&&& \ar@{-}[1,0]&&&&\ar@{-}[9,0]\\
&\ar@{-}[2,0]&&\ar@{-}[1,1]+<-0.1pc,0.1pc> && \ar@{-}[2,-2]&&\ar@/^0.1pc/ @{-}[2,2] \ar@/_0.1pc/ @{-}[2,2]&& \ar@/^0.1pc/ @{-}[2,-2]\ar@/_0.1pc/ @{-}[2,-2]\\
&&&&&&&&&&&&&\\
&\ar@{-}[1,1]+<-0.1pc,0.1pc> && \ar@{-}[2,-2]&&\ar@{-}[-1,-1]+<0.1pc,-0.1pc>\ar@{-}[2,0]&& \ar@{-}[5,0]+<0pc,-1pt> && \ar@{-}[2,0]+<0pc,-0.5pt>&&&&\\
&&&&&&&&&&&&&\\
&\ar@{-}[1,0]+<0pc,-0.5pt>&&\ar@{-}[-1,-1]+<0.1pc,-0.1pc>\ar@{-}[1,0]+<0pc,-0.5pt>&& \ar@{-}[1,0]+<0pc,-0.5pt>&&&&&&&&\\
&&&&&&&&&*+[o]+<0.37pc>[F]{\boldsymbol{\gamma}} \ar@{-}[4,0]+<0pc,-0.5pt>&&&&\\
&*+[o]+<0.37pc>[F]{\ov{\mathbf{u}}}\ar@{-}[1,0]+<0pc,-1pt>&&*+[o]+<0.37pc>[F]{\ov{\mathbf{u}}}\ar@{-}[1,0] +<0pc,-1pt> &&*+[o]+<0.37pc>[F]{\boldsymbol{\gamma}} \ar@{-}[1,0]+<0pc,-1pt>&&&& &&&&\\
&\ar@{-}`d/4pt [1,1] `[0,2] [0,2]&&&&&&&&\ar@{-}`r/4pt [1,2] [1,2]&&&&\\
&&\ar@{-}[1,0]&&&&*+<0.1pc>[F]{\,\widetilde{\theta}(\widetilde{\beta})\,}\ar@{-}[2,0]&&&&& \ar@{-}[1,1]+<-0.125pc,0.0pc> \ar@{-}[1,1]+<0.0pc,0.125pc>&&\ar@{-}[2,-2]\\
&&\ar@{-}`d/4pt [1,-2][1,-2]&&&&&&&&&&&\\
&&&&&&\ar@{-}`d/4pt [1,-6][1,-6]&&&*+[o]+<0.40pc>[F]{\mu}\ar@{-}[1,0]+<0pc,-1pt>&&\ar@{-}[3,0]&& \ar@{-}[-1,-1]+<0.125pc,0.0pc>\ar@{-}[-1,-1]+<0.0pc,-0.125pc>\ar@{-}[3,0]\\
&&&&&&&&&\ar@{-}`d/4pt [1,-9][1,-9]&&&&\\
&&&&&&&&&&&&&\\
&&&&&&&&&&&&&
}}
& \grow{\xymatrix@!0{
\\\\\\\\\\\\\\\\\\\\\\\\
\save\go+<0pt,0pt>\Drop{\txt{$=$}}\restore
}}
\grow{\xymatrix@!0{\\\\\\
\save\go+<0pt,3pt>\Drop{M}\restore \ar@{-}[20,0]&&\save\go+<0pt,3pt>\Drop{D'}\restore \ar@{-}[2,0]&& \save\go+<0pt,3pt>\Drop{D}\restore \ar@{-}[2,2] && \save\go+<0pt,3pt>\Drop{C'}\restore \ar@{-}[1,-1]+<0.125pc,0.0pc> \ar@{-}[1,-1]+<0.0pc,0.125pc> && \save\go+<0pt,3pt>\Drop{C}\restore \ar@{-}[2,0]\\
&&&&&&&&&&\\
&&\ar@{-}[2,2]&& \ar@{-}[1,-1]+<0.125pc,0.0pc> \ar@{-}[1,-1]+<0.0pc,0.125pc>\ar@{-}[-1,1]+<-0.125pc,0.0pc> \ar@{-}[-1,1]+<0.0pc,-0.125pc>&&\ar@{-}[2,2]&&\ar@{-}[1,-1]+<0.125pc,0.0pc>\ar@{-}[1,-1]+<0.0pc,0.125pc>&&\\
&&&&&&&&&&\\
&&\ar@{-}[2,-1]\ar@{-}[-1,1]+<-0.125pc,0.0pc> \ar@{-}[-1,1]+<0.0pc,-0.125pc>&&\ar@{-}[2,2]&& \ar@{-}[1,-1]+<0.125pc,0.0pc> \ar@{-}[1,-1]+<0.0pc,0.125pc>\ar@{-}[-1,1]+<-0.125pc,0.0pc> \ar@{-}[-1,1]+<0.0pc,-0.125pc>&&\ar@{-}[2,2]&&\\
&&&&&&&&&&\\
&\ar@{-}[2,0]&&&\ar@{-}[1,0]\ar@{-}[-1,1]+<-0.125pc,0.0pc> \ar@{-}[-1,1]+<0.0pc,-0.125pc>&&\ar@{-}[1,1] &&&&\ar@{-}[10,0]\\
&&&&\ar@{-}`l/4pt [1,-1] [1,-1] \ar@{-}`r [1,1] [1,1]&&&\ar@{-}[1,0]\\
&\ar@{-}[1,1]+<-0.1pc,0.1pc> && \ar@{-}[2,-2]&&\ar@/^0.1pc/ @{-}[2,2] \ar@/_0.1pc/ @{-}[2,2]&& \ar@/^0.1pc/ @{-}[2,-2]\ar@/_0.1pc/ @{-}[2,-2]\\
&&&&&&&&&&\\
&\ar@{-}[1,0]+<0pc,-0.5pt>&&\ar@{-}[-1,-1]+<0.1pc,-0.1pc>\ar@{-}[1,1]+<-0.125pc,0.0pc> \ar@{-}[1,1]+<0.0pc,0.125pc>&&\ar@{-}[2,-2]&&\ar@{-}[5,0]\\
&&&&&&&&&&\\
&*+[o]+<0.37pc>[F]{\ov{\mathbf{u}}}\ar@{-}[3,0]&&\ar@{-}[1,0]+<0pc,-1pt>&& \ar@{-}[-1,-1]+<0.125pc,0.0pc>\ar@{-}[-1,-1]+<0.0pc,-0.125pc>\ar@{-}[1,0]+<0pc,-1pt>&&&&&\\
&&&&&&&&&&\\
&&&&*+<0.1pc>[F]{\,\,\,\widetilde{\beta}\,\,\,}\ar@{-}[1,0]&&&&&&\\
&\ar@{-}`d/4pt [1,1]+<0.2pc,0pc> `[0,3] [0,3]&&&&&&\ar@{-}`l/4pt [1,-1] [1,-1] \ar@{-}`r [1,1] [1,1]&&&\\
&&\save[]+<0.2pc,0pc> \Drop{}\ar@{-}[1,0]+<0.2pc,0pc>\restore&&&& \ar@{-}[0,0]+<0pc,-0.5pt>&&\ar@{-}[1,1]+<-0.125pc,0.0pc> \ar@{-}[1,1]+<0.0pc,0.125pc>&&\ar@{-}[2,-2]\\
&&\save[]+<0.2pc,0pc> \Drop{} \ar@{-}`d/4pt [1,-2][1,-2]\restore&&&& *+[o]+<0.4pc>[F]{\mathbf{u}}\ar@{-}[1,0]&&&&\\
&&&&&& \ar@{-}`d/4pt [1,-6][1,-6]&&\ar@{-}[2,0]&&\ar@{-}[-1,-1]+<0.125pc,0.0pc> \ar@{-}[-1,-1]+<0.0pc,-0.125pc>\ar@{-}[2,0]\\
&&&&&&&&&&\\
&&&&&&&&&&
}}
\grow{\xymatrix@!0{
\\\\\\\\\\\\\\\\\\\\\\\\
\save\go+<0pt,0pt>\Drop{\txt{$=$}}\restore
}}
\grow{\xymatrix@!0{
\save\go+<0pt,3pt>\Drop{M}\restore \ar@{-}[26,0]&&&&\save\go+<0pt,3pt>\Drop{D'}\restore \ar@{-}[2,0]&& \save\go+<0pt,3pt>\Drop{D}\restore \ar@{-}[2,2] && \save\go+<0pt,3pt>\Drop{C'}\restore \ar@{-}[1,-1]+<0.125pc,0.0pc> \ar@{-}[1,-1]+<0.0pc,0.125pc> && \save\go+<0pt,3pt>\Drop{C}\restore \ar@{-}[2,0]\\
&&&&&&&&&&&\\
&&&&\ar@{-}[2,2]&& \ar@{-}[1,-1]+<0.125pc,0.0pc> \ar@{-}[1,-1]+<0.0pc,0.125pc>\ar@{-}[-1,1]+<-0.125pc,0.0pc> \ar@{-}[-1,1]+<0.0pc,-0.125pc>&&\ar@{-}[2,2]&& \ar@{-}[1,-1]+<0.125pc,0.0pc> \ar@{-}[1,-1]+<0.0pc,0.125pc>&\\
&&&&&&&&&&&&\\
&&&&\ar@{-}[3,-3]\ar@{-}[-1,1]+<-0.125pc,0.0pc> \ar@{-}[-1,1]+<0.0pc,-0.125pc>&&\ar@{-}[2,2]&& \ar@{-}[1,-1]+<0.125pc,0.0pc> \ar@{-}[1,-1]+<0.0pc,0.125pc>\ar@{-}[-1,1]+<-0.125pc,0.0pc> \ar@{-}[-1,1]+<0.0pc,-0.125pc>&&\ar@{-}[1,1]&\\
&&&&&&&&&&&\ar@{-}[11,0]\\
&&&&&&\ar@{-}[1,-2]+<0.2pc,0pc> \ar@{-}[-1,1]+<-0.125pc,0.0pc> \ar@{-}[-1,1]+<0.0pc,-0.125pc>&&\ar@{-}[1,1]&&&\\
&\ar@{-}[2,0]&&&\save[]+<0.2pc,0pc> \Drop{}\ar@{-}[1,0]+<0.2pc,0pc>\restore&&&&&\ar@{-}[4,0]&&\\
&&&&\save[]+<0.2pc,0pc> \Drop{}\ar@{-}`l/4pt [1,-1] [1,-1] \ar@{-}`r [1,2] [1,2]\restore&&&&&&&\\
&\ar@{-}[1,1]+<-0.1pc,0.1pc> && \ar@{-}[2,-2]&&&\ar@{-}[1,0]&&&&&\\
&&&&&&\ar@{-}`l/4pt [1,-1] [1,-1] \ar@{-}`r [1,1] [1,1]&&&&&\\
&\ar@{-}[3,0]+<0pc,-0.5pt>&&\ar@{-}[4,0] \ar@{-}[-1,-1]+<0.1pc,-0.1pc>&&\ar@{-}[2,0] &&\ar@{-}[1,1]+<-0.125pc,0.0pc> \ar@{-}[1,1]+<0.0pc,0.125pc>&&\ar@{-}[2,-2]&&\\
&&&&&&&&&&&\\
&&&&&\ar@/^0.1pc/ @{-}[2,2] \ar@/_0.1pc/ @{-}[2,2]&& \ar@/^0.1pc/ @{-}[2,-2]\ar@/_0.1pc/ @{-}[2,-2]&&\ar@{-}[-1,-1]+<0.125pc,0.0pc>\ar@{-}[-1,-1]+<0.0pc,-0.125pc>\ar@{-}[3,0]&&\\
&&&&&&&&&&&\\
&*+[o]+<0.37pc>[F]{\ov{\mathbf{u}}}\ar@{-}[5,0]&&\ar@{-}[1,1]+<-0.125pc,0.0pc> \ar@{-}[1,1]+<0.0pc,0.125pc>&&\ar@{-}[2,-2]&&\ar@{-}[1,0]+<0pc,-0.5pt>&&&&\\
&&&&&&&&&\ar@{-}[1,1]+<-0.125pc,0.0pc> \ar@{-}[1,1]+<0.0pc,0.125pc>&&\ar@{-}[2,-2]\\
&&&\ar@{-}[1,0]+<0pc,-1pt>&&\ar@{-}[-1,-1]+<0.125pc,0.0pc>\ar@{-}[-1,-1]+<0.0pc,-0.125pc>\ar@{-}[1,0]+<0pc,-1pt> &&*+[o]+<0.4pc>[F]{\mathbf{u}}\ar@{-}[1,0] &&&&\\
&&&&&&&\ar@{-}[4,-1]+<0.2pc,0pc>&&\ar@{-}[8,0]&&\ar@{-}[-1,-1]+<0.125pc,0.0pc> \ar@{-}[-1,-1]+<0.0pc,-0.125pc>\ar@{-}[8,0]\\
&&&&*+<0.1pc>[F]{\,\,\,\widetilde{\beta}\,\,\,}\ar@{-}[1,0]&&&&&&&\\
&\ar@{-}`d/4pt [1,1]+<0.2pc,0pc> `[0,3] [0,3]&&&&&&&&&&\\
&&\save[]+<0.2pc,0pc> \Drop{}\ar@{-}[1,0]+<0.2pc,0pc>\restore&&&&&&&&&\\
&&\save[]+<0.2pc,0pc> \Drop{}\ar@{-}`d/4pt [1,2]+<0.1pc,0pc>  `[0,4]+<0.2pc,0pc>  [0,4]+<0.2pc,0pc> \restore &&&&&&&&&\\
&&&&\save[]+<0.1pc,0pc> \Drop{}\ar@{-}[1,0]+<0.1pc,0pc> \restore&&&&&&&\\
&&&&\save[]+<0.1pc,0pc> \Drop{}  \ar@{-}`d/4pt [1,-4][1,-4] \restore&&&&&&&\\
&&&&&&&&&&&\\
&&&&&&&&&&&
}}\\
& \grow{\xymatrix@!0{
\\\\\\\\\\\\\\\\\\\\\\
\save\go+<0pt,0pt>\Drop{\txt{$=$}}\restore
}}
\grow{\xymatrix@!0{\\
\save\go+<0pt,3pt>\Drop{M}\restore \ar@{-}[22,0]&&&\save\go+<0pt,3pt>\Drop{D'}\restore \ar@{-}[2,0]&& \save\go+<0pt,3pt>\Drop{D}\restore \ar@{-}[2,2] && \save\go+<0pt,3pt>\Drop{C'}\restore \ar@{-}[1,-1]+<0.125pc,0.0pc> \ar@{-}[1,-1]+<0.0pc,0.125pc> &&& \save\go+<0pt,3pt>\Drop{C}\restore \ar@{-}[1,0]&\\
&&&&&&&&&&\ar@{-}`l/4pt [1,-1] [1,-1] \ar@{-}`r [1,1] [1,1]&\\
&&&\ar@{-}[2,2]&& \ar@{-}[1,-1]+<0.125pc,0.0pc> \ar@{-}[1,-1]+<0.0pc,0.125pc>\ar@{-}[-1,1]+<-0.125pc,0.0pc> \ar@{-}[-1,1]+<0.0pc,-0.125pc>&&\ar@{-}[2,2]&& \ar@{-}[1,-1]+<0.125pc,0.0pc> \ar@{-}[1,-1]+<0.0pc,0.125pc>&&
\ar@{-}[20,0]\\
&&&&&&&&&&&\\
&&&\ar@{-}[2,-2]\ar@{-}[-1,1]+<-0.125pc,0.0pc> \ar@{-}[-1,1]+<0.0pc,-0.125pc>&&\ar@{-}[2,2]&& \ar@{-}[1,-1]+<0.125pc,0.0pc> \ar@{-}[1,-1]+<0.0pc,0.125pc>\ar@{-}[-1,1]+<-0.125pc,0.0pc> \ar@{-}[-1,1]+<0.0pc,-0.125pc>&&\ar@{-}[18,0]&&\\
&&&&&&&&&&&\\
&\ar@{-}[3,0]&&&&\ar@{-}[1,-1]\ar@{-}[-1,1]+<-0.125pc,0.0pc>\ar@{-}[-1,1]+<0.0pc,-0.125pc>&&\ar@{-}[3,0]&&&&\\
&&&&\ar@{-}[1,0]&&&&&&&\\
&&&&\ar@{-}`l/4pt [1,-1] [1,-1] \ar@{-}`r [1,1] [1,1]&&&&&&&\\
&\ar@{-}[1,1]+<-0.1pc,0.1pc> && \ar@{-}[2,-2]&&\ar@/^0.1pc/ @{-}[2,2] \ar@/_0.1pc/ @{-}[2,2]&& \ar@/^0.1pc/ @{-}[2,-2]\ar@/_0.1pc/ @{-}[2,-2]&&&&\\
&&&&&&&&&&&\\
&\ar@{-}[2,0]+<0pc,-0.5pt>&&\ar@{-}[-1,-1]+<0.1pc,-0.1pc>\ar@{-}[1,1]+<-0.125pc,0.0pc> \ar@{-}[1,1]+<0.0pc,0.125pc>&&\ar@{-}[2,-2]&&\ar@{-}[2,0]+<0pc,-0.5pt>&&&&\\
&&&&&&&&&&&\\
&&&\ar@{-}[1,0]+<0pc,-1pt>&&\ar@{-}[1,0]+<0pc,-1pt>\ar@{-}[-1,-1]+<0.125pc,0.0pc> \ar@{-}[-1,-1]+<0.0pc,-0.125pc>&&&&&&\\
&*+[o]+<0.37pc>[F]{\ov{\mathbf{u}}}\ar@{-}[2,0]&&&&&&*+[o]+<0.4pc>[F]{\mathbf{u}}\ar@{-}[1,0]&&&&\\
&&&&*+<0.1pc>[F]{\,\,\,\widetilde{\beta}\,\,\,}\ar@{-}[1,0]&&&\ar@{-}[3,-1]+<0.2pc,0pc>&&&&\\
&\ar@{-}`d/4pt [1,1]+<0.2pc,0pc> `[0,3] [0,3]&&&&&&&&&&\\
&&\save[]+<0.2pc,0pc> \Drop{}\ar@{-}[1,0]+<0.2pc,0pc>\restore&&&&&&&&&\\
&&\save[]+<0.2pc,0pc>\Drop{}\ar@{-}`d/4pt[1,2]+<0.1pc,0pc>`[0,4]+<0.2pc,0pc>[0,4]+<0.2pc,0pc>\restore &&&&&&&&&\\
&&&&\save[]+<0.1pc,0pc> \Drop{}\ar@{-}[1,0]+<0.1pc,0pc> \restore&&&&&&&\\
&&&&\save[]+<0.1pc,0pc> \Drop{}  \ar@{-}`d/4pt [1,-4][1,-4] \restore&&&&&&&\\
&&&&&&&&&&&\\
&&&&&&&&&&&
}}
\grow{\xymatrix@!0{
\\\\\\\\\\\\\\\\\\\\\\
\save\go+<0pt,0pt>\Drop{\txt{$=$}}\restore
}}
\grow{\xymatrix@!0{
\save\go+<0pt,3pt>\Drop{M}\restore \ar@{-}[24,0]&&&\save\go+<0pt,3pt>\Drop{D'}\restore \ar@{-}[8,0]&&& \save\go+<0pt,3pt>\Drop{D}\restore \ar@{-}[2,2] && \save\go+<0pt,3pt>\Drop{C'}\restore \ar@{-}[1,-1]+<0.125pc,0.0pc> \ar@{-}[1,-1]+<0.0pc,0.125pc> &&& \save\go+<0pt,3pt>\Drop{C}\restore \ar@{-}[1,0]&\\
&&&&&&&&&&&\ar@{-}`l/4pt [1,-1] [1,-1] \ar@{-}`r [1,1] [1,1]&\\
&&&&&&\ar@{-}[2,-1]\ar@{-}[-1,1]+<-0.125pc,0.0pc> \ar@{-}[-1,1]+<0.0pc,-0.125pc>&&\ar@{-}[2,2]&& \ar@{-}[1,-1]+<0.125pc,0.0pc> \ar@{-}[1,-1]+<0.0pc,0.125pc>&&\ar@{-}[22,0]\\
&&&&&&&&&&&&\\
&&&&&\ar@{-}[2,0]&&&\ar@{-}[1,0]\ar@{-}[-1,1]+<-0.125pc,0.0pc> \ar@{-}[-1,1]+<0.0pc,-0.125pc>&&\ar@{-}[1,1]&&\\
&&&&&&&&\ar@{-}`l/4pt [1,-1] [1,-1] \ar@{-}`r [1,1] [1,1]&&&\ar@{-}[19,0]&\\
&&&&&\ar@{-}[1,1]+<-0.1pc,0.1pc> && \ar@{-}[2,-2]&&\ar@{-}[6,0]+<0pc,-0.5pt>&&&\\
&&&&&&&&&&&&\\
&&&\ar@{-}[2,2]&& \ar@{-}[1,-1]+<0.125pc,0.0pc> \ar@{-}[1,-1]+<0.0pc,0.125pc>&&\ar@{-}[7,0]+<0pc,-1pt> \ar@{-}[-1,-1]+<0.1pc,-0.1pc>&&&&&\\
&&&&&&&&&&&&\\
&&&\ar@{-}[1,-1]\ar@{-}[-1,1]+<-0.125pc,0.0pc> \ar@{-}[-1,1]+<0.0pc,-0.125pc>&&\ar@{-}[5,0]+<0pc,-1pt> &&&&&&&\\
&&\ar@{-}[1,0]&&&&&&&&&&\\
&&\ar@{-}`l/4pt [1,-1] [1,-1] \ar@{-}`r [1,1] [1,1]&&&&&&&&&&\\
&\ar@{-}[1,0]+<0pc,-0.5pt>&&\ar@{-}`d/4pt [1,2][1,2]&&&&&&*+[o]+<0.4pc>[F]{\mathbf{u}}\ar@{-}[7,0]&&&\\
&&&&&&&&&&&&\\
&*+[o]+<0.37pc>[F]{\ov{\mathbf{u}}}\ar@{-}[1,0]&&&&&&&&&&&\\
&\ar@{-}[2,3]&&&&&*+<0.1pc>[F]{\,\,\,\widetilde{\beta}\,\,\,}\ar@{-}[2,0]&&&&&&\\
&&&&&&&&&&&&\\
&&&&\ar@{-}`d/4pt [1,1] `[0,2] [0,2]&&&&&&&&\\
&&&&&\ar@{-}[1,0]&&&&&&&\\
&&&&&\ar@{-}`d/4pt [1,2] `[0,4] [0,4]&&&&&&&\\
&&&&&&&\ar@{-}[1,0]&&&&&\\
&&&&&&&\ar@{-}`d/4pt [1,-7][1,-7]\ar@{-}`d/4pt [1,-2][1,-2]&&&&&\\
&&&&&&&&&&&&\\
&&&&&&&&&&&&&
}}
\grow{\xymatrix@!0{
\\\\\\\\\\\\\\\\\\\\\\
\save\go+<0pt,0pt>\Drop{\txt{,}}\restore
}}
\end{align*}
where the first equality follows from Lemma~\ref{mengano6}, the second one from Lemma~\ref{mengano3}, the third and fourth ones are easy to check (and left to the reader), and the last one follows from Lemma~\ref{mengano4}. Since the first diagram represents the map
$$
\objectmargin{0.2pc}\objectwidth{0.2pc}
\def\objectstyle{\textstyle}
\def\labelstyle{\\textstyle}
\xymatrix @R=-2pt @C=3pc {\ov{X}_{rs}(M)\rto & \ov{X}_{r-r',s-s'}(M)\\
[m\ot\ba_{1r}]_H\ot_k \bh_{1s} \ar@{|->}[0,1] & \wt{\theta}\bigl(\wt{\vartheta}([m\ot\ba_{1r}]_H\ot_k \bh_{1s})\wt{\bullet}\wt{\theta}(\wt{\beta})\bigr)},
$$
and this map induces $(-1)^{r'(s-s')}$ times the morphism
$$
\objectmargin{0.2pc}\objectwidth{0.2pc}
\def\objectstyle{\textstyle}
\def\labelstyle{\\textstyle}
\xymatrix @R=-2pt @C=3pc {\ov{X}_{rs}(M)\rto & \ov{X}_{r-r',s-s'}(M)\\
[m\ot\ba_{1r}]_H\ot_k \bh_{1s} \ar@{|->}[0,1] & ([m\ot\ba_{1r}]_H\ot_k \bh_{1s})\star \beta},
$$
this finish the proof.\qed

\begin{lemma}\label{pirulo} Let $C:= H^{\ot_c^j}$. We have
$$
%
%
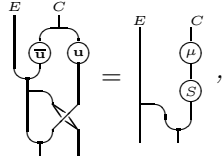
\grow{\xymatrix@!0{\\
\save\go+<0pt,3pt>\Drop{E}\restore \ar@{-}[4,0]&&&\save\go+<0.2pc,3pt>\Drop{C}\restore\save[]+<0.2pc,0pc> \Drop{}\ar@{-}[1,0]+<0.2pc,0pc>\restore&&\\
&&&\save[]+<0.2pc,0pc> \Drop{}\ar@{-}`l/4pt [1,-1] [1,-1] \ar@{-}`r [1,2] [1,2]\restore&&\\
&&&&&\\
&&*+[o]+<0.37pc>[F]{\overline{\mathbf{u}}} \ar@{-}[1,0]+<0pc,0pc>&&& *+[o]+<0.40pc>[F]{\mathbf{u}}\ar@{-}[4,0]+<0pc,0pc>\\
\ar@{-}`d/4pt [1,1] `[0,2] [0,2]&&&&&\\
&\ar@{-}[4,0]&&&&\\
& \ar@{-}`r/4pt [1,2] [1,2]&&&&\\
&&&\ar@{-}[1,1]+<-0.125pc,0.0pc> \ar@{-}[1,1]+<0.0pc,0.125pc>&&\ar@{-}[2,-2]\\
&&&&&\\
&\ar@{-}`d/4pt [1,1] `[0,2] [0,2] &&&&\ar@{-}[-1,-1]+<0.125pc,0.0pc> \ar@{-}[-1,-1]+<0.0pc,-0.125pc>\ar@{-}[2,0]\\
&&\ar@{-}[1,0]&&&\\
&&&&&
}}
\grow{\xymatrix@!0{
\\\\\\\\\\\\
\save\go+<0pt,0pt>\Drop{\txt{$=$}}\restore
}}
\grow{\xymatrix@!0{\\\\
\save\go+<0pt,3pt>\Drop{E}\restore \ar@{-}[9,0]&&&&\save\go+<0.2pc,3pt>\Drop{C}\restore \ar@{-}[1,0]\\
&&&&\\
&&&&*+[o]+<0.37pc>[F]{\mu}\ar@{-}[2,0]+<0pc,0pc>\\
&&&&\\
&&&&\\
&&&&*+[o]+<0.34pc>[F]{S}\ar@{-}[2,0]+<0pc,0pc>\\
\ar@{-}`r/4pt [1,2] [1,2]&&&&\\
&&\ar@{-}`d/4pt [1,1] `[0,2] [0,2]&&\\
&&&\ar@{-}[1,0]&\\
&&&&}}
\grow{\xymatrix@!0{
\\\\\\\\\\\\
\save\go+<0pt,0pt>\Drop{\txt{,}}\restore
}}
$$
where $\mu:=\mu_j$, $\mathbf{u}:=\mu\xcirc \gamma^{\ot_k^j}$ and $\ov{\mathbf{u}}:=\mu\xcirc \ov{\gamma}^{\ot_k^j}\xcirc\gc_j$.
\end{lemma}

\begin{proof} In fact, by Lemma~\ref{inversa de convolucion},
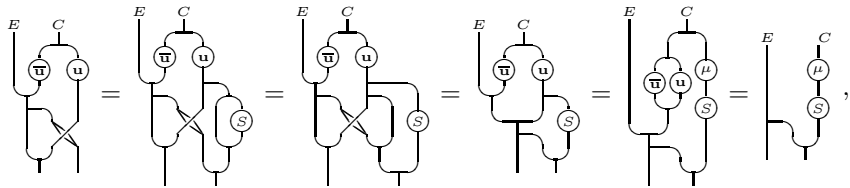
$$
%
%
\grow{\xymatrix@!0{\\
\save\go+<0pt,3pt>\Drop{E}\restore \ar@{-}[4,0]&&&\save\go+<0.2pc,3pt>\Drop{C}\restore\save[]+<0.2pc,0pc> \Drop{}\ar@{-}[1,0]+<0.2pc,0pc>\restore&&\\
&&&\save[]+<0.2pc,0pc> \Drop{}\ar@{-}`l/4pt [1,-1] [1,-1] \ar@{-}`r [1,2] [1,2]\restore&&\\
&&&&&\\
&&*+[o]+<0.37pc>[F]{\overline{\mathbf{u}}} \ar@{-}[1,0]+<0pc,0pc>&&& *+[o]+<0.40pc>[F]{\mathbf{u}}\ar@{-}[4,0]+<0pc,0pc>\\
\ar@{-}`d/4pt [1,1] `[0,2] [0,2]&&&&&\\
&\ar@{-}[4,0]&&&&\\
& \ar@{-}`r/4pt [1,2] [1,2]&&&&\\
&&&\ar@{-}[1,1]+<-0.125pc,0.0pc> \ar@{-}[1,1]+<0.0pc,0.125pc>&&\ar@{-}[2,-2]\\
&&&&&\\
&\ar@{-}`d/4pt [1,1] `[0,2] [0,2] &&&&\ar@{-}[-1,-1]+<0.125pc,0.0pc> \ar@{-}[-1,-1]+<0.0pc,-0.125pc>\ar@{-}[2,0]\\
&&\ar@{-}[1,0]&&&\\
&&&&&
}}
\grow{\xymatrix@!0{
\\\\\\\\\\\\
\save\go+<0pt,0pt>\Drop{\txt{$=$}}\restore
}}
\grow{\xymatrix@!0{
\save\go+<0pt,3pt>\Drop{E}\restore \ar@{-}[4,0]&&&\save\go+<0.2pc,3pt>\Drop{C}\restore\save[]+<0.2pc,0pc> \Drop{}\ar@{-}[1,0]+<0.2pc,0pc>\restore&&\\
&&&\save[]+<0.2pc,0pc> \Drop{}\ar@{-}`l/4pt [1,-1] [1,-1] \ar@{-}`r [1,2] [1,2]\restore&&\\
&&&&&\\
&&*+[o]+<0.37pc>[F]{\overline{\mathbf{u}}} \ar@{-}[1,0]+<0pc,0pc>&&& *+[o]+<0.40pc>[F]{\mathbf{u}}\ar@{-}[4,0]+<0pc,0pc>\\
\ar@{-}`d/4pt [1,1] `[0,2] [0,2]&&&&&\\
&\ar@{-}[4,0]&&&& \ar@{-}`r/4pt [1,2] [1,2] &&&\\
& \ar@{-}`r/4pt [1,2] [1,2]&&&&&& \ar@{-}`l/4pt [1,-1] [1,-1] \ar@{-}`r [1,1] [1,1]&\\
&&&\ar@{-}[1,1]+<-0.125pc,0.0pc> \ar@{-}[1,1]+<0.0pc,0.125pc>&&\ar@{-}[2,-2]&\ar@{-}[2,0]&&\\
&&&&&&&&*+[o]+<0.34pc>[F]{S}\ar@{-}[1,0]+<0pc,0pc>\\
&\ar@{-}`d/4pt [1,1] `[0,2] [0,2] &&&&\ar@{-}[-1,-1]+<0.125pc,0.0pc> \ar@{-}[-1,-1]+<0.0pc,-0.125pc>\ar@{-}[2,0]&\ar@{-}`d/4pt [1,1] `[0,2] [0,2]&&\\
&&\ar@{-}[3,0]&&&&&\ar@{-}[1,0]\\
&&&&&\ar@{-}`d/4pt [1,1] `[0,2] [0,2]&&&\\
&&&&&&\ar@{-}[1,0]&&\\
&&&&&&&&
}}
\grow{\xymatrix@!0{
\\\\\\\\\\\\
\save\go+<0pt,0pt>\Drop{\txt{$=$}}\restore
}}
\grow{\xymatrix@!0{
\save\go+<0pt,3pt>\Drop{E}\restore \ar@{-}[4,0]&&&\save\go+<0.2pc,3pt>\Drop{C}\restore\save[]+<0.2pc,0pc> \Drop{}\ar@{-}[1,0]+<0.2pc,0pc>\restore&&\\
&&&\save[]+<0.2pc,0pc> \Drop{}\ar@{-}`l/4pt [1,-1] [1,-1] \ar@{-}`r [1,2] [1,2]\restore&&\\
&&&&&\\
&&*+[o]+<0.37pc>[F]{\overline{\mathbf{u}}} \ar@{-}[1,0]+<0pc,0pc>&&& *+[o]+<0.40pc>[F]{\mathbf{u}}\ar@{-}[4,0]+<0pc,0pc>\\
\ar@{-}`d/4pt [1,1] `[0,2] [0,2]&&&&&\\
&\ar@{-}[4,0]&&&&\ar@{-}`r/4pt [2,4] [2,4]\\
& \ar@{-}`r/4pt [1,2] [1,2]&&&&\ar@{-}`r/4pt [1,2] [1,2]&&&&\\
&&&\ar@{-}[1,1]+<-0.125pc,0.0pc> \ar@{-}[1,1]+<0.0pc,0.125pc>&&\ar@{-}[2,-2]&&\ar@{-}[2,0]&&\ar@{-}[1,0]\\
&&&&&&&&&*+[o]+<0.34pc>[F]{S}\ar@{-}[3,0]+<0pc,0pc>\\
&\ar@{-}`d/4pt [1,1] `[0,2] [0,2] &&&&\ar@{-}[-1,-1]+<0.125pc,0.0pc> \ar@{-}[-1,-1]+<0.0pc,-0.125pc> \ar@{-}`d/4pt [1,1] `[0,2] [0,2]&&&&\\
&&\ar@{-}[3,0]&&&&\ar@{-}[1,0]&&&\\
&&&&&&\ar@{-}`d/4pt [1,1]+<0.2pc,0pc> `[0,3] [0,3]&&&\\
&&&&&&&\save[]+<0.2pc,0pc> \Drop{}\ar@{-}[1,0]+<0.2pc,0pc>\restore&&\\
&&&&&&&\save[]+<0.2pc,0pc> \Drop{}\restore&&
}}
\grow{\xymatrix@!0{
\\\\\\\\\\\\
\save\go+<0pt,0pt>\Drop{\txt{$=$}}\restore
}}
\grow{\xymatrix@!0{\\
\save\go+<0pt,3pt>\Drop{E}\restore \ar@{-}[4,0]&&&\save\go+<0.2pc,3pt>\Drop{C}\restore \save[]+<0.2pc,0pc> \Drop{}\ar@{-}[1,0]+<0.2pc,0pc>\restore&&\\
&&&\save[]+<0.2pc,0pc> \Drop{}\ar@{-}`l/4pt [1,-1] [1,-1] \ar@{-}`r [1,2] [1,2]\restore&&\\
&&&&&\\
&&*+[o]+<0.37pc>[F]{\overline{\mathbf{u}}} \ar@{-}[1,0]+<0pc,0pc>&&& *+[o]+<0.40pc>[F]{\mathbf{u}}\ar@{-}[3,0]+<0pc,0pc>\\
\ar@{-}`d/4pt [1,1] `[0,2] [0,2]&&&&&\\
&\ar@{-}[1,0]&&&&\ar@{-}`r/4pt [1,2] [1,2]\\
&\ar@{-}`d/4pt [1,2] `[0,4] [0,4] &&&&&&\\
&&&\ar@{-}[4,0]&&&&*+[o]+<0.34pc>[F]{S}\ar@{-}[2,0]+<0pc,0pc>\\
&&&\ar@{-}`r/4pt [1,2] [1,2]&&&&\\
&&&&&\ar@{-}`d/4pt [1,1] `[0,2] [0,2]&&\\
&&&&&&\ar@{-}[1,0]&\\
&&&&&&&
}}
\grow{\xymatrix@!0{
\\\\\\\\\\\\
\save\go+<0pt,0pt>\Drop{\txt{$=$}}\restore
}}
\grow{\xymatrix@!0{
\save\go+<0pt,3pt>\Drop{E}\restore \ar@{-}[8,0]&&&&\save\go+<0.2pc,3pt>\Drop{C}\restore \save[]+<0.2pc,0pc> \Drop{}\ar@{-}[1,0]+<0.2pc,0pc>\restore&&\\
&&&&\save[]+<0.2pc,0pc> \Drop{}\ar@{-}`l/4pt [1,-1] [1,-1] \ar@{-}`r [1,2] [1,2]\restore&&\\
&&&\ar@{-}[1,0] &&&\ar@{-}[1,0]\\
&&& \ar@{-}`l/4pt [1,-1] [1,-1] \ar@{-}`r [1,1] [1,1]&&&\ar@{-}[1,0]\\
&&&&&&*+[o]+<0.37pc>[F]{\mu}\ar@{-}[2,0]+<0pc,0pc>\\
&&*+[o]+<0.37pc>[F]{\overline{\mathbf{u}}} \ar@{-}[1,0]+<0pc,0pc>&& *+[o]+<0.40pc>[F]{\mathbf{u}}\ar@{-}[1,0]+<0pc,0pc>&&\\
&&\ar@{-}`d/4pt [1,1] `[0,2] [0,2]&&&&\\
&&&\ar@{-}[1,0]&&&*+[o]+<0.34pc>[F]{S}\ar@{-}[4,0]+<0pc,0pc>\\
\ar@{-}`d/4pt [1,1]+<0.2pc,0pc> `[0,3] [0,3] &&&&&&\\
&\save[]+<0.2pc,0pc> \Drop{}\ar@{-}[4,0]+<0.2pc,0pc>\restore&&&&&\\
&\save[]+<0.2pc,0pc> \Drop{} \ar@{-}`r/4pt [1,3] [1,3]\restore&&&&&\\
&&&&\ar@{-}`d/4pt [1,1] `[0,2] [0,2]&&\\
&&&&&\ar@{-}[1,0]&\\
&\save[]+<0.2pc,0pc> \Drop{}\restore&&&&&}}
\grow{\xymatrix@!0{
\\\\\\\\\\\\
\save\go+<0pt,0pt>\Drop{\txt{$=$}}\restore
}}
\grow{\xymatrix@!0{\\\\
\save\go+<0pt,3pt>\Drop{E}\restore \ar@{-}[9,0]&&&&\save\go+<0.2pc,3pt>\Drop{C}\restore \ar@{-}[1,0]\\
&&&&\\
&&&&*+[o]+<0.37pc>[F]{\mu}\ar@{-}[2,0]+<0pc,0pc>\\
&&&&\\
&&&&\\
&&&&*+[o]+<0.34pc>[F]{S}\ar@{-}[2,0]+<0pc,0pc>\\
\ar@{-}`r/4pt [1,2] [1,2]&&&&\\
&&\ar@{-}`d/4pt [1,1] `[0,2] [0,2]&&\\
&&&\ar@{-}[1,0]&\\
&&&&}}
\grow{\xymatrix@!0{
\\\\\\\\\\\\
\save\go+<0pt,0pt>\Drop{\txt{,}}\restore
}}
$$
as desired.
\end{proof}

\begin{lemma}\label{mengano} Let $C := H^{\ot_c^n}$. We have
$$
%
%
\grow{\xymatrix@!0{\\\\\\
\save\go+<0pt,3pt>\Drop{C}\restore \ar@{-}[2,0]\\
&&\\
*+[o]+<0.37pc>[F]{\ov{\mathbf{u}}}\ar@{-}[3,0]+<0pc,0pc>&&\\
&&\\
\ar@{-}`r/4pt [1,2] [1,2]&& \\
&&
}}
\grow{\xymatrix@!0{
\\\\\\\\\\
\save\go+<0pt,0pt>\Drop{\txt{$=$}}\restore
}}
\grow{\xymatrix@!0{
&\save\go+<0pt,3pt>\Drop{C}\restore \ar@{-}[1,0]&\\
&\ar@{-}`l/4pt [1,-1] [1,-1] \ar@{-}`r [1,1] [1,1]&\\
\ar@{-}[0,0]+<0pt,-2.5pt>&&\ar@{-}[0,0]+<0pt,-2.5pt>\\
&*+<0.1pc>[F]{\,\,\,\,g\,\,\,\,} \ar@{-}@<0.40pc>[2,0]+<0pt,-0.5pt> \ar@{-}@<-0.40pc>[2,0]+<0pt,-0.5pt>&\\
&&\\
&&\\
*+[o]+<0.37pc>[F]{\ov{\boldsymbol{\gamma}}} \ar@{-}[2,0]&&*+[o]+<0.37pc>[F]{\mathbf{S}}\ar@{-}[2,0]\\
&&\\
&&\\
*+[o]+<0.40pc>[F]{\mu}\ar@{-}[2,0]&&*+[o]+<0.40pc>[F]{\mu}\ar@{-}[2,0]\\
&&\\
&&
}}
\grow{\xymatrix@!0{
\\\\\\\\\\
\save\go+<0pt,0pt>\Drop{\txt{,}}\restore
}}
$$
where $g$ is the map $\gc_{2n}$, $\mathbf{S}:= S^{\ot_k^n}$, $\mu:=\mu_n$, $\ov{\boldsymbol{\gamma}} := \ov{\gamma}^{\ot_k^n}$ and $\ov{\mathbf{u}}:=\mu\xcirc \ov{\boldsymbol{\gamma}}\xcirc \gc_n$.
\end{lemma}

\begin{proof} We proceed by induction on $n$. The case $n = 1$ follows from~\cite[Lemma 10.7(2)]{G-G2}. Assume $n>1$ and let $C' := H^{\ot_c^{n-1}}$. Then,
$$
%
%
\grow{\xymatrix@!0{\\\\\\\\\\\\\\
\save\go+<0pt,3pt>\Drop{C}\restore \ar@{-}[2,0]\\
&&\\
*+[o]+<0.37pc>[F]{\ov{\mathbf{u}}}\ar@{-}[3,0]+<0pc,0pc>&&\\
&&\\
\ar@{-}`r/4pt [1,2] [1,2]&& \\
&&
}}
\grow{\xymatrix@!0{
\\\\\\\\\\\\\\\\\\
\save\go+<0pt,0pt>\Drop{\txt{$=$}}\restore
}}
\grow{\xymatrix@!0{\\\\\\\\\\\\
\save\go+<0pt,3pt>\Drop{C'}\restore \ar@{-}[1,1]+<-0.1pc,0.1pc>  &&\save\go+<0pt,3pt>\Drop{H}\restore \ar@{-}[2,-2]\\
&&\\
\ar@{-}[2,0]&&\ar@{-}[-1,-1]+<0.1pc,-0.1pc>\ar@{-}[2,0]\\
&&\\
*+[o]+<0.37pc>[F]{\ov{\gamma}}\ar@{-}[1,0]+<0pc,0pc> &&*+[o]+<0.37pc>[F]{\ov{\mathbf{u}}}\ar@{-}[1,0]+<0pc,0pc>\\
\ar@{-}`d/4pt [1,1] `[0,2] [0,2]&&\\
&\ar@{-}[2,0]&\\
&\ar@{-}`r/4pt [1,2] [1,2]& \\
&&&
}}
\grow{\xymatrix@!0{
\\\\\\\\\\\\\\\\\\
\save\go+<0pt,0pt>\Drop{\txt{$=$}}\restore
}}
\grow{\xymatrix@!0{\\\\\\\\
\save\go+<0pt,3pt>\Drop{C'}\restore \ar@{-}[2,2]+<-0.1pc,0.1pc> &&&&\save\go+<0pt,3pt>\Drop{H}\restore \ar@{-}[4,-4]\\
&&&&\\
&&&&\\
&&&&\\
\ar@{-}[1,0]&&&&\ar@{-}[-2,-2]+<0.1pc,-0.1pc>\ar@{-}[1,0]\\
&&&&\\
*+[o]+<0.37pc>[F]{\ov{\gamma}}\ar@{-}[5,0]+<0pc,0pc> &&&&*+[o]+<0.37pc>[F]{\ov{\mathbf{u}}}\ar@{-}[3,0]+<0pc,0pc>\\
&&&&\\
\ar@{-}`r/4pt [1,2] [1,2]&&&&\ar@{-}`r/4pt [1,2] [1,2] \\
&&\ar@{-}[1,1]+<-0.125pc,0.0pc> \ar@{-}[1,1]+<0.0pc,0.125pc>&&\ar@{-}[2,-2]&&\ar@{-}[2,0]\\
&&&&&&\\
\ar@{-}`d/4pt [1,1] `[0,2] [0,2] &&&&\ar@{-}[-1,-1]+<0.125pc,0.0pc>\ar@{-}[-1,-1]+<0.0pc,-0.125pc> \ar@{-}`d/4pt [1,1] `[0,2] [0,2] &&\\
&\ar@{-}[1,0]&&&&\ar@{-}[1,0]&\\
&&&&&&
}}
\grow{\xymatrix@!0{
\\\\\\\\\\\\\\\\\\
\save\go+<0pt,0pt>\Drop{\txt{$=$}}\restore
}}
\grow{\xymatrix@!0{\\
&\save\go+<0pt,3pt>\Drop{C'}\restore \ar@{-}[2,2]+<-0.1pc,0.1pc> &&&&\save\go+<0pt,3pt>\Drop{H}\restore \ar@{-}[4,-4]&\\
&&&&&&\\
&&&&&&\\
&&&&&&\\
&\ar@{-}[1,0]&&&&\ar@{-}[-2,-2]+<0.1pc,-0.1pc>\ar@{-}[1,0]&\\
&\ar@{-}`l/4pt [1,-1] [1,-1] \ar@{-}`r [1,1] [1,1]&&&&\ar@{-}`l/4pt [1,-1] [1,-1] \ar@{-}`r [1,1] [1,1]&\\
\ar@{-}[1,1]+<-0.1pc,0.1pc> && \ar@{-}[2,-2]&& \ar@{-}[0,0]+<0pt,-2.5pt> &&\ar@{-}[0,0]+<0pt,-2.5pt>\\
&&&&&*+<0.1pc>[F]{\,\,\,\,g\,\,\,\,} \ar@{-}@<0.40pc>[2,0]+<0pt,-0.5pt> \ar@{-}@<-0.40pc>[2,0]+<0pt,-0.5pt>&\\
\ar@{-}[1,0]+<0pt,-1pt>&&\ar@{-}[-1,-1]+<0.1pc,-0.1pc>\ar@{-}[1,0]+<0pt,-1pt>&&&&\\
&&&&&&\\
*+[o]+<0.37pc>[F]{\ov{\gamma}}\ar@{-}[7,0]&& *+[o]+<0.37pc>[F]{S}\ar@{-}[5,0]+<0pc,0pt>&& *+[o]+<0.37pc>[F]{\ov{\boldsymbol{\gamma}}} \ar@{-}[2,0]&&*+[o]+<0.37pc>[F]{\mathbf{S}}\ar@{-}[2,0]+<0pc,0pt>\\
&&&&&&\\
&&&&&&\\
&&&&*+[o]+<0.40pc>[F]{\mu}\ar@{-}[2,0]+<0pc,0pc>&&*+[o]+<0.40pc>[F]{\mu}\ar@{-}[4,0]+<0pc,0pc>\\
&&&&&&\\
&&\ar@{-}[1,1]+<-0.125pc,0.0pc> \ar@{-}[1,1]+<0.0pc,0.125pc>&&\ar@{-}[2,-2]&&\\
&&&&&&\\
\ar@{-}`d/4pt [1,1] `[0,2] [0,2]&&&&\ar@{-}[-1,-1]+<0.125pc,0.0pc>\ar@{-}[-1,-1]+<0.0pc,-0.125pc> \ar@{-}`d/4pt [1,1] `[0,2] [0,2]&&\\
&\ar@{-}[1,0]&&&&\ar@{-}[1,0]&\\
&&&&&&
}}
\grow{\xymatrix@!0{
\\\\\\\\\\\\\\\\\\
\save\go+<0pt,0pt>\Drop{\txt{$=$}}\restore
}}
\grow{\xymatrix@!0{
&\save\go+<0pt,3pt>\Drop{C'}\restore \ar@{-}[1,0]&&&&\save\go+<0pt,3pt>\Drop{H}\restore \ar@{-}[1,0]&\\
&\ar@{-}`l/4pt [1,-1] [1,-1] \ar@{-}`r [1,1] [1,1]&&&&\ar@{-}`l/4pt [1,-1] [1,-1] \ar@{-}`r [1,1] [1,1]&\\
\ar@{-}[2,0]&&\ar@{-}[1,1]+<-0.1pc,0.1pc> && \ar@{-}[2,-2]&&\ar@{-}[2,0]\\
&&&&&&\\
\ar@{-}[1,1]+<-0.1pc,0.1pc> && \ar@{-}[2,-2]&&\ar@{-}[-1,-1]+<0.1pc,-0.1pc>\ar@{-}[1,1]+<-0.1pc,0.1pc> && \ar@{-}[2,-2]\\
&&&&&&\\
\ar@{-}[2,0]&&\ar@{-}[-1,-1]+<0.1pc,-0.1pc>\ar@{-}[1,1]+<-0.1pc,0.1pc> && \ar@{-}[2,-2]&&\ar@{-}[-1,-1]+<0.1pc,-0.1pc>\ar@{-}[2,0]+<0pt,-2.5pt>\\
&&&&&&\\
\ar@{-}[1,1]+<-0.1pc,0.1pc> && \ar@{-}[2,-2]&&\ar@{-}[-1,-1]+<0.1pc,-0.1pc>\ar@{-}[0,0]+<0pt,-2.5pt>&&\\
&&&&&*+<0.1pc>[F]{\,\,\,\,g\,\,\,\,} \ar@{-}@<0.40pc>[4,0]+<0pt,-1pt> \ar@{-}@<-0.40pc>[1,0]&\\
\ar@{-}[3,0]+<0pt,-1pt>&&\ar@{-}[-1,-1]+<0.1pc,-0.1pc>\ar@{-}[1,1]+<-0.1pc,0.1pc> && \ar@{-}[2,-2]&&\\
&&&&&&\\
&&\ar@{-}[1,0]+<0pt,-1pt>&&\ar@{-}[1,0]+<0pt,-1pt>\ar@{-}[-1,-1]+<0.1pc,-0.1pc>&&\\
&&&&&&\\
*+[o]+<0.37pc>[F]{\ov{\gamma}}\ar@{-}[4,0]&& *+[o]+<0.37pc>[F]{\ov{\boldsymbol{\gamma}}} \ar@{-}[3,0]&&*+[o]+<0.37pc>[F]{S}\ar@{-}[4,0]&& *+[o]+<0.37pc>[F]{\mathbf{S}}\ar@{-}[3,0]\\
&&&&&&\\
&&&&&&\\
&&*+[o]+<0.40pc>[F]{\mu}\ar@{-}[1,0]&&&&*+[o]+<0.40pc>[F]{\mu}\ar@{-}[1,0]\\
%
\ar@{-}`d/4pt [1,1] `[0,2] [0,2]&&&&\ar@{-}`d/4pt [1,1] `[0,2] [0,2]&&\\
&\ar@{-}[1,0]&&&&\ar@{-}[1,0]&\\
&&&&&&
}}
\grow{\xymatrix@!0{
\\\\\\\\\\\\\\\\\\
\save\go+<0pt,0pt>\Drop{\txt{$=$}}\restore
}}
\grow{\xymatrix@!0{\\\\\\\\
&\save\go+<0pt,3pt>\Drop{C}\restore \ar@{-}[1,0]&\\
&\ar@{-}`l/4pt [1,-1] [1,-1] \ar@{-}`r [1,1] [1,1]&\\
\ar@{-}[0,0]+<0pt,-2.5pt>&&\ar@{-}[0,0]+<0pt,-2.5pt>\\
&*+<0.1pc>[F]{\,\,\,\,g\,\,\,\,} \ar@{-}@<0.40pc>[2,0]+<0pt,-0.5pt> \ar@{-}@<-0.40pc>[2,0]+<0pt,-0.5pt>&\\
&&\\
&&\\
*+[o]+<0.37pc>[F]{\ov{\boldsymbol{\gamma}}} \ar@{-}[2,0]&&*+[o]+<0.37pc>[F]{\mathbf{S}}\ar@{-}[2,0]\\
&&\\
&&\\
*+[o]+<0.40pc>[F]{\mu}\ar@{-}[2,0]&&*+[o]+<0.40pc>[F]{\mu}\ar@{-}[2,0]\\
&&\\
&&
}}
\grow{\xymatrix@!0{
\\\\\\\\\\\\\\\\\\
\save\go+<0pt,0pt>\Drop{\txt{,}}\restore
}}
$$
where the third equality follows from inductive hypothesis and~\cite[Lemma 10.7(2)]{G-G2}, the fourth one follows from~\cite[Proposition 4.3]{G-G2} and~\cite[Lemma 10.7(2)]{G-G2}, and the fifth one, from the definition of $\Delta_C$ and~\cite[Corollary 4.21]{G-G2}.
\end{proof}

\noindent\bf Proof of Theorem~\ref{morfismo D}.\rm\enspace For $0\le j\le s$, let
$$
\tau_j\colon E\ot_{\!A} E^{\ot_{\!A}^s}\ot_k \ov{A}^{\ot_k^r}\to E\ot_{\!A} E^{\ot_{\!A}^{s+1}}\ot_k \ov{A}^{\ot_k^r}
$$
be the map defined by
$$
\tau_j(a_0\gamma_{\!A}(\bv_{0s})\ot \ba_{1r}) := \sum_l 1\ot_{\!A} \gamma_{\!A} (\bv_{j+1,s}^{(l)})\ot_{\!A} a_0\gamma_{\!A}(\bv_{0j})\ot \ba_{1r}^{(l)},
$$
where $\sum \ba_{1r}^{(l)}\ot_k \bv_{j+1,s}^{(l)} := \ov{\chi}(\bv_{j+1,s} \ot_k \ba_{1r})$, and let $\wh{\tau}_j\colon \wh{X}_{rs}\longrightarrow \wh{X}_{r,s+1}$ be the map induced by $\tau_j$. By Proposition~\ref{Connes operator} we know that
$$
\wh{D}\bigl([a_0\gamma_{\!A}(\bv_{0s})\ot \ba_{1r}]\bigr) = \sum_{j=0}^s (-1)^{s+js} \wh{\tau}_j\bigl([a_0\gamma_{\!A}(\bv_{0s})\ot \ba_{1r}]\bigr)
$$
module $F^s_R(\wh{X}_{r+s})$. Hence
$$
\ov{D}\bigl([a_0\gamma(h_0)\ot\ba_{1r}]_H\ot_k\bh_{1s} \bigr) = \sum_{j=0}^s (-1)^{s+js} \theta\xcirc \wh{\tau}_j\xcirc \vartheta\bigl([a_0\gamma(h_0)\ot\ba_{1r}]_H\ot_k\bh_{1s}\bigr)
$$
module $F_R^s(\ov{X}_{n+1})$. Now, since $\theta_{r,s+1}\xcirc \wh{\tau}_j\xcirc \vartheta_{rs}$ is induced by $(-1)^r\wt{\theta}_{r,s+1}\xcirc \tau_j\xcirc \wt{\vartheta}_{rs}$, in order to finish the proof we must show that $\wt{\tau}_j = \wt{\theta}_{r,s+1}\xcirc \tau_j\xcirc \wt{\vartheta}_{rs}$. In the sequel

\begin{itemize}

\smallskip

\item[-] $\boldsymbol{\gamma}$ denotes the maps $\gamma^{\ot_k^j}$, $\gamma^{\ot_k^{s-j}}$, $\gamma^{\ot_{\!A}^j}$ and $\gamma^{\ot_{\!A}^{s-j}}$, and $\ov{\boldsymbol{\gamma}}$ denotes $\ov{\gamma}^{\ot_k^{s-j}}$,

\smallskip

\item[-] $\mu$ denotes both the maps $\mu_j$ and $\mu_{s-j}$,

\smallskip

\item[-] $\mathbf{u}$ denotes both the maps $\mu_j\xcirc\gamma^{\ot_k^j}$ and $\mu_{s-j} \xcirc \gamma^{\ot_k^{s-j}}$,

\smallskip

\item[-] $\ov{\mathbf{u}}$ denotes both the maps $\mu_j\xcirc\ov{\gamma}^{\ot_k^j}\xcirc \gc_j$ and $\mu_{s-j}\xcirc \ov{\gamma}^{\ot_k^{s-j}}\xcirc \gc_{s-j}$,

\smallskip

\item[-] $g$ denotes the map $\gc_{2s-2j}$ introduced in item~(5) of Section~\ref{hom de inv},

\smallskip

\item[-] $\mathbf{S}$ denotes the map $S^{\ot_k^{s-j}}$.

\smallskip

\end{itemize}
Let $D:=A^{\ot_k^r}$, $C_1 := H^{\ot_c^j}$ and $C_2 := H^{\ot_c^{s-j}}$. By Lemma~\ref{para d1}
$$
%
%
\grow{\xymatrix@!0{
\\\\\\\\\\\\\\\\\\\\\\\\\\\\\\
\save\go+<0pt,0pt>\Drop{\txt{$\tau_j\circ \wt{\vartheta}_{rs}= $}}\restore
}}
\grow{\xymatrix@!0{\\\\\\
\save\go+<0pt,2pt>\Drop{\circ}\ar@{-}[24,0]\restore & \save\go+<0pt,3pt>\Drop{E}\restore \ar@{-}[16,0] &&&& \save\go+<0pt,3pt>\Drop{D}\restore \ar@{-}[2,2] && \save\go+<0pt,3pt>\Drop{C_1}\restore \ar@{-}[1,-1]+<0.125pc,0.0pc> \ar@{-}[1,-1]+<0.0pc,0.125pc> && \save\go+<0pt,3pt>\Drop{C_2}\restore \ar@{-}[2,0]\\
&&&&&&&&&\\
&&&&&\ar@{-}[2,-2]\ar@{-}[-1,1]+<-0.125pc,0.0pc> \ar@{-}[-1,1]+<0.0pc,-0.125pc>&&\ar@{-}[2,2]&&\ar@{-}[1,-1]+<0.125pc,0.0pc> \ar@{-}[1,-1]+<0.0pc,0.125pc>\\
&&&&&&&&&\\
&&&\ar@{-}[1,0]&&&&\ar@{-}[1,0]\ar@{-}[-1,1]+<-0.125pc,0.0pc> \ar@{-}[-1,1]+<0.0pc,-0.125pc>&&\ar@{-}[1,1]\\
&&&\ar@{-}`l/4pt [1,-1] [1,-1] \ar@{-}`r [1,1] [1,1]&&&&\ar@{-}`l/4pt [1,-1] [1,-1] \ar@{-}`r [1,1] [1,1]&&&\ar@{-}[1,0]\\
&&\ar@{-}[2,0]&&\ar@{-}[1,1]+<-0.1pc,0.1pc> && \ar@{-}[2,-2]&&\ar@/^0.1pc/ @{-}[2,2] \ar@/_0.1pc/ @{-}[2,2]&& \ar@/^0.1pc/ @{-}[2,-2]\ar@/_0.1pc/ @{-}[2,-2]\\
&&&&&&&&&&\\
&&\ar@{-}[1,1]+<-0.1pc,0.1pc>&&\ar@{-}[2,-2]&&\ar@{-}[-1,-1]+<0.1pc,-0.1pc>\ar@{-}[3,0] &&\ar@{-}[1,0] &&\ar@{-}[1,0] \\
&&&&&&&&\ar@{-}[2,2]&&\ar@{-}[2,-2]\\
&&\ar@{-}[1,0]+<0pc,-0.5pt>&&\ar@{-}[-1,-1]+<0.1pc,-0.1pc>\ar@{-}[1,0]+<0pc,-0.5pt>&&&&&&\\
&&&&&&\ar@{-}[2,2]&&\ar@{-}[2,-2]&&\ar@{-}[13,0]\\
&&*+[o]+<0.37pc>[F]{\ov{\mathbf{u}}}\ar@{-}[2,0]&&*+[o]+<0.37pc>[F]{\ov{\mathbf{u}}}\ar@{-}[2,0] &&&&&&\\
&&&&&&\ar@{-}[3,0]&&\ar@{-}[2,0]&&\\
&&\ar@{-}`d/4pt [1,1] `[0,2] [0,2]&&&&&&&&\\
&&&\ar@{-}[1,0]&&&&&\ar@{-}[0,0]+<0pc,-0.5pt>&&\\
&\ar@{-}`d/4pt [1,1] `[0,2] [0,2]&&&&&\ar@{-}[4,-4]&& *+[o]+<0.37pc>[F]{\boldsymbol{\gamma}} \ar@{-}[8,0]+<0pc,0pc>&&\\
&&\ar@{-}[1,0]&&&&&&&&\\
&&\ar@{-}[2,2]&&&&&&&&\\
&&&&&&&&&&\\
&&\ar@{-}[1,0]+<0pc,-0.5pt>&&\ar@{-}[4,0]&&&&&&\\
&&&&&&&&&&\\
&&*+[o]+<0.37pc>[F]{\boldsymbol{\gamma}}\ar@{-}[2,0]&&&&&&&&\\
&&&&&&&&&&\\
&&&&&&&&&&
}}
\grow{\xymatrix@!0{
\\\\\\\\\\\\\\\\\\\\\\\\\\\\\\
\save\go+<0pt,0pt>\Drop{\txt{$=$}}\restore
}}
\grow{\xymatrix@!0{
\save\go+<0pt,2pt>\Drop{\circ}\ar@{-}[29,0]\restore & \save\go+<0pt,3pt>\Drop{E}\restore \ar@{-}[20,0] &&& \save\go+<0pt,3pt>\Drop{D}\restore \ar@{-}[2,2] && \save\go+<0pt,3pt>\Drop{C_1}\restore \ar@{-}[1,-1]+<0.125pc,0.0pc> \ar@{-}[1,-1]+<0.0pc,0.125pc> &&& \save\go+<0pt,3pt>\Drop{C_2}\restore \ar@{-}[1,0]&\\
&&&&&&&&&\ar@{-}`l/4pt [1,-1] [1,-1] \ar@{-}`r [1,1] [1,1]&\\
&&&&\ar@{-}[1,-1]\ar@{-}[-1,1]+<-0.125pc,0.0pc> \ar@{-}[-1,1]+<0.0pc,-0.125pc>&&\ar@{-}[2,2] &&\ar@{-}[1,-1]+<0.125pc,0.0pc> \ar@{-}[1,-1]+<0.0pc,0.125pc>&&\ar@{-}[10,0]\\
&&&\ar@{-}[6,0]&&&&&&&\\
&&&&&&\ar@{-}[1,-1]\ar@{-}[-1,1]+<-0.125pc,0.0pc> \ar@{-}[-1,1]+<0.0pc,-0.125pc>&&\ar@{-}[8,0]&&\\
&&&&&\ar@{-}[1,0]&&&&&\\
&&&&&\ar@{-}`l/4pt [1,-1] [1,-1] \ar@{-}`r [1,1] [1,1]&&&&&\\
&&&&\ar@{-}[3,2]&&\ar@{-}`d/4pt [1,2][1,2]&&&&\\
&&&&&&&&&&\\
&&&\ar@{-}`l/4pt [1,-1] [1,-1] \ar@{-}`r [1,1] [1,1]&&&&&&&\\
&&\ar@{-}[2,0]  &&\ar@{-}[1,1]+<-0.1pc,0.1pc> && \ar@{-}[2,-2]&&&&\\
&&&&&&&&&&\\
&&\ar@{-}[1,1]+<-0.1pc,0.1pc> && \ar@{-}[2,-2]&& \ar@{-}[-1,-1]+<0.1pc,-0.1pc>\ar@{-}[2,0]&&\ar@{-}[2,2]&&\ar@{-}[2,-2]\\
&&&&&&&&&&\\
&&\ar@{-}[1,0]+<0pc,-0.5pt>&&\ar@{-}[1,0]+<0pc,-0.5pt>\ar@{-}[-1,-1]+<0.1pc,-0.1pc>&&\ar@{-}[2,2]&&\ar@{-}[2,-2]&&\ar@{-}[15,0]\\
&&&&&&&&&&\\
&&*+[o]+<0.37pc>[F]{\ov{\mathbf{u}}}\ar@{-}[2,0]&&*+[o]+<0.37pc>[F]{\ov{\mathbf{u}}}\ar@{-}[2,0]&&\ar@{-}[4,0]&&\ar@{-}[2,0]&&\\
&&&&&&&&&&\\
&&\ar@{-}`d/4pt [1,1] `[0,2] [0,2]&&&&&&\ar@{-}[0,0]+<0pc,-0.5pt>&&\\
&&&\ar@{-}[1,0]&&&&&*+[o]+<0.37pc>[F]{\boldsymbol{\gamma}}\ar@{-}[10,0]+<0pc,0pc>&&\\
&\ar@{-}`d/4pt [1,1] `[0,2] [0,2]&&&&&\ar@{-}[4,-4]&&&&\\
&&\ar@{-}[1,0]&&&&&&&&\\
&&\ar@{-}[2,2]&&&&&&&&\\
&&&&&&&&&&\\
&&\ar@{-}[1,0]+<0pc,-0.5pt>&&\ar@{-}[5,0]&&&&&&\\
&&&&&&&&&&\\
&&*+[o]+<0.37pc>[F]{\boldsymbol{\gamma}}\ar@{-}[3,0]&&&&&&&&\\
&&&&&&&&&&\\
&&&&&&&&&&\\
&&&&&&&&&&
}}
\grow{\xymatrix@!0{
\\\\\\\\\\\\\\\\\\\\\\\\\\\\\\
\save\go+<0pt,0pt>\Drop{\txt{$=$}}\restore
}}
\grow{\xymatrix@!0{\\\\
\save\go+<0pt,2pt>\Drop{\circ}\ar@{-}[24,0]\restore & \save\go+<0pt,3pt>\Drop{E}\restore \ar@{-}[16,0] &&& \save\go+<0pt,3pt>\Drop{D}\restore \ar@{-}[2,2] && \save\go+<0pt,3pt>\Drop{C_1}\restore \ar@{-}[1,-1]+<0.125pc,0.0pc> \ar@{-}[1,-1]+<0.0pc,0.125pc> &&& \save\go+<0pt,3pt>\Drop{C_2}\restore \ar@{-}[1,0]&\\
&&&&&&&&&\ar@{-}`l/4pt [1,-1] [1,-1] \ar@{-}`r [1,1] [1,1]&\\
&&&&\ar@{-}[1,-1]\ar@{-}[-1,1]+<-0.125pc,0.0pc> \ar@{-}[-1,1]+<0.0pc,-0.125pc>&&\ar@{-}[2,2] &&\ar@{-}[1,-1]+<0.125pc,0.0pc> \ar@{-}[1,-1]+<0.0pc,0.125pc>&&\ar@{-}[7,0]\\
&&&\ar@{-}[7,0]&&&&&&&\\
&&&&&&\ar@{-}[1,-1]\ar@{-}[-1,1]+<-0.125pc,0.0pc> \ar@{-}[-1,1]+<0.0pc,-0.125pc>&&\ar@{-}[5,0]&&\\
&&&&&\ar@{-}[1,0]&&&&&\\
&&&&&\ar@{-}`l/4pt [1,-1] [1,-1] \ar@{-}`r [1,1] [1,1]&&&&&\\
&&&&\ar@{-}[3,1]&&\ar@{-}`d/4pt [1,2][1,2]&&&&\\
&&&&&&&&&&\\
&&&&&&&&\ar@{-}[2,2]&&\ar@{-}[3,-3]\\
&&&\ar@{-}[1,1]+<-0.1pc,0.1pc> && \ar@{-}[2,-2]&&&&&\\
&&&&&&&&&&\ar@{-}[13,0]\\
&&&\ar@{-}[2,0]&&\ar@{-}[-1,-1]+<0.1pc,-0.1pc>\ar@{-}[3,3]&&\ar@{-}[2,-2]&&&\\
&&&&&&&&&&\\
&&&\ar@{-}[2,2]&&\ar@{-}[2,-2]&&&&&\\
&&&&&&&&\ar@{-}[1,0]&&\\
&\ar@{-}[2,2]&&\ar@{-}[2,-2]&&\ar@{-}[1,0]&&&\ar@{-}`l/4pt [1,-1] [1,-1] \ar@{-}`r [1,1] [1,1]&&\\
&&&&&&&&&&\\
&\ar@{-}[1,0] &&\ar@{-}[2,0]&&*+[o]+<0.37pc>[F]{\ov{\mathbf{u}}}\ar@{-}[2,0]+<0pc,0pc>&& *+[o]+<0.37pc>[F]{\ov{\mathbf{u}}}\ar@{-}[4,0]+<0pc,0pc>&&*+[o]+<0.37pc>[F]{\boldsymbol{\gamma}} \ar@{-}[6,0]+<0pc,0pc>&\\
&&&&&&&&&&\\
&*+[o]+<0.37pc>[F]{\boldsymbol{\gamma}} \ar@{-}[4,0]+<0pc,0pc>&&\ar@{-}`d/4pt [1,1] `[0,2] [0,2]&&&&&&&\\
&&&&\ar@{-}[1,0]&&&&&&\\
&&&&\ar@{-}`d/4pt [1,1]+<0.2pc,0pc> `[0,3] [0,3] &&&&&&\\
&&&&&\save[]+<0.2pc,0pc> \Drop{}\ar@{-}[1,0]+<0.2pc,0pc>\restore&&&&&\\
&&&&&\save[]+<0.2pc,0pc> \Drop{}\restore&&&&&
}}
\grow{\xymatrix@!0{
\\\\\\\\\\\\\\\\\\\\\\\\\\\\\\
\save\go+<0pt,0pt>\Drop{\txt{.}}\restore
}}
$$
Consequently, by Lemmas~\ref{pirulo} and~\ref{mengano},
%
\allowdisplaybreaks
\begin{align*}
\grow{\xymatrix@!0{
\\\\\\\\\\\\\\\\\\\\\\\\\\\\\\\\\\\\\\
\save\go+<0pt,0pt>\Drop{\txt{$\wt{\theta}_{r,s+1}\xcirc \tau_j\xcirc \wt{\vartheta}_{rs}=$}}\restore
}} &
\grow{\xymatrix@!0{
\save\go+<0pt,2pt>\Drop{\circ}\ar@{-}[29,0]\restore && \save\go+<0pt,3pt>\Drop{E}\restore \ar@{-}[16,0] &&& \save\go+<0pt,3pt>\Drop{D}\restore \ar@{-}[2,2] && \save\go+<0pt,3pt>\Drop{C_1}\restore \ar@{-}[1,-1]+<0.125pc,0.0pc> \ar@{-}[1,-1]+<0.0pc,0.125pc> &&& \save\go+<0pt,3pt>\Drop{C_2}\restore \ar@{-}[1,0]&\\
&&&&&&&&&&\ar@{-}`l/4pt [1,-1] [1,-1] \ar@{-}`r [1,1] [1,1]&\\
&&&&&\ar@{-}[1,-1]\ar@{-}[-1,1]+<-0.125pc,0.0pc> \ar@{-}[-1,1]+<0.0pc,-0.125pc>&&\ar@{-}[2,2] &&\ar@{-}[1,-1]+<0.125pc,0.0pc> \ar@{-}[1,-1]+<0.0pc,0.125pc>&&\ar@{-}[7,0]\\
&&&&\ar@{-}[7,0]&&&&&&&\\
&&&&&&&\ar@{-}[1,-1]\ar@{-}[-1,1]+<-0.125pc,0.0pc> \ar@{-}[-1,1]+<0.0pc,-0.125pc>&&\ar@{-}[5,0]&&\\
&&&&&&\ar@{-}[1,0]&&&&&\\
&&&&&&\ar@{-}`l/4pt [1,-1] [1,-1] \ar@{-}`r [1,1] [1,1]&&&&&\\
&&&&&\ar@{-}[3,1]&&\ar@{-}`d/4pt [1,2][1,2]&&&&\\
&&&&&&&&&&&&&\\
&&&&&&&&&\ar@{-}[2,2]&&\ar@{-}[2,-2]&&\\
&&&&\ar@{-}[1,1]+<-0.1pc,0.1pc> && \ar@{-}[2,-2]&&&&\ar@{-}[2,-2]\\
&&&&&&&&&&&\ar@{-}[5,3]&&\\
&&&&\ar@{-}[2,0]&&\ar@{-}[-1,-1]+<0.1pc,-0.1pc>\ar@{-}[3,3]&&\ar@{-}[2,-2]&&\\
&&&&&&&&&&&\\
&&&&\ar@{-}[2,2]&&\ar@{-}[2,-2]&&&&&&&\\
&&&&&&&&&\ar@{-}[1,0]&&&&\\
&&\ar@{-}[2,2]&&\ar@{-}[2,-2]&&\ar@{-}[1,0]&&&\ar@{-}`l/4pt [1,-1] [1,-1] \ar@{-}`r [1,1] [1,1]&&&&&\ar@{-}[10,0]\\
&&&&&&&&&&&\\
&&\ar@{-}[1,0] &&\ar@{-}[2,0]&&*+[o]+<0.37pc>[F]{\ov{\mathbf{u}}}\ar@{-}[2,0]+<0pc,0pc>&& *+[o]+<0.37pc>[F]{\ov{\mathbf{u}}}\ar@{-}[4,0]+<0pc,0pc>&&*+[o]+<0.37pc>[F]{\boldsymbol{\gamma}} \ar@{-}[8,0]+<0pc,0pc>&\\
&&&&&&&&&&&\\
&&*+[o]+<0.37pc>[F]{\boldsymbol{\gamma}} \ar@{-}[6,0]+<0pc,0pc>&&\ar@{-}`d/4pt [1,1] `[0,2] [0,2]&&&&&&&\\
&&&&&\ar@{-}[1,0]&&&&&&\\
&&&&&\ar@{-}`d/4pt [1,1]+<0.2pc,0pc> `[0,3] [0,3] &&&&&&\\
&&&&&&\save[]+<0.2pc,0pc> \Drop{}\ar@{-}[1,0]+<0.2pc,0pc>\restore&&&&&\\
&&&&&&\save[]+<0.2pc,0pc> \Drop{}\ar@{-}[1,0]\restore  &&&&&&&&\\
&&\ar@{-}`r/4pt [1,2] [1,2]&&&&\ar@{-}[1,0]\ar@{-}`r/4pt [1,2] [1,2]&&&&\ar@{-}`r/4pt [1,2] [1,2]&&&&\\
&&&&\ar@{-}[1,1]+<-0.125pc,0.0pc> \ar@{-}[1,1]+<0.0pc,0.125pc>&&\ar@{-}[2,-2]&& \ar@{-}[1,1]+<-0.125pc,0.0pc> \ar@{-}[1,1]+<0.0pc,0.125pc>&&\ar@{-}[2,-2]&&\ar@{-}[1,1]+<-0.125pc,0.0pc> \ar@{-}[1,1]+<0.0pc,0.125pc>&&\ar@{-}[2,-2]\\
&&*+[o]+<0.37pc>[F]{\mu}\ar@{-}[2,0]+<0pc,0pc>&&&&&&&&&&&&\\
&&&&\ar@{-}[6,0]&&\ar@{-}[-1,-1]+<0.125pc,0.0pc>\ar@{-}[-1,-1]+<0.0pc,-0.125pc> \ar@{-}[1,1]+<-0.125pc,0.0pc> \ar@{-}[1,1]+<0.0pc,0.125pc>&&\ar@{-}[2,-2]&& \ar@{-}[-1,-1]+<0.125pc,0.0pc>\ar@{-}[-1,-1]+<0.0pc,-0.125pc>
\ar@{-}[1,1]+<-0.125pc,0.0pc> \ar@{-}[1,1]+<0.0pc,0.125pc>&&\ar@{-}[2,-2] &&\ar@{-}[-1,-1]+<0.125pc,0.0pc>\ar@{-}[-1,-1]+<0.0pc,-0.125pc>\ar@{-}[10,0]\\
\ar@{-}`d/4pt [1,1] `[0,2] [0,2]&&&&&&&&&&&&&&\\
&\ar@{-}[6,0]&&&&&\ar@{-}[1,0]&&\ar@{-}[-1,-1]+<0.125pc,0.0pc>\ar@{-}[-1,-1]+<0.0pc,-0.125pc> \ar@{-}[1,1]+<-0.125pc,0.0pc> \ar@{-}[1,1]+<0.0pc,0.125pc>&&\ar@{-}[2,-2] &&\ar@{-}[-1,-1]+<0.125pc,0.0pc>\ar@{-}[-1,-1]+<0.0pc,-0.125pc>\ar@{-}[8,0]&&\\
&&&&&&&&&&&&&&\\
&&&&&&*+[o]+<0.37pc>[F]{\mu}\ar@{-}[2,0]+<0pc,0pc>&&\ar@{-}[6,0] &&\ar@{-}[-1,-1]+<0.125pc,0.0pc> \ar@{-}[-1,-1]+<0.0pc,-0.125pc> \ar@{-}[6,0]&&&&\\
&&&&&&&&&&&&&&\\
&&&&\ar@{-}`d/4pt [1,1] `[0,2] [0,2]&&&&&&&&&&\\
&&&&&\ar@{-}[1,0]&&&&&&&&&\\
&\ar@{-}`d/4pt [1,2] `[0,4] [0,4]&&&&&&&&&&&&&\\
&&&\ar@{-}[1,0]&&&&&&&&&&&\\
&&&&&&&&&&&&&&
}}
\grow{\xymatrix@!0{
\\\\\\\\\\\\\\\\\\\\\\\\\\\\\\\\\\\\\\
\save\go+<0pt,0pt>\Drop{\txt{$=$}}\restore
}}
\grow{\xymatrix@!0{\\\\
\save\go+<0pt,2pt>\Drop{\circ}\ar@{-}[30,0]\restore && \save\go+<0pt,3pt>\Drop{E}\restore \ar@{-}[16,0] &&& \save\go+<0pt,3pt>\Drop{D}\restore \ar@{-}[2,2] && \save\go+<0pt,3pt>\Drop{C_1}\restore \ar@{-}[1,-1]+<0.125pc,0.0pc> \ar@{-}[1,-1]+<0.0pc,0.125pc> &&& \save\go+<0pt,3pt>\Drop{C_2}\restore \ar@{-}[1,0]&\\
&&&&&&&&&&\ar@{-}`l/4pt [1,-1] [1,-1] \ar@{-}`r [1,1] [1,1]&\\
&&&&&\ar@{-}[1,-1]\ar@{-}[-1,1]+<-0.125pc,0.0pc> \ar@{-}[-1,1]+<0.0pc,-0.125pc>&&\ar@{-}[2,2] &&\ar@{-}[1,-1]+<0.125pc,0.0pc> \ar@{-}[1,-1]+<0.0pc,0.125pc>&&\ar@{-}[7,0]\\
&&&&\ar@{-}[7,0]&&&&&&&\\
&&&&&&&\ar@{-}[1,-1]\ar@{-}[-1,1]+<-0.125pc,0.0pc> \ar@{-}[-1,1]+<0.0pc,-0.125pc>&&\ar@{-}[5,0]&&\\
&&&&&&\ar@{-}[1,0]&&&&&\\
&&&&&&\ar@{-}`l/4pt [1,-1] [1,-1] \ar@{-}`r [1,1] [1,1]&&&&&\\
&&&&&\ar@{-}[3,1]&&\ar@{-}`d/4pt [1,2][1,2]&&&&\\
&&&&&&&&&&&\\
&&&&&&&&&\ar@{-}[2,2]&&\ar@{-}[2,-2]\\
&&&&\ar@{-}[1,1]+<-0.1pc,0.1pc> && \ar@{-}[2,-2]&&&&\ar@{-}[2,-2]\\
&&&&&&&&&&&\ar@{-}[5,3] &&&&\\
&&&&\ar@{-}[2,0]&&\ar@{-}[-1,-1]+<0.1pc,-0.1pc>\ar@{-}[2,2]&&\ar@{-}[2,-2]&&\\
&&&&&&&&&&\\
&&&&\ar@{-}[2,2]&&\ar@{-}[2,-2]&&\ar@{-}[1,1]+<0.2pc,0pc>&&\\
&&&&&&&&&\save[]+<0.2pc,0pc> \Drop{}\ar@{-}[1,0]+<0.2pc,0pc>\restore&&&\\
&&\ar@{-}[2,2]&&\ar@{-}[2,-2]&&\ar@{-}[1,0]&&&\save[]+<0.2pc,0pc> \Drop{}\ar@{-}`l/4pt [1,-1] [1,-1] \ar@{-}`r [1,2] [1,2]\restore&&&&&\ar@{-}[10,0]\\
&&&&&&&&\ar@{-}[1,0]&&&\ar@{-}[1,0]&&&\\
&&\ar@{-}[2,0]&&\ar@{-}[1,0]&&*+[o]+<0.37pc>[F]{\ov{\mathbf{u}}}\ar@{-}[1,0]+<0pc,0pc>&&\ar@{-}`l/4pt [1,-1] [1,-1] \ar@{-}`r [1,1] [1,1]&&&\ar@{-}[1,1]&&&\\
&&&&\ar@{-}`d/4pt [1,1] `[0,2] [0,2]&&&&&&&&\ar@{-}[7,0]&&\\
&&*+[o]+<0.37pc>[F]{\boldsymbol{\gamma}} \ar@{-}[7,0]+<0pc,0pc>  &&&\ar@{-}[1,0]&& *+[o]+<0.37pc>[F]{\ov{\mathbf{u}}}\ar@{-}[1,0]+<0pc,0pc>&& *+[o]+<0.40pc>[F]{\mathbf{u}}\ar@{-}[1,0]+<0pc,0pc>&&&&&\\
&&&&&\ar@{-}`d/4pt [1,1] `[0,2] [0,2] &&&&\ar@{-}[1,1]&&&&&\\
&&&&&&\ar@{-}[4,0]&&&&\ar@{-}[2,0]&&&&\\
&&&&&&\ar@{-}`r/4pt [1,2] [1,2]&&&&&&&&\\
&&&&&&&&\ar@{-}[1,1]+<-0.125pc,0.0pc> \ar@{-}[1,1]+<0.0pc,0.125pc>&&\ar@{-}[2,-2]&&&&\\
&&&&&&&&&&&&&&\\
&&\ar@{-}`r/4pt [1,2] [1,2]&&&&\ar@{-}`d/4pt [1,1] `[0,2] [0,2]&&&&\ar@{-}[-1,-1]+<0.125pc,0.0pc>\ar@{-}[-1,-1]+<0.0pc,-0.125pc> \ar@{-}[2,0]&&\ar@{-}[1,1]+<-0.125pc,0.0pc> \ar@{-}[1,1]+<0.0pc,0.125pc>&&\ar@{-}[2,-2]\\
&&&&\ar@{-}[1,1]&&&\ar@{-}[1,0]&&&&&&&\\
&&*+[o]+<0.37pc>[F]{\mu}\ar@{-}[2,0]+<0pc,0pc>&&&\ar@{-}[1,1]+<-0.125pc,0.0pc> \ar@{-}[1,1]+<0.0pc,0.125pc>&&\ar@{-}[2,-2]&&&\ar@{-}[1,1]+<-0.125pc,0.0pc> \ar@{-}[1,1]+<0.0pc,0.125pc>&&\ar@{-}[2,-2] &&\ar@{-}[-1,-1]+<0.125pc,0.0pc> \ar@{-}[-1,-1]+<0.0pc,-0.125pc>\ar@{-}[6,0]\\
&&&&&&&&&&&&&&\\
\ar@{-}`d/4pt [1,1] `[0,2] [0,2] &&&&&\ar@{-}[2,0]&&\ar@{-}[-1,-1]+<0.125pc,0.0pc> \ar@{-}[-1,-1]+<0.0pc,-0.125pc> \ar@{-}[1,1]&&&\ar@{-}[1,0]&& \ar@{-}[-1,-1]+<0.125pc,0.0pc> \ar@{-}[-1,-1]+<0.0pc,-0.125pc>\ar@{-}[4,0] &&\\
&\ar@{-}[1,0]&&&&&&&\ar@{-}[1,1]+<-0.125pc,0.0pc> \ar@{-}[1,1]+<0.0pc,0.125pc>&&\ar@{-}[2,-2]&&&&\\
&\ar@{-}`d/4pt [1,2] `[0,4] [0,4]&&&&&&&&&&&&&\\
&&&\ar@{-}[1,0]&&&&&\ar@{-}[1,0]&&\ar@{-}[1,0]\ar@{-}[-1,-1]+<0.125pc,0.0pc>\ar@{-}[-1,-1]+<0.0pc,-0.125pc>&&&&\\
&&&&&&&&&&&&&&
}}
\grow{\xymatrix@!0{
\\\\\\\\\\\\\\\\\\\\\\\\\\\\\\\\\\\\\\
\save\go+<0pt,0pt>\Drop{\txt{$=$}}\restore
}}
\grow{\xymatrix@!0{\\\\\\\\
\save\go+<0pt,2pt>\Drop{\circ}\ar@{-}[26,0]\restore && \save\go+<0pt,3pt>\Drop{E}\restore \ar@{-}[14,0] &&& \save\go+<0pt,3pt>\Drop{D}\restore \ar@{-}[2,2] && \save\go+<0pt,3pt>\Drop{C_1}\restore \ar@{-}[1,-1]+<0.125pc,0.0pc> \ar@{-}[1,-1]+<0.0pc,0.125pc> &&& \save\go+<0pt,3pt>\Drop{C_2}\restore \ar@{-}[1,0]&\\
&&&&&&&&&&\ar@{-}`l/4pt [1,-1] [1,-1] \ar@{-}`r [1,1] [1,1]&\\
&&&&&\ar@{-}[1,-1]\ar@{-}[-1,1]+<-0.125pc,0.0pc> \ar@{-}[-1,1]+<0.0pc,-0.125pc>&&\ar@{-}[2,2] &&\ar@{-}[1,-1]+<0.125pc,0.0pc> \ar@{-}[1,-1]+<0.0pc,0.125pc>&&\ar@{-}[7,0]\\
&&&&\ar@{-}[5,0]&&&&&&&\\
&&&&&&&\ar@{-}[1,-1]\ar@{-}[-1,1]+<-0.125pc,0.0pc> \ar@{-}[-1,1]+<0.0pc,-0.125pc>&&\ar@{-}[5,0]&&\\
&&&&&&\ar@{-}[1,0]&&&&&\\
&&&&&&\ar@{-}`l/4pt [1,-1] [1,-1] \ar@{-}`r [1,1] [1,1]&&&&&\\
&&&&&\ar@{-}[2,2]&&\ar@{-}`d/4pt [1,2][1,2]&&&&\\
&&&&\ar@{-}[1,1]&&&&&&&&\\
&&&&&\ar@{-}[1,1]+<-0.1pc,0.1pc> && \ar@{-}[2,-2]&&\ar@{-}[5,5]&&\ar@{-}[2,-2]\\
&&&&&&&&&&&\\
&&&&&\ar@{-}[2,0]&&\ar@{-}[-1,-1]+<0.1pc,-0.1pc>\ar@{-}[2,2]&&\ar@{-}[2,-2]&&\\
&&&&&&&&&&&&\\
&&&&&\ar@{-}[2,2]&&\ar@{-}[2,-2]&&\ar@{-}[2,2]&&&&\\
&&\ar@{-}[3,3]&&&&&&&&&&&&\ar@{-}[9,0]\\
&&&&&\ar@{-}[2,-2]&&\ar@{-}[1,0]&&&&\ar@{-}[1,0]&&&\\
&&&&&&&&&&& \ar@{-}`l/4pt [1,-1] [1,-1] \ar@{-}`r [1,1] [1,1]&&&\\
&&&\ar@{-}[1,-1]&&\ar@{-}[1,0]&&*+[o]+<0.37pc>[F]{\ov{\mathbf{u}}}\ar@{-}[1,0]+<0pc,0pc>&&&&&\ar@{-}[6,0]&&\\
&&\ar@{-}[1,0]&&&\ar@{-}`d/4pt [1,1] `[0,2] [0,2]&&&&&*+[o]+<0.40pc>[F]{\mu}\ar@{-}[2,0]+<0pc,0pc> &&&&\\
&&&&&&\ar@{-}[4,0]&&&&&&&&\\
&&*+[o]+<0.37pc>[F]{\boldsymbol{\gamma}}\ar@{-}[3,0]+<0pc,0pc>&&&&&&&&&&&&\\
&&&&&& \ar@{-}`r/4pt [1,2] [1,2]&&&&*+[o]+<0.34pc>[F]{S}\ar@{-}[1,0]+<0pc,0pc>&&&&\\
&& \ar@{-}`r/4pt [1,2] [1,2]&&&&&&\ar@{-}`d/4pt [1,1] `[0,2] [0,2]&&&&&&\\
&&&&\ar@{-}[1,1]+<-0.125pc,0.0pc> \ar@{-}[1,1]+<0.0pc,0.125pc>&&\ar@{-}[2,-2]&&&\ar@{-}[1,0]&&& \ar@{-}[1,1]+<-0.125pc,0.0pc> \ar@{-}[1,1]+<0.0pc,0.125pc>&&\ar@{-}[2,-2]\\
&&*+[o]+<0.40pc>[F]{\mu}\ar@{-}[2,0]+<0pc,0pc>&&&&&&&\ar@{-}[1,1]&&&&&\\
&&&&\ar@{-}[3,0]&&\ar@{-}[-1,-1]+<0.125pc,0.0pc>\ar@{-}[-1,-1]+<0.0pc,-0.125pc>\ar@{-}[2,2]&&&& \ar@{-}[1,1]+<-0.125pc,0.0pc> \ar@{-}[1,1]+<0.0pc,0.125pc>&&\ar@{-}[2,-2]&& \ar@{-}[-1,-1]+<0.125pc,0.0pc>\ar@{-}[-1,-1]+<0.0pc,-0.125pc>\ar@{-}[5,0]\\
\ar@{-}`d/4pt [1,1] `[0,2] [0,2]&&&&&&&&&&&&&&\\
&\ar@{-}[1,0]&&&&&&&\ar@{-}[1,1]+<-0.125pc,0.0pc> \ar@{-}[1,1]+<0.0pc,0.125pc>&&\ar@{-}[2,-2]&& \ar@{-}[-1,-1]+<0.125pc,0.0pc>\ar@{-}[-1,-1]+<0.0pc,-0.125pc>\ar@{-}[3,0]&&\\
&\ar@{-}`d/4pt [1,1]+<0.2pc,0pc> `[0,3] [0,3]&&&&&&&&&&&&&\\
&&\save[]+<0.2pc,0pc> \Drop{}\ar@{-}[1,0]+<0.2pc,0pc>\restore&&&&&&\ar@{-}[1,0]&& \ar@{-}[-1,-1]+<0.125pc,0.0pc>\ar@{-}[-1,-1]+<0.0pc,-0.125pc>\ar@{-}[1,0]&&&&\\
&&\save[]+<0.2pc,0pc> \Drop{}\restore&&&&&&&&&&&&
}}\\
\grow{\xymatrix@!0{
\\\\\\\\\\\\\\\\\\\\\\\\\\\\\\\\\\\\\\
\save\go+<0pt,0pt>\Drop{\txt{$=$}}\restore
}} &
\grow{\xymatrix@!0{\\\\\\\\
\save\go+<0pt,2pt>\Drop{\circ}\ar@{-}[25,0]\restore && \save\go+<0pt,3pt>\Drop{E}\restore \ar@{-}[14,0] &&& \save\go+<0pt,3pt>\Drop{D}\restore \ar@{-}[2,2] && \save\go+<0pt,3pt>\Drop{C_1}\restore \ar@{-}[1,-1]+<0.125pc,0.0pc> \ar@{-}[1,-1]+<0.0pc,0.125pc> &&& \save\go+<0pt,3pt>\Drop{C_2}\restore \ar@{-}[1,0]&\\
&&&&&&&&&&\ar@{-}`l/4pt [1,-1] [1,-1] \ar@{-}`r [1,1] [1,1]&\\
&&&&&\ar@{-}[1,-1]\ar@{-}[-1,1]+<-0.125pc,0.0pc> \ar@{-}[-1,1]+<0.0pc,-0.125pc>&&\ar@{-}[2,2] &&\ar@{-}[1,-1]+<0.125pc,0.0pc> \ar@{-}[1,-1]+<0.0pc,0.125pc>&&\ar@{-}[7,0]\\
&&&&\ar@{-}[5,0]&&&&&&&\\
&&&&&&&\ar@{-}[1,-1]\ar@{-}[-1,1]+<-0.125pc,0.0pc> \ar@{-}[-1,1]+<0.0pc,-0.125pc>&&\ar@{-}[5,0]&&\\
&&&&&&\ar@{-}[1,0]&&&&&\\
&&&&&&\ar@{-}`l/4pt [1,-1] [1,-1] \ar@{-}`r [1,1] [1,1]&&&&&\\
&&&&&\ar@{-}[2,2]&&\ar@{-}`d/4pt [1,2][1,2]&&&&\\
&&&&\ar@{-}[1,1]&&&&&&&\\
&&&&&\ar@{-}[1,1]+<-0.1pc,0.1pc> && \ar@{-}[2,-2]&&\ar@{-}[5,5]&&\ar@{-}[2,-2]\\
&&&&&&&&&&&\\
&&&&&\ar@{-}[2,0]&&\ar@{-}[-1,-1]+<0.1pc,-0.1pc>\ar@{-}[2,2]&&\ar@{-}[2,-2]&&\\
&&&&&&&&&&&&\\
&&&&&\ar@{-}[2,2]&&\ar@{-}[2,-2]&&\ar@{-}[2,2]&&&&\\
&&\ar@{-}[3,3]&&&&&&&&&&&&\ar@{-}[12,0]\\
&&&&&\ar@{-}[2,-2]&&\ar@{-}[2,0]&&&&\ar@{-}[2,0]&&&\\
&&&&&&&&&&&&&&\\
&&&\ar@{-}[5,0]&&\ar@{-}[5,0]&&&&&& \ar@{-}`l/4pt [1,-1] [1,-1] \ar@{-}`r [1,1] [1,1]&&&\\
&&&&&&&*+[o]+<0.40pc>[F]{\ov{\mathbf{u}}}\ar@{-}[4,0]+<0pc,0pc>&&&&&\ar@{-}[8,0]&&\\
&&&&&&&&&&*+[o]+<0.40pc>[F]{\mu}\ar@{-}[2,0]+<0pc,0pc>&&&&\\
&&&&&&&&&&&&&&\\
&&&&&&&&&&&&&&\\
&&&\ar@{-}`l/4pt [1,-1] [1,-1] \ar@{-}`r [1,1] [1,1]&&\ar@{-}`d/4pt [1,1] `[0,2] [0,2]&&&&&*+[o]+<0.34pc>[F]{S}\ar@{-}[3,0]+<0pc,0pc>&&&&\\
&&&&\ar@{-}[3,0]&&\ar@{-}[3,0]&&&&&&&&\\
&&*+[o]+<0.4pc>[F]{\mathbf{u}}\ar@{-}[1,0]+<0pc,0pc>&&&&\ar@{-}`r/4pt [1,2] [1,2]&&&&&&&&\\
\ar@{-}`d/4pt [1,1] `[0,2] [0,2]&&&&&&&&\ar@{-}`d/4pt [1,1] `[0,2] [0,2]&&&&&&\\
&\ar@{-}[2,0]&&&\ar@{-}[1,1]+<-0.125pc,0.0pc> \ar@{-}[1,1]+<0.0pc,0.125pc> &&\ar@{-}[2,-2]&&&\ar@{-}[1,0]&&&\ar@{-}[1,1]+<-0.125pc,0.0pc> \ar@{-}[1,1]+<0.0pc,0.125pc>&&\ar@{-}[2,-2]\\
&&&&&&&&&\ar@{-}[1,1]&&&&&\\
&\ar@{-}`d/4pt [1,1]+<0.2pc,0pc> `[0,3] [0,3]&&&&&\ar@{-}[-1,-1]+<0.125pc,0.0pc> \ar@{-}[-1,-1]+<0.0pc,-0.125pc>\ar@{-}[2,2] &&&&\ar@{-}[1,1]+<-0.125pc,0.0pc> \ar@{-}[1,1]+<0.0pc,0.125pc>&&\ar@{-}[2,-2]&& \ar@{-}[-1,-1]+<0.125pc,0.0pc>\ar@{-}[-1,-1]+<0.0pc,-0.125pc>\ar@{-}[4,0]\\
&&\save[]+<0.2pc,0pc> \Drop{}\ar@{-}[3,0]+<0.2pc,0pc>\restore&&&&&&&&&&&&\\
&&&&&&&&\ar@{-}[1,1]+<-0.125pc,0.0pc> \ar@{-}[1,1]+<0.0pc,0.125pc>&&\ar@{-}[2,-2] &&\ar@{-}[-1,-1]+<0.125pc,0.0pc>\ar@{-}[-1,-1]+<0.0pc,-0.125pc>\ar@{-}[2,0]&&\\
&&&&&&&&&&&&&&\\
&&\save[]+<0.2pc,0pc> \Drop{}\restore&&&&&&&&\ar@{-}[-1,-1]+<0.125pc,0.0pc> \ar@{-}[-1,-1]+<0.0pc,-0.125pc>&&&&
}}
\grow{\xymatrix@!0{
\\\\\\\\\\\\\\\\\\\\\\\\\\\\\\\\\\\\\\
\save\go+<0pt,0pt>\Drop{\txt{$=$}}\restore
}}
\grow{\xymatrix@!0{\\\\
&\save\go+<0pt,3pt>\Drop{E}\restore \ar@{-}[16,0]&&&& \save\go+<0pt,3pt>\Drop{D}\restore \ar@{-}[2,2] && \save\go+<0pt,3pt>\Drop{C_1}\restore \ar@{-}[1,-1]+<0.125pc,0.0pc> \ar@{-}[1,-1]+<0.0pc,0.125pc> &&& \save\go+<0pt,3pt>\Drop{C_2}\restore \ar@{-}[1,0]&\\
&&&&&&&&&&\ar@{-}`l/4pt [1,-1] [1,-1] \ar@{-}`r [1,1] [1,1]&\\
&&&&&\ar@{-}[1,-1]\ar@{-}[-1,1]+<-0.125pc,0.0pc> \ar@{-}[-1,1]+<0.0pc,-0.125pc>&&\ar@{-}[2,2] &&\ar@{-}[1,-1]+<0.125pc,0.0pc> \ar@{-}[1,-1]+<0.0pc,0.125pc>&&\ar@{-}[7,0]\\
&&&&\ar@{-}[5,0]&&&&&&&\\
&&&&&&&\ar@{-}[1,-1]\ar@{-}[-1,1]+<-0.125pc,0.0pc> \ar@{-}[-1,1]+<0.0pc,-0.125pc>&&\ar@{-}[5,0]&&\\
&&&&&&\ar@{-}[1,0]&&&&&\\
&&&&&&\ar@{-}`l/4pt [1,-1] [1,-1] \ar@{-}`r [1,1] [1,1]&&&&&\\
&&&&&\ar@{-}[2,2]&&\ar@{-}`d/4pt [1,2][1,2]&&&&\\
&&&&\ar@{-}[1,1]&&&&&&&\\
&&&&&\ar@{-}[1,1]+<-0.1pc,0.1pc> && \ar@{-}[2,-2]&&\ar@{-}[5,5]&&\ar@{-}[2,-2]\\
&&&&&&&&&&&&&&\\
&&&&&\ar@{-}[2,0]&&\ar@{-}[-1,-1]+<0.1pc,-0.1pc> \ar@{-}[4,4]&&\ar@{-}[2,-2]&&&&&\\
&&&&&&&&&&&&&&\\
&&&&&\ar@{-}[2,2]&&\ar@{-}[2,-2]&&&\\
&&&&&&&&&&&&&&\ar@{-}[14,0]  \\
&&&&&\ar@{-}[1,-2]&&\ar@{-}[2,0]&&&&\ar@{-}[1,0] &&&\\
&\ar@{-}[2,2]&&\ar@{-}[2,-2]&&&&&&&&\ar@{-}`l/4pt [1,-1] [1,-1] \ar@{-}`r [1,1] [1,1]&&&\\
&&&&&&&&&&\ar@{-}[1,0] && \ar@{-}[11,0]&&\\
&\ar@{-}[3,0]&&\ar@{-}[7,0]&&&&*+[o]+<0.40pc>[F]{\ov{\mathbf{u}}}\ar@{-}[5,0]+<0pc,0pc>&&&&&&&\\
&&&&&&&&&&*+[o]+<0.40pc>[F]{\mu}\ar@{-}[2,0]+<0pc,0pc>&&&&\\
&&&&&&&&&&&&&&\\
&\ar@{-}`l/4pt [1,-1] [1,-1] \ar@{-}`r [1,1] [1,1]&&&&&&&&&&&&&\\
\ar@{-}[2,0]&&\ar@{-}[5,0]&\ar@{-}`r/4pt [1,2] [1,2]&&&&\ar@{-}`r/4pt [1,2] [1,2]&&&*+[o]+<0.34pc>[F]{S}\ar@{-}[5,0]+<0pc,0pc> &&&&\\
&&&&&\ar@{-}[1,1]+<-0.125pc,0.0pc> \ar@{-}[1,1]+<0.0pc,0.125pc>&&\ar@{-}[2,-2]&&\ar@{-}[2,0]&&&&&\\
*+[o]+<0.4pc>[F]{\mathbf{u}}\ar@{-}[5,0]+<0pc,0pc>&&&&&&&&&&&&&&\\
&&&\ar@{-}`d/4pt [1,1] `[0,2] [0,2]&&&&\ar@{-}[-1,-1]+<0.125pc,0.0pc>\ar@{-}[-1,-1]+<0.0pc,-0.125pc> \ar@{-}`d/4pt [1,1] `[0,2] [0,2]&&&&&&&\\
&&&&\ar@{-}[1,0]&&&&\ar@{-}[1,0]&&&&&&\\
&&\ar@{-}[1,1]+<-0.125pc,0.0pc> \ar@{-}[1,1]+<0.0pc,0.125pc>&&\ar@{-}[2,-2]&&&&\ar@{-}`d/4pt [1,1] `[0,2] [0,2]&&&&&&\\
&&&&&&&&&\ar@{-}[1,0]&&&\ar@{-}[1,1]+<-0.125pc,0.0pc> \ar@{-}[1,1]+<0.0pc,0.125pc>&&\ar@{-}[2,-2]\\
\ar@{-}`d/4pt [1,1] `[0,2] [0,2]&&&&\ar@{-}[-1,-1]+<0.125pc,0.0pc>\ar@{-}[-1,-1]+<0.0pc,-0.125pc>  \ar@{-}[3,4]&&&&&\ar@{-}[1,1]&&&&&\\
&\ar@{-}[4,0]&&&&&&&&&    \ar@{-}[1,1]+<-0.125pc,0.0pc> \ar@{-}[1,1]+<0.0pc,0.125pc>&&\ar@{-}[2,-2]&& \ar@{-}[-1,-1]+<0.125pc,0.0pc>\ar@{-}[-1,-1]+<0.0pc,-0.125pc>\ar@{-}[4,0]\\
&&&&&&&&&&&&&&\\
&&&&&&&&\ar@{-}[1,1]+<-0.125pc,0.0pc> \ar@{-}[1,1]+<0.0pc,0.125pc>&&\ar@{-}[2,-2] &&\ar@{-}[-1,-1]+<0.125pc,0.0pc>\ar@{-}[-1,-1]+<0.0pc,-0.125pc>\ar@{-}[2,0]&&\\
&&&&&&&&&&&&&&\\
&&&&&&&&&&\ar@{-}[-1,-1]+<0.125pc,0.0pc>\ar@{-}[-1,-1]+<0.0pc,-0.125pc>&&&&
}}
\grow{\xymatrix@!0{
\\\\\\\\\\\\\\\\\\\\\\\\\\\\\\\\\\\\\\
\save\go+<0pt,0pt>\Drop{\txt{$=$}}\restore
}}
\grow{\xymatrix@!0{
&\save\go+<0pt,3pt>\Drop{E}\restore \ar@{-}[16,0] &&& \save\go+<0pt,3pt>\Drop{D}\restore \ar@{-}[2,2] &&  \save\go+<0pt,3pt>\Drop{C}\restore \ar@{-}[1,-1]+<0.125pc,0.0pc> \ar@{-}[1,-1]+<0.0pc,0.125pc> &&&\save\go+<0.2pc,3pt>\Drop{C'}\restore  \save[]+<0.2pc,0pc> \Drop{}\ar@{-}[1,0]+<0.2pc,0pc>\restore&&\\
&&&&&&&&& \save[]+<0.2pc,0pc> \Drop{}\ar@{-}`l/4pt [1,-1] [1,-1] \ar@{-}`r [1,2] [1,2]\restore&&\\
&&&&\ar@{-}[1,-1]\ar@{-}[-1,1]+<-0.125pc,0.0pc> \ar@{-}[-1,1]+<0.0pc,-0.125pc>&&\ar@{-}[2,2]&& \ar@{-}[1,-1]+<0.125pc,0.0pc> \ar@{-}[1,-1]+<0.0pc,0.125pc>&&& \ar@{-}[6,0] \\
&&&\ar@{-}[3,0]&&&&&&&&\\
&&&&&&\ar@{-}[-1,1]+<-0.125pc,0.0pc> \ar@{-}[-1,1]+<0.0pc,-0.125pc>\ar@{-}[1,0]&&\ar@{-}[1,1]&&&\\
&&&&&&\ar@{-}`l/4pt [1,-1] [1,-1] \ar@{-}`r [1,1] [1,1]&&&\ar@{-}[3,0]&&\\
&&&\ar@{-}[1,1]+<-0.1pc,0.1pc> && \ar@{-}[2,-2]&&\ar@{-}`d/4pt [1,2][1,2]&&&&\\
&&&&&&&&&&&\\
&&&\ar@{-}[2,0]&&\ar@{-}[-1,-1]+<0.1pc,-0.1pc>\ar@{-}[7,7]&&&&\ar@{-}[6,6]&&\ar@{-}[10,-10]\\
&&&&&&&&&&&&&&&\\
&&&\ar@{-}[5,5]&&&&&&&&&&&&\\
&&&&&&&&&&&&&&&\\
&&&&&&&&&&&&&&&\\
&&&&&&&&&&&&&&&\\
&&&&&&&&&&&&&&&\ar@{-}[17,0]\\
&&&&&&&&\ar@{-}[1,0]&&&&\ar@{-}[1,0]&&&\\
&\ar@{-}[2,2]&&&&&&&\ar@{-}`l/4pt [1,-1] [1,-1] \ar@{-}`r [1,1] [1,1]&&&&\ar@{-}`l/4pt [1,-1] [1,-1] \ar@{-}`r [1,1] [1,1]&&&\\
&&&&&&&\ar@{-}[1,0]+<0pc,-3pt>&&\ar@{-}[1,0]+<0pc,-3pt>&&\ar@{-}[1,0]+<0pc,-0.5pt>&&\ar@{-}[14,0]&&\\
&\ar@{-}[1,0]&&\ar@{-}[11,0]&&&&&&&&&&&&\\
&\ar@{-}`l/4pt [1,-1] [1,-1] \ar@{-}`r [1,1] [1,1]&&&&&&&*+<0.1pc>[F]{\,\,\,g\,\,\,} \ar@{-}@<0.40pc>[2,0]+<0pc,-0.5pt> \ar@{-}@<-0.40pc>[2,0]+<0pc,-0.5pt>&&& *+[o]+<0.40pc>[F]{\mu}\ar@{-}[2,0]+<0pc,-0.5pt> &&&&\\
\ar@{-}[1,0]+<0pc,0pc>&&\ar@{-}[11,0]&&&&&&&&&&&&&\\
&&&&&&&&&&&&&&&\\
*+[o]+<0.4pc>[F]{\mathbf{u}}\ar@{-}[11,0]+<0pc,0pc>&&&&&&&*+[o]+<0.37pc>[F]{\ov{\boldsymbol{\gamma}}} \ar@{-}[2,0]+<0pc,-0.5pt>&&*+[o]+<0.37pc>[F]{\mathbf{S}}\ar@{-}[2,0]+<0pc,-0.5pt>&& *+[o]+<0.37pc>[F]{S}\ar@{-}[9,0]&&&&\\
&&&&&&&&&&&&&&&\\
&&&&&&&&&&&&&&&\\
&&&&&&&*+[o]+<0.40pc>[F]{\mu}\ar@{-}[2,0]+<0pc,0pc>&&*+[o]+<0.40pc>[F]{\mu}\ar@{-}[4,0]+<0pc,0pc>&&&&&&\\
&&&\ar@{-}`r/4pt [1,2] [1,2]&&&&&&&&&&&&\\
&&&&&\ar@{-}[1,1]+<-0.125pc,0.0pc> \ar@{-}[1,1]+<0.0pc,0.125pc>&&\ar@{-}[2,-2]&&&&&&&&\\
&&&&&&&&&&&&&&&\\
&&&\ar@{-}`d/4pt [1,1] `[0,2] [0,2]&&&&\ar@{-}[-1,-1]+<0.125pc,0.0pc>\ar@{-}[-1,-1]+<0.0pc,-0.125pc> \ar@{-}`d/4pt [1,1] `[0,2] [0,2]&&&&&&&&\\
&&&&\ar@{-}[1,0]&&&&\ar@{-}[1,0]&&&&&&&\\
&&\ar@{-}[1,1]+<-0.125pc,0.0pc> \ar@{-}[1,1]+<0.0pc,0.125pc>&&\ar@{-}[2,-2]&&&&\ar@{-}`d/4pt [1,1]+<0.2pc,0pc> `[0,3] [0,3]&&&&&\ar@{-}[1,1]+<-0.125pc,0.0pc> \ar@{-}[1,1]+<0.0pc,0.125pc>&&\ar@{-}[2,-2]\\
&&&&&&&&&\save[]+<0.2pc,0pc> \Drop{}\ar@{-}[1,0]+<0.2pc,0pc>\restore &&&&&&\\
\ar@{-}`d/4pt [1,1] `[0,2] [0,2]&&&&\ar@{-}[3,4]\ar@{-}[-1,-1]+<0.125pc,0.0pc>\ar@{-}[-1,-1]+<0.0pc,-0.125pc> &&&&& \save[]+<0.2pc,0pc> \Drop{}\ar@{-}[1,1]\restore&&&&\ar@{-}[1,-1]&& \ar@{-}[-1,-1]+<0.125pc,0.0pc> \ar@{-}[-1,-1]+<0.0pc,-0.125pc>\ar@{-}[6,0]\\
&\ar@{-}[5,0]&&&&&&&&&\ar@{-}[1,1]+<-0.125pc,0.0pc> \ar@{-}[1,1]+<0.0pc,0.125pc>&&\ar@{-}[2,-2]&&&\\
&&&&&&&&&&&&&&&\\
&&&&&&&&\ar@{-}[1,1]+<-0.125pc,0.0pc> \ar@{-}[1,1]+<0.0pc,0.125pc>&&\ar@{-}[2,-2]&& \ar@{-}[-1,-1]+<0.125pc,0.0pc>\ar@{-}[-1,-1]+<0.0pc,-0.125pc>\ar@{-}[3,0]&&&\\
&&&&&&&&&&&&&&&\\
&&&&&&&&\ar@{-}[1,0]&&\ar@{-}[-1,-1]+<0.125pc,0.0pc>\ar@{-}[-1,-1]+<0.0pc,-0.125pc>\ar@{-}[1,0]&&&&&\\
&&&&&&&&&&&&&&&}}
\grow{\xymatrix@!0{
\\\\\\\\\\\\\\\\\\\\\\\\\\\\\\\\\\\\\\
\save\go+<0pt,0pt>\Drop{\txt{,}}\restore
}}
\end{align*}
as desired.\qed

\end{document}